\PassOptionsToPackage{german,english}{babel}
\documentclass[twoside,openright,titlepage,fleqn,headinclude,
                11pt,a4paper,BCOR5mm,footinclude,abstractoff 
                ]{scrreprt}
\newcommand{\myTitle}{{\LARGE A}SYMPTOTIC {\LARGE R}ESULTS FOR {\LARGE R}EPRESENTATIONS OF {\LARGE F}INITE {\LARGE G}ROUPS\xspace}
\newcommand{\myDegree}{Erlangung der naturwissenschaftlichen Doktorw\"urde\\
(Dr.\ sc.\ nat.)\xspace}
\newcommand{\myName}{Dario De Stavola\xspace}

\newcommand{\myFaculty}{Mathematisch-naturwissenschaftlichen Fakult\"at\xspace}

\newcommand{\myUni}{\protect{Universit\"at Z\"urich}\xspace}
\newcommand{\myLocation}{Z\"urich\xspace}
\newcommand{\myTime}{2017\xspace}
\usepackage[latin1]{inputenc} 
\usepackage[ngerman,american]{babel}           
\usepackage[square,numbers]{natbib} 
\usepackage[fleqn]{amsmath} 
\usepackage{classicthesis-ldpkg} 
\usepackage[eulerchapternumbers,
            subfig,beramono,eulermath,parts]{classicthesis}
\hypersetup{%
    colorlinks=true, linktocpage=true, pdfstartpage=3, pdfstartview=FitV,%
    breaklinks=true, pdfpagemode=UseNone, pageanchor=true, pdfpagemode=UseOutlines,%
    plainpages=false, bookmarksnumbered, bookmarksopen=true, bookmarksopenlevel=1,%
    hypertexnames=true, pdfhighlight=/O,
    urlcolor=Black, linkcolor=Black, citecolor=Black, 
    pdftitle={\myTitle},%
    pdfauthor={\textcopyright\ \myName, \myUni, \myFaculty},%
    pdfsubject={},%
    pdfkeywords={},%
    pdfcreator={pdfLaTeX},%
    pdfproducer={LaTeX with hyperref and classicthesis}%
}

\usepackage{graphicx,color}

\usepackage{amsmath,amsfonts,amssymb,amsthm,mathrsfs,amsopn,tikz,todonotes,float,mathtools}
\usepackage{latexsym}
\usepackage[numbers]{natbib}
\usepackage[enableskew]{youngtab}
\usepackage[left=2.5cm,top=2.5cm,bottom=2.5cm,right=2.5cm]{geometry}
\usepackage{hyperref}
\usepackage[all,cmtip]{xy}
\usepackage{bbm}

\usepackage{latexsym}
\usepackage{hyperref}
\DeclareMathAlphabet{\mathbbold}{U}{bbold}{m}{n}

\newtheorem{theorem}{Theorem}[chapter]
\newtheorem{lemma}[theorem]{Lemma}
\newtheorem{proposition}[theorem]{Proposition}
\newtheorem{corollary}[theorem]{Corollary}
\newtheorem{conjecture}[theorem]{Conjecture}

\theoremstyle{remark}
\newtheorem{proofpart}{Part}
\newtheorem{example}{Example}[chapter]
\theoremstyle{definition}
\newtheorem{defi}[theorem]{Definition}

\theoremstyle{remark}
\newtheorem{remark}[theorem]{Remark}

\numberwithin{section}{chapter}
\numberwithin{equation}{chapter}

\DeclareMathOperator{\rk}{rk}
\DeclareMathOperator{\Aut}{Aut}
\DeclareMathOperator{\Ind}{Ind}
\DeclareMathOperator{\Mat}{Mat}
\DeclareMathOperator{\cat}{Cat}
\DeclareMathOperator{\id}{Id}
\DeclareMathOperator{\col}{col}

\DeclareMathOperator{\supp}{Supp}
\DeclareMathOperator{\SYT}{SYT}
\DeclareMathOperator{\GL}{GL}
\DeclareMathOperator{\reg}{Reg}

\DeclareMathOperator{\Irr}{Irr}
\DeclareMathOperator{\Pl}{Pl}
\DeclareMathOperator{\SPl}{SPl}
\DeclareMathOperator{\sch}{sch}
\DeclareMathOperator{\scl}{scl}

\DeclareMathOperator{\sh}{sh}
\DeclareMathOperator{\crs}{crs}
\DeclareMathOperator{\nst}{nst}
\DeclareMathOperator{\SInd}{SInd}
\DeclareMathOperator{\Res}{Res}
\DeclareMathOperator{\Sing}{Sing}
\DeclareMathOperator{\Leb}{Leb}
\DeclareMathOperator{\PGL}{PGL}
\DeclareMathOperator{\proj}{proj}
\DeclareMathOperator{\sgn}{sgn}
\DeclareMathOperator{\arm}{arm}
\DeclareMathOperator{\leg}{leg}
\DeclareMathOperator{\ctr}{ctr}
\DeclareMathOperator{\tr}{tr}

\DeclareMathOperator{\Cl}{Cl}
\DeclareMathOperator{\ch}{ch}
\newcommand{\ten}{$10$}
\newcommand{\eleven}{$11$}
\newcommand{\twelve}{$12$}
\newcommand{\thirteen}{$13$}
\newcommand{\fourteen}{$14$}
\newcommand{\fifteen}{$15$}
\newcommand{\sixteen}{$16$}
\newcommand{\seventeen}{$17$}
\newcommand{\eighteen}{$18$}
\newcommand{\nineteen}{$19$}
\newcommand{\twenty}{$20$}
\newcommand{\minusone}{$-1$}
\newcommand{\minustwo}{$-2$}
\newcommand{\minusthree}{$-3$}
\newcommand{\minusfour}{$-4$}
\newcommand{\smallo}{o}
\newcommand{\stat}{stat}
\newcommand{\provauno}{\begin{array}{c}\yng(1)\end{array}}
\newcommand{\provaduea}{\begin{array}{c}\yng(2)\end{array}}
\newcommand{\provadueb}{\begin{array}{c}\yng(1,1)\end{array}}
\newcommand{\provatrea}{\begin{array}{c}\yng(3)\end{array}}
\newcommand{\provatreb}{\begin{array}{c}\yng(2,1)\end{array}}
\newcommand{\provatrec}{\begin{array}{c}\yng(1,1,1)\end{array}}
\newcommand{\provaquattroa}{\begin{array}{c}\yng(4)\end{array}}
\newcommand{\provaquattrob}{\begin{array}{c}\yng(3,1)\end{array}}
\newcommand{\provaquattroc}{\begin{array}{c}\yng(2,2)\end{array}}
\newcommand{\provaquattrod}{\begin{array}{c}\yng(2,1,1)\end{array}}
\newcommand{\provaquattroe}{\begin{array}{c}\yng(1,1,1,1)\end{array}}
\newcommand{\vuoto}{ }

\newcommand{\R}{\mathbb{R}}
\newcommand{\C}{\mathbb{C}}
\newcommand{\Z}{\mathbb{Z}}
\newcommand{\N}{\mathbb{N}}
\newcommand{\K}{\mathcal{K}}
\newcommand{\q}{\mathbf{q}}
\newcommand{\p}{\mathbf{p}}
\newcommand{\Y}{\mathbb{Y}}
\newcommand{\wt}{wt}
\newcommand\numberthis{\addtocounter{equation}{1}\tag{\theequation}}
\usepackage{a4wide}
\usepackage{nopageno}

\begin{document}
\frenchspacing
\raggedbottom
\selectlanguage{american}
\pagenumbering{roman}
\pagestyle{empty}

\begin{titlepage}
\phantom{.}
\vspace{1cm}
    \begin{center}
        \large  
        \begingroup
            \color{Maroon}\spacedallcaps{\myTitle} \\ \bigskip
        \endgroup

\vspace{2cm}

	Dissertation\\
\medskip
	zur\\
\medskip
        \myDegree \\
\medskip
	vorgelegt der\\
\medskip
        \myFaculty \\
\medskip
	der\\
\medskip
        \myUni \\ \bigskip\bigskip\bigskip
\medskip
	von\\
\medskip
	\spacedlowsmallcaps{\myName}\\
\medskip
aus\\
\medskip
Italien

\vspace{2cm}

Promotionskommission\\
\medskip
Prof.\ Dr.\ Valentin F\'eray (Vorsitz)\\
Prof.\ Dr.\ Benjamin Schlein (Leitung der Dissertation)\\
Prof.\ Dr.\ Ashkan Nikeghbali\\
\vspace{1cm}
        \myLocation, \myTime

        \vfill                      

    \end{center}        
\end{titlepage}

\cleardoublepage
\pdfbookmark[1]{Abstract}{Abstract}
\chapter*{Abstract}
Representation theory of finite groups portrays a marvelous crossroad of group theory, algebraic combinatorics, and probability. In particular a probability, called the Plancherel measure, arises naturally from representation theory, and in this thesis we study asymptotic questions for random Plancherel distributed representations. We consider two types of problems:
\begin{enumerate}
 \item Irreducible representations of the symmetric group:
 
 The irreducible characters of the symmetric group, which are indexed by integer partitions, have been widely studied since the foundation of representation theory in the late nineteen century. In the last thirty years the community has moved towards what is known as the ``dual approach'', in which characters $\chi^{\lambda}_{\rho}$ are considered functions of $\lambda$, with parameter $\rho$, rather than the other way around. This method is particularly effective when considering asymptotic questions. 

  We extend a celebrated result of Kerov on the asymptotic of Plancherel distributed characters by studying partial trace and partial sum of a representation matrix. We decompose each of these objects into a main term and a remainder, and in each case we prove a central limit theorem for the main term. We apply these results to prove a law of large numbers for the partial sum. Our main tool is the expansion of symmetric functions evaluated on Jucys-Murphy elements. 
\item Variations of representation theory:

The theory of projective representations is almost as old as the classical one, but received much less coverage. Only recently the interest sparked again, in particular in the symmetric group case. Here the projective irreducible representations are indexed by strict partitions or, equivalently, by shifted Young diagrams. The combinatorics of projective representations seem to share interesting resemblances with their traditional counterpart. We convert a multirectangular description of classical Young diagrams which was developed by Stanley to shifted Young diagrams. We obtain results in the strict partitions case which are similar to Stanley's in the traditional case. Notably, the renormalized character expressed as a function of the multirectangular coordinates is a polynomial whose coefficients seem to exhibit a positivity property. We prove that this is true for the leading term of this polynomial, and pose the general case as a conjecture. We also examine a limit shape result proved by Ivanov on the asymptotic of Plancherel distributed strict partitions, and we extend it to the uniform topology.

In a different subject, we consider the problem of the representations of the upper unitriangular group. The study of these representations has proven to be difficult, if not intractable. There is an alternative theory, called supercharacter theory, which is nowadays considered to be a valid surrogate. Later on Diaconis and Isaacs generalized this approach for general finite groups. We show that a generalization of the Plancherel measure, called the superplancherel measure, can be defined for any supercharacter theory, and maintains many useful properties. We study this measure for a particular supercharacter theory of the upper unitriangular group, in which supercharacters are indexed by set partitions. We prove a limit shape result for a random set partition according to this distribution. We also give a description of the asymptotical behavior of two set partition statistics related to the supercharacters. The study of these statistics when the set partitions are uniformly distributed has been done by Chern, Diaconis, Kane and Rhoades.

\end{enumerate}

\begin{otherlanguage}{german}
\cleardoublepage\pdfbookmark[1]{Zusammenfassung}{Zusammenfassung}
\chapter*{Zusammenfassung}

Die Darstellungstheorie endlicher Gruppen zeigt eine wunderbare Kreuzung von Gruppentheorie, algebraischer Kombinatorik und Wahrscheinlichkeitstheorie. Insbesondere ergibt sich eine nat{\"u}rliche Wahrscheinlichkeit, die als Plancherel-Mass bezeichnet wird, aus der Darstellungstheorie. In dieser Arbeit werden asymptotische Fragen f{\"u}r zuf{\"a}llige Plancherel-verteilte Darstellungen untersucht. Wir betrachten zwei Arten von Problemen:
\begin{enumerate}
 \item Irreduzible Darstellungen der symmetrischen Gruppe:
 
 Die irreduziblen Charaktere der symmetrischen Gruppe, die durch Partitionen von ganzen Zahlen indiziert sind, wurden seit der Gr{\"u}ndung der Darstellungstheorie im sp{\"a}ten neunzehnten Jahrhundert ausgiebig untersucht. In den letzten dreissig Jahren hat sich die Gemeinschaft dem so genannten dualen Ansatz zugewandt, in welchem die Zeichen $\chi^{\lambda}_{\rho}$ als Funktionen von $\lambda$ gesehen werden, mit Parameter $\rho$, und nicht umgekehrt. Diese Methode ist besonders effektiv f{\"u}r asymptotische Fragen.
 
  Wir erweitern ein ber{\"u}hmtes Resultat von Kerov, welches das asymptotische Verhalten von Plancherel-verteilten Charakteren beschreibt, indem wir die Teilspur und die Teilsumme einer Darstellungsmatrix untersuchen. Wir zerlegen jedes dieser Objekte in einen Hauptterm und einen Rest, und in beiden F{\"a}llen beweisen wir ein zentraler Grenzwertsatz f{\"u}r den Hauptterm. Wir wenden diese Resultate an, um ein Gesetz der grossen Zahlen f{\"u}r die Teilsumme zu beweisen. Unser Hauptwerkzeug ist die Erweiterung symmetrischer Funktionen, die an Jucys-Murphy-Elementen ausgewertet werden.
\item Variationen der Darstellungstheorie:

Die Theorie der projektiven Darstellungen ist fast so alt wie die klassische Darstellungstheorie, aber hat viel weniger Beachtung gefunden. Erst k{\"u}rzlich flammte das Interesse erneut auf, insbesondere im Falle symmetrischer Gruppen. Hier werden die projektiven irreduziblen Darstellungen durch strikt Partitionen oder, {\"a}quivalent, durch verschobene Young-Diagramme indiziert. Die Kombinatorik der projektiven Darstellungen scheint interessante {\"A}hnlichkeiten mit ihrem traditionellen Gegenst{\"u}ck zu haben. Wir konvertieren eine multirechteckige Beschreibung klassischer Young-Diagramme, die von Stanley entwickelt wurde, zu verschobenen Young-Diagrammen. Wir erhalten Resultate f{\"u}r strikt Partitionen, die Stanley's Resultaten f{\"u}r traditionelle Partitionen {\"a}hneln. Bemerkenswerterweise ist der normierte Charakter, ausgedr{\"u}ckt in Funktion der multirechteckigen Koordinaten, ein Polynom, dessen Koeffizienten eine Positivit{\"a}tseigenschaft zu zeigen scheinen. Wir beweisen, dass dies f{\"u}r den Leitkoeffizient dieses Polynoms gilt, und stellen den allgemeinen Fall als Vermutung auf. Wir untersuchen auch ein von Ivanov bewiesenes Grenzform-Ergebnis {\"u}ber das asymptotische Verhalten von Plancherel-Mass verteilten strikt Partitionen und erweitern es auf die uniforme Topologie.

Zu einem anderen Thema betrachten wir das Problem der Darstellungen der Gruppe der oberen unipotenten Dreiecksmatrizen. Das Studium dieser Darstellungen hat sich als schwierig, wenn nicht gar unl{\"o}sbar erwiesen. Es gibt eine alternative Theorie, die sogenannte Supercharakter-Theorie, die heutzutage als ein g{\"u}ltiger Ersatz angesehen wird. Sp{\"a}ter verallgemeinerten Diaconis und Isaacs diesen Ansatz zu allgemeinen endlichen Gruppen. Wir zeigen, dass eine Verallgemeinerung des Plancherel-Masses, das sogenannte Superplancherel-Mass, f{\"u}r jede Supercharakter-Theorie definiert werden kann und viele n{\"u}tzliche Eigenschaften beibeh{\"a}lt. Wir untersuchen dieses Mass f{\"u}r eine bestimmte Supercharakter-Theorie der Gruppe der oberen unipotenten Dreiecksmatrizen, in der Supercharaktere durch Mengenpartitionen indiziert werden. Wir zeigen ein Grenzformergebnis f{\"u}r eine zuf{\"a}llige Mengenpartition nach dieser Verteilung. Wir beschreiben auch das asymptotische Verhalten zweier Mengenpartitionstatistiken in Bezug auf die Supercharaktere. Die Untersuchung dieser Statistiken, wenn die Mengenpartitionen gleichm{\"a}ssig verteilt sind, wurde von Chern, Diaconis, Kane und Rhoades durchgef{\"u}hrt.
\end{enumerate}

\end{otherlanguage}
\cleardoublepage
\pdfbookmark[1]{Acknowledgments}{acknowledgments}

\begin{flushright}{\slshape    
   Remember to put a deep motivational quote in the acknowledgements. \medskip
    --- Federico L. G. Barco, \textit{Re:Re:Fwd:Re:Fwd:Re:Re: epic fails video 2012}}
\end{flushright}

\bigskip

\begingroup
\let\clearpage\relax
\let\cleardoublepage\relax
\let\cleardoublepage\relax
\chapter*{Acknowledgments}
~
\\

First and foremost my deepest gratitude goes to Prof. Dr. Valentin F{\'e}ray. His dedication lead me through these years like the Levant on a kite. A kite that seemed on many occasions eager to head towards the trees but was saved by Valentin's assistance and attention to\linebreak the details.

To Prof. Hora and Prof. Lecouvey many thanks for their detailed reviews and useful observations. My appreciation goes also to Prof. Matsumoto for our fruitful discussions.

In this acknowledgements, and in my heart, a special place is reserved for my research line-up, for providing assistance or distraction: Mathilde, Per, Marko, Jehanne, Raul, and our\linebreak pleasant recent additions Jacopo and Benedikt; I would like to thank in particular Helen for\linebreak helping me and prompting me to my best without (almost) actually making the effort.\linebreak
Deep appreciation goes also to Antonio for pushing me to work more, Violetta for pushing me\linebreak (efficiently) to work less, Yannick for pushing me, Michel, Candia, Paola, and all other\linebreak secretaries, professors, post docs, PhDs and students who contributed to the beauty\linebreak of these years. Thanks to the Oldskool for all the unihockey, and the Marcelians for\linebreak nerding with me.

Of all the people who helped me out throughout my life, a special recognition is\linebreak necessary to the friends who played a part in making me who I am today. Thanks to the\linebreak entire team of Firelions, and in particular to Gian Paolo and Alberta: you may be last in the\linebreak league, but you will always be first to me. Thank you Fede for being my best friend, your vicinity encourages me always to be a better person. Thank you ``Virginians'' for being a unique source of fun. Thank you Wale for being always close to my heart despite being\linebreak often too far for my taste. 

Thank you, finally, mum and dad, and thank you Marco for always supporting me via your\linebreak outstanding love. You are the best one could ask for.

\endgroup
\thispagestyle{scrheadings}
\cleardoublepage
\refstepcounter{dummy}
\pdfbookmark[1]{\contentsname}{tableofcontents}
\setcounter{tocdepth}{1}
\tableofcontents 


    \refstepcounter{dummy}
    \addcontentsline{toc}{chapter}{\listfigurename}
    \listoffigures

    \vspace*{8ex}

\chapter*{List of the main symbols}
\addcontentsline{toc}{chapter}{List of Symbols}

\begin{tabular}{llr}
 $\mathbb{C}[G]$	& Group algebra of the finite group $G$ &	p 13\\
 $Z(\C[G])$	& Center of the group algebra &			p 14\\
 $\chi^{\lambda}(\sigma)$	&Character indexed by $\lambda$ and computed in $\sigma$	&p 14\\
 $\Mat_d(\C)$	& Group of matrices $d\times d$ with entries in $\C$& p 14\\
 $\Irr(G)$ 	&Set of irreducible characters of $G$	&p 15\\
 $P_{Pl}$	&Plancherel measure	& p 16\\
 $\mathcal{CF}(G)$	&Algebra of class functions of $G$	&p 16\\
 $\Res^G_H(\chi)$	&Restriction to $H\leq G$ of the function $\chi$	&p 16\\
  $\Ind^L_G(\chi)$	&Induction to $L\geq G$ of the function $\chi$	&p 16\\
  $l(\lambda)$		&Length of the partition $\lambda$	&p 17\\
  $|\lambda|$	&Size of the partition $\lambda$	&p 17\\
  $m_i(\lambda)$	&Multiplicity of $i$ in the partition $\lambda$	&p 17\\
  $K_{\rho}$		&Set of permutations of cycle type $\rho$	&p 17\\
  $z_{\rho}$		&Size of the centralizer of a permutation of cycle type $\rho$	&p 17\\
  $\mu\nearrow\lambda$	&The partition $\mu$ has $|\mu|=|\lambda|-1$ and $\mu_i\leq\lambda_i$ for all $i$	&p 18\\
  $c(\Box)$		&Content of $\Box$		&p 18\\
  $\mathcal{P}(n)$	&Set of partitions of $n$	&p 19\\
  $\pi^{\lambda}$	&\begin{tabular}{l} Irreducible representation of the symmetric group $S_{|\lambda|}$\\ indexed by $\lambda$\end{tabular}&p 19\\
  $\dim\lambda$		&Dimension of the irreducible representation indexed by $\lambda$		&p 19\\
  $\SYT(\lambda)$	&Set of standard Young tableaux of shape $\lambda$		&p 19\\
  $d_k(T)$		&\begin{tabular}{l} Signed distance between $k$ and $k+1$ in the standard\\ Young tableaux $T$\end{tabular}&p 20\\
  $\Lambda$		&Ring of symmetric functions					&p 23\\
  $c(\sigma)$		&Cycle type of the permutation $\sigma$				&p 24\\
  $\sgn(\sigma)$	&Sign of a permutation						&p 25\\
  $\rk(\cdot)$		&Rank function in a Bratteli diagram				&p 27\\
  $\kappa(\cdot,\cdot)$	&Multiplicity function in a Bratteli diagram			&p 28\\
  $\tr(\cdot,\cdot),\ctr(\cdot,\cdot)$	&Transition and co-transition measure 		&p 29\\
  $\mathbb{Y}_n$	&Set of young diagrams of size $n$				&p 30\\
  $\mathcal{D}^0$		&Continual diagrams centered in $0$			&p 31\\
  $\mathbb{A}$		&Algebra of polynomial function on Young diagrams		&p 32\\  
  $n^{\downarrow k}$	&Falling factorial			&p 33\\
  $p^{\sharp}_{\rho}(\cdot)$	&Renormalized character of the symmetric group computed on $\rho$	&p 33\\
  $\mathbb{E}_P[X]$		&Average of the random variable $X$ according to the probability $P$		&p 33

  \end{tabular}
  
  \begin{tabular}{llr}
  $\overset{d}\to,$ $\overset{p}\to,$ $\overset{a.s.}\to$	&\begin{tabular}{l}Types of convergence (resp. in distribution, in probability,\\ and almost surely)\end{tabular}	&p 33\\
  $\mathcal{H}_m(x)$		&Modified Hermite polynomials				&p 33\\
  $\Lambda^*$			&Algebra of shifted symmetric functions			&p 34\\
  $\mathcal{C}_{\lambda}$	&Multiset of contents					&p 37\\
  $f(\Xi)$			&\begin{tabular}{l} Projective limit of a symmetric function $f$ computed on\\ Jucys-Murphy elements\end{tabular}&p 38\\
  $\phi(z;\lambda)$		&Generating function of $\lambda$			&p 38\\
  $\lambda(\cdot)$		&Piece-wise linear function associated to a partition $\lambda$	&p 38\\
  $K_{\omega}(z)$		&Kerov transform of the diagram $\omega$		&p 40\\
  $C_{\mu}(z)$			&Cauchy transform of the measure $\mu$			&p 40\\
  $F_{\mu}$			&Distribution function of the measure $\mu$		&p 40\\
  $PT_u^{\lambda}(\sigma)$, $PS_u^{\lambda}(\sigma)$	&\begin{tabular}{l} Renormalized partial trace and partial sum up to $u$\\ of the irreducible matrix $\pi^{\lambda}(\sigma)$ \end{tabular}&p 46, 68\\
  $TS^{\lambda}(\sigma)$	&Renormalized total sum of $\pi^{\lambda}(\sigma)$ &p 46\\
  $MT_u^{\lambda}(\sigma),RT_u^{\lambda}(\sigma)$	&Main term and remainder of the partial trace &p 47\\
  $\mathcal{N}(0,v^2)$	&Normal random variable of variance $v^2$	&p 47\\
  $\supp(\sigma)$	&Support of the permutation $\sigma$	&p 47\\
  $\sh(T)$	&Shape of the standard Young tableaux $T$	&p 50\\
  $\wt(\cdot)$	&Weight of a partition or permutation	&p 50\\
  $o_P(n^{\beta})$	&Small $o$ in probability of $n^{\beta}$	&p 56\\
  $O_P(n^{\beta})$	&Big $o$ in probability of $n^{\beta}$	&p 56\\
  
  $\tilde{p}_{\nu}(x_1,\ldots)$	&Modified power sum indexed by the partition $\nu$	&p 57\\
  $MS_u^{\lambda}(\sigma),RS_u^{\lambda}(\sigma)$	&Main term and remainder of the partial sum &p 69\\
  $\q\times\p$	&Partition in multirectangular coordinates	&p 75\\
  $\varkappa^{\lambda}_{\rho}$	&Spin character indexed by $\lambda$ and computed in $\rho$	&p 76\\
  $K_k(x_1,\ldots,x_k)$	&Kerov polynomial	&p 77\\
  $\tilde{R}_j$	&\begin{tabular}{l} Free cumulant of the transition measure \\for strict partitions \end{tabular}&p 77\\
  \begin{tabular}{l}$OP_n,$ $DP^+_n,$\\$DP^-_n,$ $DP_n$\end{tabular}	&\begin{tabular}{l} Sets of partitions of $n$, resp. with odd parts,\\ with distinct parts and even sign, with distinct\\ parts and odd sign, and with distinct parts\end{tabular}&p 80\\
  $P_{\lambda|m}(x_1,\ldots,x_m)$ 	&Olshanski's supersymmetric polynomial	&p 83\\
  $g^{\lambda}$	&\begin{tabular}{l}Dimension of the irreducible spin representation\\ indexed by $\lambda$\end{tabular}		&p 83\\
  $P^n_{strict}(\lambda)$		&Strict Plancherel measure		&p 83\\
  $D(\lambda)$	&Double diagram of the Young diagram $\lambda$		&p 85\\
  \end{tabular}

  \begin{tabular}{llr}
  $\tilde{p}^{\sharp}_k(\cdot)$		&\begin{tabular}{l}Renormalized spin character evaluated on an odd\\ cycle of length $k$\end{tabular}		&p 86\\
  $\scl(G),$ $\sch(G)$		&Sets of resp. superclasses and supercharacters of $G$					&p 104\\
  $\SPl_G(\chi)$		&Superplancherel measure of the supercharacter $\chi$				&p 106\\
  $\SInd^G_H(\phi)$		&Superinduction of the function $\phi$ 			&p 108\\
  \begin{tabular}{l}$\dim\pi,$ $\crs(\pi)$,\\ $\nst(\pi),$ $d(\pi)$\end{tabular}		&\begin{tabular}{l}Statistics on the set partition $\pi$, resp. dimension,\\ number of crossings, number of nestings, and number of arcs\end{tabular}		&p 110\\
  $\mu_{\pi}$			&measure in $\R^2$ associated to the set partition $\pi$		&p 118\\
  $\Delta$			&Subset of $\R^2$							&p 118\\
  $\Gamma$			&Subprobabilities on $\Delta$ with subuniform marginals			&p 119\\
  $\tilde{\Gamma}$		&Subprobabilities on $\Delta$ with uniform marginals			&p 121\\
  $d_{L-P}(\cdot,\cdot)$	&L{\'e}vy-Prokhorov metric						&p 125  
  \end{tabular}

%

\pagestyle{scrplain}

\pagenumbering{arabic}
\cleardoublepage \myPart{Introduction} \label{ch: intro}
\cleardoublepage\section{Representation theory}

Character theory originated with a series of three letters written by Frobenius to Dedekind in April 1896. His goal was to factorize some polynomials arising from group theory, but he quickly realized that the coefficients of the factors behaved like the traces of a matrix. This prompted Frobenius, with the help of his student Schur, to define thus the concept of matrix representation.
\medskip

Informally speaking representation theory is a branch of group theory whose aim is to study a group $G$ by considering a group of matrices which maintain properties similar to $G$. From the very beginning it was apparent that representations are strictly related to \emph{characters}, which are the traces of the matrices associated to a representation. Characters can be seen as a map $G\to\C$ constant on the conjugacy classes of $G$ and with $\chi(\id_G)\in\N$, where $\id_G$ is the neutral element of $G$. The value $\chi(\id_G)=\dim\chi$ is called the \emph{degree} of the character (or of the representation). Characters in general can be decomposed into characters of smaller degree, called \emph{irreducible}. The set of irreducible characters $\Irr(G)$ forms a basis for the algebra of functions $G\to\C$ which are constant on conjugacy classes.

In many instances representation theory proved to have a rich structure and connections with several fields of mathematics. One of the most brilliant example of this is the \emph{symmetric group}.

\section{The symmetric group}

The \emph{symmetric group} $S_n$ is the group of bijections $\{1,\ldots,n\}\to\{1,\ldots,n\}$. A well developed theory of the irreducible characters of the symmetric group $S_n$ was established since the pioneering works of Frobenius, Schur and Young, who showed that irreducible characters are indexed by integer partitions or, equivalently, by Young diagrams. It took another half a century before mathematicians took interest in the \emph{infinite} symmetric group $S_{\infty}$, with the honorable mention of Thoma's theorem \cite{Tho64}. From that moment the community, and among them the Russian school of Vershik and Olshanski, started investigating the \emph{asymptotic} of representation theory of the symmetric group, which can be considered as an intermediate step between the finite and infinite case (see for example \cite{vershik2003two} and \cite{vershik1995asymptotic}). Problems arising from the combinatorics of permutations could be stated in the language of asymptotic representations, and gave a further push in the development of the field. 

In the 1990s a new approach, called ``dual combinatorics'', was introduced. The idea was to consider the irreducible character $\chi^{\lambda}(\sigma)$, for a partition $\lambda$ of $n$ (which we write $\lambda\vdash n$) and a permutation $\sigma\in S_n$, as a function of $\lambda$ rather than $\sigma$. We thus fix $\sigma$ in $S_k$, with $k\leq n$, and let $\lambda$ grow with $n$. The approach works because if $\sigma$ is in $ S_k$ we can consider it as a permutation in $S_n$ by letting the points $k+1,k+2,\ldots, n$ be fixed points. In this way $\chi^{\lambda}(\sigma)$ makes sense when the size of $\lambda$ increases.

When studying asymptotic problems on a family of groups one should fix a probability on the group itself. The most natural one is the uniform distribution function, where each permutation is taken uniformly at random. The \emph{Robinson-Schensted-Knuth correspondence} allows us to translate the research of the asymptotic of some statistics on random permutations, in particular the length of longest increasing subsequences, to integer partitions, see \cite{RobinsonRSK} and \cite{SchenstedRSK}. The right analogue of the uniform measure on permutations is the \emph{Plancherel measure} on the set of partitions $\mathcal{P}(n)=\{\lambda\vdash n\}$.

More generally, the Plancherel measure appears naturally from the study of irreducible characters: for a general finite group $G$ we have 
\[\sum\limits_{\chi\in\Irr(G)}\frac{\chi(\id_G)^2}{|G|}=1,\]
so that for each group $G$ we can associate a probability measure $P_{Pl}$ with the set $\Irr(G)$:
\[P_{\Pl}(\chi)=\frac{\chi(\id_G)^2}{|G|}.\]

At the time two questions were the main focus of asymptotic representation theory for $S_n$:
\begin{enumerate}
 \item What does the Young diagram $\lambda$ look like when $\lambda$ is a random integer partition of $n$ taken with the Plancherel measure and $n\to \infty$?
 \item What does the character $\chi^{\lambda}(\sigma)$ look like when the size of the random Plancherel distributed partition $\lambda$ increases and $\sigma$ is fixed? 
\end{enumerate}
The first question was solved by Vershik and Kerov \cite{KerovVershik1977} and, independently, by Logan and Shepp \cite{logan1977variational} in 1977. When dual combinatorics was introduced Ivanov and Olshanski \cite{ivanov2002kerov} wrote a new proof, based on the ideas of Kerov. In short, the Young diagram $\lambda$ can be seen, through a change of coordinates, as a piece-wise linear function $\lambda(x)$ in $\R$. The moments of this function, called $\overline{p}_k(\lambda)$, are defined as
\[\overline{p}_k(\lambda):=k(k-1)\int_{\R}x^{k-2}\frac{\lambda(x)-|x|}{2}\,dx.\]
These moments are studied and shown to converge to $\overline{p}_k(\Omega)$, for $\lambda\vdash n$ distributed with the Plancherel measure and $n\to\infty$. Since the functions $\overline{p}_k$ fulfill some regularity conditions then we derive that $\lambda(x)$ converges to $\Omega(x)$ in the uniform topology. The function $\Omega$ is called the \emph{limit shape}:
\[\Omega(x):=\left\{\begin{array}{ll}\frac{2}{\pi}(x\arcsin\frac{x}{2}+\sqrt{4-x^2})& \mbox{ if }|x|\leq 2,\\|x|&\mbox{ if }|x|\geq 2.\end{array}\right.\]
The success of the dual approach is mostly due to the realization that the character $\chi^{\lambda}(\sigma)$, after some renormalization, lives in the algebra $\mathbb{A}:=\R[\overline{p}_2(\lambda),\overline{p}_3(\lambda),\ldots]$ and thus we can express the characters in terms of the moments of the diagrams and vice versa. This surprising fact led to a second order asymptotic result on random Young diagrams and a central limit theorem for characters, solving thus question 2.

The algebra $\mathbb{A}$ revealed to be deeply connected with the algebra of symmetric functions: if we consider the parts of $\lambda=(\lambda_1,\lambda_2,\ldots,\lambda_l)$ as formal variables, then the elements of $\mathbb{A}$ become symmetric functions when computed in the variables $\lambda_i:=x_i+i$. This led to the study of shifted symmetric functions, introduced in \cite{OkOl1998}.

A renewed interest in the subject arose with the proof that the joint distribution of properly scaled largest parts of a random partition is equal to the joint distribution of the largest eigenvalues of a Gaussian random Hermitian matrix, see \cite{borodin2000asymptotics}, \cite{baik1999distribution} and \cite{BaikDeiftJohansson1999second_row}. The connection, although mysterious, appears to go deeper: the Gaussian process for the global fluctuations of random Plancherel distributed partition around the limit shape resembles the second order asymptotic for the eigenvalues of a Gaussian random Hermitian matrix, see \cite{Johansson1998} and \cite{Johansson2001}. Other random matrix ensembles share these similarities, see \cite{diaconis2001linear}, \cite{diaconis1994eigenvalues} and \cite{johansson1997random}.
\medskip

In the next sections we introduce our contributions to the subjects of combinatorics and representation theory. On the one hand we develop the study of Plancherel distributed Young diagrams; on the other hand we consider similar questions in analogue contexts.

\section{Partial sum for representations of the symmetric group}
We recall the central limit theorem for Plancherel distributed normalized characters of the symmetric group $\hat{\chi}^{\lambda}_{(\rho,1,\ldots,1)}=\chi^{\lambda}_{(\rho,1,\ldots,1)}/\dim\lambda$, where $\chi^{\lambda}_{(\rho,1,\ldots,1)}$ is the character associated to $\lambda$ calculated on a permutation of cycle type $(\rho,1,\ldots,1)$ and $\dim\lambda=\chi^{\lambda}(\id_{S_n})$. This fundamental result has been proved by Kerov \cite{ivanov2002kerov} and, independently, by Hora \cite{hora1998central}.

Given $m\in \N_{\geq 0}$ let $\mathcal{H}_m(x)$ be the $m$-th \emph{modified Hermite polynomial} defined by 
\[x\mathcal{H}_m(x)=\mathcal{H}_{m+1}(x)+m\mathcal{H}_{m-1}(x)\quad \mbox{ and }\quad\mathcal{H}_0(x)=1,\quad\mathcal{H}_1(x)=x.\] Consider $\rho$ a fixed partition of $r$ and $\lambda\vdash n$ a Plancherel distributed partition, with $r\leq n$; set $m_k(\rho)$ to be the number of parts of $\rho$ which are equal to $k$. Then for $n\to\infty$
 \begin{equation*}
  n^{\frac{|\rho|-m_1(\rho)}{2}}\hat{\chi}^{\lambda}_{(\rho,1,\ldots,1)}\overset{d}\to \prod_{k\geq 2} k^{m_k(\rho)/2} \mathcal{H}_{m_k(\rho)}(\xi_k),
 \end{equation*}
 where $\{\xi_k\}_{k\geq 2}$ are i.i.d. standard gaussian variables, and $\overset{d}\to$ means convergence in distribution.
 
 \subsubsection{Our contribution}
A natural step following this result is to examine the asymptotic of irreducible representation matrices. More precisely, let $\pi^{\lambda}(\sigma)$ be the irreducible representation matrix associated to a partition $\lambda\vdash n$ and computed on a permutation $\sigma\in S_k$, where as before we set $S_k\subseteq S_{k+1}\subseteq\ldots\subseteq S_n$. Rather than the normalized character 
\[\hat{\chi}^{\lambda}(\sigma)=\sum_{i\leq \dim\lambda}\frac{\pi^{\lambda}(\sigma)_{i,i}}{\dim\lambda}\]
we consider the partial trace up to a certain index: if $u\in[0,1]$ we set the \emph{partial trace} to be
\[PT_u^{\lambda}(\sigma):=\sum_{i\leq u\dim\lambda}\frac{\pi^{\lambda}(\sigma)_{i,i}}{\dim\lambda}.\]
Since the entries of the matrix $\pi^{\lambda}(\sigma)$ depend on the choice of the basis, we study the explicit construction of the Young orthogonal representation given in Section \ref{section young orthogonal}. Our results hold also for other ``famous'' constructions, as discussed in Chapter \ref{ch: partial sum}.

We prove a decomposition formula for the partial trace. Recall that for $\lambda\vdash n$ and $\mu\vdash n-1$ we write $\mu\nearrow\lambda$ if $\mu_i\leq\lambda_i$ for each $i$. We show that there exist $\bar{u}\in[0,1]$ and $\bar{\mu}\nearrow\lambda$ such that
\begin{equation}\label{cio}
 PT_u^{\lambda}(\sigma)=\sum_{\substack{\mu\nearrow\lambda\\\mu<\bar{\mu}}}\frac{\dim\mu}{\dim\lambda}\hat{\chi}^{\mu}(\sigma)+\frac{\dim\bar{\mu}}{\dim\lambda}PT_{\bar{u}}^{\bar{\mu}}(\sigma),
\end{equation}
where the sum runs over partitions $\mu\nearrow\lambda$ which are smaller than $\bar{\mu}$ in the reverse lexicographic order. We call the first sum on the right hand side the \emph{main term of the partial trace} $MT_u^{\lambda}(\sigma)$ and the second the \emph{remainder for the partial trace} $RT_u^{\lambda}(\sigma)$. We focus on asymptotic questions when $\lambda\vdash n$ is distributed with the Plancherel measure and $n\to\infty$. Our main result in Chapter \ref{ch: partial sum} is the following:
\begin{theorem}
 Let as before $\{\xi_k\}_{k\geq 2}$ be a sequence of independent standard Gaussian variables and set $\rho\vdash r$ to be the cycle type of $\sigma\in S_r$. Then 
 \[n^{\frac{|\rho|-m_1(\rho)}{2}}MT_{u}^{\lambda}(\sigma)\overset{d}\to u\cdot \prod_{k\geq 2} k^{m_k(\rho)/2} \mathcal{H}_{m_k(\rho)}(\xi_k)\qquad\mbox{ as }n\to\infty.\]
\end{theorem}
In particular the case $u=1$ corresponds to Hora-Kerov's theorem (which we use it in our proof). Note that there is no additional randomness compared to the asymptotic of the total trace.

Unfortunately we cannot obtain a law of large numbers for the partial trace because of the remainder $RT_u^{\lambda}(\sigma)$. We conjecture that the remainder is asymptotically smaller than the main term, but we cannot prove it. We discuss this problem in Section \ref{section conjecture}, where we present some numerical evidence for our conjecture.

We move a step farther, considering the \emph{partial sum}
\[PS_{u,v}^{\lambda}(\sigma):=\sum_{\substack{i\leq u\dim\lambda\\j\leq v\dim\lambda}}\frac{\pi^{\lambda}(\sigma)_{i,j}}{\dim\lambda}.\]
A decomposition similar to \eqref{cio} holds for the partial sum as well, so that the partial sum can be written as a sum of a main term and a remainder. We prove a central limit theorem for the main term and a law of large numbers for the remainder, which combined hold a law of large numbers for the partial sum:
\begin{corollary}
 Consider a random partition $\lambda\vdash n$ distributed with the Plancherel measure and $\sigma\in S_r$ fixed. then
\[ PS_{u,v}^{\lambda}(\sigma)\overset{p}\to \min\{u,v\}\cdot \sum_{\nu\vdash r}\frac{\dim\nu^2}{r!}PS_{1,1}^{\nu}(\sigma),\]
where $\overset{p}\to$ means convergence in probability.
\end{corollary}
Although the results are mostly of probabilistic nature, we rely heavily on tools from algebraic combinatorics. More precisely, we reduce the study of the main term $MT^{\lambda}(\sigma)$ to the study of the asymptotic of $\hat{\chi}^{\lambda}(\sigma)-\hat{\chi}^{\mu}(\sigma)$, where $\lambda\vdash n$ is distributed with the Plancherel measure and $\mu\nearrow \lambda$. The difference $\hat{\chi}^{\lambda}(\sigma)-\hat{\chi}^{\mu}(\sigma)$ has a nice description in terms of linear combinations of particular elements of the group algebra $\C[S_n]$, called \emph{Jucys-Murphy elements}. They allow us to translate the problem in terms of some statistics of the Young diagrams $\lambda$ and $\mu$, which assume similar values when $n\to\infty$.
\section{Projective representations of the symmetric group}
Projective representation theory was introduced by Schur in the early 1900s, but received little coverage until recently. In the past 30 years the theory developed in a way that resembles the classic representation theory of the symmetric group. In the projective case, the projective (or spin) characters are indexed by strict partitions, and the role of standard Young tableaux is played by a shifted version. A dual approach similar to the one of Kerov was considered by Ivanov in \cite{ivanov2004gaussian} and \cite{ivanov2006plancherel}. The projective characters are closely related to \emph{supersymmetric functions}, which form a subalgebra of the algebra of symmetric functions. Similarly to the classical case, Ivanov described the asymptotic of opportunely renormalized projective characters and shifted Young diagrams for Plancherel distributed strict partitions.

The parallelism between classical and projective representation theory of the symmetric group seems to go deeper, and in Chapter \ref{ch: strict partitions} we exploit yet another feature, which is derived from a different point of view introduced by Stanley: for the study of the classical irreducible characters he proposed in \cite{stanley2003irreducible} a new parametrization for Young diagrams called \emph{multirectangular coordinates}. This parametrization revealed new interesting combinatorics hidden in the renormalized character. Indeed, by treating the multirectangular coordinates of a Young diagram as formal variables, the normalized character evaluated on a cycle of length $k$ becomes a polynomial in these variables. Stanley proved that the coefficients of these polynomials are all integers, and conjectured that they are all nonnegative. In \cite{Stanley-preprint2006} he improved the conjecture, proposing a combinatorial formula for the coefficients. The formula (and the conjecture) were proven by F\'eray in \cite{feray2010stanley} (see also \cite{FeraySniady2011}).

\subsubsection{Our contribution}

We apply a similar approach to the projective representation theory of the symmetric group, introducing multirectangular coordinates for shifted Young diagrams. As in the classical case, the renormalized character written as a function on the multirectangular coordinates is a polynomial, and the coefficients seem to be nonnegative. These coefficients are not integers though, but we prove that they belong to $\Z/2$. Moreover, we show that the terms of highest degree in this polynomial have indeed nonnegative coefficient.

Stanley polynomials are strictly related to Kerov polynomials, which describe the renormalized characters in terms of the free cumulants (see \cite{Biane2003}, \cite{LassalleJackFreeCumulants} and \cite{FeraySniady2011} for more details). As with Stanley polynomials, Kerov polynomials were conjectured to have nonnegative integer coefficients, and thus to have a hidden combinatorial structure. The conjecture was proved in \cite{F'eray2008} through the F\'eray-Stanley formula.

Sho Matsumoto conjectured in a private communication the existence of projective versions of Kerov polynomials with similar properties. In particular he conjectures that the leading term of the normalized spin character evaluated on the cycle $(1,\ldots, k)$, $k$ odd, is the $k+1$-th free cumulant of the spin transition measure, multiplied by $1/2$. We prove this conjecture in Chapter \ref{ch: strict partitions}.

\section{Upper unitriangular matrices}
Consider the group $U_n=U_n(\mathbb{K})$ of upper triangular matrices with coefficients in a finite field $\mathbb{K}$ and $1$ in the main diagonal. The irreducible character theory of $U_n$ has been proven to be a wild problem (see \cite{gudivok1990classes}) and considered intractable by many. In a series of papers (\cite{andre1995basic}, \cite{andre1995basicb}, \cite{andre1996coadjoint}, \cite{andre1998regular}, \cite{andre2001basic}, \cite{andre2002basic} and \cite{andre2004hecke}) Andr\'e developed an alternative method to the study of irreducible character theory of $U_n$. Rather than looking at irreducible characters, he evaluates \emph{supercharacters} and \emph{superclasses}, where the former are sums of irreducible characters and the latter are unions of conjugacy classes, such that supercharacters are constant on superclasses. In \cite{diaconis2008supercharacters} Diaconis and Isaacs formalize the concept of supercharacter theory for arbitrary finite groups.

For each finite group there are at least two supercharacter theories, called the \emph{trivial }supercharacter theories, but in general there can be more. For the upper unitriangular group $U_n$ two nontrivial supercharacter theories caught the interest of the community: in the first the supercharacters (and superclasses) are indexed by \emph{set partitions} of the set $[n]=\{1,\ldots,n\}$; in the second the supercharacters and superclasses are indexed by \emph{colored set partitions}, that is, set partitions equipped with a function from the set of arcs to $\mathbb{K}^{\times}$. The theory, although recent, is gaining increasing coverage from the community, and its spreading towards different directions, which we mention below.
\begin{enumerate}
 \item \textbf{Formalization of supercharacter theory.} The main question in this area is to describe the supercharacter theories for a general finite group $G$. For example, Diaconis and Isaacs (\cite{diaconis2008supercharacters}) noticed how a theorem of Brauer could be used to build supercharacter theories from a group $A$ that acts via automorphisms on $G$. Aliniaeifard \cite{aliniaeifard2015normal} and Hendrickson \cite{hendrickson2012supercharacter} proved two different constructions of supercharacter theories through normal subgroups of $G$. Other results focus on counting the supercharacter theories for particular groups (see \cite{hendrickson2008supercharacter} and \cite{burkett2017groups}). Interestingly, in \cite{keller2014generalized} Keller showed that for each finite group $G$ there exists a unique minimal supercharacter theory such that the supercharacters take integer values. It is still an open problem whether the uncolored set partition supercharacter theory is the minimal integral one for $U_n$.
 \item \textbf{Combinatorics of set partitions.} As expected, once a bridge has been built between representation theory and combinatorics, both areas can beneficiate from the other. In the two aforementioned supercharacter theories for $U_n$ (namely, with colored and uncolored set partitions as indices) the supercharacters can be described by an explicit formula which relies on certain statistics of the set partitions called the dimension $\dim\pi$ and the number of crossings $\crs\pi$ for a set partition $\pi$. Chern, Diaconis, Kane and Rhoades in \cite{chern2014closed} and \cite{chern2015central} proved asymptotic results for these statistics when the set partition is taken uniformly at random. One of the most promising discoveries brought by this bridge is that the Hopf algebra of superclass functions for the ``uncolored set partitions'' supercharacter theory is isomorphic to the Hopf algebra of symmetric functions in noncommuting variables, see \cite{aguiar2012supercharacters}, \cite{bergeron2006grothendieck}, \cite{bergeron2013supercharacter}, \cite{baker2014antipode}.
 \item \textbf{Results on $U_n$}. The known supercharacter tables for the upper unitriangular group represent a good surrogate for the irreducible character tables of $U_n$. For example, Arias-Castro, Diaconis and Stanley \cite{arias2004super}, and then Diaconis and Hough \cite{diaconis2015random}, analyzed random walks on unipotent matrix group with the aid of supercharacter theory. Another problem which gained from supercharacter theory includes the counting of irreducible characters of $U_n$, see \cite{le2010counting}, \cite{goodwin2009calculating}, \cite{isaacs2007counting}, \cite{marberg2011combinatorial}.
 \item \textbf{Number theory.} In a different direction, some mathematicians considered using supercharacter theory to prove number theoretic results. See for example \cite{fowler2014ramanujan} and \cite{brumbaugh2014supercharacters} for some new one-line proofs of Ramanujan identities obtained with this approach.
\end{enumerate}

\subsubsection{Our contribution}

Our results range between the first 2 directions. We prove that to each supercharacter theory we can associate a \emph{superplancherel measure}. This is a generalization of the Plancherel measure, in the sense that if the supercharacter theory is actually the irreducible one, then the superplancherel and Plancherel measures coincide. We compute explicitly the superplancherel measure for the uncolored set partition supercharacter theory of $U_n$. We prove a limit shape result for superplancherel distributed set partitions. Informally speaking our result claims that a random superplancherel distributed set partition of $n$ is asymptotically close to the set partition $\{\{1,n\},\{2,n-1\},\{3,n-2\},\ldots\}$ whose arc representation is

\[
 \begin{tikzpicture}[scale=0.5]
\draw [fill] (0,0) circle [radius=0.05];
\node [below] at (0,0) {1};
\draw [fill] (1,0) circle [radius=0.05];
\node [below] at (1,0) {2};
\draw [fill] (2,0) circle [radius=0.05];
\node [below] at (2,0) {3};
\draw [fill] (3.5,0) circle [radius=0.03];
\draw [fill] (3.75,0) circle [radius=0.03];
\draw [fill] (4,0) circle [radius=0.03];
\draw [fill] (5.5,0) circle [radius=0.05];
\draw [fill] (6.5,0) circle [radius=0.05];
\draw [fill] (7.5,0) circle [radius=0.05];
\node [below] at (7.5,0) {n};
\draw (0,0) to[out=70, in=110] (7.5,0);
\draw (1,0) to[out=70, in=110] (6.5,0);
\draw (2,0) to[out=70, in=110] (5.5,0);
\end{tikzpicture}\]
The two statistics mentioned above, the dimension and the number of crossings, play an important role in the study of superplancherel distributed set partitions, and we prove
 \[\frac{\dim(\pi)}{n^2}\to\frac{1}{4}\qquad\mbox{and}\qquad \frac{\crs(\pi)}{n^2}\to 0\quad\mbox{ almost surely}\] 
when $\pi$ is a superplancherel distributed set partition and $n\to\infty$.

Our approach is to see the set partition $\pi$ as a measure $\mu_{\pi}$ of the unit square $[0,1]^2$, which allows us to translate the statistics $\dim(\pi)$ and $\crs(\pi)$ into integrals over $\mu_{\pi}$. With an entropy argument we prove that there exists a unique measure $\Omega$ which maximizes the superplancherel measure when $n$ grows, and we conclude that $\mu_{\pi}\to\Omega$ in the weak* convergence, which is a natural topology in this space of measures.

\section{A (very short) word of conclusion}
 We have presented our results in three different directions related to random objects coming from representation theory (or variants). It is impossible to fully exhaust the knowledge of a topic and, as often happens in mathematics, the more we know the more we ignore. At the end of chapters \ref{ch: partial sum}, \ref{ch: strict partitions} and \ref{ch: supercharacters} we include possible developments in the field and future research.

\section{Structure of the thesis}
The next two chapters of the thesis are devolved to preliminaries of asymptotic representation theory, with a focus on the symmetric group $S_n$; the last three chapters present new results obtained during this thesis. In particular, in Chapter \ref{ch: repr theory} we recall the basics of representation theory, the symmetric group and its connections with symmetric functions. In Chapter \ref{ch: probability} we describe Bratteli diagrams and a particular class of them, which arises from representation theory. We then consider the symmetric group and discuss the main results of the dual approach, including the law of large numbers for Plancherel distributed Young diagrams and irreducible characters (after opportune renormalization). We present the main feature of this approach, that is, that normalized characters can be written as polynomials evaluated on certain parameters of Young diagrams. These polynomials are called \emph{shifted symmetric functions}. We conclude describing asymptotic results on the transition and co-transition measures.

In Chapter \ref{ch: partial sum} we study the partial trace and partial sum of representation matrices of the symmetric group, their decomposition into main terms and remainders, and their asymptotic. We conclude the chapter with a future research section, which includes a conjecture that would imply a central limit theorem for the partial trace.

In Chapter \ref{ch: strict partitions} we discuss the multirectangular coordinates for shifted Young diagrams. We begin with an overview of projective representation theory and the particular case of $S_n$. We then introduce multirectangular coordinates for strict partitions and a conjecture on the coefficients of the normalized character expressed in these new coordinates. We prove this conjecture for the main term, and discuss in Section \ref{sec: future strict partitions} possible approaches to solve the whole conjecture. As for the asymptotic, we recall in Section \ref{section: asymptotics} the results of Ivanov and we translate them in the uniform topology.

In Chapter \ref{ch: supercharacters} we cover the basics of supercharacter theory and we introduce the new concept of superplancherel measure. We compute it for the ``uncolored set partition'' supercharacter theory for $U_n$ and prove a law of large numbers for it. In Section \ref{sec: future supercharacters} we mention possible future developments on this topic; in particular we suggest a new approach to the theory, based on the original results of Frobenius.

\cleardoublepage \myPart{Preliminaries} \label{p:1}
\cleardoublepage\chapter{Representation theory and the symmetric group}\label{ch: repr theory}
The study of characters of finite groups started with Gauss at the beginning of the 19th century, but it took almost a century before Frobenius generalized the concept to non abelian groups. Since then, the theories of representations and characters have had a major impact in many areas of mathematics, physics, and chemistry.

The particular case of representation theory of the symmetric group has proved to be both useful, for its connections with other fields of mathematics and quantum mechanics, and beautiful in its own regard, especially for its link with symmetric functions. In the following we present an introduction on the subject of representation theory and character theory (Section \ref{section: repr theory}), following \cite[Chapter 1]{sagan2013symmetric}. We also define the Plancherel measure, which will be the main object of study of this thesis. In Section \ref{section: the sym group} we focus on the symmetric group, for which we present an explicit construction, the Young orthogonal representation. We investigate this explicit construction in Chapter \ref{ch: partial sum}. We conclude the Chapter with an overview on symmetric functions.

\section{Introduction to representation theory}\label{section: repr theory}
The following is a brief introduction to representation theory over the complex field $\C$. Throughout the chapter $G$ will be a fixed finite group. 
\begin{defi}
 Let $V$ be a $\C$-vector space of dimension $d$. A \emph{group representation} is a homomorphism $\pi\colon G\to\GL(V)$, where $\GL(V)$ is the group of automorphisms $V\to V$. 
\end{defi}
If we fix a basis for the vector space $V$ we can identify $\GL(V)$ with $\GL_d(\C)$, the group of square matrices of dimension $d$ over $\C$, so that $\pi$ can be seen as a map $\pi\colon G\to \GL_d(\C)$.

If $\pi\colon G\to \GL(V)$ is a representation then $V$ is said to be a (left) $G$-\emph{module}. The \emph{dimension} of a representation is the dimension of the associated $G$-module.
\begin{example}
 For a finite group $G=\{g_1,\ldots,g_n\}$ the $G$-module
 \[\C[G]:=\{\sum_{i=1}^nc_ig_i\mbox{ s.t. }c_i\in\C\}\]
 is called the \emph{group algebra} of $G$. It is a $G$-module for the (left) \emph{regular representation} $\pi^{\reg}$ defined by
 \[\pi^{\reg}(g)\left(\sum_{i=1}^nc_ig_i\right):=\sum_{i=1}^nc_i(gg_i).\] 
\end{example}
\begin{defi}
 Let $G$ be a finite group and $\pi\colon G\to\GL(V)$ a representation. A \emph{submodule} is a vector subspace $W\leq V$ which is closed under the action of $G$: if $w\in W$ then $\pi(g)(w)\in W$ for all $g\in G$. 
\end{defi}
It is easy to see that the vector spaces $\{0\}$ and $V$ are submodules of $V$, and are called the \emph{trivial submodules}. If $V$ has nontrivial submodules then it is called \emph{reducible}, and otherwise it is said to be \emph{irreducible}.
\begin{defi}
 Let $G$ be a finite group, $V_1$ and $V_2$ two $G$-modules with representations respectively $\pi_1\colon G\to \GL(V_1)$ and $\pi_2\colon G\to \GL(V_2)$. A $G$-\emph{homomorphism} $\theta\colon v_1\to V_2$ is a vector space homomorphism such that 
 \[\theta(\pi_1(g)(v))=\pi_2(g)(\theta(v))\qquad\mbox{ for each }g\in G\mbox{ and }v\in V.\]
 If $\theta$ is an isomorphism then we say that $\theta $ is a $G$-\emph{isomorphism} and we write equivalently $V_1\cong V_2$ and $\pi_1\cong\pi_2$.
\end{defi}
\begin{theorem}[Maschke's Theorem]
Let $G$ be a finite group and $V$ a nonzero $G$-module. Then there exist irreducible $G$ submodules $W_1,\ldots, W_l$ such that 
\[V\cong W_1\oplus\ldots\oplus W_l.\] 
\end{theorem}
Let $\pi\colon G\to \GL(V)$ be a representation of dimension $d$. If we fix a basis for $V$ then $\pi(g)$ is a $d$ dimensional matrix for each $g$, and Maschke's theorem can be stated thus: there exist irreducible representations $\pi_1,\ldots,\pi_l$ such that 
\[T\pi(g)T^{-1}=\left[\begin{array}{ccc}\pi_1(g)&&\mathbbold{0}\\&\ddots&\\\mathbbold{0}&&\pi_l(g) \end{array}\right]\]
for some matrix $T$ that does not depend on $g$.
\begin{proposition}[Schur's Lemma]
Let $V_1$, $V_2$ be irreducible $G$-modules. If $\theta\colon V_1\to V_2$ is a nonzero $G$-homomorphism then $\theta$ is a $G$-isomorphism. 
\end{proposition}
\begin{remark}
 The number of irreducible representations is equal to the number of conjugacy classes of $G$, and in particular is finite. The previous two results show that every representation is described, up to $G$-isomorphism, by the multiplicities of the irreducible representations appearing into its decomposition.
\end{remark}
An easy corollary of Schur's lemma is that if $\pi$ is an irreducible representation of dimension $n$ and $T$ is a matrix such that $T\pi(g)=\pi(g)T$ for all $g\in G$ then $T$ is a scalar matrix, that is, there exists $c\in \C$ such that $T=c\mathbb{1}_n$. In particular, if $z$ is an element of the center of the group algebra, $z\in Z(\C[G])$, then there exists $c\in \C$ such that $\pi(z)=c\mathbb{1}_n$.
\begin{defi}
 Let $V\cong \C^d$ be a $G$-module and $\pi\colon G\to\GL_d(\C)$ its representation. The \emph{character} $\chi$ associated to $\pi$ is the map
 \[\begin{array}{cccccc}\chi&\colon&G&\to&\C&\\
 &&g&\mapsto&\tr(\pi(g))&=\sum\limits_{i=1}^d\pi(g)_{i,i}. 
 \end{array}\] 
\end{defi}
Note that the trace of a matrix is constant up to matrix conjugation, hence the character is well defined independently of the choice of the basis for $V$. A character is \emph{irreducible} if the underlying $G$-module is irreducible. We call $\Irr(G)$ the set of irreducible characters of $G$.

We summarize in the next proposition some fundamental properties of characters.
\begin{proposition}
 Let $\pi\colon G\to \GL_d(\C)$ be a representation of dimension $d$ with character $\chi$, then
 \begin{itemize}
  \item $\chi(\id_G)=d$, where $\id_G$ is the identity for the group $G$. We say that $\chi(\id_G)$ is the \emph{dimension} of the character;
  \item the character $\chi$ is a \emph{class function}, that is, is a function which is constant up to conjugation: if $g,h\in G$ then $\chi(h^{-1}gh)=\chi(g)$;
  \item $\chi(g^{-1})=\overline{\chi(g)}$, where $\overline{(\cdot)}$ is the complex conjugate;
  \item let $\tilde{\pi}\colon G\to \Mat_{\tilde{d}}(\C)$ be another representation with character $\tilde{\chi}$, then 
  \[\pi\cong\tilde{\pi}\mbox{ if and only if }\chi(g)=\tilde{\chi}(g)\mbox{ for each }g\in G;\]
  \item if $\pi $ is the regular representation, $\pi=\pi^{\reg}$, then 
  \[\chi^{\reg}(g)=\left\{\begin{array}{lc}|G|&\mbox{ if }g=\id_G\\0&\mbox{otherwise.}   \end{array}\right.\]
  Moreover, if $\chi_1,\ldots,\chi_l$ are the distinct irreducible characters of $G$ of dimension, respectively, $d_1,\ldots, d_l$, then  
  \[\chi^{\reg}=\sum\limits_{i=1}^ld_i\chi_i.\]  
In particular the dimension of $\chi^{\reg}$ is 
\begin{equation}\label{eq:plancherel formula}
 \chi^{\reg}(\id_G)=|G|=\sum\limits_{i=1}^l\chi_i(\id_G)^2.
\end{equation}
 \end{itemize}
 \end{proposition}
 Let $K\subseteq G$ be a conjugacy class of $G$ and $\chi$ a character of $G$. Since the characters are class functions we can write $\chi(K)$ rather than $\chi(g)$, where $g\in K$.
 
 Let $z\in Z(\C[G])$ and recall that for each irreducible representation $\pi$ there exists $c\in \C$ such that $\pi(z)=c\mathbb{1}_n$. Let $\chi$ be the irreducible character associated to $\pi$, then 
 \begin{equation}\label{eq: consequence schur lemma}
   \chi(gz)= \tr(\pi(gz))= \tr(\pi(g)\cdot\pi(z))=\tr(\pi(g)\cdot c\mathbb{1}_n)=\tr(\pi(g))\cdot c=\frac{\chi(g)\chi(z)}{\chi(\id_G)}.
 \end{equation}
 This fact, together with the observation that characters are class functions and that the sum of elements in a conjugacy class are in the center of the group algebra, allows us to shift the study of multiplication of character values to the study of the multiplication table of the group, and vice versa. 
 
Equation \eqref{eq:plancherel formula} shows that we can associate a discrete measure with the set of irreducible characters. It is called the \emph{Plancherel measure} $P_{\Pl}$: if $\chi$ is an irreducible character for the group $G$ then
\begin{equation}\label{eq: plancherel measure}
P_{\Pl}(\chi)=\frac{\chi(\id_G)^2}{|G|}. 
\end{equation}
\begin{defi}
 Let $f_1,f_2\colon G\to\C$ be two functions. The \emph{Frobenius inner product} of $f_1$ and $f_2$ is
 \[\langle f_1,f_2\rangle:=\frac{1}{|G|}\sum\limits_{g\in G}f_1(g)\overline{f_2(g)}.\]
\end{defi}
\begin{proposition}\label{prop: character relations}
 Let $G$ be a finite group.
 \begin{itemize}
  \item \emph{Character relations of the first kind:} if $\chi_1,\chi_2$ are irreducible characters of $G$ then 
  \[\langle\chi_1,\chi_2\rangle=\delta_{\chi_1,\chi_2},\]
  where $\delta$ is the Kronecker delta.
  \item \emph{Character relations of the second kind:} if $K_1,K_2$ are conjugacy classes of $G$ then 
  \[\sum_{\chi\in\Irr(G)}\chi(K_1)\overline{\chi(K_2)}=\frac{|G|}{|K_1|}\delta_{K_1,K_2}.\]
 \end{itemize}
\end{proposition}
\begin{defi}
 The \emph{character table} of a finite group $G$ is the square matrix $[\chi(K)]_{\chi,K}$, where the entries are indexed by the irreducible characters and the conjugacy classes of $G$. Note that the character table is an orthonormal matrix because of the previous proposition.
\end{defi}
 Consider the $\C$-algebra of class functions $Cl(G)$, where the basic operations are pointwise addition and multiplication, equipped with the inner product defined above. If $K_1,\ldots, K_l$ are the conjugacy classes of $G$ then a trivial orthogonal basis for $Cl(G)$ is $\{\epsilon_{K_1},\ldots\epsilon_{K_l}\}$, where
 \[\epsilon_{K_i}(g)=\left\{\begin{array}{cc}1&\mbox{ if }g\in K_i\\0&\mbox{ otherwise}.\end{array}\right. \]
Hence the dimension of $Cl(G)$ is equal to the number of conjugacy classes. The character relations of the first kind show that the set $\Irr(G)$ forms an orthonormal basis for $Cl(G)$.
\begin{defi} Let $G$ be a finite group and $\chi\colon G\to \C$ is a character of $G$. Suppose that $H$ is a subgroup of $G$ and that $G$ is a subgroup of $L$, then
 \begin{itemize}
  \item the \emph{restricted character} $\Res_H^G(\chi)\colon H\to\C$ is defined as $\Res_H^G(\chi)(h):=\chi(h)$ for $h\in H$.
  \item the \emph{induced character} $\Ind_G^L(\chi)\colon L\to\C$ is defined as 
  \[\Ind_G^L(\chi)(l):=\frac{1}{|G|}\sum\limits_{x\in L}\chi^0(x^{-1}lx)\]
  where $l\in L$ and 
  \[\chi^0(l):=\left\{\begin{array}{cc}\chi(l)&\mbox{ if }l\in G\\0&\mbox{ otherwise.} \end{array}\right.\]
 \end{itemize}
\end{defi}
\begin{proposition}[Frobenius reciprocity]\label{prop: frobenius reciprocity}
 Let $H\leq G$ be a subgroup of $G$ set $\chi_H,\chi_G$ to be characters respectively of $H$ and $G$. Then
 \[\langle\Ind_H^G(\chi_H),\chi_G\rangle=\langle\chi_H,\Res_H^G(\chi_G)\rangle. \]
\end{proposition}
We will prove Propositions \ref{prop: character relations} and \ref{prop: frobenius reciprocity} in Chapter \ref{p:4} in the more general context of supercharacter theory.

\section{The symmetric group}\label{section: the sym group}

\begin{defi}
 The \emph{symmetric group} $S_n$ is the group of functions 
 \[\{\sigma\colon \{1,\ldots,n\}\to\{1,\ldots,n\}\mbox{ such that }\sigma \mbox{ is a bijection}\}\]
with multiplication being the composition of functions. Its cardinality is $n!=n(n-1)\cdot\ldots\cdot 1$ and its elements are called \emph{permutations}.
 \end{defi}
If $\sigma\in S_n$ and $\sigma(i)=i$ then $i$ is said to be a \emph{fixed point} for $\sigma$. A permutation with $n-2$ fixed points is called a \emph{transposition}, and a transposition $\sigma\in S_n$ with $\sigma(i)=i+1$ for some  $i<n$ is called an \emph{adjacent transposition}.

We will write permutations in cycle notation, that is
\[\sigma=(i_1,\sigma(i_1),\sigma^2(i_1),\ldots,\sigma^{\rho_1-1}(i_1))\cdot(i_2,\ldots,\sigma^{\rho_2-1}(i_2))\cdot\ldots\cdot(i_r,\ldots,\sigma^{\rho_r-1}(i_r)),\]
where each positive integer $i=1,\ldots, n$ appears only once in the right hand side, and $\sigma^{\rho_a}(i_a)=i_a$. The factor $(i_a,\sigma(i_a),\ldots,\sigma^{\rho_a-1}(i_a))$ is called a \emph{cycle} of length $\rho_a$. The \emph{cycle type} of $\sigma$ is $\rho=(\rho_1,\ldots,\rho_r)$, ordered such that $\rho_1\geq\ldots\geq\rho_r$. For example, the permutation $(1,3,2)(4,5)$ is the bijection 
\[1\mapsto 3,\quad 2\mapsto 1,\quad 3 \mapsto 2,\quad 4\mapsto 5,\quad 5\mapsto 4;\]
this permutation has cycle type $(3,2)$.

An integer \emph{partition} (or, in short, a partition) $\lambda$ of $n$ is a sequence of nonnegative integers $\lambda=(\lambda_1,\ldots,\lambda_l)$ with $\lambda_1\geq\lambda_2\geq\ldots\geq\lambda_l$ and $\sum_i\lambda_i=n$. The elements $\lambda_i$ of a partition $\lambda$ are called the \emph{parts} of $\lambda$. We call $n$ the \emph{size} of $\lambda$, and we will write $\lambda\vdash n$. The \emph{length} $l(\lambda)$ of $\lambda$ is the number of parts. We will sometimes identify a partition with a longer version by adding zero parts. A partition can be also represented as an infinite sequence $(1^{m_1(\lambda)},2^{m_2(\lambda)},\ldots)$, where $m_k(\lambda)$ is the multiplicity of the number $k$ in the partition $\lambda$. By abuse of notation we will also write $\lambda=(1^{m_1(\lambda)},2^{m_2(\lambda)},\ldots)$. Note that the size of $\lambda$ is $|\lambda|=\sum_i i\cdot m_i(\lambda)$ and the length is $l(\lambda)=\sum_i m_i(\lambda)$.

It is well known that integer partitions of $n$ index conjugacy classes of $S_n$. If $\rho\vdash n$ is a partition of $n$, then we call $K_{\rho}$ the conjugacy class indexed by $\rho$, where
\[K_{\rho}=\{\sigma\in S_n\mbox{ s. t. }\sigma \mbox{ has cycle type }\rho\}.\]
The size of $K_{\rho}$ is $|K_{\rho}|= n!/z_{\rho}$, where 
\[z_{\rho}=\prod_{i=1}^{\infty}i^{m_i(\lambda)}\cdot m_i(\lambda)!\]

In this thesis we will deal with asymptotic problems of representation theory, and thus we will consider, for all $n$, the natural inclusion $S_n\hookrightarrow S_{n+1}$ whose image of $\sigma\in S_n$ is $\sigma$ with the additional fixed point $n+1$. The direct limit of this inclusions is the \emph{infinite symmetric group}
\[S_{\infty}:=\{\sigma\colon\N\to\N\mbox{ s. t. }\sigma \mbox{ has a finite number of nonfixed points}\}.\]

\subsection{Young diagrams}\label{sec: young diagrams}

Set $\lambda=(\lambda_1,\lambda_2,\ldots)\vdash n$ to be a partition of $n$. We associate $\lambda$ with its \emph{Young diagram} $S(\lambda)$ represented in English notation, where
\[S(\lambda)=\{(a,b)\in \N^2\mbox{ s. t. }1\leq a\leq l(\lambda),1\leq b\leq \lambda_i\}.\]
We will represent the elements of $S(\lambda)$ as boxes; see Figure \ref{example young diagrams} for an example of Young diagram. We write $\Box=(a,b)\in\lambda$ if $\Box\in S(\lambda)$. We say that $\Box\in\lambda$ is an \emph{outer corner} if there exists another Young diagram with the same shape as $\lambda$ without that box. Likewise, a $\Box\notin\lambda$ is an \emph{inner corner} if there exists a Young diagram with the same shape of $\lambda$ together with $\Box$. For $\Box=(a,b)$ the \emph{content} is defined as $c(\Box)=c(a,b)=b-a$. Of particular interest are the contents of outer and inner corners of $\lambda$, and we call them respectively $y_j$ and $x_j$, ordered in a way such that $x_1<y_1<x_2<\ldots<y_d<x_{d+1}$, where $d$ is the number of outer corners. For an outer corner $\Box$ such that $y_j=c(\Box)$ for some $j$, the partition of $n-1$ obtained from $\lambda$ by removing that box is called $\mu_j$, and we write $\mu_j\nearrow\lambda$, or $c(\lambda/\mu_j)=y_j$. We call such a partition a \emph{subpartition} of $\lambda$. Similarly, $\Lambda_j$ indicates the partition of $n+1$ obtained from $\lambda$ by adding the inner corner box of content $x_j$. 
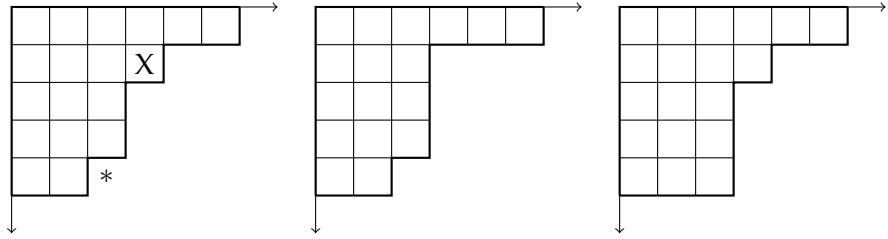
\begin{figure}

 \begin{center}
 \begin{tikzpicture}[scale=0.5]
 \draw[<->] (0,-1)--(0,5)--(7,5);
\draw[thick](0,0)--(2,0)--(2,1)--(3,1)--(3,3)--(4,3)--(4,4)--(6,4)--(6,5);
\draw[thick](0,0)--(0,5)--(6,5);
\draw[thick](8,0)--(10,0)--(10,1)--(11,1)--(11,4)--(12,4)--(14,4)--(14,5);
\draw[<->] (8,-1)--(8,5)--(15,5);
\draw[thick](8,0)--(8,5)--(14,5);
\draw[thick](16,0)--(19,0)--(19,3)--(20,3)--(20,4)--(22,4)--(22,5);
\draw[<->](16,-1)--(16,5)--(23,5);
\draw[thick](16,0)--(16,5)--(22,5);
\draw (1,0)--(1,5);
\draw (2,1)--(2,5);
\draw (3,3)--(3,5);
\draw (4,4)--(4,5);
\draw (5,4)--(5,5);
\draw (9,0)--(9,5);
\draw (10,1)--(10,5);
\draw (11,4)--(11,5);
\draw (12,4)--(12,5);
\draw (13,4)--(13,5);
\draw (17,0)--(17,5);
\draw (18,0)--(18,5);
\draw (19,3)--(19,5);
\draw (20,4)--(20,5);
\draw (21,4)--(21,5);
\draw (0,1)--(2,1);
\draw (0,2)--(3,2);
\draw (0,3)--(3,3);
\draw (0,4)--(4,4);
\draw (8,1)--(10,1);
\draw (8,2)--(11,2);
\draw (8,3)--(11,3);
\draw (8,4)--(11,4);
\draw (16,1)--(19,1);
\draw (16,2)--(19,2);
\draw (16,3)--(19,3);
\draw (16,4)--(20,4);
\node at (2.5,0.5){$\ast$};
\node at (3.5,3.5){X};
 \end{tikzpicture}
 \end{center}
 \caption[Example of Young diagrams]{Example of Young diagrams. Starting from the left: $\lambda=(6,4,3,3,2)$, $\mu_3=(6,3,3,3,2)$ and $\Lambda_2=(6,4,3,3,3)$. }\label{example young diagrams}
 \end{figure}

 \begin{example}\label{duh}
 In Figure \ref{example young diagrams} we show three Young diagrams. The first is the integer partition $\lambda=(6,4,3,3,2)\vdash 18$ where we stress out an inner corner $\begin{array}{c}   \young(\ast) \end{array}$ of content $-2$ and an outer corner $\begin{array}{c}   \young(X) \end{array}$ of content $2$. The second is a subpartition corresponding to removing the box $\begin{array}{c}   \young(X) \end{array}$ from $\lambda$, while the third has the same shape as $\lambda$ with an additional box corresponding to $\begin{array}{c}   \young(\ast) \end{array}$.
\end{example}

For a partition $\lambda\vdash n$ a \emph{standard Young tableau} $T$ is the Young diagram $\lambda$ with each box filled with a different number from $1$ to $n$, such that the entries in each row and column are increasing.

Since the set $\mathcal{P}(n)$ of partitions of $n$ indexes the conjugacy classes of $S_n$, by the results described in the previous section we see that the number of irreducible representations of $S_n$ is also $\mathcal{P}(n)$. We will present in the next section an explicit construction that will associated to every partition $\lambda\vdash n$ an irreducible representation. Given $\lambda\vdash n,$ we call $\pi^{\lambda}$ the associated representation, $\chi^{\lambda}$ its character, and $\dim\lambda$ its dimension, $\dim\lambda=\chi^{\lambda}(\id_{S_n})$. It is a known fact that the set
\[\SYT(\lambda):=\{T: T \mbox{ is a standard Young tableau of shape }\lambda\}\]
indexes a basis for $\pi^{\lambda}$, so that $\dim\lambda=|\SYT(\lambda)|$. See \cite{sagan2013symmetric} for an introduction on the subject.
\begin{example}\label{example on tableaux}
We present the set of standard Young tableaux of shape $\lambda=(3,2)$.
 \[ T_1=\begin{array}{c}\young(123,45)\end{array}\quad T_2=\begin{array}{c}\young(124,35)\end{array}\quad T_3=\begin{array}{c}\young(134,25)\end{array}\quad T_4=\begin{array}{c}\young(125,34)\end{array}\quad T_5=\begin{array}{c}\young(135,24)\end{array}.\]
\end{example}
 Let $\lambda\vdash n$ be a partition, which we identify with its Young diagram in English coordinates. For a box $\Box\in\lambda$ of coordinates $(i,j)$ we define the arm length
\[\arm_{\Box}=\arm_{(i,j)}=\mbox{ number of boxes in }\lambda\mbox{ exactly to the right of }\Box.\]
Similarly, the leg length is
\[\leg_{\Box}=\leg_{(i,j)}=\mbox{ number of boxes in }\lambda\mbox{ exactly below }\Box.\]
Set the \emph{hook length} for a box $(i,j)$ in a Young diagram $\lambda$ to be 
\[h_{(i,j)}=\arm_{(i,j)}+\leg_{(i,j)}+1.\]
We call $\mathcal{H}_{\lambda}$ the multiset of hook lengths of $\lambda$, $\mathcal{H}_{\lambda}=\{h_\Box\mbox{ s.t. }\Box\in\lambda\}$, and $H_\lambda$ the \emph{hook product:}
\[H_\lambda=\prod_{\Box\in\lambda}h_\Box.\]
In Figure \ref{figure: hook lengths} we show the Young diagram $(6,4,3,3,2)$ in which we write the hook length $h(i,j)$ inside each box.

The number of standard Young Tableaux of shape $\lambda$ can be calculated by the well known \emph{hook length formula}:
\[\dim\lambda=\frac{n!}{H(\lambda)}.\]

\begin{figure}
\begin{center}
$\young(\ten97421,7641,542,431,21)\qquad \qquad\young(012345,\minusone012,\minustwo\minusone0,\minusthree\minustwo\minusone,\minusfour\minusthree)$
\end{center}
\caption[Hook lengths and contents of a Young diagram]{On the left, the hook lengths of the partition $\lambda=(6,4,3,3,2)$; on the right, the contents of $\lambda$.}\label{figure: hook lengths}
\end{figure}

\subsection{An explicit construction: Young's orthogonal representation}\label{section young orthogonal}

We recall the definition of Young's orthogonal representation, see \cite{young1977collected} for the original introduction, and \cite{GreeneRationalIdentity} for a more modern description. We need three preliminary definitions; set $\lambda\vdash n$, $k\leq n$, and $T$ a standard Young tableau of shape $\lambda$, then  
\begin{enumerate}
 \item given $\Box\in \lambda$ recall that the \emph{content} $c(\Box)$ is the difference between the row index and the column index of the box. For $k\leq n$ we also write $c_k(T)$ for the content of the box containing the number $k$ in the tableau $T$. For example, in the tableau $T_1$ of Example \ref{example on tableaux}, we have $c_4(T_1)=-1$, $c_1(T_1)=c_5(T_1)=0$, $c_2(T_1)=1$ and $c_3(T_1)=2$.
 \item For $k\leq n-1$ the \emph{signed distance} between $k$ and $k+1$ in the tableau $T$ is 
 \[d_k(T):= c_k(T)-c_{k+1}(T).\]
 For example $d_1(T_1)=d_2(T_1)=d_4(T_1)=-1$; $d_3(T_1)=3$.
 \item If $k$ and $k+1$ are in different columns and rows we define $(k,k+1)T$ as the standard Young tableau equal to $T$ but with the boxes containing $k$ and $k+1$ exchanged. In the previous example $(3,4)T_1=T_2$, while $(2,3)T_1$ is not defined.
\end{enumerate}
\begin{defi}\label{def young orthogonal}
Let $(k,k+1)\in S_n$ be a transposition and $\lambda\vdash n$. The Young's orthogonal representation defines the matrix associated with \hbox{$ \pi^{\lambda}((k,k+1)) $} entrywise in the following way: if $T$ and $S$ are standard Young tableaux of shape $\lambda$, then
\[\pi^{\lambda}((k,k+1))_{T,S}=\left\{\begin{array}{cr}
1/d_k(T)&\mbox{ if } T=S;\\
\sqrt{1-\frac{1}{d_k(T)^2}}&\mbox{ if } (k,k+1)T=S;\\
0&\mbox{else.}\end{array}\right.
\]
\end{defi}

Notice that the adjacent transpositions $(k,k+1)$, $k=1,\ldots, n-1$, generate the group $S_n$, hence $\pi^{\lambda}(\sigma)$ is well defined for all $\sigma\in S_n$. The proof that the matrices built this way satisfy the braid relations can be found in \cite{GreeneRationalIdentity}.
\begin{example}\label{example of matrices}
 Consider $\lambda=(3,2)$, $\sigma=(2,4,3)$ in cycle notation. We compute $\pi^{\lambda}((2,4,3)):$
  \[\pi^{\lambda}((2,4,3))=\pi^{\lambda}((3,4)(2,3))=\pi^{\lambda}((3,4))\pi^{\lambda}((2,3))\]
   \[=\left[\begin{array}{ccccc}
    -1/3&{\scriptstyle\sqrt{8/9}}&0&0&0\\
    {\scriptstyle\sqrt{8/9}}&1/3&0&0&0\\
    0&0&1&0&0\\
    0&0&0&1&0\\
    0&0&0&0&-1
   \end{array}\right]\cdot
  \left[\begin{array}{ccccc}
    1&0&0&0&0\\
    0&-1/2&{\scriptstyle\sqrt{3/4}}&0&0\\
    0&{\scriptstyle\sqrt{3/4}}&1/2&0&0\\
    0&0&0&-1/2&{\scriptstyle\sqrt{3/4}}\\
    0&0&0&{\scriptstyle\sqrt{3/4}}&1/2
   \end{array}\right]\]
 \[=  \left[\begin{array}{ccccc}
    -1/3&-{\scriptstyle\sqrt{2/9}}&{\scriptstyle\sqrt{2/3}}&0&0\\
    {\scriptstyle\sqrt{8/9}}&-1/6&{\scriptstyle\sqrt{1/12}}&0&0\\
    0&{\scriptstyle\sqrt{3/4}}&1/2&0&0\\
    0&0&0&-1/2&{\scriptstyle\sqrt{3/4}}\\
    0&0&0&-{\scriptstyle\sqrt{3/4}}&-1/2
   \end{array}\right]\]
\end{example}

We define also the \emph{last letter order} in $\SYT(\lambda)$: let $T,S$ be standard Young tableaux of shape $\lambda$, then $T\leq S$ if the box containing $n$ lies in a row in $T$ which is lower that $S$; if $n$ lies in the same row in both tableaux, then we look at the rows containing $n-1$, and so on. Notice that in the list of tableaux of shape $(3,2)$ of Example \ref{example on tableaux} the tableaux are last letter ordered; also, the entries of the matrices in Example \ref{example of matrices} are ordered accordingly. We write $\pi^{\lambda}(\sigma)_{i,j}$ instead of $\pi^{\lambda}(\sigma)_{T_i,T_j}$, where $T_i$ is the $i-$th tableau of shape $\lambda$ in the last letter order.
\section{Symmetric functions}\label{section: symmetric functions}
\subsection{The ring of symmetric polynomials}
In this section we recall the fundamental notions of the theory of symmetric functions, following \cite{sagan2013symmetric}, \cite{McDo} and \cite{Stanley:EC2}.
 
 A \emph{weak composition} is a sequence $\underline{\alpha}=(\alpha_1,\alpha_2,\ldots,\alpha_l)$ of nonnegative integers. Note that a partition is a composition whose parts are in nonincreasing order. The definitions of size and length are the same as for partitions. Let $\underline{x}=(x_1,\ldots,x_n)$ be a sequence of formal variables. For a composition $\underline{\alpha}$ of length $n$ we write $\underline{x}^{\underline{\alpha}}=x_1^{\alpha_1}\cdot\ldots\cdot x_n^{\alpha_n}$. Let $\C[\underline{x}]$ be the ring of polynomials in the variables $x_1,\ldots, x_n$, the symmetric group $S_n$ acts on $\C[\underline{x}]$ by permuting the variables: if $\sigma\in S_n$ and $f(x_1,\ldots,x_n)\in \C[\underline{x}]$ then 
 \[\sigma f(x_1,\ldots,x_n)=f(x_{\sigma(1)},x_{\sigma(2)},\ldots,x_{\sigma(n)}).\]
\begin{defi}
 The \emph{ring of symmetric polynomials }$\Lambda_n$ is the subring of $\C[\underline{x}]$ of polynomials invariant under the action of $S_n$: 
 \[\Lambda_n=\C[\underline{x}]^{S_n}.\] 
\end{defi}
Define $\Lambda_n^k$ to be the subspace of homogeneous symmetric polynomials of degree $k$, then 
\[\Lambda_n=\bigoplus_{k\geq 0}\Lambda^k_n,\]
where we set $\Lambda_n^0=\C$. This implies that $\Lambda_n$ is a grading ring, since if $f_1\in\Lambda_n^{k_1}$ and $f_2\in\Lambda_n^{k_2}$ then $f_1\cdot f_2\in\Lambda_n^{k_1+k_2}$.
\begin{defi}
 Let $\lambda=(\lambda_1,\ldots,\lambda_l)$ be a partition of length $l\leq n$. We consider it as a partition of length $n$ by adding zero parts. The \emph{monomial symmetric polynomial} corresponding to $\lambda$ is
 \[m_{\lambda}(\underline{x})=\sum_{\underline{\alpha}}\underline{x}^{\underline{\alpha}},\]
 where the sum runs over all distinct permutations of $\lambda$.
 \end{defi}
 \begin{example}
  Let $\underline{x}=(x_1,x_2,x_3)$ and $\lambda=(3,2,0)$, then
  \[m_{\lambda}(x_1,x_2,x_3)=x_1^3x_2^2+x_1^3x_3^2+x_2^3x_3^2+x_1^2x_2^3+x_1^2x_3^3+x_2^2x_3^3.\]  
 \end{example}

\begin{proposition}
 The set $\{m_{\lambda}\mbox{ s. t. } l(\lambda)\leq n\mbox{ and } |\lambda|=k\}$ is a basis for $\Lambda^k_n$, and $\{m_{\lambda}\mbox{ s. t. } l(\lambda)\leq n\}$ is a basis for $\Lambda_n$. 
\end{proposition}
We present three more families of symmetric polynomials: set 
\begin{itemize}
 \item $p_k(\underline{x}):=m_{(k)}(\underline{x})=\sum\limits_{i=1}^n x_i^k;$
 \item $e_k(\underline{x}):=m_{(1^k)}(\underline{x})=\sum\limits_{1\leq i_1<\ldots<i_k\leq n} x_{i_1}\cdot x_{i_2}\cdot\ldots\cdot x_{i_k},$ where the partition $(1^k)$ is the partition $(1,\ldots, 1)$ of size $k$;
 \item $h_k(\underline{x}):=\sum_{\lambda\vdash k}m_{\lambda}(\underline{x})=\sum\limits_{1\leq i_1\leq\ldots\leq i_k\leq n} x_{i_1}\cdot x_{i_2}\cdot\ldots\cdot x_{i_k}.$ 
\end{itemize}

\begin{defi}
 Let $\lambda=(\lambda_1,\ldots,\lambda_l)$ be a partition. We define:
 \begin{itemize}
  \item the \emph{power sum symmetric polynomial} $p_{\lambda}(\underline{x})=p_{\lambda_1}(\underline{x})\cdot\ldots\cdot p_{\lambda_l}(\underline{x})$;
  \item the \emph{elementary symmetric polynomial} $e_{\lambda}(\underline{x})=e_{\lambda_1}(\underline{x})\cdot\ldots\cdot e_{\lambda_l}(\underline{x})$;
  \item the \emph{complete homogeneous symmetric polynomial} $h_{\lambda}(\underline{x})=h_{\lambda_1}(\underline{x})\cdot\ldots\cdot h_{\lambda_l}(\underline{x})$;
 \end{itemize}
\end{defi}
\begin{example}Let, as before, $\underline{x}=(x_1,x_2,x_3)$ and $\lambda=(3,2,0)$, then
\begin{itemize}
 \item $p_{\lambda}(x_1,x_2,x_3)=(x_1^3+x_2^3+x_3^3)(x_1^2+x_2^2+x_3^2)=x_1^5+x_1^3x_2^2+x_1^3x_3^2+x_2^3x_3^2+x_2^5+x_1^2x_2^3+x_1^2x_3^3+x_2^2x_3^3+x_3^5;$
 \item $e_{\lambda}(x_1,x_2,x_3)=(x_1x_2x_3)(x_1x_2+x_1x_3+x_2x_3)=x_1^2x_2^2x_3+x_1^2x_2x_3^2+x_1x_2^2x_3^2;$
 \item $\begin{aligned}[t]h_{\lambda}(x_1,x_2,x_3)&=\begin{multlined}[t](x_1^2+x_2^2+x_3^2+x_1x_2+x_1x_3+x_2x_3)\cdot\\\hspace{-.5cm}\cdot(x_1^3+x_2^3+x_3^3+x_1^2x_2+x_1^2x_3+x_2^2x_3+x_1x_2^2+x_1x_3^2+x_2x_3^2+x_1x_2x_3)\end{multlined}\\ 
 &=\begin{multlined}[t]
     x_1^5 + 2 x_1^4 x_2 + 3 x_1^3 x_2^2 + 3 x_1^2 x_2^3 + 2 x_1 x_2^4 + x_2^5 + 2 x_1^4 x_3 +  4 x_1^3 x_2 x_3 +\\ 5 x_1^2 x_2^2 x_3 + 4 x_1 x_2^3 x_3 + 2 x_2^4 x_3 + 3 x_1^3 x_3^2 +  5 x_1^2 x_2 x_3^2 + 5 x_1 x_2^2 x_3^2 + 3 x_2^3 x_3^2 + \\3 x_1^2 x_3^3 + 4 x_1 x_2 x_3^3 +  3 x_2^2 x_3^3 + 2 x_1 x_3^4 + 2 x_2 x_3^4 + x_3^5
   \end{multlined}
\end{aligned}$
\end{itemize}

\end{example}

\subsection{Symmetric functions}
Let $m\geq n$ and consider the projection map $\C[x_1,\ldots,x_m]\to\C[x_1,\ldots,x_n]$. If we restrict this map to $\Lambda^k_n$ we obtain the projection
\[\begin{array}{ccc}
\Lambda^k_m&\to&\Lambda^k_n\\
m_{\lambda}(x_1,\ldots,x_m)&\mapsto&\left\{\begin{array}{ll}m_{\lambda}(x_1,\ldots,x_n,0,\ldots,0)=m_{\lambda}(x_1,\ldots,x_n)&\mbox{ if }l(\lambda)\leq n\\0&\mbox{ otherwise}\end{array}\right.
  \end{array}\]
for $\lambda\in\mathcal{P}(k)$. Hence we can define the inverse limit $\Lambda^k$ of $\Lambda^k_n$ as the \emph{ring of symmetric functions of degree $k$}:
\[\Lambda^k=\lim_{\leftarrow}\Lambda^k_n.\]
When dealing with functions in this ring we will sometimes drop the argument and write $f$ for $f(\underline{x})$. On the other hand, if $f\in \Lambda$ is a symmetric function we will write $f_{\vert n}$ for its restriction to $\Lambda_n$, that is, 
\[f_{\vert n}(x_1,\ldots,x_n)=f(x_1,\ldots,x_n,0,0,\ldots).\]

For a partition $\lambda\in \mathcal{P}(k)$ we call monomial functions $m_{\lambda}$ the image of the monomial polynomial $m_{\lambda}(x_1,\ldots,x_n)$ under the inverse limit. Specifically, $m_{\lambda}$ is the sequence 
\[m_{\lambda}=m_{\lambda}(x_1,x_2,\ldots)=\left(m_{\lambda}(x_1),m_{\lambda}(x_1,x_2),m_{\lambda}(x_1,x_2,x_3),\ldots\right).\]
Similarly, let $e_{\lambda},p_{\lambda},h_{\lambda}$ be respectively the \emph{elementary, power sum, and complete homogeneous symmetric functions}.
\begin{proposition}
 The sets $\{m_{\lambda}, \lambda\vdash k\}$, $\{p_{\lambda}, \lambda\vdash k\}$, $\{e_{\lambda}, \lambda\vdash k\}$, and $\{h_{\lambda}, \lambda\vdash k\}$ are $\C$-bases for the ring $\Lambda^k$. 
\end{proposition}
\begin{defi}
 The ring $\Lambda=\bigoplus_{k\geq 0}\Lambda^k$ is the \emph{ring of symmetric functions}, which is a graded ring spanned by $\{m_{\lambda}, \lambda\mbox{ a partition}\}$.
 
\end{defi}

\subsection{Schur functions}
Let $\lambda=(\lambda_1,\ldots,\lambda_l)$ be a partition. A \emph{semistandard Young tableau} $T$ of shape $\lambda$ is a filling of the Young diagram $\lambda$ with positive integers such that every row is weakly increasing and every column is strictly increasing. Let $SSYT(\lambda)$ be the set of semistandard Young tableaux of shape $\lambda$.

For $T\in SSYT(\lambda)$, define 
\[\underline{x}^T:=\prod_{\Box\in \lambda}x_{T(\Box)},\]
where $T(\Box)$ is the entry $\Box=(i,j)$ in $T$.
\begin{defi}\label{defi: schur functions}
 The \emph{Schur function} associated to the partition $\lambda$ is
 \[s_{\lambda}(\underline{x}):=\sum_{T\in SSYT(\lambda)}\underline{x}^T.\]
\end{defi}
It is nontrivial to see that Schur functions are symmetric, and in fact they form a basis for $\Lambda$. More precisely
\begin{proposition}
 The set $\{s_{\lambda}(\underline{x})\mbox{ s. t. }l(\lambda)\leq n, |\lambda|=k\}$ is a basis for $\Lambda^k_n$.
\end{proposition}
\begin{example}
 There are $15$ semistandard Young tableaux of shape $\lambda=(3,2)$ such that each entry is $\leq 3$. The corresponding Schur function is
 \begin{multline*}
   s_{(3,2)}(x_1,x_2,x_3)=x_1^3x_2^2+x_1^3x_2x_3+x_1^3X_3^2+  x_1^2x_2^3+2x_1^2x_2^2x_3+2x_1^2x_2x_3^2+x_1^2x_3^3+\\   x_1x_2^3x_3+2x_1x_2^2x_3^2+x_1x_2x_3^3+x_2^3x_3^2+x_2^2x_3^3.
 \end{multline*}
 \end{example}
Schur functions have several interesting properties; we recall some basic ones, each of which can be taken as an equivalent definition for Schur function. For a proof of these statements, see \cite{sagan2013symmetric},\cite{McDo} or \cite{Stanley:EC2}. Consider a partition $\lambda\vdash k$ with $l(\lambda)\leq n$.
\begin{itemize}
 \item If $\mu$ is another partition of $k$ then the \emph{Kostka number} of $(\lambda,\mu)$ is 
 \[K_{\lambda,\mu}=\sharp\{T\in SSYT(\lambda)\mbox{ s. t. }T\mbox{ has weight }\mu\},\]
 where a semistandard Young tableau $T$ has weight $\mu$ if the number $i$ appears $\mu_i$ times in $T$. In particular $K_{\lambda,\lambda}=1$. Then 
 \[s_{\lambda}=\sum_{\mu\vdash k}K_{\lambda,\mu} m_{\mu}.\]
 \item The Schur functions are uniquely defined via their expansion in the basis of complete homogeneous symmetric functions:
 \[s_{\lambda}=\det(h_{\lambda_i-i+j}).\]
 Equivalently,
 \[s_{\lambda'}=\det(e_{\lambda_i-i+j}),\]
 where $\lambda'$ is the conjugate of $\lambda$. These formulas are known as the \emph{Jacobi-Trudi} identities.
 \item The following is known as the determinantal formula: 
  \begin{equation}\label{eq: def schur via determinantal formula}
   s_{\lambda}(x_1,\ldots, x_l)=\frac{\det(x_i^{\lambda_j+n-j})}{\det(x_i^{n-j})}.
 \end{equation}
 
\item The Schur functions can be defined through their expansion in the power sum symmetric functions:
\begin{equation}\label{eq: schur expansion via power sums}
 s_{\lambda}=\frac{1}{n!}\sum_{\sigma\in S_n}\chi^{\lambda}(\sigma)p_{c(\sigma)}=\sum_{\mu\vdash n}\frac{\chi^{\lambda}_{\mu}}{z_{\mu}}p_{\mu},
\end{equation}
where $c(\sigma)$ is the cycle type of $\sigma$, and $\chi^{\lambda}$ is the irreducible character indexed by $\lambda$. 
\end{itemize}
The last property shows us that irreducible characters appear in the transition matrix between Schur functions and power sums, which is a cornerstone result of the deep connections between the representation theory of the symmetric group and symmetric functions.

Let $\Cl^n=\Cl(S_n)$ be the algebra of class functions on $S_n$, with $\{\chi^{\lambda}\}_{\lambda\vdash n}$ as an orthonormal basis. For $f\in \Cl^n$ and $\mu\vdash n$ we write $f(\mu)=f(\sigma)$, where $\sigma$ is any permutation of cycle type $c(\sigma)=\mu$.
\begin{defi}
 The \emph{Frobenius characteristic map} is
 \[\begin{array}{cccc}\ch^{(n)}\colon&\Cl^n&\to&\Lambda^n\\&f&\mapsto&\sum\limits_{\mu\vdash n}\frac{1}{z_{\mu}}f(\mu)p_{\mu}\end{array} \]
\end{defi}
\begin{example}
 \begin{itemize}
  \item If $\id$ is the identity function on $S_n$ then
  \[\ch^{(n)}(\id)=\sum_{\mu\vdash n}\frac{1}{z_{\mu}}p_{\mu}=h_n.\]
  \item If $\sgn(\sigma)$ is the sign of the permutation $\sigma$ then $\sgn(\mu)=(-1)^{n-l(\mu)}$ and
   \[\ch^{(n)}(\sgn)=\sum_{\mu\vdash n}\frac{(-1)^{n-l(\mu)}}{z_{\mu}}p_{\mu}=e_n.\]
   \item On the orthonormal basis $\{\chi^{\lambda}\}$ the Frobenius characteristic map acts thus
   \[\ch^{(n)}(\chi^{\lambda})=\sum_{\mu\vdash n}\frac{\chi^{\lambda}_{\mu}}{z_{\mu}}p_{\mu}=s_{\lambda}.\]
 \end{itemize}
\end{example}

 The ring $\Lambda$ inherits an inner product which makes the map $\ch$ an isometry: for $f,g\in\Cl$
 \[\langle \ch(f),\ch(g)\rangle:=\langle f,g\rangle.\]
 In particular
 \[\langle s_{\lambda},s_{\mu}\rangle=\delta_{\lambda,\mu},\qquad \langle m_{\lambda},h_{\mu}\rangle=\delta_{\lambda,\mu},\qquad \langle p_{\lambda},p_{\mu}\rangle=z_{\lambda}\delta_{\lambda,\mu}.\]

\cleardoublepage\chapter{The dual approach}\label{ch: probability}
\section{Introduction}
In the last 30 years the subject of asymptotic representations of the symmetric group benefited from the dual combinatorics approach, and gathered a lot of attention, spreading in several directions. In this chapter we attempt but a brief introduction in some of the different approaches in this vast and growing area.

We start with an overview on Bratteli diagrams (Section \ref{sec: bratteli}), which represent the right setting for asymptotic questions of representation theory for a general finite group $G$. Bratteli diagrams (sometimes referred to as Bratteli-Vershik diagrams in the representation theory environment) will appear also in Chapter \ref{ch: strict partitions}, where we study the projective characters of the symmetric group indexed by strict partitions, and Chapter \ref{ch: supercharacters}, where we introduce a new Bratteli diagram of set partitions motivated by a supercharacter theory of the upper unitriangular group. Two measures arise naturally from Bratteli diagrams: the \emph{transition} and \emph{co-transition} measures. These measures are connected with the Plancherel measure in the case of Bratteli diagrams associated to the representation theory of the group $G$.

We focus afterwards on the symmetric group, showing the change of coordinates for Young diagrams called the \emph{Russian coordinates}, ideal for the study of asymptotic questions. In Section \ref{section: asymptotic of diagrams} we state the limit shape theorem. We present the aforementioned moments $\overline{p}_k(\lambda)$ and the algebra $\mathbb{A}$, called the \emph{algebra of polynomial functions on Young diagrams}. We describe in Section \ref{section: asymptotic p sharp} the central limit theorem for opportunely renormalized characters. We then introduce in Section \ref{sec: shifted symmetric functions} the \emph{algebra of shifted symmetric functions} $\Lambda^{\ast}$, which is just $\mathbb{A}$ seen from an algebraic perspective, and the algebra of partial permutations $\mathcal{B}_{\infty}$ in Section \ref{sec: partial permutations}. We then describe in Section \ref{sec: generating function} the connection between $\overline{p}_k(\lambda)$ and the character $\chi^{\lambda}_{(k,1^{n-k})}$, that is, the character $\chi^{\lambda}$ evaluated on the single cycle of length $k$. This connection relies on the generating function for the partition $\lambda$, called $\phi(\lambda;z)$, where $z\in\C$. The function $\phi(\lambda;z)$ is also useful for the computation of the asymptotic of the transition measure (Section \ref{sec: the transition measure}). For the symmetric group it was originally proved by Kerov \cite{kerov1993transition} that the transition distribution function associated to a random Plancherel partition $\lambda$ converges to the semicircular law. From his proof we can easily deduce a similar result for the co-transition measure, which appears in Section \ref{sec: the co-transition}.


\section{Bratteli diagrams}\label{sec: bratteli}
\begin{defi}
 A \emph{directed graph} is an ordered pair $\Gamma=(V,E)$ where $V$ is called the \emph{vertex set} and $E=\{(v,w)\mbox{ s. t. }v,w\in V\}$ is called the \emph{edge set}.
 
 A directed graph is \emph{graded} if there exists a function $\rk\colon V\to\Z$, named the \emph{rank function}, such that for each $(v,w)\in E$ we have $\rk(w)=\rk(v)+1$. 
\end{defi}
We sometimes identify a directed graded graph $\Gamma$ with its set of vertices, and we write $\Gamma=\bigcup_{n\geq 0}\Gamma_n$, where
\[\Gamma_n=\{v\in V\mbox{ s. t. }\rk(v)=n\}.\]
We consider graphs in which multiple edges are allowed or, more generally, graphs with a multiplicity function $\kappa\colon E\to\R_{\geq 0}$, which counts the multiplicity of an edge.
\begin{defi}\label{defi: bratteli diagram}
 A \emph{Bratteli diagram} is a directed graded graph $\Gamma=\bigcup_{n\geq 0}\Gamma_n$ such that
 \begin{enumerate}
  \item The set of vertices $\Gamma_0$ has a single vertex called $\emptyset$.
  \item Every vertex has at least one outgoing edge.
  \item All $\Gamma_n$ are finite.  
 \end{enumerate}
 If $\lambda\in \Gamma_n$ and $\Lambda\in \Gamma_{n+1}$ we write $\lambda\nearrow\Lambda$ if $(\lambda,\Lambda)\in E$. 
\end{defi}
Let $m\geq n$ and consider a sequence $\lambda_n\nearrow\lambda_{n+1}\nearrow\ldots\nearrow\lambda_m$ with $\lambda_i\in \Gamma_i$. We call such a sequence a \emph{path} from $\lambda_n$ to $\lambda_m$. The \emph{weight} of a path is the product of the multiplicities of its edges: $\prod_{i=n}^{m-1} \kappa(\lambda_i,\lambda_{i+1})$. For a vertex $\lambda\in \Gamma_n$ its \emph{dimension} $\dim\lambda$ is the sum of the weights of the paths from $\emptyset$ to $\lambda$:
\[\dim\lambda=\sum_{\emptyset\nearrow\lambda_1\nearrow\ldots\nearrow\lambda}\prod_{i=0}^{n-1} \kappa(\lambda_i,\lambda_{i+1}),\]
where the sum is taken over all the possible paths $\emptyset\nearrow\lambda_1\nearrow\ldots\nearrow\lambda$ from $\emptyset$ to $\lambda$. It is clear that for $\Lambda\in\Gamma_{n+1}$ then
\begin{equation}\label{eq: equivalent definition dimension}
 \dim\Lambda=\sum_{\lambda\in \Gamma_n}\dim\lambda\cdot\kappa(\lambda,\Lambda)\qquad\mbox{ and }\qquad \dim\emptyset=1.
\end{equation}
Note that Equation \eqref{eq: equivalent definition dimension} can be considered an equivalent definition for the dimension $\dim\Lambda$. 
\begin{defi}
 Let $\Gamma$ be a Bratteli diagram with multiplicity function $\kappa$. A \emph{harmonic function} $\phi\colon \Gamma\to\R_{\geq 0}$ is a function such that 
 \begin{enumerate}
  \item $\phi(\emptyset)=1;$
  \item $\phi(\lambda)=\sum\limits_{\Lambda\colon\lambda\nearrow\Lambda}\kappa(\lambda,\Lambda)\phi(\Lambda)$.  
 \end{enumerate}
\end{defi}
Equivalently, by setting $M_n(\lambda)=\phi(\lambda)\cdot\dim\lambda$ one obtains a set of measures $\{M_n\}$ on $\Gamma_n$ such that 
\begin{equation}\label{eq: coherent measures}
 M_n(\lambda)=\sum_{\Lambda\colon \lambda\nearrow\Lambda}\frac{\dim\lambda\cdot\kappa(\lambda,\Lambda)}{\dim\Lambda}M_{n+1}(\Lambda).
\end{equation}
A set of measures $\{M_n\}$ which respects Equation \eqref{eq: coherent measures} is called \emph{coherent}. It is known that the measures $M_n$ are actually probabilities, see for example \cite{kerov1994boundary}.
\begin{defi}\label{def: transition and co-transition}
 Let $\Gamma$ be a Bratteli diagram, $\{M_n\}$ a set of coherent measures, and let $\lambda\nearrow\Lambda$. Then
 \begin{itemize}
  \item The measure $\tr(\lambda,\Lambda)=\frac{\dim\lambda\cdot M_{n+1}(\Lambda)\kappa(\lambda,\Lambda)}{\dim\Lambda\cdot M_n(\lambda)}$ is called the \emph{transition measure}.
  \item The measure $\ctr(\lambda,\Lambda)=\frac{\dim\lambda\cdot\kappa(\lambda,\Lambda)}{\dim\Lambda}$ is called the \emph{co-transition measure}.
 \end{itemize} 
\end{defi}
Since $\{M_n\}$ are coherent measures then it is immediate to see that
\[\sum_{\Lambda\in\Gamma_{n+1}}\tr(\lambda,\Lambda)=1,\qquad \sum_{\lambda\in\Gamma_{n}}M_n(\lambda)\tr(\lambda,\Lambda)=M_{n+1}(\Lambda)\]
and, similarly
\[\sum_{\lambda\in\Gamma_{n}}\ctr(\lambda,\Lambda)=1,\qquad \sum_{\Lambda\in\Gamma_{n+1}}M_{n+1}(\Lambda)\ctr(\lambda,\Lambda)=M_{n}(\lambda).\]

\subsection{Bratteli diagrams and representation theory}\label{section: bratteli for representation}
Let $G_0=\emptyset\hookrightarrow G_1=\{\id_{G_1}\}\hookrightarrow G_2\hookrightarrow\ldots$ be a sequence of groups. In this section we build a Bratteli diagram associated to the sets of irreducible characters of these groups. Set $\Gamma_n$ to be a set of indices for $\Irr(G_n)$ so that $\lambda\in\Gamma_n$ if $\chi^{\lambda}\in\Irr(G_n)$. We consider the directed, graded graph $\Gamma=(V,E)$ where $V=\bigcup\Gamma_n$ and $\rk(\lambda)=n$ if $\lambda\in\Gamma_n$; for $\lambda\in\Gamma_n$, $\Lambda\in\Gamma_{n+1}$ we set $(\lambda,\Lambda)\in E$ if $\langle\Ind_{G_n}^{G_{n+1}}(\chi^{\lambda}),\chi^{\Lambda}\rangle\neq 0$, and moreover we define the multiplicity of the edge $\kappa(\lambda,\Lambda)=\langle\Ind_{G_n}^{G_{n+1}}(\chi^{\lambda}),\chi^{\Lambda}\rangle$. It is clear then that $\Gamma$ is a Bratteli diagram. 

We claim that, for $\Lambda\in\Gamma_{n+1}$,
\[\chi^{\Lambda}(\id_{G_{n+1}})=\sum_{\emptyset\nearrow\ldots\nearrow\Lambda}\prod_{i=1}^n\kappa(\lambda_i,\lambda_{i+1}),\]
so that the ``representation theoretic'' definition of the dimension of an irreducible character and the ``Bratteli diagram theoretic'' definition of the dimension of a vertex coincide. Let $\dim\lambda=\chi^{\lambda}(\id_{G_n})$, then from Equation \eqref{eq: equivalent definition dimension} to prove the claim it is enough to show that 
\[\dim\Lambda=\chi^{\Lambda}(\id_{G_{n+1}})=\sum_{\lambda\in\Gamma_n}\dim\lambda\cdot\kappa(\lambda,\Lambda).\]
Note that 
\[\sum_{\lambda\in\Gamma_n}\dim\lambda\cdot\kappa(\lambda,\Lambda)=\sum_{\lambda\in\Gamma_n}\chi^{\lambda}(\id_{G_n})\langle\Ind_{G_n}^{G_{n+1}}(\chi^{\lambda}),\chi^{\Lambda}\rangle=\sum_{\lambda\in\Gamma_n}\chi^{\lambda}(\id_{G_n})\langle\chi^{\lambda},\Res_{G_n}^{G_{n+1}}(\chi^{\Lambda})\rangle.\]
The claim is apparent when one realizes that $\chi^{\Lambda}(\id_{G_{n+1}})=\Res_{G_n}^{G_{n+1}}(\chi^{\Lambda})(\id_{G_n})$ and 
\[\Res_{G_n}^{G_{n+1}}(\chi^{\Lambda})=\sum_{\lambda\in\Gamma_n}\langle\chi^{\lambda},\Res_{G_n}^{G_{n+1}}(\chi^{\Lambda})\rangle\chi^{\lambda}\]
since $\Irr(G_n)$ forms an orthonormal basis for the class functions on $G_n$.

A consequence of Frobenius reciprocity is that if $P_{\Pl}^n$ is the Plancherel measure of $G_n$, then $\{P_{\Pl}^n\}$ is a set of coherent measures, according to Equation \eqref{eq: coherent measures}. As a consequence, the function $\phi(\lambda):=\dim\lambda/|G_n|$ is a harmonic function. The transition and co-transition measures are respectively
\[\tr(\lambda,\Lambda)=\frac{|G_n|}{|G_{n+1}|}\frac{\dim\Lambda}{\dim\lambda}\langle\Ind_{G_n}^{G_{n+1}}(\chi^{\lambda}),\chi^{\Lambda}\rangle,\qquad \ctr(\lambda,\Lambda)=\frac{\dim\lambda}{\dim\Lambda}\langle\Ind_{G_n}^{G_{n+1}}(\chi^{\lambda}),\chi^{\Lambda}\rangle.\]

\begin{example}\label{example: bratteli integer partitions}
 The \emph{Young lattice} is defined to be the graph $\Y=\bigcup\Y_n$ of all Young diagrams, where $\Y_n=\{\lambda\vdash n\}$, and we draw an edge from $\lambda$ to $\Lambda$ (and write $\lambda\nearrow\Lambda$) if $\Lambda$ can be obtained from $\lambda$ by adding one box. See Figure \ref{figure: young lattice} for the beginning of the Young lattice. It is immediate to see that this is a Bratteli diagram and, moreover, this graph is simple, that is, all edge multiplicities are either $0$ or $1$.
 
 For $\lambda\in\Y_n$ one can identify paths $\emptyset\nearrow\ldots\nearrow\lambda$ with standard Young diagrams $T\in\SYT(\lambda)$, so that the definition of $\dim\lambda$ according to the theory of Bratteli diagrams correspond to the definition of $\dim\lambda$ in representation theory:
 \[\dim\lambda=\sharp\SYT(\lambda).\]
 In this case $\tr(\lambda,\Lambda)=\frac{\dim\Lambda}{(n+1)\dim\lambda}$ and $\ctr(\lambda,\Lambda)=\frac{\dim\lambda}{\dim\Lambda}$ if $\lambda\nearrow\Lambda$. More on the topic of the Bratteli diagram associated to symmetric groups can be found in \cite{fulman2005stein} and \cite{kerov1994boundary}.
\end{example}

\begin{figure}\scalebox{.5}{$\xymatrix{\vuoto&&\vuoto&\vuoto&\vuoto&&\vuoto&\vuoto&\vuoto&&\vuoto\\
 &\provaquattroa\ar[lu]\ar[ru]&&\provaquattrob\ar[lu]\ar[ru]\ar[u]&&\provaquattroc\ar[lu]\ar[ru]&&\provaquattrod\ar[lu]\ar[ru]\ar[u]&&\provaquattroe\ar[lu]\ar[ru]\\
 &&\provatrea\ar[lu]\ar[ru]&&&\provatreb\ar[llu]\ar[rru]\ar[u]&&&\provatrec\ar[lu]\ar[ru]\\
 &&&\provaduea\ar[lu]\ar[rru]&&&&\provadueb\ar[llu]\ar[ru]\\
 &&&&&\provauno\ar[llu]\ar[rru]\\
 &&&&&\emptyset\ar[u]}$}\caption{Beginning of the Young lattice up to $\Y_4$.}\label{figure: young lattice}\end{figure}
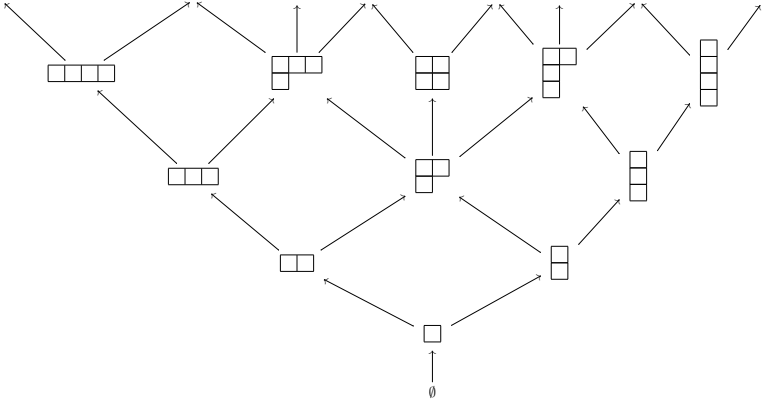

In  Chapters \ref{ch: strict partitions} and \ref{ch: supercharacters} we will present two other examples of Bratteli lattices inherited from representation theory: the first is a lattice of strict partitions (or, equivalently, shifted Young diagrams) and the second is a lattice of set partitions.

\section{Asymptotic of Plancherel distributed Young diagrams}\label{section: asymptotic of diagrams}
In 1977 the asymptotical behavior of Plancherel distributed Young diagrams was proved independently by Logan and Shepp \cite{logan1977variational} and Kerov and Vershik \cite{kerov1977asymptotics}. The result was particularly relevant for its connections via the RSK algorithm to asymptotic problems of statistics on uniform random permutations, for example the problem of the largest increasing subsequence of a uniform random partition (see for example \cite{romik2015surprising}). In this section we recall this result, introducing some tools which will be fundamental in the study of the asymptotic of Plancherel distributed random characters of the symmetric group. The majority of the results appearing here can be found in the article of Ivanov and Olshanski \cite{ivanov2002kerov}, with most of the ideas due to Kerov, see \cite{kerov1993gaussian}, \cite{KerovOlshanski1994} and \cite{IvanovKerov1999}.

Fix $n\in \N$ and let $\mathbb{Y}_n=\{\lambda\vdash n\}$ be the set of Young diagrams of size $n$, written in English notation as in Section \ref{section: repr theory}. Let $P_{\Pl}^n$ be the Plancherel measure on $\mathbb{Y}_n$, that is,
\[P_{\Pl}^n(\lambda):=\frac{\dim\lambda^2}{n!}\]
for $\lambda\in\mathbb{Y}_n$, where $\dim\lambda=\chi^{\lambda}(\id_{S_n})$. In order to describe the asymptotic behavior of Plancherel distributed random Young diagrams, we need to define a notion of convergence; this will be provided by the presentation of Young diagrams in Russian coordinates, defined as follows: consider $\lambda\in\mathbb{Y}_n\subset \R_{\geq 0}^2$ and define $\partial\lambda$ to be the polygonal line which overlaps the line $x=0$ for $y\geq l(\lambda)$, goes through the border of $\lambda$ in $\R_{>0}^2$, and then overlaps the line $y=0$ for $x\geq \lambda_1$. See Figure \ref{example russian coordinates} for an example. The same diagram $\lambda$ in Russian notation corresponds to the change of coordinates
\[\left\{\begin{array}{c}r=y-x\\s=x+y\end{array}\right.\]
\begin{figure}
 \begin{center}
 \begin{tikzpicture}[scale=0.5]
 \draw[<->] (0,-2)--(0,5)--(8,5);
\draw(0,0)--(2,0)--(2,1)--(3,1)--(3,3)--(4,3)--(4,4)--(6,4)--(6,5);
\draw(0,0)--(0,5)--(6,5);
\node [below] at (8,5){$x$};
\node [right] at (0,-2){$y$};
\draw (1,0)--(1,5);
\draw (2,1)--(2,5);
\draw (3,3)--(3,5);
\draw (4,4)--(4,5);
\draw (5,4)--(5,5);
\draw (0,1)--(2,1);
\draw (0,2)--(3,2);
\draw (0,3)--(3,3);
\draw (0,4)--(4,4);
\draw[very thick](0,-2)--(0,0)--(2,0)--(2,1)--(3,1)--(3,3)--(4,3)--(4,4)--(6,4)--(6,5)--(8,5);
\node [below] at (2.5,0.5){$\partial\lambda$};
 \end{tikzpicture}\qquad
 \begin{tikzpicture}[scale=0.75]
  \draw[->](-4,0)--(4,0);
  \node [above] at (4,0){$r$};
\node [right] at (0,5){$s$};
  \draw[->](0,-1)--(0,5);
  \draw (-3,3)--(0,0)--(3,3);
  \draw[very thick] (-4,4)--(-2.5,2.5)--(-1.5,3.5)--(-1,3)--(-0.5,3.5)--(0.5,2.5)--(1,3)--(1.5,2.5)--(2.5,3.5)--(3,3)--(4,4);
  \draw [thin] (-2,2)--(-1,3);
  \draw [thin] (-1.5,1.5)--(0,3);
  \draw [thin] (-1,1)--(0.5,2.5);
  \draw [thin] (-0.5,0.5)--(1.5,2.5);
  \draw [thin] (0.5,0.5)--(-2,3);
  \draw [thin] (1,1)--(-1,3);
  \draw [thin] (1.5,1.5)--(0.5,2.5);
  \draw [thin] (2,2)--(1.5,2.5);
  \draw [thin] (2.5,2.5)--(2,3); 
  \node[above] at (2.5,3.5){$\lambda(r)$};
 \end{tikzpicture}

 \end{center}
 \caption[Young diagram in English and Russian coordinates]{On the left, the Young diagram $\lambda=(6,4,3,3,2)$ in English coordinates with highlighted the border $\partial \lambda$. On the right, the same Young diagram in Russian coordinates.}\label{example russian coordinates}
 \end{figure}
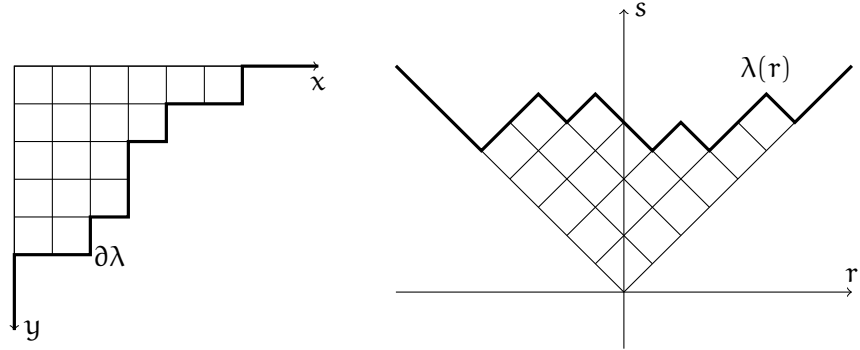
as in figure \ref{example russian coordinates}. Then the border $\partial\lambda$ can be identified with a piece-wise linear function $\lambda(\cdot)$ such that $\lambda(x)=|x|$ for $|x|$ large enough. The following definition appeared first in \cite{kerov1993transition}.
\begin{defi}
 A \emph{continual diagram centered in $0$} is a function $\omega(x)$ on $\R$ such that:
 \begin{enumerate}
  \item the function $\omega$ is $1-$Lipschitz, that is, $|\omega(x_1)-\omega(x_2)|\leq |x_1-x_2|$ for all $x_1,x_2\in\R$;
  \item we have $\omega(x)=|x|$ for $|x|$ large enough.
 \end{enumerate}
The set of all continual diagrams centered in $0$ is denoted $\mathcal{D}^0$. 
\end{defi}
The transformation into Russian coordinates corresponds to an embedding $\mathbb{Y}_n\hookrightarrow\mathcal{D}^0$. Note that the area between $\lambda(x)$ and $|x|$ is exactly $n$, that is,
\[\int_{-\infty}^{\infty}(\lambda(x)-|x|)\,dx=n.\]
Set $\overline{\lambda}(\cdot)$ to be the continual diagram defined by 
\[\overline{\lambda}(x)=\frac{1}{\sqrt{n}}\lambda(\sqrt{n}\cdot x), \qquad x\in\R.\]
We can now define the asymptotic shape of a Plancherel distributed Young diagram.
\begin{defi}
 The continual diagram $\Omega(x)$ centered in $0$ is defined as
 \[\Omega(x):=\left\{\begin{array}{ll}\frac{2}{\pi}(x\arcsin\frac{x}{2}+\sqrt{4-x^2})& \mbox{ if }|x|\leq 2,\\|x|&\mbox{ if }|x|\geq 2.\end{array}\right.\]
\end{defi}
\begin{theorem}[Law of large numbers for Young diagrams]\label{theorem: Law of large numbers for Young diagrams}
Let $\lambda$ range over $\mathbb{Y}_n$, and consider $\overline{\lambda}(\cdot)$ as a random function in the probability space $(\mathbb{Y}_n,P_{\Pl}^n)$. Then 
\[\lim\limits_{n\to\infty}\sup\limits_{x\in\R}|\overline{\lambda}(x)-\Omega(x)|=0\]
in probability.
\end{theorem}
In Figure \ref{example limit shape} we show the function $\Omega$ sided by a large partition.

Theorem \ref{theorem: Law of large numbers for Young diagrams} is based on the study of the moments of the random variable
\begin{equation}\label{eq: def p tilde}
 \overline{p}_k(\overline{\lambda}(\cdot)):=k(k-1)\int_{-\infty}^{\infty}\frac{\overline{\lambda}(x)-|x|}{2}x^{k-2}\,dx,\qquad k=2,3,\ldots
\end{equation}

\begin{figure}
 \begin{center}
 \begin{tikzpicture}[scale=1.3]
\draw(-2.5,0)--(2.5,0);
  \draw(0,-.1)--(0,2.5);
      	\foreach \x in {-12,...,12}
     		\draw (\x/5,0pt) -- (\x/5,-1pt);
  \draw (-2,0pt) -- (-2,-2pt);
  \node[below] at (-2,0){-2};
  \draw (-1,0pt) -- (-1,-2pt);
  \node[below] at (-1,0){-1};
  \draw (1,0pt) -- (1,-2pt);
  \node[below] at (1,0){1};
  \draw (2,0pt) -- (2,-2pt);
  \node[below] at (2,0){2};
      	\foreach \y in {0,...,12}
     		\draw (0pt,\y/5) -- (-1pt,\y/5);
  \draw (0pt,1) -- (-2pt,1);
  \node[left] at (0,1){1};
  \draw (0pt,2) -- (-2pt,2);
  \node[left] at (0,2){2};
 \draw[thick,domain=-2:2] plot (\x, {(2/pi)*(\x * asin(\x/2)/180*pi+sqrt(4-\x*\x))}) ;
 \draw[thick,domain=-2.5:-2] plot (\x, {-\x}) ;
  \draw[thick,domain=2:2.5] plot (\x, {\x}) ;
  \draw (-2.5,2.5)--(0,0)--(2.5,2.5);
  \node[above left] at (1.5,1.7){$\Omega$};
 \end{tikzpicture}\qquad
 \includegraphics[width=7.3cm,height=4cm]{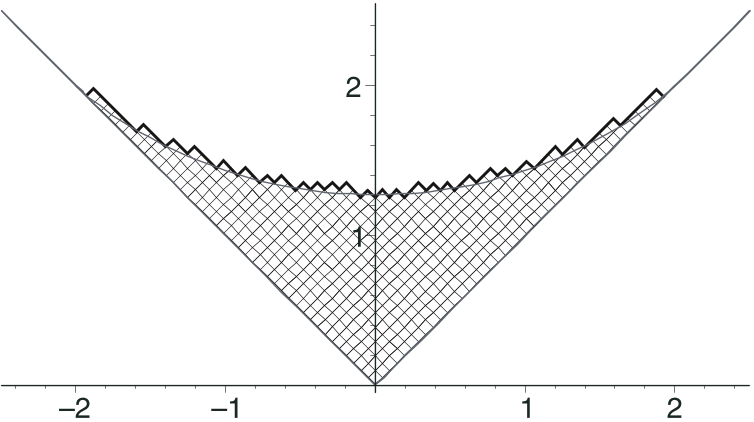}

 \end{center}
 \caption[The limit shape for a Young diagram]{On the left, the diagram $\Omega$. On the right, a large random partition (the image is taken from \cite{borodin2000asymptotics}).}\label{example limit shape}
 \end{figure}
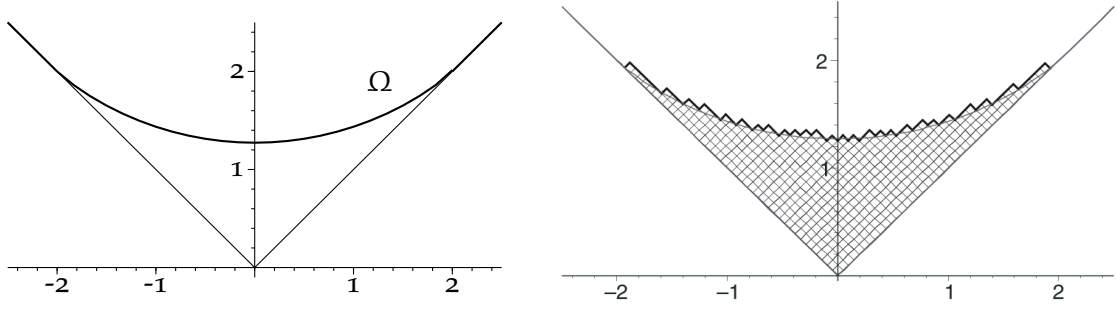

\begin{defi}
 The algebra $\mathbb{A}:=\R[\overline{p}_2,\overline{p}_3,\ldots]$ is called the \emph{algebra of polynomial functions} on the set $\mathbb{Y}=\bigcup_n\mathbb{Y}_n$.
\end{defi}
Two other bases play an important role on the moments of \eqref{eq: def p tilde}:
\begin{itemize}
 \item The \emph{shifted power sums}
 \[p^{\ast}_k(\lambda):=\sum_{i=1}^{l(\lambda)}(\lambda_i-i)^k-(-i)^k.\]
 \item The functions $p^{\sharp}_k$:
 \[p^{\sharp}_k(\lambda):=\left\{\begin{array}{ll}n^{\downarrow k}\frac{\chi^{\lambda}_{(k,1^{n-k})}}{\dim\lambda}&\mbox{ if }n:=|\lambda|\geq k,\\0&\mbox{ otherwise,}\end{array}\right.\]
 where $n^{\downarrow k}=n\cdot(n-1)\cdot\ldots\cdot(n-k+1)$ and $(k,1^{n-k})=(k,1,\ldots,1)\vdash n$. 
\end{itemize}
The proofs that the two families $\{p^{\ast}_k\}_{k\geq 1}$ and $\{p^{\sharp}_k\}_{k\geq 1}$ each form a basis for $\mathbb{A}$ can be found in \cite[Propositions 1.4 and 2.5 and Corollary 4.3]{ivanov2002kerov}. The family $\{p^{\sharp}_k\}_{k\geq 1}$ is particularly relevant since it encodes essentially the irreducible characters of the symmetric group, and arguments of representation theory and combinatorics can be used to study its asymptotic behavior. For example, it is easy to check that its average is
\[\mathbb{E}_{P_{\Pl}^n}[p^{\sharp}_k]:=\sum_{\lambda\vdash n}\frac{\dim\lambda^2}{n!}p^{\sharp}_k(\lambda)=\left\{\begin{array}{ll}n&\mbox{ if }k=1;\\0&\mbox{ otherwise.}\end{array}\right.\]
\section{Asymptotic of \texorpdfstring{$p^{\sharp}$}{Lg}}\label{section: asymptotic p sharp}
The following theorem appears in \cite[Theorem 6.1]{ivanov2002kerov} and relies on the results in \cite{IvanovKerov1999}. A different proof can be found in \cite{hora1998central}.
\begin{theorem}
 Denote with $p^{\sharp^{(n)}}_k$ the random variable defined on $(\mathbb{Y}_n,P_{\Pl}^n)$ obtained by restricting $p^{\sharp}_k$ to $\mathbb{Y}_n$. Let $\{\xi_k\}_{k=2,3,\ldots}$ be independent standard Gaussian random variables. Then, as $n\to\infty$,
 \[\left\{n^{-\frac{k}{2}}\cdot p^{\sharp^{(n)}}_k\right\}_{k=2,3,\ldots}\overset{d}\to\{\sqrt{k}\cdot\xi_k\}_{k=2,3,\ldots} \]
where $\overset{d}\to$ means convergence in distribution function. 
\end{theorem}
The functions $p^{\sharp}_k$ can be generalized to any cycle type: let $\rho\vdash k$ and $\lambda\vdash n$, then
 \[p^{\sharp}_{\rho}(\lambda):=\left\{\begin{array}{ll}n^{\downarrow k}\frac{\chi^{\lambda}_{(\rho,1^{n-k})}}{\dim\lambda}&\mbox{ if }n:=|\lambda|\geq k,\\0&\mbox{ otherwise.}\end{array}\right.\]
In \cite[Section 6]{ivanov2002kerov} the authors show a central limit theorem for $p^{\sharp}_{\rho}$, mentioned in the Introduction \ref{ch: intro}, which we present below. Let $\mathcal{H}_m(x)$, for $m\in\N_{\geq 0}$, be the $m$-th \emph{modified Hermite polynomial} of degree $m$ defined by the recurrence relation $x\mathcal{H}_m(x)=\mathcal{H}_{m+1}(x)+m\mathcal{H}_{m-1}(x)$ and initial data $\mathcal{H}_0(x)=1$ and $\mathcal{H}_1(x)=x$. For $\rho\vdash k$ and $j=1,2,\ldots,$ let $m_j(\rho)$ be the multiplicity of $j$ in $\rho$, that is, the number of parts in $\rho$ equal to $j$.
\begin{theorem}\label{theorem: asymptotic p sharp}
 Let as before $\{\xi_k\}_{k=2,3,\ldots}$ be independent standard Gaussian random variables. Then
 \[\left\{n^{-\frac{k+m_1(\rho)}{2}}p^{\sharp^{(n)}}_{\rho}\right\}_{\rho}\overset{d}\to\left\{\prod_{k\geq 2}k^{\frac{m_k(\rho)}{2}}\mathcal{H}_{m_k(\rho)}(\xi_k)\right\}_{\rho},\]
 where $\rho$ ranges over the set of all partitions.
\end{theorem}

A useful point of view on the functions $p^{\sharp}$ was introduced by Okounkov and Olshanski in \cite{OkOl1998}, where these polynomials are seen as symmetric functions modulo a shift of variables. This prompted the study of the algebra of shifted symmetric functions, which we recall in the next section.

\section{Shifted symmetric functions}\label{sec: shifted symmetric functions}
\begin{defi}
 A polynomial $f(x_1,\ldots,x_n)$ is called \emph{shifted symmetric} if $f(x_1+1,x_2+2,\ldots,x_n+n)$ is symmetric. Call $\Lambda^{\ast}_n$ the algebra of shifted symmetric polynomials in $n$ variables. This algebra has a filtration through the degree of the polynomial. 
\end{defi}

As in the non shifted case we can consider the projection $\Lambda^{\ast}_{n+1}\to\Lambda^{\ast}_n$ which sets the $n+1$-th variable to $0$. This is a morphism of filtered algebras, so that $\Lambda^{\ast}=\lim_{\leftarrow}\Lambda^{\ast}_n$ is a filtered algebra, called the \emph{algebra of shifted symmetric functions}. For $f\in \Lambda^{\ast}$ we write $f_{\vert n}\in \Lambda^{\ast}_n$ to indicate the polynomial corresponding to $f$ in which $x_{n+1}=0=x_{n+2}=\ldots.$

We present some examples for functions in $\Lambda^{\ast}$: let $\rho\vdash n$, then we define
\begin{itemize}
 \item \emph{The shifted power sums} $p_{\rho}^{\ast}(x_1,x_2,\ldots)=\prod_{i=1}^{l(\rho)}p_{\rho_i}^{\ast},$ where
 \[p_k^{\ast}=\sum_{i\geq1}\left((x_i-i)^k-(-i)^k\right).\]
 It is clear then that the algebra $\mathbb{A}$ of polynomial functions on Young diagrams and the algebra $\Lambda^{\ast}$ of shifted symmetric functions coincide. It is often preferred to manipulate the two algebras separately, since $\mathbb{A}$ carries a ``geometric'' approach, while $\Lambda^{\ast}$ an ``algebraic'' one.
 \item \emph{The shifted elementary symmetric functions} $e_{\rho}^{\ast}(x_1,x_2,\ldots)=\prod_{i=1}^{l(\rho)}e_{\rho_i}^{\ast},$ where
 \[e_k^{\ast}=\sum_{1\leq i_1<i_2<\ldots<i_k}(x_{i_1}-i_1)\cdot\ldots\cdot(x_{i_k}-i_k).\]
  \item \emph{The shifted complete homogeneous symmetric functions} $h_{\rho}^{\ast}(x_1,x_2,\ldots)=\prod_{i=1}^{l(\rho)}h_{\rho_i}^{\ast},$ where
 \[h_k^{\ast}=\sum_{1\leq i_1\leq i_2\leq \ldots\leq i_k}(x_{i_1}-i_1)\cdot\ldots\cdot(x_{i_k}-i_k).\]
\end{itemize}
By \cite[Corollary 1.6]{OkOl1998} each set $\{p_{\rho}^{\ast}\}_{\rho},$ $\{e_{\rho}^{\ast}\}_{\rho}$ and $\{h_{\rho}^{\ast}\}_{\rho}$ forms a basis for $\Lambda^{\ast}$, where $\rho$ runs over all partitions. Note that the functions just described are \emph{stable}, that is, 
\[p_{\rho}^{\ast}(x_1,x_2,\ldots,x_n,0)=p_{\rho}^{\ast}(x_1,x_2,\ldots,x_n),\]
and the same holds for $e_{\rho}^{\ast}$ and $h_{\rho}^{\ast}$.
\begin{defi}
 Let $\rho$ be a partition such that $l(\rho)\leq n$. The \emph{shifted Schur function} associated to $\rho$ is 
 \[s_{\rho}^{\ast}(x_1,\ldots,x_n)=\frac{\det\left((x_i+n-i)^{\downarrow(\rho_j+n-j)}\right)}{\det\left((x_i+n-i)^{\downarrow(n-j)}\right)},\]
 where $n^{\downarrow k}=n\cdot(n-1)\cdot\ldots\cdot(n-k+1)$.
\end{defi}
A number of properties of the shifted Schur functions are proved in \cite{OkOl1998} by Okounkov and Olshanski. We state the most relevant ones:
\begin{itemize}
 \item The shifted Schur functions are stable: if $\rho$ is a partition with $l(\rho)\leq n$ then
 \[s_{\rho}^{\ast}(x_1,\ldots,x_n)=s_{\rho}^{\ast}(x_1,\ldots,x_n,0),\]
 so that $s_{\rho}^{\ast}$ makes sense as an element of $\Lambda^{\ast}$.
 \item Shifted Schur functions form a basis of $\Lambda^{\ast}$.
 \item There exists a combinatorial presentation of shifted Schur functions, analogous to Definition \ref{defi: schur functions}:
 \[s_{\rho}^{\ast}(x_1,x_2,\ldots)=\sum_{T}\prod_{\Box\in\rho}(x_{T(\Box)}-c(\Box)),\]
where the sum runs over \emph{reverse tableaux} of shape $\rho$, that is, fillings of the Young diagram of shape $\rho$ which are weakly decreasing along each row and strictly decreasing down each column, and $c(\Box)$ is the content of $\box\in\rho$.
\item \emph{Vanishing theorem:} let $\rho=(\rho_1,\ldots,\rho_r)$ and $\lambda=(\lambda_1,\ldots,\lambda_l)$ be partitions. We define $s_{\rho}^{\ast}(\lambda):=s_{\rho}^{\ast}(\lambda_1,\ldots,\lambda_l)$, and recall that $\rho\subset \lambda$ if $\rho_i\leq \lambda_i$ for all $i$. Then
\[s_{\rho}^{\ast}(\lambda)=\left\{\begin{array}{ll}0&\mbox{ unless }\rho\subset \lambda\\H(\rho)&\mbox{ if }\lambda=\rho,\end{array}\right.\]
where $H(\rho)$ is the hook product of $\rho$. This theorem was proved by Okounkov in \cite{Okounkov1996a}.
\end{itemize}
Consider the isomorphism 
\[\begin{array}{cccc}\phi^{(n)}\colon& \Lambda_n&\to&\Lambda_n^{\ast}\\&s_{\lambda\vert n}&\mapsto&s_{\lambda\vert n}^{\ast}\end{array}\]
and its inverse limit $\phi\colon \Lambda\to\Lambda^{\ast}$.  It was proven in \cite[Equation (15.21)]{OkOl1998} that the image of a power sum symmetric function $p_{\rho}$ under this isomorphism is exactly $p^{\sharp}_{\rho}$: \[p^{\sharp}_{\rho}=\phi(p_{\rho}).\]

Equation \eqref{eq: schur expansion via power sums} implies that, for a partition $\lambda\vdash n$,
\[s^{\ast}_{\lambda}=\sum_{\rho\vdash n}\frac{\chi^{\lambda}_{\rho}}{z_{\rho}}p_{\rho}^{\sharp}.\]

\subsection{The multiplication table of normalized characters}
Consider the multiplication table of the normalized character, that is, the structure constants $f^{\gamma}_{\rho_1,\rho_2}$ for partitions $\gamma,\rho_1,\rho_2\vdash n$ such that for each $\lambda$
\[\hat{\chi}^{\lambda}_{\rho_1}\cdot \hat{\chi}^{\lambda}_{\rho_2}=\sum_{\gamma\vdash n}f^{\gamma}_{\rho_1,\rho_2}\hat{\chi}^{\lambda}_{\gamma}.\]
As hinted in the introduction, the multiplication table for renormalized characters is closely related to the multiplication table of the central elements of the group algebra. Let $\rho\vdash n$ and set 
\[K_{\rho}=\frac{1}{\#\{\sigma \mbox{ of cycle type }\rho\}}\sum\sigma\in Z(\C[S_n]),\]
where the sum runs over the permutations $\sigma\in S_n$ of cycle type $\rho$ and $Z(\C[S_n])$ is the center of $\C[S_n]$. The set $\{K_{\rho}\}_{\rho\vdash n}$ generates $Z(\C[S_n])$ (see \cite{IvanovKerov1999}). Note that if $\sigma\in S_n$ is of cycle type $\rho$ then $\chi^{\lambda}(\sigma)=\chi^{\lambda}(K_{\rho})$. Hence Equation \eqref{eq: consequence schur lemma} implies that, for $\rho_1,\rho_2\vdash n$,
\[\hat{\chi}^{\lambda}_{\rho_1}\cdot\hat{\chi}^{\lambda}_{\rho_2}=\hat{\chi}^{\lambda}(K_{\rho_1})\cdot \hat{\chi}^{\lambda}(K_{\rho_2})=\hat{\chi}^{\lambda}(K_{\rho_1}\cdot K_{\rho_2}).\]
Therefore
\[K_{\rho_1}\cdot K_{\rho_2}=\sum_{\gamma}f^{\gamma}_{\rho_1,\rho_2}K_{\gamma}.\]

To attack the problem of studying the multiplication table of the elements $K_{\rho}$, $\rho\vdash n$, Ivanov and Kerov in \cite{IvanovKerov1999} introduced a new algebra, which can be considered as an extension of the group algebra $\C[S_n]$, which is called the \emph{partial permutation algebra} and it is more apt to the purpose. We devote the next section to describe some basic results involving this algebra, which we will need in Chapter \ref{ch: partial sum}.

\section{The partial permutation algebra}\label{sec: partial permutations}

 In this section we recall some results in the theory of partial permutations, introduced in \cite{IvanovKerov1999}. All the definitions and results in this section can be found in \cite{feray2012partial}.
 \begin{defi}
  A \emph{partial permutation} is a pair $(\sigma,D)$, where $D\subset \mathbb{N}$ and $\sigma$ is a bijection $D\to D$.
 \end{defi}
We call $P_n$ the set of partial permutations $(\sigma,D)$ such that $D\subseteq \{1,\ldots,n\}$. The set $P_n$ is endowed with the operation $(\sigma,D)\cdot(\sigma',D')=(\tilde{\sigma}\cdot\tilde{\sigma}',D\cup D')$, where $\tilde{\sigma}$ is the bijection from $D\cup D'$ to itself defined by $\tilde{\sigma}_{\vert D}=\sigma$ and $\tilde{\sigma}_{\vert D'\setminus D}=\id$; the same holds for $\tilde{\sigma}'$.
\smallskip

Define the algebra $\mathcal{B}_n=\mathbb{C}[P_n]$. There is an action of the symmetric group $S_n$ on $P_n$ defined by $\tau\cdot (\sigma,D):=(\tau\sigma\tau^{-1},\tau(D))$, and we call $\mathcal{A}_n$ the subalgebra of $\mathcal{B}_n$ of the invariant elements under this action. For a partition $\rho$ and an integer $n$, set
\[\alpha_{\rho;n}:=\sum\limits_{\substack{D\subseteq\{1,\ldots,n\},\rho\vdash|D|\\ \sigma\in S_D\mbox{ of type } \rho}}(\sigma,D).\]
The next Proposition can be found in \cite[Proposition 2.1]{feray2012partial}.
\begin{proposition}
 The family $(\alpha_{\rho;n})_{|\rho|\leq n}$ forms a linear basis for $\mathcal{A}_n$.
\end{proposition}
There is a natural projection $\mathcal{B}_{n+1}\to\mathcal{B}_n$ which sends to $0$ the partial permutations whose support contains $n+1$. The projective limit through this projection is called $\mathcal{B}_{\infty}:=\lim\limits_{\leftarrow}\mathcal{B}_n$, and similarly $\mathcal{A}_{\infty}:=\lim\limits_{\leftarrow}\mathcal{A}_n$. The family $\alpha_{\rho}:=(\alpha_{\rho;n})_{n\geq 1}$ forms a linear basis of $\mathcal{A}_{\infty}$.
\smallskip

Let us fix some notation: we write $n^{\downarrow k}$ for the $k$-th falling factorial, which is the number $n^{\downarrow k}= n(n-1)\cdot\ldots\cdot(n-k+1)$, and $z_{\rho}=\prod_i \rho_i\prod_i m_i(\rho)!$ for the order of the centralizer of a permutation of type $\rho$. We consider the morphism of algebras $\varphi_n\colon\mathcal{A}_n\to Z(\mathbb{C}[S_n])$ which sends $(\sigma,D)$ to $\sigma$, where $Z(\mathbb{C}[S_n])$ is the center of the group algebra $\mathbb{C}[S_n]$. On the basis $\{\alpha_{\rho;n}\}$ this morphism acts as follows (see \cite[Equation 4]{feray2012partial}):
\begin{equation}\label{act of phi}
 \varphi_n(\alpha_{\rho;n})= \frac{n^{\downarrow|\rho|}}{z_{\rho}}K_{(\rho, 1^{n-|\rho|})};\qquad\mbox{ in particular } \varphi_n(\alpha_{1^k;n})=\frac{n^{\downarrow k}}{k!}\id_{S_n}.
\end{equation}
Therefore, for $\lambda\vdash n$, $\hat{\chi}^{\lambda}(\varphi_n(\alpha_{\rho;n}))=\frac{1}{z_{\rho}}|\lambda|^{\downarrow |\rho|}\hat{\chi}^{\lambda}_{\rho},$ so that 
\begin{equation}\label{eq: uff}
 \hat{\chi}^{\lambda}(\varphi_n(\alpha_{\rho;n}))=p^{\sharp}_{\rho}(\lambda)/z_{\rho}\mbox{ if }|\rho|\leq n.
\end{equation}
In \cite[Theorem 9.1]{IvanovKerov1999} Ivanov and Kerov showed that, when considering the projective limit, Equation \eqref{eq: uff} gives an isomorphism of algebras:
  \begin{lemma}\label{lemma isomorphism}
  There is an isomorphism $F:\mathcal{A}_{\infty}\to\Lambda^{\ast}$ that sends $\alpha_{\rho}$ to $p_{\rho}^{\sharp}/z_{\rho}$.
 \end{lemma}
To summarize, we have described the following isomorphic algebras: the algebra of symmetric functions, the algebra of shifted symmetric functions, the algebra of polynomial functions on $\mathbb{Y}$, and the algebra of partial permutations invariant by conjugation:
 \[\Lambda\cong\Lambda^{\ast}=\mathbb{A}\cong\mathcal{A}_{\infty}.\]
In order to investigate the multiplication table of $p^{\sharp}_{\rho}$ one can instead consider the algebra $\mathcal{A}_{\infty}$ generated by $\{\alpha_{\rho}\}$. This strategy is particularly successful when the goal is to evaluate top degree components in the right choice of filtration for $\mathcal{A}_{\infty}$. We will employ this strategy in Section \ref{section: multiplication table p sharp}.
\medskip

We recall the definition of \emph{Jucys-Murphy element}, described for example in \cite{Jucys1966} and \cite{Murphy1981}, and its generalization as a partial permutation. 
\begin{defi}
The $i$-th \emph{Jucys-Murphy element} is the element of the group algebra $\mathbb{C}[S_n]$ defined by $J_i:=(1,i)+(2,i)+\ldots+(i-1,i)$, $J_1=0$.
\smallskip

 The \emph{partial Jucys-Murphy element }$\xi_i$ is the element of $\mathcal{B}_n$ defined by $\xi_i:=\sum\limits_{j<i}((j,i),\{j,i\}),$ $\xi_1:=0.$
\end{defi}

From a result of Jucys (see \hbox{\cite[Equation (12)]{Jucys1974}}) we have that
\begin{equation}\label{ciao}
 \hat{\chi}^{\lambda}(\varphi_n(f(\xi_1,\ldots,\xi_n)))=f(\mathcal{C}_{\lambda}),
\end{equation}
 where $f$ is a symmetric function and $\mathcal{C}_{\lambda}$ is the multiset of contents of the diagram $\lambda$, $\mathcal{C}_{\lambda}=\{c(\Box),\Box\in\lambda\}$, and $\varphi_n$ is defined in \eqref{act of phi}.

The next Proposition can be found in \cite[Proposition 2.4]{feray2012partial}.
\begin{proposition}
Let $f$ be a symmetric function, then $f(\xi_1,\ldots, \xi_n)\in\mathcal{A}_n$. Moreover, the sequence $f_n=f(\xi_1,\ldots,\xi_n)$ is an element of the projective limit $\mathcal{A}_{\infty}$, which we denote $f(\Xi)$.
\end{proposition}
For a partition $\nu=(\nu_1,\ldots,\nu_q)$ and $n\geq |\nu|$ we consider the power sum evaluated on the Jucys-Murphy elements \hbox{$p_{\nu}(\xi_1,\ldots,\xi_n)=\prod_{i=1}^q(\xi_1^{\nu_i}+\ldots +\xi_n^{\nu_i})$}. F\'eray proved in \cite[Proposition 2.5 and proof]{feray2012partial} that
\begin{equation}\label{eq:feray}
p_{\nu}(\Xi)=\prod_i m_i(\nu)!\hspace{0.1cm}\alpha_{\nu+\underline{1}}+\sum\limits_{|\rho|<|\nu|+q}c_{\rho}\alpha_{\rho},
\end{equation}
with $\nu+\underline{1}=(\nu_1+1,\ldots,\nu_q+1)$ and $c_{\rho}$ non-negative integers.

\section{The generating function of a partition}\label{sec: generating function}
\begin{defi}
 Let $\lambda\in\mathbb{Y}=\bigcup_n\mathbb{Y}_n$. The \emph{generating function} of $\lambda$ is the rational function
 \[\phi(z;\lambda):=\prod_{i\geq 1}\frac{z+i}{z+i-\lambda_i}.\] 
\end{defi}
As showed in \cite[Proposition 1.4]{ivanov2002kerov} the expansion of $\log\phi(z;\lambda)$ in descending powers of $z$ around $z=\infty$ is
\[\log\phi(z;\lambda)=\sum_{k\geq 1}\frac{p^{\ast}_k(\lambda)}{k}z^{-k},\]
so that $\mathbb{A}$ is generated over $\R$ by the coefficients of the expansion of $\phi(z;\lambda)$ or, equivalently, $\log\phi(z;\lambda)$.

A corollary of a result due to Frobenius (see \cite[Example 1.7.7]{McDo} and \cite[Proposition 3.2]{ivanov2002kerov}) is that for any $k\in\N_{\geq 1}$ and $\lambda\in\mathbb{Y}$, the function $p^{\sharp}_k(\lambda)$ is equal to the coefficient of $z^{-1}$ in the expansion of 
\[-\frac{1}{k}z^{\downarrow k}\frac{\phi(z;\lambda)}{\phi(z-k;\lambda)}\]
in descending powers of $z$ around $z=\infty$. We will work in Chapter \ref{ch: strict partitions} with similar objects in projective representation theory.

The generating function $\phi$ has another description in terms of the contents of particular boxes of the partition $\lambda$. Recall that a box $\Box\in\lambda$ is an outer corner for $\lambda$ if there exists a partition $\mu\nearrow\lambda$ such that they differ only on $\Box$. Similarly, a box $\Box\notin\lambda$ is an inner corner for $\lambda$ if there exists a partition $\Lambda$, $\lambda\nearrow\Lambda$, such that they differ only on $\Box$. If $\Box=(i,j)$ then its content is $c(\Box)=j-i$.

Consider a partition $\lambda$ pictured in Russian coordinates and its piece-wise linear function $\lambda(\cdot)$. The contents of inner and outer corners are exactly the abscissa of, respectively, the maxima and minima of $\lambda(\cdot)$. Recall from Section \ref{sec: young diagrams} that the abscissa of the maxima and the minima of $\lambda(\cdot)$ are respectively $\{y_j\}_{j=1,\ldots,m}$ and $\{x_j\}_{j=0,\ldots,m}$, as showed in Figure \ref{example interlacing sequence}. Then it is easy to see that the sequence of maxima and minima is \emph{interlacing}, that is, $x_0<y_1<\ldots<y_m<x_m$.

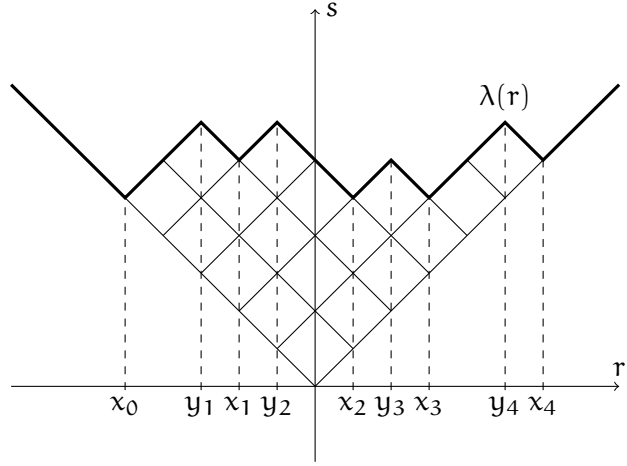
\begin{figure}
 \begin{center}
 \begin{tikzpicture}
  \draw[->](-4,0)--(4,0);
  \node [above] at (4,0){$r$};
\node [right] at (0,5){$s$};
  \draw[->](0,-1)--(0,5);
  \draw (-3,3)--(0,0)--(3,3);
  \draw[very thick] (-4,4)--(-2.5,2.5)--(-1.5,3.5)--(-1,3)--(-0.5,3.5)--(0.5,2.5)--(1,3)--(1.5,2.5)--(2.5,3.5)--(3,3)--(4,4);
  \draw [thin] (-2,2)--(-1,3);
  \draw [thin] (-1.5,1.5)--(0,3);
  \draw [thin] (-1,1)--(0.5,2.5);
  \draw [thin] (-0.5,0.5)--(1.5,2.5);
  \draw [thin] (0.5,0.5)--(-2,3);
  \draw [thin] (1,1)--(-1,3);
  \draw [thin] (1.5,1.5)--(0.5,2.5);
  \draw [thin] (2,2)--(1.5,2.5);
  \draw [thin] (2.5,2.5)--(2,3); 
  \node[above] at (2.5,3.5){$\lambda(r)$};
  \draw[dashed](-2.5,-0.05)--(-2.5,2.5);
  \node[below] at (-2.5,0){$x_0$};
  \draw[dashed](-1.5,-0.05)--(-1.5,3.5);
  \node[below] at (-1.5,0){$y_1$};
  \draw[dashed](-1,-0.05)--(-1,3);
  \node[below] at (-1,0){$x_1$};
  \draw[dashed](-0.5,-0.05)--(-0.5,3.5);
  \node[below] at (-.5,0){$y_2$};
  \draw[dashed](.5,-0.05)--(.5,2.5);
  \node[below] at (.5,0){$x_2$};
  \draw[dashed](1,-0.05)--(1,3);
  \node[below] at (1,0){$y_3$};
  \draw[dashed](1.5,-0.05)--(1.5,2.5);
  \node[below] at (1.5,0){$x_3$};
  \draw[dashed](2.5,-0.05)--(2.5,3.5);
  \node[below] at (2.5,0){$y_4$};
  \draw[dashed](3,-0.05)--(3,3);
  \node[below] at (3,0){$x_4$};
 \end{tikzpicture}

 \end{center}
 \caption[Maxima and minima for a Young diagram]{The Young diagram $\lambda=(6,4,3,3,2)$ in Russian coordinates and the interlacing sequence of maxima and minima.}\label{example interlacing sequence}
 \end{figure}
The following proposition connects the generating function $\phi(z;\lambda)$ with the sequence of maxima and minima of $\lambda(\cdot)$, and can be found in \cite{ivanov2002kerov}.
\begin{proposition}\label{prop: gen function in terms of x_j and y_j}
 Let $\lambda\in\mathbb{Y}$ with maxima and minima respectively $\{y_j\}$ and $\{x_j\}$ as before. Then 
 \[\frac{\phi(z-1;\lambda)}{\phi(z;\lambda)}=\frac{z\prod_{i=1}^m(z-y_i)}{\prod_{i=0}^m (z-x_i)}.\] 
\end{proposition}

\section{The transition measure}\label{sec: the transition measure}
Recall that given $\lambda,\Lambda\in\mathbb{Y}$ the transition measure associated to $\lambda,\Lambda$ is
\[\tr(\lambda,\Lambda):=\left\{\begin{array}{ll}\frac{\dim\Lambda}{|\Lambda|\dim\lambda}&\mbox{ if }\lambda\nearrow\Lambda;\\0&\mbox{ otherwise.} \end{array}\right.\]
If we draw $\lambda\vdash n$ in Russian coordinates we can write the transition measure as a discrete probability on $\R$ as follows: if $x_j$ is a minimum for $\lambda(\cdot)$, for $j=0,\ldots,m$, we call $\Lambda_j$ the partition of $n+1$ obtained from $\lambda$ by adding a box at the position $x_j$. Then we define 
\[\tr_{\lambda}(v):=\left\{\begin{array}{ll}\frac{\dim\Lambda_j}{(n+1)\cdot\dim\lambda}&\mbox{ if }v=x_j\mbox{ for some }j;\\0&\mbox{ otherwise.} \end{array}\right.\]
In \cite{kerov1993transition} Kerov extended the definition of transition measure to continual diagrams centered in $0$, which relies on the following proposition, proven in \cite{kerov1993asymptotics}. Let $\mathcal{M}^0$ be the set of probability measures on $\R$ with compact support and first moment equal to zero. Equip $\mathcal{M}^0$ with the weak* convergence of measures, so that it is a topological space.
\begin{proposition}
 There is a homomorphism between the spaces $\mathcal{M}^0$ and $\mathcal{D}^0$, where $\mathcal{D}^0$ is endowed with the topology of uniform convergence. For $\omega\in\mathcal{D}^0$, call $\sigma(x):=\frac{1}{2}(\omega(x)-|x|)$. Then the measure associated to $\omega$ is $\mu\in\mathcal{M}^0$ such that
 \[\frac{1}{z}\exp\int_{\R}\frac{\sigma'(x)\,dx}{x-z}=\int_{\R}\frac{\mu(dx)}{z-x},\]
 where $z\in\C\setminus I$ for a sufficiently large interval $I\subset\R$. The left hand side is called the \emph{Kerov transform} $K_{\omega}(z)$ of the diagram $\omega$, while the right hand side is called the \emph{Cauchy transform} of the measure $\mu$ and it is written $C_{\mu}(z)$.
\end{proposition}
We will use this result in Chapter \ref{ch: strict partitions}, where we will manipulate strict partitions to obtain a continual diagram centered in $0$.

A diagram $\omega\in\mathcal{D}^0$ is \emph{rectangular} if it is piece-wise linear and $\omega'(x)=\pm 1$ for almost all $x\in\R$. Obviously partitions pictured in Russian coordinates can be seen as rectangular diagrams.

A rectangular diagram $\omega\in\mathcal{D}^0$ is uniquely determined by the interlacing sequence of maxima and minima $\{y_j\},\{x_j\}$. The definition of generating function for a partition can be generalized to rectangular diagrams using Proposition \ref{prop: gen function in terms of x_j and y_j}. We call such a function $\phi(z;\omega)$. The transition measure associated to the rectangular diagram $\omega$ is discrete and to each $x_j$ attaches the weight
\[\tr_{\omega}(x_j):=\frac{\prod\limits_{1\leq i\leq m}(x_j-y_i)}{\prod\limits_{\substack{0\leq i\leq m\\i\neq j}}(x_j-x_i)}\]
and $\tr_{\omega}(x)=0$ if $x\neq x_j$ for each $j=0,\ldots, m$. See \cite[Section 2]{kerov1993transition} for more details. In particular the Cauchy transform of $\tr_{\omega}$ is 
\[C_{\tr_{\omega}}(z)=\frac{1}{z}\frac{\phi(z-1;\lambda)}{\phi(z;\lambda)}.\]

\subsection{Asymptotic of the transition measure}
In \cite{kerov1993transition} Kerov described the asymptotic behavior of the transition distribution function when $n\to\infty$ and $\lambda\in\mathbb{Y}_n$ is a random partition taken with the Plancherel measure. In \cite{ivanov2002kerov} the authors develop Kerov's ideas to obtain a second order asymptotic result. 

As it is expected, some renormalization is in order: for $v\in\R$ define the renormalized transition distribution function
\[F_{\tr}^{\lambda}(v):=\int_{-\infty}^{v\sqrt{n}}\tr_{\lambda}(dx)= \sum\limits_{x_j\leq v\sqrt{n} }\frac{\dim\Lambda_j}{(n+1)\dim\lambda},\]
where the sum ranges over all partitions $\Lambda_j$ of $n+1$ such that the partitions $\Lambda_j$ and $\lambda$ differ only of one box, whose content is $x_j$, and $x_j\leq v\sqrt{n}$.
\begin{defi}
 The \emph{semicircular distribution function} is defined as
 \[F_{sc}(v):=\frac{1}{2\pi}\int_{-2}^v \sqrt{4-t^2}\,dt.\]
\end{defi}
Kerov showed the following:
\begin{proposition}
 Set $F^{(n)}_{\tr}(v)$ to be a random function on the probability space $(\mathbb{Y}_n,P_{\Pl}^n)$. Consider the space of infinite paths on the Bratteli diagram of partitions, then almost surely
 \[\lim\limits_{n\to \infty}F^{(n)}_{\tr}(v)=F_{sc}(v),\]
 
\end{proposition}
\section{The co-transition measure}\label{sec: the co-transition}
In this section we study the co-transition measure, defined in Definition \ref{def: transition and co-transition}. As with the transition counterpart, we can extend the definition of co-transition measure to rectangular diagrams: let $\omega\in\mathcal{D}^0$ be a rectangular diagram with interlacing sequence of maxima and minima $x_0<y_1<\ldots<x_m$, then
\[\ctr_{\omega}(y_j):=-\frac{1}{n}\frac{\prod\limits_{0\leq i\leq m}(y_j-x_i)}{\prod\limits_{\substack{1\leq i\leq m\\i\neq j}}(y_j-y_i)},\qquad \ctr(y)=0\mbox{ if }y\neq y_j\mbox{ for all }j=1,\ldots, m.\]
See \cite[Lemma 3.3]{kerov1997anisotropic}. Kerov computed the Cauchy transform of the co-transition measure as
\begin{equation}\label{eq: co-tr}
 C_{\ctr_{\omega}}(z)=z-\frac{z\phi(z;\omega)}{\phi(z-1;\omega)}=z-\frac{1}{C_{\ctr_{\omega}(z)}}.
\end{equation}
Equation \eqref{eq: co-tr} allows us to compute the asymptotic of the renormalized co-transition measure for Plancherel random partitions.
\begin{corollary}\label{convergence of co-transition}
Let $v\in\R$ and $\lambda\vdash n$ be distributed with the Plancherel measure. Consider as before the space of infinite paths on the Bratteli diagram of partitions, then almost surely
 \[\lim_{n \to +\infty}F_{\ctr}^{\lambda}(v)=F_{sc}(v).\]
\end{corollary}
\begin{proof}
We prove this by looking at the Cauchy transforms of the normalized co-transition measure and semicircular distribution function. It is showed in \cite[Corollary 1]{geronimo2003necessary} that pointwise convergence of Cauchy transforms implies convergence in distribution, which again implies convergence of distribution functions at continuity points (recall that the semicircular distribution function is continuous everywhere). It is well known (see for example \cite{kerov1993transition}) that 
\[C_{sc}(z)=\frac{z-\sqrt{z^2-4}}{2},\]
where $C_{sc}(z)$ is the Cauchy transform of the semicircular distribution function. Since $C_{\ctr_{\lambda}}(z)=z-1/C_{\tr_{\lambda}}(z)$ and 
\[C_{\tr_{\lambda}}(z)\overset{a.s.}\to C_{sc}(z)\]
then 
\[C_{\ctr_{\lambda}}(z)=z-\frac{1}{C_{\tr_{\lambda}}(z)}\overset{a.s.}\to z-\frac{2}{z-\sqrt{z^2-4}}=\frac{z-\sqrt{z^2-4}}{2}=C_{sc}(z)\]
and the corollary is proved.
\end{proof}

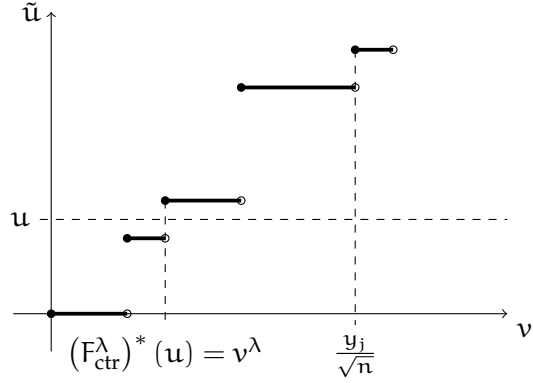
\begin{figure}
\begin{center}
 \begin{tikzpicture}[scale=0.5]
\draw[->] (-1,0) -- (12,0);
\node at (12,0)[below right]{$v$};
\draw[->] (0,-1) -- (0,8);
\node at (0,8)[left]{$\tilde{u}$};
\draw[line width=.5mm](0,0)--(2,0);
\draw[fill=black] (0,0) circle (.10cm);
\draw (2,0) circle (.10cm);
\draw[line width=.5mm](2,2)--(3,2);
\draw[fill=black] (2,2) circle (.10cm);
\draw (3,2) circle (.10cm);
\draw[line width=.5mm](3,3)--(5,3);
\draw[fill=black] (3,3) circle (.10cm);
\draw (5,3) circle (.10cm);
\draw[line width=.5mm](5,6)--(8,6);
\draw[fill=black] (5,6) circle (.10cm);
\draw (8,6) circle (.10cm);
\draw[line width=.5mm](8,7)--(9,7);
\draw[fill=black] (8,7) circle (.10cm);
\draw (9,7) circle (.10cm);
 \node at (-0.3,2.5)[left] {$u$};
 \draw[dashed] (-0.3,2.5)--(12,2.5);
 \draw[dashed] (3,3)--(3,-0.3);
 \node at (3,-0.3)[below]{$\left(F_{\ctr}^{\lambda}\right)^{*}(u)=v^{\lambda}$};
  \draw[dashed] (8,7)--(8,-0.3);
   \node at (8,-0.3)[below]{$\frac{y_j}{\sqrt{n}}$};
 \end{tikzpicture}
\end{center}
 \caption[Example of the co-transition distribution function]{Example of graph of $\tilde{u}=F_{\ctr}^{\lambda}(v)$.}\label{picture co transition measure}
 \end{figure}
We conclude the section with a result which will be useful in the next chapter. Set 
\[\left(F_{\ctr}^{\lambda}\right)^{*}(u):=\sup\left\{z\in\mathbb{R} \mbox{ s.t. } F_{\ctr}^{\lambda}(z)\leq u\right\}.\]
We consider $\left(F_{\ctr}^{\lambda}\right)^{*}$ as the inverse of the step function $F_{\ctr}^{\lambda}$. Fix $u\in\R$, we call $v^{\lambda}:=\left(F_{\ctr}^{\lambda}\right)^{*}(u)$. We want to show that if $\lambda\vdash n$ is distributed with the Plancherel measure then $F_{\ctr}^{\lambda}(v^{\lambda})$ converges to $u$.
\begin{lemma}\label{behaviour of co-transition}
Let $u \in \mathbb{R}$ be a real number and let $v^\lambda = (F_{ct}^\lambda)^*(u)$. Consider a random variable $\lambda \mapsto F_{\ctr}^{\lambda}(v^{\lambda})$. Then 
 \[F_{\ctr}^{\lambda}(v^{\lambda})=\sum_{y_j\leq v^{\lambda}\sqrt{n}}\frac{\dim\mu_j}{\dim\lambda}\overset{p}\to u,\]
as $n \to \infty$, where $\lambda$ is sampled with the Plancherel measure, and $\xrightarrow{p}$ denotes convergence in probability.
\end{lemma}
\begin{proof}
We show in Figure \ref{picture co transition measure} an example of a normalized co-transition distribution function. Since $F_{\ctr}^{\lambda}$ is right continuous then $u\leq F_{\ctr}^{\lambda}(v^{\lambda})$. Moreover, since the co-transition distribution function is a step function, for each element of the image $\tilde{u}\in F_{\ctr}^{\lambda}(\mathbb{R})$ there exists $j$ such that $F_{\ctr}^{\lambda}(\frac{y_j}{\sqrt{n}})=\tilde{u}$. Call $\bar{j}$ the index corresponding to $F_{\ctr}^{\lambda}(v^{\lambda})$, that is, $F_{\ctr}^{\lambda}(\frac{y_{\bar{j}}}{\sqrt{n}})=F_{\ctr}^{\lambda}(v^{\lambda})$. Hence $\sum_{j<\bar{j}}\dim\mu_j/\dim\lambda\leq u$. Thus
\begin{itemize}
 \item $v^{\lambda}\leq \frac{y_{\bar{j}}}{\sqrt{n}}$, since $F_{\ctr}^{\lambda}(\frac{y_{\bar{j}}}{\sqrt{n}})\geq u$;
 \item $v^{\lambda}\geq \frac{y_{\bar{j}}}{\sqrt{n}}$, since for each $\epsilon>0$, $F_{\ctr}^{\lambda}(\frac{y_{\bar{j}}}{\sqrt{n}}-\epsilon)\leq u$.
\end{itemize}
Thus $v^{\lambda}= \frac{y_{\bar{j}}}{\sqrt{n}}$. Therefore
\[\sum_{j<\bar{j}}\frac{\dim\mu_j}{\dim\lambda} =F_{\ctr}^{\lambda}(v^{\lambda})-\frac{\dim\mu_{\bar{j}}}{\dim\lambda}\leq u\leq F_{\ctr}^{\lambda}(v^{\lambda}).\]
 Equivalently
 \[-\frac{\dim\mu_{\bar{j}}}{\dim\lambda}\leq u- F_{\ctr}^{\lambda}(v^{\lambda})\leq 0,\]
 and $\max \frac{\dim\mu_{j}}{\dim\lambda}\overset{p}\to 0$ because of the convergence of the normalized co-transition distribution function towards an atom free distribution function proved in the previous Corollary. 
 \end{proof}

\cleardoublepage

\pagestyle{scrheadings}

\cleardoublepage \myPart{New Results in Asymptotic of Representation Theory} 
\chapter{Partial sums of representation matrices}\label{ch: partial sum}
This chapter corresponds to the article \cite{de2016partial}, which has been submitted.
\section{Introduction}\label{sec:in}
Let $\lambda$ be a partition of $n$, in short $\lambda\vdash n$, represented as a Young diagram in English notation. Recall from Chapter \ref{ch: repr theory} that a filling of the boxes of $\lambda$ with numbers from $1$ to $n$, increasing towards the right and downwards, is called a \emph{standard Young tableau}. We call $\dim\lambda$ the number of standard Young tableaux of shape $\lambda$. We fix $n$ and we associate to each $\lambda$ the weight $\frac{(\dim\lambda)^2}{n!}$, which defines the \emph{Plancherel measure}.
\smallskip

Let us recall briefly three results for the study of the asymptotics of Plancherel distributed random partitions. They relate algebraic combinatorics, representation theory of the symmetric group, combinatorics of permutations, and random matrix theory. 
\begin{enumerate}
 \item The partitions of $n$ index the irreducible representations of the symmetric group $S_n$. For each $\lambda\vdash n$ the dimension of the corresponding irreducible representation is $\dim\lambda$. A natural question concerns the asymptotics of the associated characters when $\lambda$ is distributed with the Plancherel measure. A central limit theorem was proved with different techniques by Kerov, \cite{kerov1993gaussian}, \cite{ivanov2002kerov}, and Hora, \cite{hora1998central}. We presented this result in Section \ref{section: asymptotic p sharp} in the language of $p^{\sharp}$. We state it again here in terms of the renormalized character: consider a Plancherel distributed partition $\lambda\vdash n$ and $\rho$ a fixed partition of $r$ for $r\leq n$; set $m_k(\rho)$ to be the number of parts of $\rho$ which are equal to $k$, and $\hat{\chi}^{\lambda}_{(\rho,1,\ldots,1)}=\chi^{\lambda}_{(\rho,1,\ldots,1)}/\dim\lambda$ the renormalized character associated to $\lambda$ calculated on a permutation of cycle type $(\rho,1,\ldots,1)$. Then for $n\to\infty$
 \begin{equation}\label{intro2}
  n^{\frac{|\rho|-m_1(\rho)}{2}}\hat{\chi}^{\lambda}_{(\rho,1,\ldots,1)}\overset{d}\to \prod_{k\geq 2} k^{m_k(\rho)/2} \mathcal{H}_{m_k(\rho)}(\xi_k),
 \end{equation}
 where $\mathcal{H}_m(x)$ is the $m$-th modified Hermite polynomial defined in Section \ref{section: asymptotic p sharp}, $\{\xi_k\}_{k\geq 2}$ are i.i.d. standard gaussian variables, and $\overset{d}\to$ means convergence in distribution.
 
\item The Robinson-Schensted-Knuth algorithm allows us to interpret the longest increasing subsequence of a uniform random permutation as the first row of a Plancherel distributed partition. This motivates the study of the shape of a random partition. A limit shape result was proved independently by \cite{kerov1977asymptotics} and \cite{logan1977variational}, then extended to a central limit theorem by Kerov, \cite{ivanov2002kerov}. For an extensive introduction on the topic, see \cite{romik2015surprising}.

\item More recently it was proved that, after rescaling, the limiting distribution of the longest $k$ rows of a Plancherel distributed partition $\lambda$ coincides with the limit distribution of the properly rescaled $k$ largest eigenvalues of a random Hermitian matrix taken from the Gaussian Unitary Ensemble. See for example \cite{borodin2000asymptotics} and references therein. Such similarities also occur for fluctuations of linear statistics, see \cite{ivanov2002kerov}.
\end{enumerate}
In the aftermath of Kerov's result \eqref{intro2}, a natural step in the study of the characters of the symmetric group is to look at the representation matrix rather than just the trace. We consider thus, for a real valued matrix $A$ of dimension $N$ and $u\in [0,1]$, the partial trace and partial sum defined, respectively, as

\[PT_u(A):=\sum_{i\leq uN}\frac{A_{i,i}}{N},\qquad PS_u(A):= \sum_{i,j\leq uN}\frac{A_{i,j}}{N}.\]
We study these values for representation matrices of $S_n$: let $\lambda$ be a partition of $n$ and $\sigma$ a permutation in $S_r$, $r\leq n$; we consider $\sigma$ as a permutation of $S_n$ in which the points $r+1,\ldots, n$ are fixed points. We call $\pi^{\lambda}$ the irreducible representation of $S_n$ associated to $\lambda$. Thus, $\pi^{\lambda}(\sigma)$ is a square matrix of dimension $\dim\lambda$ whose entries are complex numbers. We study the values of $PT_u(\pi^{\lambda}(\sigma))$ and $PS_u(\pi^{\lambda}(\sigma))$ as functions of $\lambda$. In particular, we consider $\lambda$ a random partition distributed with the Plancherel measure, and we study the random functions $PT_u(\pi^{\lambda}(\sigma))$ and $PS_u(\pi^{\lambda}(\sigma))$ when $n$ grows. These partial sums are obviously not invariant by isomorphisms of representations, hence we consider an explicit natural construction of irreducible representations (the Young's orthogonal representation, defined in Section \ref{section young orthogonal}). We define \emph{subpartitions} $\mu_j$ of $\lambda$, denoted $\mu_j\nearrow\lambda$, the partitions of $n-1$ obtained from $\lambda$ by removing one box. Suppose that $\sigma\in S_r$ with $r\leq n-1$, then the orthogonal representation (and the right choice for the order of the basis elements) allows us to write the representation matrix $\pi^\lambda(\sigma)$ as a block diagonal matrix, where the blocks are $\pi^{\mu}(\sigma)$ for $\mu\nearrow\lambda$. In other words
\[\pi^\lambda(\sigma)=\bigoplus_{\mu\nearrow\lambda}\pi^{\mu}(\sigma)\]
if $\sigma\in S_r$ with $r\leq n-1$. This property will be proven in Proposition \ref{decomposition of the representation matrix}. We deduce a decomposition of the partial trace and partial sum of a representation matrix: we will show that there exist $\bar{j}\in\mathbb{N}$ and $\bar{u}\in[0,1]$ such that
\begin{equation}\label{intro1}
 PT_u^{\lambda}(\sigma):=PT_u(\pi^{\lambda}(\sigma))=\sum_{j<\bar{j}}\frac{\dim\mu_j}{\dim\lambda}PT_1^{\mu_j}(\sigma)+\frac{\dim\mu_{\bar{j}}}{\dim\lambda}PT_{\bar{u}}^{\mu_{\bar{j}}}(\sigma),
\end{equation}
\begin{equation}\label{intro3}
 PS_u^{\lambda}(\sigma):=PS_u(\pi^{\lambda}(\sigma))=\sum_{j<\bar{j}}\frac{\dim\mu_j}{\dim\lambda}PS_1^{\mu_j}(\sigma)+\frac{\dim\mu_{\bar{j}}}{\dim\lambda}PS_{\bar{u}}^{\mu_{\bar{j}}}(\sigma).
\end{equation}
Here $PT_1^{\mu_j}(\sigma)=\hat{\chi}^{\mu_j}(\sigma)$, while 
\[PS^{\lambda}_1(\sigma)=\sum_{i,j\leq \dim\lambda}\frac{\pi^{\lambda}(\sigma)_{i,j}}{\dim\lambda}=:TS^{\lambda}(\sigma)\]
is the \emph{total sum} of the matrix $\pi^{\lambda}(\sigma)$.

In the first (resp. the second) decomposition we call the first term the \emph{main term of the partial trace} $MT_u^{\lambda}(\sigma)$ (resp. \emph{main term of the partial sum} $MS_u^{\lambda}(\sigma)$) and the second the \emph{remainder for the partial trace} $RT_u^{\lambda}(\sigma)$ (resp. the \emph{remainder for the partial sum} $RS_u^{\lambda}(\sigma)$). The decompositions show that the behavior of partial trace and partial sum depend on, respectively, total trace and total sum. 

We consider Plancherel distributed partitions $\lambda\vdash n$. We will obtain asymptotic results on the partial trace and partial sum with the aid of Kerov's result on the asymptotic of $p^{\sharp}_{\rho}$ described in Section \ref{section: asymptotic p sharp}; we prove a central limit theorem for the total sum (Theorem \ref{convergence of total sum}): to each $\sigma \in S_r$ we associate the two values 
\[m_{\sigma}:= \mathbb{E}_{\Pl}^{r}\left[TS^{\nu}(\sigma)\right]\qquad \mbox{and}\qquad v_{\sigma}:=\binom{r}{2}\mathbb{E}_{\Pl}^{r}\left[\hat{\chi}^{\nu}_{(2,1,\ldots,1)} TS^{\nu}(\sigma)\right],\]
where $\mathbb{E}_{\Pl}^r[X^{\nu}]$ is the average of the random variable $X^{\nu}$ considered with the Plancherel measure $(\dim\nu)^2/r!$ for $\nu\vdash r$. Then
\begin{theorem}\label{1}
Fix $\sigma\in S_{r}$ and let $n\geq r$. Consider $\lambda\vdash n$ a random partition distributed with the Plancharel measure, so that $TS^{\lambda}(\sigma)$ is a random function on the space of partitions of $n$. Then, for $n\to\infty$
 \[n\cdot(TS^{\lambda}(\sigma)-m_{\sigma})\overset{d}\to \mathcal{N}(0,2 v_{\sigma}^2),\]
 where $\mathcal{N}(0,2 v_{\sigma}^2)$ is a normal random variable of variance $2 v_{\sigma}^2$ (provided that $v_{\sigma}\neq 0$) and $\overset{d}\to$ means convergence in distribution.
\end{theorem}
The idea is to show that, for $\sigma\in S_r$, the total sum $TS^{\lambda}(\sigma)$ can be written as linear combination of $\{\hat{\chi}^{\lambda}(\tau)\}_{\tau\in S_r}$. This will be proven in Theorem \ref{convergence of total sum}.
\bigskip

When investigating the partial trace, we focus on the main term, and we prove the following theorem:
\begin{theorem}\label{ora}
 let $\sigma\in S_r$ be a permutation of cycle type $\rho$ and $u\in [0,1]$. Let $\{\xi_k\}_{k\geq 2}$ be a sequence of independent standard Gaussian variables. Then, for a Plancharel distributed partition $\lambda\vdash n$,
 \[n^{\frac{|\supp(\sigma)|}{2}}MT_{u}^{\lambda}(\sigma)\overset{d}\to u\cdot \prod_{k\geq 2} k^{m_k(\rho)/2} \mathcal{H}_{m_k(\rho)}(\xi_k)\qquad\mbox{ as }n\to\infty,\]
 where $|\supp(\sigma)|$ is the size of the support of the permutation $\sigma$, $\mathcal{H}_n(x)$ is the $m$-th modified Hermite polynomial, and $m_k(\rho)$ is the number of parts of $\rho$ equal to $k$.
\end{theorem}
In short, this theorem states that if we condition the partial trace of a representation matrix by the total trace, the partial trace appears to be deterministic. We actually prove multivariate version of the theorems stated in the introduction. In particular, for Theorem \ref{ora}, the joint distributions of $n^{\frac{|\supp(\sigma_i)|}{2}}MT_{u_i}^{\lambda}(\sigma_i)$ will converge to modified Hermite polynomials of the same Gaussian variables for a family of permutations $\{\sigma_i\}$ and a family of real numbers $\{u_i\}$. Notice that this result generalizes \eqref{intro2}, although we use Kerov's result in the proof.

Informally, the main idea to prove Theorem \ref{ora} is to show that when $n$ grows,
\[\sum_{j<\bar{j}}\frac{\dim\mu_j}{\dim\lambda}\hat{\chi}^{\mu_j}(\sigma)\sim \left(\sum_{j<\bar{j}}\frac{\dim\mu_j}{\dim\lambda}\right)\hat{\chi}^{\lambda}(\sigma).\]
\smallskip

To achieve this result we need to estimate the asymptotic of $\hat{\chi}^{\lambda}(\sigma)-\hat{\chi}^{\mu}(\sigma)$, for $\lambda$ distributed with the Plancherel measure and $\mu\nearrow\lambda$. This will be done in Section \ref{Asymptotic of the main term}.
\medskip

We cannot unfortunately prove asymptotic results on the remainder $RT_u^{\lambda}(\sigma)$, although we conjecture that, for a permutation $\sigma\in S_r$ of cycle type $\rho$, a real number $u\in[0,1]$, a random partition $\lambda\vdash n$ distributed with the Plancherel measure ($n\geq r$),
\begin{equation}\label{conjecture on remainder}
n^{\frac{|\rho|-m_1(\rho)}{2}}RT_u^{\lambda}(\sigma)\overset{p}\to 0 
\end{equation}
where $\overset{p}\to$ means convergence in probability. In Section \ref{section conjecture} we describe a different conjecture, which would imply the one above, involving quotient of dimensions of irreducible representations. We give some numerical evidence. Our conjecture would imply that for $n,r,u,\sigma,\rho$ defined as before and a random Plancharel distributed $\lambda$, 
\[n^{\frac{|\rho|-m_1(\rho)}{2}}PT_u^{\lambda}(\sigma)\overset{d}\to u\cdot \prod_{k\geq 2} k^{m_k(\rho)/2} \mathcal{H}_{m_k(\rho)}(\xi_k).\]
\bigskip

Regarding the partial sum, our results on the total sum and the main term of the partial sum allow us to prove a law of large numbers (Theorem \ref{convergence of partial sum}):
\begin{theorem}
 Let $u\in [0,1]$, $\sigma\in S_r$ and $\lambda\vdash n$ a random partition as before. Then 
 \[ PS_u^{\lambda}(\sigma)\overset{p}\to u\cdot m_{\sigma}\]
 in probability as $n\to\infty$.
\end{theorem}
It is easy to see that the remainder for the partial sum goes asymptotically to zero, but we do not know how fast. For the same reasons as above, we cannot thus present a central limit theorem for the partial sum. Nevertheless, we prove a central limit theorem for the main term of the partial sum (Corollary \ref{convergence of G}): 
\begin{theorem}\label{4}
  Let $u\in [0,1]$, $\sigma\in S_r$ and $\lambda\vdash n$ a Plancherel distributed random partition. Consider $m_{\sigma}$ and $v_{\sigma}$ defined above. Then for $n\to\infty$  
 \[n\cdot\left(MS_u^{\lambda}(\sigma)-u\cdot m_{\sigma}\right)\overset{d}\to u\cdot\mathcal{N}(0,2 v_{\sigma}^2)\qquad\mbox{ as }n\to\infty,\]
 where $\mathcal{N}(0,2 v_{\sigma}^2)$ is a normal random variable of variance $2 v_{\sigma}^2$, provided that $v_{\sigma}\neq 0$.
\end{theorem}
Notice that Theorem \ref{1} can be seen as a corollary of Theorem \ref{4}, since for $u=1$ we have 
\[MS_1^{\lambda}(\sigma)=\sum_{\mu\nearrow\lambda}\frac{\dim\mu}{\dim\lambda}TS^{\mu}(\sigma)=TS^{\lambda}(\sigma).\]
We will anyway prove Theorem \ref{1} independently.
\smallskip

For the partial trace case, we show a multivariate generalization of the previous two theorems.
\medskip
\begin{remark}
 Although our results depend on the Young's orthogonal representation (Definition \ref{def young orthogonal}), the only property necessary is the decomposition of the representation matrix proved in Proposition \ref{decomposition of the representation matrix}. It is worth mentioning that such decomposition holds for other ``famous'' constructions, such as the Young's seminormal representation and the Young's natural representation (see, for example, \cite{GreeneRationalIdentity} and \cite{sagan2013symmetric} for their definitions). Indeed, suppose that $T$ is a basis vector of the irreducible module associated to $\lambda$, $\lambda\vdash n$, then these explicit constructions of the representation matrix define the action of a transposition $(k,k+1)$ on $T$, and this action does not depend on $n$, provided that $k<n$. The right choice of the order of the basis elements is anyway necessary.
 
 More precisely, if one wishes to prove our results in the setting of Young's seminormal or natural representation, he would need to change the proofs of Lemma \ref{lemma: op} and Lemma \ref{bound on entries} (the statements still hold), while Example \ref{ex wow} is conditioned to Young's orthogonal representation, and thus does not work with a different construction. The other proofs remain unchanged.
\end{remark}

The partial trace $PT_u(A)$ and the partial sum $PS_u(A)$ have been studied by D'Aristotile, Diaconis and Newman in \cite{d2003brownian} for the case in which $A$ is a random matrix of the Gaussian Unitary Ensemble (GUE). The authors showed that in this case both partial trace and partial sum, after normalization, converge to Brownian motion. Thus the GUE case has a ``higher degree of randomness'' than the partial sum and partial trace of a representation matrix, in the sense that the first order asymptotic is a stochastic process, while in the representation matrix case the first order asymptotic is deterministic, conditionally given the total trace.

\medskip
In section 2 we introduce the partial trace, and describe its decomposition through a study of the Young orthogonal representation. Moreover, we analyze the multiplication table of shifted symmetric functions, which will improve some results of \cite{ivanov2002kerov} necessary for the proof of the central limit theorem of the main term. In section 3 we prove the result concerning the asymptotics of the main term of the partial trace. In section 4 we study the total and partial sum, while in section 5 we describe a conjecture which would imply a convergence result on the partial trace.

\section{Preliminaries}\label{section Preliminaries}
\subsection{The partial trace and its main term}\label{chapter the partial trace}
Let $\lambda\vdash n$  and $\sigma\in S_r$ be a permutation of the set $\{1,\ldots,r\}$ with $r\leq n$. We see $S_r$ as a subgroup of $S_n$ by considering $r+1,\ldots,n$ as fixed points in the permutation $\sigma$, so that $\pi^{\lambda}(\sigma)$ is well defined. Unless stated otherwise, $\pi^{\lambda}(\sigma)$ will be considered an explicit matrix over the complex numbers via the construction of the Young orthogonal representation. Define the \emph{character} $\chi^{\lambda}(\sigma)$ as the trace of $\pi^{\lambda}(\sigma)$:
\[\chi^{\lambda}(\sigma):=\sum_{i\leq\dim\lambda}\pi^{\lambda}(\sigma)_{i,i}.\]
The character is a class function, so that if $\rho$ is the cycle type of $\sigma$ we will often write $\chi^{\lambda}_{\rho}$ instead of $\chi^{\lambda}(\sigma)$. Adding fixed points to a permutation corresponds to adding several $1$ to the cycle type. Hence, by definition, $\chi^{\lambda}_{\rho}=\chi^{\lambda}_{(\rho, 1^{n-r})}$, where if $\rho=(\rho_1,\ldots,\rho_l)$ then $(\rho, 1^{n-r})=(\rho_1,\ldots,\rho_l,1,\ldots,1)$ . The \emph{normalized character }is 
\[\hat{\chi}^{\lambda}_{\rho}:=\frac{\chi^{\lambda}_{\rho}}{\dim\lambda}.\]

\begin{defi}
 Let $\lambda\vdash n$, $\sigma\in S_r$ with $r\leq n$, and $u\in [0,1]$. Define the partial trace as
 \[PT_{u}^{\lambda}(\sigma):=\sum_{i\leq u \dim\lambda}\frac{\pi^{\lambda}(\sigma)_{i,i}}{\dim\lambda}.\]

\end{defi}
When $u=1$ then the partial trace corresponds to the normalized trace 
\[\hat{\chi}^{\lambda}(\sigma)=\frac{\chi^{\lambda}(\sigma)}{\dim\lambda}=\sum_{i\leq \dim\lambda}\frac{\pi^{\lambda}(\sigma)_{i,i}}{\dim\lambda}=PT_1^{\lambda}(\sigma).\]
\begin{proposition}\label{decomposition of the representation matrix}
 Let $\lambda\vdash n$ and $\sigma\in S_r$ with $r\leq n-1$. Let $\mu_1,\mu_2,\ldots,\mu_{d+1}$ be the subpartitions of $\lambda$. Consider the matrix $\pi^{\lambda}(\sigma)$ constructed with the Young's orthogonal representation, and the standard Young tableaux (which form a basis for the matrix) arranged with the last letter order defined in Section \ref{section young orthogonal}. Then 
 \begin{equation}
 \pi^{\lambda}(\sigma)= \left[
\begin{array}{c c c c c}
\pi^{\mu_1}(\sigma) &\mathbbold{0}&\mathbbold{0}&\cdots&\mathbbold{0}\\ 
\mathbbold{0} &\pi^{\mu_2}(\sigma) &\mathbbold{0}&&\vdots\\
\mathbbold{0}&\mathbbold{0}&\pi^{\mu_3}(\sigma) &\ddots\\
\vdots&\vdots&\ddots&\ddots&\mathbbold{0}\\
\mathbbold{0}&\mathbbold{0}&&\mathbbold{0}&\pi^{\mu_{d+1}}(\sigma) 
\end{array}\right].
\end{equation}
\end{proposition}
\begin{proof}
 Since adjacent transpositions of $S_n$ generate the whole group, we can suppose without loss of generality that $\sigma=(k,k+1)$. Let $T,S\in \SYT(\lambda)$. Call \emph{restriction to }$n-1$ of $T$ the standard Young tableau obtained by removing from $T$ the box containing $n$. Set $\tilde{T},\tilde{S}$ to be the restrictions to $n-1$ of respectively $T$ and $S$. Call $\sh(\tilde{T}),\sh(\tilde{S})$ the shapes of $\tilde{T},\tilde{S}$ respectively. Since $k+1\leq n-1$, then if $\sh(\tilde{T})\neq\sh(\tilde{S})$ then $\pi^{\lambda}((k,k+1))_{T,S}=0$ by the definition of Young's orthogonal representation. Suppose $\sh(\tilde{T})=\sh(\tilde{S})=\mu_j$ for some $j\leq d+1$, then it is easy to see that
 \begin{itemize}
  \item $T=S$ if and only if $\tilde{T}=\tilde{S};$
  \item $(k,k+1)T=S$ if and only if $(k,k+1)\tilde{T}=\tilde{S};$
  \item $d_k(T)=d_k(\tilde{T}).$
 \end{itemize}
Therefore if $\sh(\tilde{T})=\sh(\tilde{S})=\mu_j$ then $\pi^{\lambda}((k,k+1))_{T,S}=\pi^{\mu_j}((k,k+1))_{\tilde{T},\tilde{S}}$.

Suppose now that $\sh(\tilde{T})=\mu_i$ and $\sh(\tilde{S})=\mu_j$ with $i<j$. Recall that $\mu_j$ is the subpartition of $\lambda$ obtained by removing the box of content $x_j$. We have $i<j$ if and only if $x_i<x_j$. Equivalently, the box containing $n$ lies in a row in $T$ which is lower that $S$, and by definition this happens if and only if $T<S$ in the last letter order. This concludes the proof.
\end{proof}

Hence $\hat{\chi}^{\lambda}(\sigma)=\sum_j \frac{\dim\mu_j}{\dim\lambda}\hat{\chi}^{\mu_j}(\sigma)$. This decomposition of the total trace can be easily generalized to the partial trace:
\begin{proposition}[Decomposition of the partial trace]\label{first order decomposition of the partial trace}
 Fix $u\in [0,1]$, $\lambda\vdash n$ and $\sigma\in S_r$ with $r\leq n-1$. Set $\left(F_{\ctr}^{\lambda}\right)^{*}(u)=\frac{y_{\bar{j}}}{\sqrt{n}}=v^{\lambda}$ for some $\bar{j}$ as in the proof of Lemma \ref{behaviour of co-transition}. Define
 \[\bar{u}=\frac{\dim\lambda}{\dim\mu_{\bar{j}}}\left(u-\sum_{j<\bar{j}}\dim\mu_j\right)<1. \]
 Then
\begin{equation}\label{first decomposition of partial trace}
 PT_u^{\lambda}(\sigma)=\sum_{j<\bar{j}}\frac{\dim\mu_j}{\dim\lambda}\hat{\chi}^{\mu_j}(\sigma)+\frac{\dim\mu_{\bar{j}}}{\dim\lambda}PT_{\bar{u}}^{\mu_{\bar{j}}}(\sigma).
\end{equation}
\end{proposition}
\begin{proof}
 We can decompose the partial trace as
 \begin{equation}\label{i}
  PT_u^{\lambda}(\sigma)=\sum_{i\leq u\dim\lambda}\frac{\pi^{\lambda}(\sigma)_{i,i}}{\dim\lambda}=\sum_{i\leq \sum\limits_{j<\bar{j}}\dim\mu_j}\frac{\pi^{\lambda}(\sigma)_{i,i}}{\dim\lambda}+\sum_{\sum\limits_{j<\bar{j}}\dim\mu_j<i\leq u\dim\lambda}\frac{\pi^{\lambda}(\sigma)_{i,i}}{\dim\lambda}.
 \end{equation}
 The first sum in the right hand side  of the previous equation is 
 \[\sum_{i\leq \sum\limits_{j<\bar{j}}\dim\mu_j}\frac{\pi^{\lambda}(\sigma)_{i,i}}{\dim\lambda}=\sum_{j<\bar{j}}\frac{\dim\mu_j}{\dim\lambda}\hat{\chi}^{\mu_j}(\sigma)\]
 by the previous proposition.
 
 We consider now the second sum of the right hand side of \eqref{i}: we have $u\leq \sum\limits_{j\leq\bar{j}}\frac{\dim\mu_j}{\dim\lambda}$ by the definition of $\bar{j}$, thus
 \[\sum\limits_{j<\bar{j}}\dim\mu_j<i\leq u\dim\lambda\leq \sum\limits_{j\leq\bar{j}}\dim\mu_j.\]
Hence $\pi^{\lambda}(\sigma)_{i,i}=\pi^{\mu_{\bar{j}}}(\sigma)_{\tilde{i},\tilde{i}}$, where $\tilde{i}=i-\sum\limits_{j<\bar{j}}\dim\mu_j$, so that
\[0<\tilde{i}\leq u\dim\lambda-\sum\limits_{j<\bar{j}}\dim\mu_j=\bar{u}\dim\mu_{\bar{j}}.\]
Therefore 
\[\sum_{\sum\limits_{j<\bar{j}}\dim\mu_j<i\leq u\dim\lambda}\frac{\pi^{\lambda}(\sigma)_{i,i}}{\dim\lambda}=\sum_{\tilde{i}\leq \bar{u}\dim\mu_{\bar{j}}}\frac{\pi^{\mu_{\bar{j}}}(\sigma)_{\tilde{i},\tilde{i}}}{\dim\lambda}=\frac{\dim\mu_{\bar{j}}}{\dim\lambda}PT_{\bar{u}}^{\mu_{\bar{j}}}(\sigma),\]
and the proposition is proved.
\end{proof}
Set $\bar{j}=\bar{j}(\lambda,u)$ such that $y_{\bar{j}}=\sqrt{n}\cdot\left(F_{\ctr}^{\lambda}\right)^{*}(u)=\sqrt{n}\cdot v^{\lambda}$. We define $\sum_{y_j\leq v^{\lambda}\sqrt{n}}\frac{\dim\mu_j}{\dim\lambda}\hat{\chi}^{\mu_j}(\sigma)$ to be the \emph{main term of the partial trace} (denoted $MT^{\lambda}_u(\sigma)$), while $\frac{\dim\mu_{\bar{j}}}{\dim\lambda}PT_{\bar{u}}^{\mu_{\bar{j}}}(\sigma)$ is called the \emph{remainder} (denoted $R^{\lambda}_u(\sigma)$). The main goal of this paper is to establish the asymptotic behavior of the main term of the partial trace when $\lambda\vdash n$ is a random partition distributed with the Plancherel measure and $n\to\infty$. We set some notation: if a permutation $\sigma$ has cycle type $\rho$ we write $m_k(\sigma)=m_k(\rho)$ for the number of parts of $\rho$ which are equal to $k$ or, equivalently, for the number of cycles of $\sigma$ which are of length $k$. We define the \emph{weights} of $\sigma$ and $\rho$ as $\wt(\sigma):=\wt(\rho):=|\supp(\sigma)|=|\rho|-m_1(\rho).$ We recall now a result due to \cite{kerov1993gaussian}, equivalent to Theorem \ref{theorem: asymptotic p sharp}, with a complete proof given by \cite{ivanov2002kerov} and, independently, by \cite{hora1998central}:
\begin{theorem}\label{convergence of characters}
Consider a sequence of independent standard Gaussian random variables $\{\xi_k\}_{k\geq 2}$, and let $\mathcal{H}_m(x)$, $m\geq 1$, be the modified Hermite polynomial of degree $m$ defined by the recurrence relation $x\mathcal{H}_m(x)=\mathcal{H}_{m+1}(x)+m\mathcal{H}_{m-1}(x)$ and initial data $\mathcal{H}_0(x)=1$ and $\mathcal{H}_1(x)=x$. Let $\rho_1\vdash r_1, \rho_2\vdash r_2,\ldots$ be a sequence of partitions. Let $\lambda$ range over the space of partitions of $n$ equipped with the Plancherel measure $P_{\Pl}$, so that $\hat{\chi}_{\rho_i}^{\lambda}$ is a random function (assume $\hat{\chi}_{\rho_i}^{\lambda}=0$ if $r_i>n$). The asymptotic behavior of the irreducible character $\hat{\chi}^{\lambda}$ is:
\[\left\{n^{\frac{\wt(\rho_i)}{2}}\hat{\chi}^{\lambda}_{\rho_i}\right\}_{i\geq 1}\overset{d}\to \left\{\prod_{k\geq 2} k^{m_k(\rho_i)/2} \mathcal{H}_{m_k(\rho_i)}(\xi_k)\right\}_{i\geq 1},\]
where $\overset{d}\to$ denotes convergence of random variables in distribution.
\end{theorem}
Our result is the following:
\begin{theorem}\label{convergence of transformed co-transition}
 let $\sigma_1,\sigma_2,\ldots$ be permutations of cycle type respectively $\rho_1,\rho_2,\ldots$, and $u_1,u_2,\ldots$ numbers in $ [0,1]$. Let $\{\xi_k\}_{k\geq 2}$ be a sequence of independent standard Gaussian variables. Consider $\mathcal{H}_m(x)$ and $m_k(\rho_i)$ as before. Then for a Plancherel distributed partition $\lambda\vdash n$ and $n\to\infty$ 
 \[\left\{n^{\frac{\wt(\rho_i)}{2}}MT_{u_i}^{\lambda}(\sigma_i)\right\}_{i\geq1}\overset{d}\to \left\{u_i\cdot \prod_{k\geq 2} k^{m_k(\rho_i)/2} \mathcal{H}_{m_k(\rho_i)}(\xi_k)\right\}_{i\geq1}.\]
\end{theorem}
We will prove the Theorem in section \ref{Asymptotic of the main term}.

Notice that the procedure of Proposition \ref{first order decomposition of the partial trace} can be iterated:

\begin{proposition}\label{decomposition of the partial trace}
 Let $\sigma\in S_r$ and $\lambda\vdash n$. Set $s$ such that $n-s>r$, then there exists a sequence of partitions $\mu^{(0)}\nearrow\mu^{(1)}\nearrow\ldots\nearrow \mu^{(s)}=\lambda$ and a sequence of real numbers $0\leq u_{0},\ldots,u_s<1$ such that
 \[PT_u^{\lambda}(\sigma)=\sum_{i=1}^{s}\frac{\dim\mu^{(i)}}{\dim\lambda}MT_{u^{(i)}}^{\mu^{(i)}}(\sigma)+\frac{\dim\mu^{(0)}}{\dim\lambda}PT_{u^{(0)}}^{\mu^{(0)}}(\sigma).\]
\end{proposition}
\begin{proof}
 we prove the proposition inductively on $s$, with the case $s=0$ being trivial and the case $s=1$ corresponding to Proposition \ref{first order decomposition of the partial trace}. Fix $s>1$ and suppose that the statement is true up to $s-1$ for a general partition $\lambda$ and any $u$. let $\bar{j}$ be such that $\left(F_{\ctr}^{\lambda}\right)^{*}(u)=\frac{y_{\bar{j}}}{\sqrt{n}}=v^{\lambda}$ and
\[ \bar{u}=\frac{\dim\lambda}{\dim\mu_{\bar{j}}}\left(u-\sum_{j<\bar{j}}\dim\mu_j\right).\]
 By \eqref{first decomposition of partial trace}
 \[PT_u^{\lambda}(\sigma)=\sum_{j<\bar{j}}\frac{\dim\mu_j}{\dim\lambda}\hat{\chi}^{\mu_j}(\sigma)+\frac{\dim\mu_{\bar{j}}}{\dim\lambda}PT_{\bar{u}}^{\mu_{\bar{j}}}(\sigma).\]
 Set  $\mu^{(s-1)}=\mu_{\bar{j}}$ and $u^{(s-1)}=\bar{u}$. Then by the inductive step
 \[PT_{u^{(s-1)}}^{\mu^{(s-1)}}(\sigma)=\sum_{i=1}^{s-1}\frac{\dim\mu^{(i)}}{\dim\mu^{(s-1)}}MT_{u^{(i)}}^{\mu^{(i)}}(\sigma)+\frac{\dim\mu^{(0)}}{\dim\mu^{(s-1)}}PT_{u^{(0)}}^{\mu^{(0)}}(\sigma)\]
for a sequence of partitions $\mu^{(0)}\nearrow \mu^{(1)}\nearrow\ldots\nearrow \mu^{(s-1)}=\mu_{\bar{j}}$ and real numbers $0\leq u_{0},\ldots,u_{s-1}<1$. Set $\mu^{(s)}=\lambda$, $u^{(s)}=u$, then
\begin{align*}
 PT_u^{\lambda}(\sigma)&=\begin{multlined}[t][10.5cm]
                          MT^{\mu^{(s)}}_{u^{(s)}}(\sigma)+\frac{\dim\mu^{(s-1)}}{\dim\lambda}\sum_{i=1}^{s-1}\frac{\dim\mu^{(i)}}{\dim\mu^{(s-1)}}MT_{u^{(i)}}^{\mu^{(i)}}(\sigma)\\+\frac{\dim\mu^{(s-1)}}{\dim\lambda}\frac{\dim\mu^{(0)}}{\dim\mu^{(s-1)}}PT_{u^{(0)}}^{\mu^{(0)}}(\sigma)
                         \end{multlined}
\\
 &=\sum_{i=1}^{s}\frac{\dim\mu^{(i)}}{\dim\lambda}MT_{u^{(i)}}^{\mu^{(i)}}(\sigma)+\frac{\dim\mu^{(0)}}{\dim\lambda}PT_{u^{(0)}}^{\mu^{(0)}}(\sigma)
\end{align*} 
and the proposition is proved.
\end{proof}

 \subsection{The multiplication table of \texorpdfstring{$p^{\sharp}$}{Lg}}\label{section: multiplication table p sharp}
In this section we study the multiplication table of $p^{\sharp}_{\rho}$, for a partition $\rho\vdash k$, defined in Chapter \ref{ch: probability}. Recall that $\{p^{\sharp}_{\rho}\}$ forms a basis for the algebra of shifted symmetric functions $\Lambda^{\ast}$, where $\rho$ ranges over all partitions, $\rho\in\mathbb{Y}$. A first study of the multiplication table was done in \cite[Section 4]{ivanov2002kerov}, which we use as a starting point.
 We are interested in two filtrations of the algebra $\Lambda^{\ast}$, which are described in \hbox{\cite[section 10]{IvanovKerov1999}}:
\begin{itemize}
 \item $\deg(p_{\rho}^{\sharp})_1:=|\rho|_1:=|\rho|+m_1(\rho)$. This filtration is usually referred as the \emph{Kerov filtration};
 \item $\deg(p_{\rho}^{\sharp})_{\mathbb{N}}:=|\rho|_{\mathbb{N}}:=|\rho|+l(\rho)$, where $l(\rho)=\sum_{i\in\mathbb{N}}m_i(\rho)$ is the number of parts of $\rho$.
\end{itemize}
Our goal in this section is to develop \cite[Proposition 4.12]{ivanov2002kerov}, which gives information about the top degree term of $p_{\rho}^{\sharp}\cdot p_{\theta}^{\sharp}$ for partitions $\rho$ and $\theta$, where top degree refers to the Kerov filtration. Consequently, we obtain results on the product $\alpha_{\rho}\cdot\alpha_{\theta}$, where $\alpha_{\rho},\alpha_{\theta}$ are elements of the algebra $\mathcal{A}_{\infty}$ of partial permutations invariant by conjugation defined in Section \ref{sec: partial permutations}. 
\smallskip

We write $V^{<}_k$ for the vector space whose basis is $\{ p^{\sharp}_{\rho}$ such that $|\rho|_1<k\}$. Notice that since $\deg(\cdot)_1$ is a filtration, 
 \begin{equation}\label{property of filtration}
p_{\rho}^{\sharp}\cdot V^{<}_k\subseteq V^{<}_{|\rho|_1+k}.  
\end{equation}
By abuse of notation we often write $p_{\rho}^{\sharp}+V^{<}_k$ to indicate $p_{\rho}^{\sharp}$ plus a linear combination of elements with $\deg(\cdot)_1<k$. We recall a result of Ivanov and Olshanski \cite[Proposition 4.12]{ivanov2002kerov}:
\begin{lemma}\label{prop 4.12 of IO}
 For any partition $\rho$ and $k\geq 2$,
 \begin{equation}
  p_{\rho}^{\sharp}p_{k}^{\sharp}=p_{\rho\cup k}^{\sharp}+\left\{\begin{array}{lr}
                                                                  k\cdot m_k(\rho)\cdot p_{(\rho\setminus k)\cup 1^k}^{\sharp}+V^{<}_{|\rho|_1+k},&\mbox{if }m_k(\rho)\geq 1\\
                                                                  V^{<}_{|\rho|_1+k},&\mbox{if }m_k(\rho)=0,
                                                                 \end{array}\right.
 \end{equation}
where the partition $\rho\setminus k$ is obtained by removing one part equal to $k$ and $\rho\cup k$ is obtained by adding one part equal to $k$.
\end{lemma}
\begin{lemma}
 For a partition $\rho$,
 \[p_{\rho}^{\sharp}p_{1^k}^{\sharp}=p_{(\rho, 1^k)}^{\sharp}+V^{<}_{|\rho|_1+2k}.\]
\end{lemma}
\begin{proof}
 We prove the statement by induction on $k$, the case $k=1$ being proved by Ivanov and Olshanski in \hbox{\cite[Proposition 4.11]{ivanov2002kerov}}. This implies that for any $k$
 \[p_{1^k}^{\sharp}p_1^{\sharp}=p_{1^{k+1}}^{\sharp}+V^{<}_{2k+2}.\]
 Fix $k>1$ and suppose that the statement is true up to $k$, then
\begin{align*}
 p_{\rho}^{\sharp}p_{1^{k+1}}^{\sharp}&=p_{\rho}^{\sharp}p_{1^k}^{\sharp}p_{1}^{\sharp}+p_{\rho}^{\sharp}V^{<}_{2k+2}\\
 &=p_{(\rho, 1^k)}^{\sharp}p_1^{\sharp}+V^{<}_{|\rho|_1+2k+2}\\
 &=p_{(\rho,1^k, 1)}^{\sharp}+V^{<}_{|\rho|_1+2k+2},
\end{align*}
where we repetitively used property \eqref{property of filtration}.
\end{proof}
Define $\tilde{V}^=_k$ as the vector space with basis $\{p_{\theta}^{\sharp}$ such that $|\theta|_1=k$ and $m_1(\theta)>0\}$, a slightly different version of property \eqref{property of filtration} holds:
\begin{lemma}\label{property of filtration 2}
For a positive integer $k$ and a partition $\rho$
 \begin{equation}
p_{\rho}^{\sharp}\cdot \tilde{V}^=_k\subseteq \tilde{V}^=_{|\rho|_1+k}+ V^{<}_{|\rho|_1+k}.  
\end{equation}
\end{lemma}
\begin{proof}
 It is enough to show that, for a partition $\theta$ such that $|\theta|_1=k$ and $m_1(\theta)>0$, 
 \begin{equation}\label{this}
 p_{\rho}^{\sharp}\cdot p^{\sharp}_{\theta}\in \tilde{V}^=_{|\rho|_1+|\theta|_1}+ V^{<}_{|\rho|_1+|\theta|_1}. 
 \end{equation}
Set $\theta=(\tilde{\theta}, 1)$. By the previous lemma  $p_{\theta}^{\sharp}=p_{\tilde{\theta}}^{\sharp}\cdot p_1^{\sharp}+V^<_{|\tilde{\theta}|_1+2}$ and by \eqref{property of filtration} we have
\[p_{\rho}^{\sharp}\cdot p_{\theta}^{\sharp}=p_{\rho}^{\sharp}\cdot p_{\tilde{\theta}}^{\sharp}\cdot p_1^{\sharp}+V^<_{|\rho|_1+|\tilde{\theta}|_1+2}.\]
Set $C_{\rho,\tilde{\theta}}^{\tau}$ to be the structure constants in the basis $\{p_{\tau}^{\sharp}\}$ of the product $p_{\rho}^{\sharp}\cdot p_{\tilde{\theta}}^{\sharp}$, that is, 
\[p_{\rho}^{\sharp}\cdot p_{\tilde{\theta}}^{\sharp}=\sum_{|\tau|_1\leq |\rho|_1+|\tilde{\theta}|_1}C_{\rho,\tilde{\theta}}^{\tau}p_{\tau}^{\sharp},\]
where the restriction $|\tau|_1\leq |\rho|_1+|\tilde{\theta}|_1$ is a consequence of \eqref{property of filtration}. Thus 
\begin{align*}
p_{\rho}^{\sharp}\cdot p^{\sharp}_{\theta}&=p_{\rho}^{\sharp}\cdot p_{\tilde{\theta}}^{\sharp}\cdot p_1^{\sharp}+V^<_{|\rho|_1+|\tilde{\theta}|_1+2}\\
&=\sum_{|\tau|_1\leq |\rho|_1+|\tilde{\theta}|_1}C_{\rho,\tilde{\theta}}^{\tau}p_{\tau}^{\sharp}\cdot p_1^{\sharp} +V^<_{|\rho|_1+|\tilde{\theta}|_1+2}.
\end{align*}
We apply the previous lemma: $p_{\tau}^{\sharp}\cdot p_1^{\sharp}=p_{(\tau, 1)}^{\sharp}+V_{|\tau|_1+2}^<$, so that 
\[p_{\rho}^{\sharp}\cdot p^{\sharp}_{\theta}=\sum_{|\tau|_1\leq |\rho|_1+|\tilde{\theta}|_1}C_{\rho,\tilde{\theta}}^{\tau}\left(p_{(\tau, 1)}^{\sharp}+V_{|\tau|_1+2}^<\right)+V^<_{|\rho|_1+|\tilde{\theta}|_1+2}.\]
Notice that, for $|\tau|_1\leq |\rho|_1+|\tilde{\theta}|_1$, we have $p_{(\tau, 1)}^{\sharp}\in  \tilde{V}^=_{|\rho|_1+|\tilde{\theta}|_1+2}+V^<_{|\rho|_1+|\tilde{\theta}|_1+2}$. Hence
\[p_{\rho}^{\sharp}\cdot p^{\sharp}_{\theta}\in  \tilde{V}^=_{|\rho|_1+|\tilde{\theta}|_1+2}+V^<_{|\rho|_1+|\tilde{\theta}|_1+2}=\tilde{V}^=_{|\rho|_1+|\theta|_1}+ V^{<}_{|\rho|_1+|\theta|_1}.\]
This proves \eqref{this} and hence the lemma.
\end{proof}
\begin{proposition}
 Let $\rho$, $\theta$ be partitions. Then
 \[p_{\rho}^{\sharp}\cdot p_{\theta}^{\sharp}=p_{\rho\cup\theta}^{\sharp}+\tilde{V}^=_{|\rho|_1+|\theta|_1}+V^<_{|\rho|_1+|\theta|_1}.\]
\end{proposition}
\begin{proof}
 We prove the statement by induction on the number $l(\theta)-m_1(\theta)$ of parts of $\theta$ larger than $1$, with the initial case being $\theta=1^k$, shown in the previous lemma.
 \smallskip
 
 Consider now the claim true for $p_{\rho}^{\sharp}\cdot p_{\tilde{\theta}}^{\sharp}$ for some $\tilde{\theta}$, that is,
 \begin{equation}\label{ok}
  p_{\rho}^{\sharp}\cdot p_{\tilde{\theta}}^{\sharp}=p_{\rho\cup\tilde{\theta}}^{\sharp}+\tilde{V}^=_{|\rho|_1+|\tilde{\theta}|_1}+ V^<_{|\rho|_1+|\tilde{\theta}|_1}.
 \end{equation}
Set $\theta=\tilde{\theta}\cup k$ for $k\geq 2$. Then by Lemma \ref{prop 4.12 of IO}
 \begin{align*}
  p_{\theta}^{\sharp}=p_{\tilde{\theta}\cup k}^{\sharp}&=p_{\tilde{\theta}}^{\sharp}\cdot p_{k}^{\sharp}+\left\{\begin{array}{lr}
                                                                -k\cdot m_k(\tilde{\theta})\cdot p_{(\tilde{\theta}\setminus k)\cup 1^k}^{\sharp}+V^<_{|\tilde{\theta}|_1+k},&\mbox{if }m_k(\tilde{\theta})\geq 1\\
                                                                  V^<_{|\theta|_1},&\mbox{if }m_k(\tilde{\theta})=0,
                                                                 \end{array}\right.\\
&= p_{\tilde{\theta}}^{\sharp}\cdot p_{k}^{\sharp}+\tilde{V}^=_{|\theta|_1}+V^<_{|\theta|_1},
 \end{align*}
 since $|(\tilde{\theta}\setminus k)\cup 1^k|_1=|\tilde{\theta}\cup k|_1=|\theta|_1.$ Because of Property \eqref{property of filtration} and Lemma \eqref{property of filtration 2},
 \begin{equation}\label{bla}
 p_{\rho}^{\sharp}\cdot p_{\theta}^{\sharp}=p_{\rho}^{\sharp}\cdot p_{\tilde{\theta}}^{\sharp}\cdot p_k^{\sharp}+\tilde{V}^=_{|\rho|_1+|\theta|_1}+V^<_{|\rho|_1+|\theta|_1}. 
 \end{equation}
 We need thus to evaluate the term $p_{\rho}^{\sharp}\cdot p_{\tilde{\theta}}^{\sharp}\cdot p_k^{\sharp}$. We apply the inductive step \eqref{ok} and Lemma \ref{prop 4.12 of IO}:
  \begin{align*}
   p_{\rho}^{\sharp}\cdot p_{\tilde{\theta}}^{\sharp}\cdot p_k^{\sharp}&=p_{\rho\cup\tilde{\theta}}^{\sharp}\cdot p_k^{\sharp}+\tilde{V}^=_{|\rho|_1+|\tilde{\theta}|_1+k}+ V^<_{|\rho|_1+|\tilde{\theta}|_1+k}\\
   &=p_{\rho\cup\tilde{\theta}\cup k}^{\sharp}+\tilde{V}^=_{|\rho|_1+|\tilde{\theta}|_1+k}+ V^<_{|\rho|_1+|\tilde{\theta}|_1+k}.
  \end{align*}
We substitute the previous expression in \eqref{bla}
\[p_{\rho}^{\sharp}\cdot p_{\theta}^{\sharp}=p_{\rho\cup\tilde{\theta}\cup k}^{\sharp}+\tilde{V}^=_{|\rho|_1+|\tilde{\theta}|_1+k}+V^<_{|\rho|_1+|\tilde{\theta}|_1+k},\]
which concludes the proof.
\end{proof}
We can obtain a similar result in the algebra $\mathcal{A}_{\infty}$ by applying the isomorphism $F^{-1}$ defined in \ref{lemma isomorphism}. It is easy to see that $F^{-1}(\tilde{V}^=_k)$ is the space of linear combinations of $\alpha_{\rho}$ such that $|\rho|_1=k$ and $m_1(\rho)>0$. Similarly, $F^{-1}(V^<_k)$ is the space of linear combinations of $\alpha_{\rho}$ such that $|\rho|_1<k$.
\begin{corollary}\label{corollary on alpha}
 Let $\rho$, $\theta$ be partitions. Then
 \[\alpha_{\rho}\cdot \alpha_{\theta}=\alpha_{\rho\cup\theta}+F^{-1}(\tilde{V}^=_{|\rho|_1+|\theta|_1})+F^{-1}(V^<_{|\rho|_1+|\theta|_1}).\]
\end{corollary}

\section{Asymptotic of the main term of the partial trace}\label{Asymptotic of the main term}
Let $\lambda\vdash n$ be a random partition distributed with the Plancherel measure. We say that a function $X$ on the set of partitions is $X\in o_P(n^{\beta})$ if $n^{-\beta}X(\lambda)\overset{p}\to 0$. Similarly, $X\in O_P(n^{\beta})$ is \emph{stochastically bounded} by $n^{\beta}$ if for any $\epsilon>0$ there exists $M>0$ such that 
\[P_{\Pl}(|X(\lambda)n^{-\beta}|>M)\leq \epsilon.\]
For example, as a consequence of Kerov's result on the convergence of characters (see Theorem \ref{convergence of characters}), $\hat{\chi}^{\lambda}_{\rho}\in O_P(n^{-wt(\rho)/2})$. If a function $X(\lambda,\mu_j)$ depends also on a subpartition $\mu_j\nearrow\lambda$, by $X(\lambda,\mu_j)\in o_P(n^{\beta})$ we mean that $\max_j X(\lambda,\mu_j)\in o_P(n^{\beta})$, and similarly for the notion of stochastic boundedness.

In this section we prove Theorem \ref{convergence of transformed co-transition}. The main step is to prove that for a fixed partition $\rho$, a random Plancherel distributed partition $\lambda\vdash n$ and a subpartition $\mu\nearrow \lambda$ we have \hbox{$\hat{\chi}^{\lambda}_{\rho}-\hat{\chi}^{\mu}_{\rho}\in o_P(n^{-\frac{\wt(\rho)}{2}})$}.
\bigskip

For a partition $\nu=(\nu_1,\ldots,\nu_q)$ we define slightly modified power sums $\tilde{p}_{\nu}$; we will show (Lemma \ref{lemma between mu and lambda} and Equation \eqref{ciao}) that 
\begin{multline}\label{ecco1}
 \hat{\chi}^{\lambda}(\varphi_n(\tilde{p}_{\nu}(\xi_1,\ldots,\xi_n)))-\hat{\chi}^{\mu}(\varphi_{n-1}(\tilde{p}_{\nu}(\xi_1,\ldots,\xi_{n-1})))=\\\tilde{p}_{\nu}(\mathcal{C}_{\lambda})-\tilde{p}_{\nu}(\mathcal{C}_{\mu})\in o_P(n^{\frac{|\nu|+q}{2}}).
\end{multline}
In order to translate this result on a bound on $\hat{\chi}^{\lambda}_{\rho}-\hat{\chi}^{\mu}_{\rho}$ we need to study the expansion of the modified power sums evaluated on $\Xi$, the infinite set of Jucys-Murphy elements in the theory of partial permutations:
\[ \tilde{p}_{\nu}(\xi_1,\ldots,\xi_n)=\sum_{\varsigma}c_{\varsigma}\alpha_{\varsigma;n}.\]
With the right choice of $\nu$ and the right filtration we will prove (Proposition \ref{final corollary}) that 
\begin{align}\label{ecco2}
\tilde{p}_{\nu}(\mathcal{C}_{\lambda})-\tilde{p}_{\nu}(\mathcal{C}_{\mu})&= \hat{\chi}^{\lambda}\circ\varphi_n\left(\sum_{\varsigma}c_{\varsigma}\alpha_{\varsigma;n}\right) -\hat{\chi}^{\mu}\circ\varphi_{n-1}\left(\sum_{\varsigma}c_{\varsigma}\alpha_{\varsigma;n-1}\right) \\&=\hat{\chi}^{\lambda}_{\rho}-\hat{\chi}^{\mu}_{\rho}+ o_P(n^{-\frac{\wt(\rho)}{2}}).
\end{align}
\bigskip
Comparing \eqref{ecco1} and \eqref{ecco2}, this gives 
\[\hat{\chi}^{\lambda}_{\rho}-\hat{\chi}^{\mu}_{\rho}\in o_P(n^{-\frac{\wt(\rho)}{2}}).\]
\begin{remark}
It is easy to see that $\hat{\chi}^{\lambda}_{\rho}-\hat{\chi}^{\mu}_{\rho}\in o_P(n^{-\frac{|\rho|-l(\rho)}{2}})$, where $l(\rho)$ is the number of parts of $\rho$; for example, considering balanced diagrams, in \cite{feray2007asymptotics} the authors prove an explicit formula for the normalized character, which implies $\hat{\chi}^{\lambda}_{\rho}-\hat{\chi}^{\mu}_{\rho}\in o_P(n^{-\frac{|\rho|-l(\rho)}{2}})$. Alternatively one can look at the descriptions of normalized characters expressed as polynomials in terms of free cumulants, studied by Biane in \cite{Biane2003}. The action of removing a box from a random partition $\lambda$ affects free cumulants in a sense described in \cite{dolkega2010explicit}, which provides the aforementioned bound. On the other hand we needed a stronger result, namely Proposition \ref{final corollary}, and for this reason we introduce the modified power sums.
\end{remark}
\bigskip

Through the section $\sigma$ will be a fixed permutation, $\rho$ its cycle type, and we set the partition $\nu=(\nu_1,\ldots,\nu_q)$ such that $\rho=\nu+\underline{1}=(\nu_1+1,\ldots,\nu_q+1,1\ldots,1)$.
 \subsection{Modified power sums}
 \begin{defi}
 Let $k,n$ be positive integers, and $\{x_1,\ldots,x_n\}$ formal variables. The $k-$th \emph{modified power sum} is 
 \[\tilde{p}_k(x_1,\ldots,x_n):=p_k(x_1,\ldots,x_n)-\cat\left(\frac{k}{2}\right)\cdot\left(\frac{k}{2}\right)! \hspace{0.3cm}\alpha_{(1^{\frac{k}{2}+1});n},\]
 where, for a partition $\varsigma$, $\alpha_{\varsigma;n}$ was defined in Section \ref{sec: partial permutations}, and $\cat(l)=\binom{2k}{l}\frac{1}{l+1}$ is the $l$-th Catalan number if $l$ is an integer, and $0$ otherwise.
 
  For a partition $\nu=(\nu_1,\ldots,\nu_q)$ set
   \[\tilde{p}_{\nu}(x_1,\ldots,x_n)=\prod_{i=1}^q\tilde{p}_{\nu_i}(x_1,\ldots,x_n).\]
   Moreover, if $\{x_1,x_2,\ldots\}$ is an infinite sequence of formal variables, we set
   \[\tilde{p}_{\nu}(x_1,x_2,\ldots)=\prod_{i=1}^q\left(p_{\nu_i}(x_1,x_2,\ldots)-\cat\left(\frac{\nu_i}{2}\right)\cdot\left(\frac{\nu_i}{2}\right)! \hspace{0.3cm}\alpha_{(1^{\frac{\nu_i}{2}+1})}\right).\]
 \end{defi}
 \bigskip
 
 \begin{lemma}\label{power sums on partial jucys murphy}
 Set $\nu=(\nu_1,\ldots,\nu_q)$, then
  \begin{equation}\label{p tilde}
   \tilde{p}_{\nu}(\Xi)=\sum_{\varsigma}c_{\varsigma}\alpha_{\varsigma},
  \end{equation}
  for some non-negative integers $c_{\varsigma}$. We have
  \begin{enumerate}
   \item the sum runs over the partitions $\varsigma$ such that $|\varsigma|_1\leq |\nu|_{\mathbb{N}}$;
   \item $c_{\nu+\underline{1}}=\prod_i (m_i(\nu)!)$;
   \item if $c_{\varsigma}\neq 0$, $|\varsigma|_1=|\nu|_{\mathbb{N}}$ and $\varsigma\neq \nu+\underline{1}$ then $m_1(\varsigma)>0$.
  \end{enumerate}
 \end{lemma}
\bigskip

\begin{proof}
\begin{proofpart}
Since $\deg(\alpha_{\varsigma})_1=|\varsigma|_1$ is a filtration, it is enough to prove the statement for $q=1$, that is, $\nu=(k)$ for some positive integer $k$. We consider
\begin{equation}\label{power}
 p_k(\Xi)=\sum_{h\geq 2}\quad\sum_{1\leq j_1,\ldots,j_k<h}\left((j_1,h),\{j_1,h\}\right)\cdot\ldots\cdot\left((j_k,h),\{j_k,h\}\right).
\end{equation}
Let $\left(\sigma,D\right)=\left((j_1,h)\cdot\ldots\cdot(j_k,h),\{j_1,\ldots,j_k,h\}\right)$ be a term of the previous sum and let $\varsigma\vdash |D|$ be the cycle type of $\sigma$. This is the outline of the proof: firstly, we show an upper bound for $|\varsigma|_1,$ that is, $|\varsigma|_1\leq k+2$. Then we check some conditions that $\sigma$ must satisfy in order to have $|\varsigma|_1=k+2$ (if such $\sigma$ exists). We prove that $|\varsigma|_1= k+2$ if and only if $(\sigma,D)=(\id,D)$ with $|D|=k/2+1$ and $k $ even. Finally, we calculate how many times the partial permutation $(\id,D)$, with $|D|=k/2+1$, appears in the modified power sum $p_k(\Xi)$. This number is the coefficient of $\alpha_{(1^{\frac{k}{2}+1})}$ in $p_k(\Xi)$, and we show that it is equal to $\cat\left(\frac{k}{2}\right)\cdot\left(\frac{k}{2}\right)!$. Since we defined $\tilde{p}_{k}(\Xi)=p_{k}(\Xi)-\cat\left(\frac{k}{2}\right)\cdot\left(\frac{k}{2}\right)!\cdot\alpha_{(1^{\frac{k}{2}+1})}$ this shows that all the terms in $\tilde{p}_k(\Xi)$ satisfy $|\varsigma|_1\leq k+1=|\nu|_{\mathbb{N}}$ when $\nu=(k)$. Moreover, the fact that $c_{\varsigma}\in\Z_{\geq 0}$ is a direct consequence of \eqref{eq:feray}.
\medskip 

We first estimate
\begin{align*}
 &|\varsigma|_1=|\{j_1,\ldots,j_k,h\}|+m_1(\varsigma)\\
 &= 2\#\{j_i, \mbox{ s.t. } \sigma j_i=j_i\}+\#\{j_i, \mbox{ s.t. } \sigma j_i\neq j_i\}+2\delta(h\mbox{ is fixed by }\sigma)+\delta(h\mbox{ is not fixed by }\sigma),
\end{align*}
where $\delta$ is the Kronecker delta. We stress out that, when counting the set cardinalities above, we do not count multiplicities; for example, if 
\[\sigma=(1,5)(2,5)(2,5)(1,5)(3,5)=(1)(2)(3,5),\qquad\mbox{ then}\]
\[\#\{j_i, \mbox{ s.t. } \sigma j_i=j_i\}=\#\{1,2\}=2,\qquad \#\{j_i, \mbox{ s.t. } \sigma j_i\neq j_i\}=\#\{3\}=1,\qquad\varsigma=(2,1,1).\]
Notice that in the sequence $(j_1,\ldots,j_k)$ all fixed points of $\sigma$ must appear at least twice, while non fixed points must appear at least once, hence
\[2\#\{j_i, \mbox{ s.t. } \sigma j_i=j_i\}+\#\{j_i, \mbox{ s.t. } \sigma j_i\neq j_i\}\leq k,\]
while obviously,
\[2\delta(h\mbox{ is fixed by }\sigma)+\delta(h\mbox{ is not fixed by }\sigma)\leq 2.\]
Thus $|\varsigma|_1\leq k+2.$
\medskip

Suppose there exists $\sigma$ such that $|\varsigma|_1=k+2$, then from the proof of the inequality $|\varsigma|_1\leq k+2$ we know that $\sigma$ satisfies
\begin{equation}\label{property at most twice}
 \mbox{for each }i,\mbox{ }j_i\mbox{ appears at most twice;}
\end{equation}
\begin{equation}\label{property of exactly twice}
 j_i \mbox{ is fixed by }\sigma\mbox{ iff it appears exactly twice in the multiset }\{j_1,\ldots,j_k\};
\end{equation}
\begin{equation}\label{property of h fixed}
 h\mbox{ is a fixed point}.
\end{equation}
 We prove by induction on $k$ that if $|\varsigma|_1=k+2$ and $\alpha_{\varsigma}$ appears in the sum \eqref{power} with nonzero coefficient then $k$ is even and $\varsigma=(1^{\frac{k}{2}+1})$. If $k=1$ then $\varsigma=(2)$ and $|\varsigma|_1=2<3=k+2$. If $k=2$ then $\varsigma$ can be either $\varsigma=(3)$ or $\varsigma=(1,1)$. By our request that $|\varsigma|_1=4=k+2$ we see that we must have $\varsigma=(1,1)$. Consider now that the statement is true up to $k-1$ and $\sigma=(j_1,h)\cdot\ldots\cdot(j_k,h)$. By property \eqref{property of h fixed} $h$ is fixed, so that $j_k$ must appear at least twice, and by \eqref{property at most twice} $j_k$ appears exactly twice. Hence, by \eqref{property of exactly twice} $j_k$ is fixed, thus it exists a unique $l<k-1$ such that $j_{l+1}=j_k$ and 
\[\sigma=(j_1,h)\ldots(j_l,h)\cdot(j_k,h)\cdot(j_{l+2},h)\ldots(j_{k-1},h)(j_k,h)=\tau_1(j_k,h)\tau_2(j_k,h)=\tau_1\tau_2,\]
where 
\[\tau_1=\left\{\begin{array}{lr}   \id &\mbox{if }l=0\\(j_1,h)\ldots(j_l,h)&\mbox{if }l>0,        \end{array}\right.\qquad\quad\tau_2=\left\{\begin{array}{lr}   \id &\mbox{if }l=k-2\\(j_{l+2},j_k)\ldots(j_{k-1},j_k)&\mbox{if }l<k-2.        \end{array}\right.\]
If $l=0$ or $l=k-2$ we can apply induction on either $\tau_1$ or $\tau_2$ and the result follows, so consider $0<l<k-2$. We claim that the sets $\{j_1,\ldots j_l\}$ and $\{j_{l+2},\ldots,j_k\}$ are disjoint. Notice first that $j_k$ does not appear in $\{j_1,\ldots j_l\}$. Suppose $j_a=j_b$ for some $a\leq l<b$, and choose $b$ minimal with this property. Then $j_b$ is not fixed by $\tau_2$, that is, $\tau_2(j_b)=j_{\tilde{b}}$ with either $\tilde{b}<b$ or $\tilde{b}=k$. In the first case ($\tilde{b}<b$), by minimality of $b$, $j_{\tilde{b}}$ does not appear in $\tau_1$; in the second case ($\tilde{b}=k$) we know that $j_k$ does not appear in $\{j_1,\ldots j_l\}$. Hence in either case $\sigma(j_b)=j_{\tilde{b}}$. This is a contradiction, since $j_b$ appears twice in $\sigma$ and therefore must be fixed (property \eqref{property of exactly twice}). This proves that $\{j_1,\ldots j_l\}\cap\{j_{l+2},\ldots,j_k\}=\emptyset$.

Therefore $\tau_1$ and $\tau_2$ respect properties \eqref{property at most twice},\eqref{property of exactly twice} and \eqref{property of h fixed}, and we can apply the inductive hypothesis to obtain that $\tau_1=\id=\tau_2$ and both $l$ and $k-l-2$ are even. We conclude hence that if $|\varsigma|_1=k+2$ and $\alpha_{\varsigma}$ appears in the sum \eqref{p tilde} then $\varsigma=(1^{\frac{k}{2}+1})$ and $k$ is even. 
\medskip 

We assume now that $k$ is even and we calculate the coefficient of $\alpha_{(1^{\frac{k}{2}+1})}$ in $p_k(\Xi)$, which is equal to the coefficient of $\left(\id,D\right)=\left((j_1,h)\cdot\ldots\cdot(j_k,h),\{j_1,\ldots,j_k,h\}\right)$ in the sum \eqref{power}, for a fixed $D$ of cardinality \hbox{$k/2+1$}. In order to compute this coefficient we count the number of lists $L=\left((j_1,h),\ldots,(j_k,h)\right)$ such that \hbox{$\{j_1,\ldots,j_k,h\}=D$} and $(j_1,h)\ldots(j_k,h)=\id$. We call $\mathcal{L}_D$ the set of these lists.
\smallskip
Define a set partition of a set $X$ to be a set of subsets of $X$ (called \emph{blocks} of the set partition) such that $X$ is the disjoint union of these subsets. Define a \emph{pair set partition} of the set $[k]=\{1,\ldots,k\}$ to be a set partition in which the blocks have cardinality $2$. Fix a pair set partition $A=\{(r_1,s_1\},\ldots,\{r_{\frac{k}{2}},s_{\frac{k}{2}}\}\}$ of $[k]$ into pairs such that $r_a<s_a$ for all $a$. Such a set partition is said to be \emph{crossing} if $r_a<r_b<s_a<s_b$ for some $a,b\leq k/2$, otherwise the set partition is said to be \emph{non crossing}. Calling $\mathcal{S}_k$ the set of non crossing pair set partitions, it is known that $|\mathcal{S}_k|=\cat(k/2)$, see \cite{hora1998central}. We build a map $\psi_k\colon \mathcal{L}_D\to\mathcal{S}_k$ and we prove that this map is $\left(\frac{k}{2}\right)!$-to-one, which implies that $|\mathcal{L}_D|=c_{(1^{\frac{k}{2}+1})}=\cat\left(\frac{k}{2}\right)\cdot\left(\frac{k}{2}\right)!$.
\smallskip

Let $L\in\mathcal{L}_D$, $L=(L_1,\ldots,L_k)=((j_1,h),\ldots,(j_k,h))$; we have proven that, since 
\[(j_1,h)\cdot\ldots\cdot(j_k,h)=\id\qquad \mbox{and}\qquad |D|=|\{j_1,\ldots,j_k,h\}|=k/2+1,\]
then each element $(j_i,h)$ must appear exactly twice in $L$. We construct a pair partition $\psi_k(L)$ such that a pair $\{r,s\}\in \psi_k(L)$ iff $j_r=j_s$. By \cite[Lemma 2]{hora1998central} this set partition is non crossing. This map is clearly surjective, although not injective: every permutation $\gamma$ acting on $D=\{j_1,\ldots,j_k,h\}$ which fixes $h$ acts also on $\mathcal{L}_D$: $\gamma(L)=((\gamma(j_1),h),\ldots,(\gamma(j_k),h))$. Notice that $h>j_1,\ldots,j_k$ and this is why $\gamma(h)=h$ in order to have an action on $\mathcal{L}_D$. Moreover $\psi_k(L)=\psi_k(L')$ if and only if $L=\gamma(L')$ for some $\gamma$ in $S_{D\setminus \{h\}}$. Thus $\psi_k$ is a $\left(\frac{k}{2}\right)!$-to-one map, hence $c_{(1^{\frac{k}{2}+1})}=\cat\left(\frac{k}{2}\right)\cdot\left(\frac{k}{2}\right)!$ and the first part of the proof is concluded.
\end{proofpart}
\begin{proofpart}
 The second statement is a consequence \eqref{eq:feray} and it is shown in \cite[section 2]{feray2012partial}.
\end{proofpart}
\begin{proofpart}
 This claim is proven by induction on $q$, the number of parts of $\nu$. The initial case is $q=1$ and
 \[\tilde{p}_k(\Xi)=\sum_{|\varsigma|_1\leq k+1}c_{\varsigma}\alpha_{\varsigma}.\]
 We consider the possible $\alpha_{\varsigma}$ appearing in the sum such that $|\varsigma|_1=k+1$ and $m_1(\varsigma)=0$, so that $|\varsigma|=k+1$. Expanding the sum in a way similar to Equation \eqref{power}, we see that $|\varsigma|=k+1=\#\{j_1,\ldots,j_k,h\}$, and all the elements in this set must be pairwise different. Hence $\sigma=(j_1,h)\cdot\ldots\cdot (j_k,h)=(h,j_k,\ldots,j_1)$ and $\varsigma=(k+1)$. Thus the statement for $q=1$ is proved.
 
 Let $\nu=(\nu_1,\ldots,\nu_q)$ with $q\geq 2$. Set $\tilde{\nu}=(\nu_1,\ldots,\nu_{q-1})$ and suppose, by the induction hypothesis, that the assertion is true for $\tilde{p}_{\tilde{\nu}}(\Xi)$. Then
 \begin{align*}
  \tilde{p}_{\nu}(\Xi)&=\tilde{p}_{\tilde{\nu}}(\Xi)\tilde{p}_{\nu_q}(\Xi)\\
  &=\left(\sum_{|\tilde{\varsigma}|_1\leq |\tilde{\nu}|_{\mathbb{N}}}c_{\tilde{\varsigma}}\alpha_{\tilde{\varsigma}}\right)\cdot\left(\sum_{|\theta|_1\leq \nu_q+1}c_{\theta}\alpha_{\theta}\right)\\
  &=\sum_{\substack{|\tilde{\varsigma}|_1\leq |\tilde{\nu}|_{\mathbb{N}}\\|\theta|_1\leq \nu_q+1}}c_{\tilde{\varsigma}}c_{\theta}\cdot \alpha_{\tilde{\varsigma}}\alpha_{\theta}.
 \end{align*}
We apply now Corollary \ref{corollary on alpha} and obtain
\[\tilde{p}_{\nu}(\Xi)=\sum_{\substack{|\tilde{\varsigma}|_1\leq |\tilde{\nu}|_{\mathbb{N}}\\|\theta|_1\leq \nu_q+1}}c_{\tilde{\varsigma}}c_{\theta}\alpha_{\tilde{\varsigma}\cup\theta}+F^{-1}(\tilde{V}^=_{|\nu|_{\mathbb{N}}})+F^{-1}(V^<_{|\nu|_{\mathbb{N}}}).\]
 Hence there exists only one term in the previous sum such that $|\varsigma|_1=|\tilde{\varsigma}\cup\theta|_1=|\nu|_{\mathbb{N}}$ and $m_1(\varsigma)=0$, that is, $\varsigma=\nu+\underline{1}$, and the proof is completed. \qedhere
\end{proofpart}
\end{proof}

\begin{lemma}\label{newlemma}
Recall that $\varphi\colon\mathcal{A}_n\to Z(\mathbb{C}[S_n])$ is the homomorphism that sends $(\sigma,D)$ to $\sigma$. Consider a partition $\rho$ such that $|\rho|\leq n$ and a random Plancherel distributed partition $\lambda\vdash n$. Then $\hat{\chi}^{\lambda}(\varphi_n(\alpha_{\rho;n}))\in O_P( n^{\frac{|\rho|_1}{2}})$, where $|\rho|_1=|\rho|+m_1(\rho)$.
\end{lemma}
\begin{proof}
From Equation \eqref{act of phi}, $\hat{\chi}^{\lambda}(\varphi_n(\alpha_{\rho;n}))=\frac{n^{\downarrow|\rho|}}{z_{\rho}}\hat{\chi}^{\lambda}_{\rho}$. Since $\hat{\chi}^{\lambda}_{\rho}\in O_P(n^{-\wt(\rho)/2})$ the lemma follows.
\end{proof}

Recall that for a partition $\lambda$, $\mathcal{C}_{\lambda}$ is the multiset of contents of $\lambda$. For a subpartition $\mu_j\nearrow\lambda$ such that the content $c(\lambda/\mu_j)=y_j$, it is clear that $\mathcal{C}_{\lambda}=\mathcal{C}_{\mu_j}\cup\{y_j\}$.
 \begin{lemma}\label{lemma between mu and lambda} Fix a partition $\nu=(\nu_1,\ldots,\nu_q)$ and let $\xi_1,\ldots, \xi_n$ be the partial Jucys-Murphy elements defined in Section \ref{sec: partial permutations}. Consider a random Plancherel distributed partition $\lambda$ and $\mu_j\nearrow\lambda$. Then
  \[\hat{\chi}^{\lambda}(\varphi_n(\tilde{p}_{\nu}(\xi_1,\ldots,\xi_n)))-\hat{\chi}^{\mu_j}(\varphi_{n-1}(\tilde{p}_{\nu}(\xi_1,\ldots,\xi_{n-1})))=\tilde{p}_{\nu}(\mathcal{C}_{\lambda})-\tilde{p}_{\nu}(\mathcal{C}_{\mu_j})\in o_P(n^{\frac{|\nu|+q}{2}});\]
  that is, $\max_j \left(\hat{\chi}^{\lambda}(\varphi_n(\tilde{p}_{\nu}(\xi_1,\ldots,\xi_n)))-\hat{\chi}^{\mu_j}(\varphi_{n-1}(\tilde{p}_{\nu}(\xi_1,\ldots,\xi_{n-1})))\right)\in o_P(n^{\frac{|\nu|+q}{2}})$.
 \end{lemma}
\begin{proof}
The first equality come from the fact that $\hat{\chi}^{\lambda}\circ\varphi_n$ applied to a symmetric function of partial Jucys-Murphy elements equals the same symmetric function evaluated on the contents of $\lambda$ (see Equation \eqref{ciao}). 
We compute first $\tilde{p}_k(\mathcal{C}_{\lambda})$ for $k\in\mathbb{N}$ through equations \eqref{act of phi} and \eqref{ciao}:
\begin{align*}
 \tilde{p}_k(\mathcal{C}_{\lambda})&=\hat{\chi}^{\lambda}(\varphi_n(\tilde{p}_k(\xi_1,\ldots,\xi_n)))\\
 &=\hat{\chi}^{\lambda}(\varphi_n(p_k(\xi_1,\ldots,\xi_n)))-\cat\left(\frac{k}{2}\right)\cdot\left(\frac{k}{2}\right)!\hspace{0.3cm}\hat{\chi}^{\lambda}(\varphi_n(\alpha_{1^{\frac{k}{2}+1};n}))\\
 &=p_k(\mathcal{C}_{\lambda})-\frac{\cat(k/2)}{\frac{k}{2}+1}n^{\downarrow\left(\frac{k}{2}+1\right)}.
\end{align*}
Hence, for $\mu_j\nearrow\lambda$,
\begin{align*}
 \tilde{p}_k(\mathcal{C}_{\lambda})-\tilde{p}_k(\mathcal{C}_{\mu_j})&=p_k(\mathcal{C}_{\lambda})-p_k(\mathcal{C}_{\mu_j})-\frac{\cat(k/2)}{\frac{k}{2}+1}\left(n^{\downarrow\left(\frac{k}{2}+1\right)}-(n-1)^{\downarrow\left(\frac{k}{2}+1\right)}\right)\\
 &=p_k(\mathcal{C}_{\lambda})-p_k(\mathcal{C}_{\mu_j})-\cat\left(\frac{k}{2}\right)(n-1)^{\downarrow\frac{k}{2}}.
\end{align*}
Notice that $p_k(\mathcal{C}_{\lambda})=p_k(\mathcal{C}_{\mu_j}\cup y_j)=p_k(\mathcal{C}_{\mu_j})+y_j^k.$ Then 
\[\tilde{p}_k(\mathcal{C}_{\lambda})-\tilde{p}_k(\mathcal{C}_{\mu_j})=y_j^k-\cat\left(\frac{k}{2}\right)(n-1)^{\downarrow\frac{k}{2}}.\]
It is clear that $y_j<\max\{\lambda_1,\lambda_1'\}$ since $y_j$ is the content of a box of $\lambda$ (here $\lambda_1$ is the longest part of $\lambda$ and $\lambda_1'$ is the number of parts of $\lambda$). Let now $\lambda$ be a random partition of $n$ distributed with the Plancharel measure. It is shown in \cite[Lemma 1.5]{romik2015surprising} that, with probability that goes to $1$, both $\lambda_1$ and $\lambda_1'$ are smaller than $3\sqrt{n}$, hence for each subpartition $\mu_j$ of $\lambda$, one has $y_j\in O_P(\sqrt{n})$. Therefore we obtain that $ \tilde{p}_k(\mathcal{C}_{\lambda})-\tilde{p}_k(\mathcal{C}_{\mu_j})\in O_P(n^{\frac{k}{2}})\in o_P(n^{\frac{k+1}{2}})$.

 In the general case $\nu=(\nu_1,\ldots,\nu_q)$ one has
 \begin{align*}
  \tilde{p}_{\nu}(\mathcal{C}_{\lambda})&=\prod_{i=1}^q\tilde{p}_{\nu_i}(\mathcal{C}_{\lambda})\\
  &=\prod_{i=1}^q\tilde{p}_{\nu_i}(\mathcal{C}_{\mu_j}\cup y_j)\\
  &=\prod_{i=1}^q\left(\tilde{p}_{\nu_i}(\mathcal{C}_{\mu_j})+y_j^{\nu_i}-\cat\left(\frac{\nu_i}{2}\right)(n-1)^{\downarrow\frac{\nu_i}{2}}   \right)\\
  &=\sum_{A\subseteq\{1,\ldots,q\}}\prod_{i\in A}\tilde{p}_{\nu_i}(\mathcal{C}_{\mu_j})\prod_{i\notin A}\left(y_j^{\nu_i}-\cat\left(\frac{\nu_i}{2}\right)(n-1)^{\downarrow\frac{\nu_i}{2}}\right).
 \end{align*}
Therefore 
  \[\tilde{p}_{\nu}(\mathcal{C}_{\lambda})-\tilde{p}_{\nu}(\mathcal{C}_{\mu_j})=\sum_{A\subsetneq\{1,\ldots,q\}}\prod_{i\in A}\tilde{p}_{\nu_i}(\mathcal{C}_{\mu_j})\prod_{i\notin A}\left(y_j^{\nu_i}-\cat\left(\frac{\nu_i}{2}\right)(n-1)^{\downarrow\frac{\nu_i}{2}}\right).  \]

We use now Lemma \ref{power sums on partial jucys murphy} and \ref{newlemma}, which show that the factor $\prod_{i\in A}\tilde{p}_{\nu_i}(\mathcal{C}_{\mu_j})$ is in $O_P(n^{\frac{1}{2}(\sum_{i\in A}\nu_i+|A|)})$ and 
\[\prod_{i\notin A}\left(y_j^{\nu_i}+\cat\left(\frac{\nu_i}{2}\right)(n-1)^{\downarrow\frac{\nu_i}{2}}\right)\in O_P(n^{\frac{1}{2}\sum_{i\notin A}\nu_i}).\]
 Therefore $\tilde{p}_{\nu}(\mathcal{C}_{\lambda})-\tilde{p}_{\nu}(\mathcal{C}_{\mu_j})\in O_P(n^l)$, with 
 \[l=\max_{A\subsetneqq\{1,\ldots,q\}}\frac{1}{2}\sum_i\nu_i+\frac{|A|}{2}<\frac{|\nu|+q}{2}.\qedhere\] 
\end{proof}
\begin{proposition}\label{final corollary}
Let as before $\rho$ be a partition of $r$ with $r\leq n$. For a Plancherel distributed partition $\lambda\vdash n$ and $\mu\nearrow\lambda$ then
\[\hat{\chi}^{\lambda}_{\rho}-\hat{\chi}^{\mu}_{\rho}\in o_P(n^{-\frac{\wt(\rho)}{2}}).\]
\end{proposition}
\begin{proof}
Set $\nu=(\nu_1,\ldots, \nu_q)$ such that $\rho=(\nu_1+1,\ldots,\nu_q+1,1,\ldots,1)$. obviously, $\hat{\chi}^{\lambda}_{\nu+\underline{1}}=\hat{\chi}^{\lambda}_{\rho}$ and $\hat{\chi}^{\mu}_{\nu+\underline{1}}=\hat{\chi}^{\mu}_{\rho}$ (although $\nu+\underline{1}\neq\rho$ in general); moreover, $|\nu|+q=\wt(\rho)$. We want to prove that $\hat{\chi}^{\lambda}_{\nu+\underline{1}}-\hat{\chi}^{\mu}_{\nu+\underline{1}}\in o_P(n^{-\frac{|\nu|+q}{2}})$. We prove the statement by induction on $|\nu+\underline{1}|_1=|\nu|+q$. The initial case is $\nu+\underline{1}=(1)$ and $\hat{\chi}^{\lambda}_{(1)}=1=\hat{\chi}^{\mu}_{(1)}$, so the proposition is trivially true.
 
  Consider the random function $n^{-\frac{|\nu|+q}{2}}\left(\tilde{p}_{\nu}(\mathcal{C}_{\lambda})-\tilde{p}_{\nu}(\mathcal{C}_{\mu})\right)$, which belongs to $o_P(1)$ because of the previous lemma. It can be rewritten as
\[n^{-\frac{|\nu|+q}{2}}\left(\tilde{p}_{\nu}(\mathcal{C}_{\lambda})-\tilde{p}_{\nu}(\mathcal{C}_{\mu})\right)=n^{-\frac{|\nu|+q}{2}}\sum_{\substack{\varsigma\mbox{ s.t.}\\|\varsigma|_1\leq |\nu|+q}}\left(\hat{\chi}^{\lambda}\circ\varphi_n-\hat{\chi}^{\mu}\circ\varphi_{n-1}\right)(  c_{\varsigma}\alpha_{\varsigma;n})\overset{p}\to 0,\]
where the coefficients $c_{\varsigma}$ are described in Lemma \ref{power sums on partial jucys murphy}.
\smallskip

 Equivalently
 \begin{multline*}
  n^{-\frac{|\nu|+q}{2}}\left(\tilde{p}_{\nu}(\mathcal{C}_{\lambda})-\tilde{p}_{\nu}(\mathcal{C}_{\mu})\right)=
  n^{-\frac{|\nu|+q}{2}}\sum_{|\varsigma|_1\leq |\nu|+q}\frac{c_{\varsigma}}{z_{\varsigma}}\left(n^{\downarrow|\varsigma|}\hat{\chi}^{\lambda}_{\varsigma}-(n-1)^{\downarrow|\varsigma|}\hat{\chi}^{\mu}_{\varsigma}\right)\\
  =n^{-\frac{|\nu|+q}{2}}\sum_{|\varsigma|_1\leq |\nu|+q}\frac{c_{\varsigma}}{z_{\varsigma}}\left(\hat{\chi}^{\lambda}_{\varsigma}(n^{\downarrow|\varsigma|}-(n-1)^{\downarrow|\varsigma|})+(n-1)^{\downarrow|\varsigma|}(\hat{\chi}^{\lambda}_{\varsigma}-\hat{\chi}^{\mu}_{\varsigma})\right).
  \end{multline*}
  We split the previous sum and notice that $n^{-\frac{|\nu|+q}{2}}\sum_{\varsigma}\frac{c_{\varsigma}}{z_{\varsigma}}|\varsigma|(n-1)^{\downarrow(|\varsigma|-1)}\hat{\chi}^{\lambda}_{\varsigma}\overset{p}\to 0$. Indeed 
  \[(n-1)^{\downarrow(|\varsigma|-1)}\cdot n^{-\frac{|\nu|+q}{2}}\leq n^{\frac{|\nu|+q}{2}-1}\qquad \mbox{and}\qquad n^{\frac{|\nu|+q}{2}-1} \hat{\chi}^{\lambda}_{\varsigma}\overset{p}\to 0,\]
  since $|\varsigma|_1\leq |\nu|+q$.
  
  We deal with the sum $n^{-\frac{|\nu|+q}{2}}\sum_{\varsigma}\frac{c_{\varsigma}}{z_{\varsigma}}\left((n-1)^{\downarrow|\varsigma|}(\hat{\chi}^{\lambda}_{\varsigma}-\hat{\chi}^{\mu}_{\varsigma})\right)$. We separate in this sum the terms with $|\varsigma|_1=|\nu|+q$ and $\varsigma\neq \nu+\underline{1}$, the terms with $|\varsigma|_1<|\nu|+q$, and the term corresponding to $\varsigma=\nu+\underline{1}$.
\begin{itemize}
 \item Case $|\varsigma|_1=|\nu|+q$ and $\varsigma\neq \nu+\underline{1}$: we want to estimate 
 \[n^{-\frac{|\nu|+q}{2}}\sum_{\substack{|\varsigma|_1=|\nu|+q\\m_1(\varsigma)>0}}\frac{c_{\varsigma}}{z_{\varsigma}}\left((n-1)^{\downarrow|\varsigma|}(\hat{\chi}^{\lambda}_{\varsigma}-\hat{\chi}^{\mu}_{\varsigma})\right),\]
where the restriction $m_1(\varsigma)>0$ is a consequence of Lemma \ref{power sums on partial jucys murphy}, part 3. We consider one term of the previous sum and we write $\tilde{\nu}:=\varsigma-\underline{1}$, removed of the  parts equal to zero. Notice that, as before, $\hat{\chi}^{\lambda}_{\varsigma}= \hat{\chi}^{\lambda}_{\tilde{\nu}+\underline{1}}$, and $\hat{\chi}^{\mu}_{\varsigma}= \hat{\chi}^{\mu}_{\tilde{\nu}+\underline{1}}$. Thus $|\tilde{\nu}+\underline{1}|<|\varsigma|_1=|\nu|+q$ and we can apply the induction hypothesis. Therefore $\hat{\chi}^{\lambda}_{\tilde{\nu}+\underline{1}}-\hat{\chi}^{\mu}_{\tilde{\nu}+\underline{1}}\in o_P(n^{-\frac{|\tilde{\nu}+\underline{1}|}{2}})$ and
\[n^{-\frac{|\varsigma|_1}{2}}\cdot (n-1)^{\downarrow|\varsigma|}(\hat{\chi}^{\lambda}_{\varsigma}-\hat{\chi}^{\mu}_{\varsigma})\in o_P\left(n^{\frac{|\varsigma|_1}{2}-\frac{m_1(\varsigma)}{2}-\frac{|\tilde{\nu}+\underline{1}|}{2}}\right)=o_P(1).\]
 \item Case $|\varsigma|_1<|\nu|+q$: we apply induction again and obtain $(\hat{\chi}^{\lambda}_{\varsigma}-\hat{\chi}^{\mu}_{\varsigma})\in o_P(n^{-\frac{|\varsigma|-m_1(\varsigma)}{2}})$. Therefore
 \[n^{-\frac{|\nu|+q}{2}}\cdot (n-1)^{\downarrow|\varsigma|}(\hat{\chi}^{\lambda}_{\varsigma}-\hat{\chi}^{\mu}_{\varsigma})\in o_P\left(n^{-\frac{|\nu|+q}{2}+\frac{|\varsigma|_1}{2}}\right)\subseteq o_P(1).\]
\end{itemize}
We obtain thus that 
\[n^{-\frac{|\nu|+q}{2}}\cdot (n-1)^{|\nu|+q}\frac{\prod_i m_i(\nu)!}{z_{\nu+\underline{1}}}(\hat{\chi}^{\lambda}_{\nu+\underline{1}}-\hat{\chi}^{\mu}_{\nu+\underline{1}})+o_P(1)\overset{p}\to 0,\]
which proves the statement.
\end{proof}
 \begin{proof}[Proof of Theorem \ref{convergence of transformed co-transition}]
Let $\sigma\in S_r$ with cycle type $\rho$ and let $\lambda$ be a random partition of $n$ distributed with the Plancharel measure. Set $v^{\lambda}=\left(F_{\ctr}^{\lambda}\right)^{*}(u)$ as in Proposition \ref{first order decomposition of the partial trace}. We have
\[ MT^{\lambda}_u(\sigma)=\sum_{ y_j\leq v^{\lambda}\sqrt{n}}\frac{\dim\mu_j}{\dim\lambda}\hat{\chi}^{\mu_j}_{\rho}=\sum_{ y_j\leq v^{\lambda}\sqrt{n}}\frac{\dim\mu_j}{\dim\lambda}\hat{\chi}^{\lambda}_{\rho}+o_P(n^{-\frac{\wt(\rho)}{2}}),\]
since $\sum \dim\mu_j/\dim\lambda\leq 1$, and the previous proposition. Hence
\[n^{\frac{\wt(\rho)}{2}} MT^{\lambda}_u(\sigma)=n^{\frac{\wt(\rho)}{2}}\sum_{ y_j\leq v^{\lambda}\sqrt{n}}\frac{\dim\mu_j}{\dim\lambda}\hat{\chi}^{\lambda}_{\rho}+o_P(1).\]
Finally, by Lemma \ref{behaviour of co-transition} and Theorem \ref{convergence of characters}, we obtain that given $\sigma_1,\sigma_2,\ldots$ permutations of cycle type respectively $\rho_1,\rho_2,\ldots$ and $u_1,u_2,\ldots\in [0,1]$ and calling $\{\xi_k\}_{k\geq 2}$ a family of independent standard Gaussian variables, then 
 \[\left\{n^{\frac{\wt(\rho_i)}{2}}MT_{u_i}^{\lambda}(\sigma_i)\right\}=\left\{n^{\frac{\wt(\rho_i)}{2}}\sum_{ y_j\leq v_i^{\lambda}\sqrt{n}}\frac{\dim\mu_j}{\dim\lambda}\hat{\chi}_{\rho_i}^{\lambda}+o_P(1)\right\}\]
 so that 
 \[\left\{n^{\frac{\wt(\rho_i)}{2}}MT_{u_i}^{\lambda}(\sigma_i)\right\} \overset{d}\to \left\{u_i\cdot \prod_{k\geq 2} k^{m_k(\rho_i)/2} \mathcal{H}_{m_k(\rho_i)}(\xi_k)\right\},\]
 for $i\geq 1$.
\end{proof}

\section{Sum of the entries of an irreducible representation}

In this chapter our goal is to describe the sum of the entries of the matrix associated to a Young's orthogonal representation up to a certain index (depending on the dimension of the representation). 
We stress out that the objects we study really depend on the representation matrix, and change, for example, under isomorphism of the representation. 
Some calculations are similar to those in the previous chapter: first we consider the sum of all the entries in the matrix (before this role was played by the trace), and then we study the sum of the entries whose indices $(i,j)$ satisfy \hbox{$i\leq u\dim\lambda$}, \hbox{$j\leq u\dim\lambda$}, while before we were considering the partial trace.
\subsection{Total sum}
\begin{defi}
Let $\lambda\vdash n$ and $\sigma\in S_r$ with $r\leq n$. Let $\pi^{\lambda}(\sigma)$ be the irreducible representation matrix built with the Young's orthogonal representation. The normalized total sum is
\[TS^{\lambda}(\sigma):=\sum_{i,j\leq \dim\lambda} \frac{\pi^{\lambda}(\sigma)_{i,j}}{\dim\lambda}\]
\end{defi}
The following is the main result of the section:
\begin{theorem}\label{convergence of total sum}
Fix $\sigma_1\in S_{r_1},\sigma_2\in S_{r_2},\ldots$ and let $\lambda\vdash n$ be a random Plancharel distributed partition. Define the real numbers
\[m_{\sigma_i}:= \mathbb{E}_{\Pl}^{r_i}\left[TS^{\nu}(\sigma_i)\right]\qquad \mbox{and}\qquad v_{\sigma_i}:=\binom{r_i}{2}\mathbb{E}_{\Pl}^{r_i}\left[\hat{\chi}^{\nu}_{(2,1,\ldots,1)} TS^{\nu}(\sigma_i)\right],\]
where $\mathbb{E}_{\Pl}^r[X^{\nu}]$ is the average of the random variable $X^{\nu}$ considered with the Plancherel measure $(\dim\nu)^2/r!$ for $\nu\vdash r$.
 Then
 \[\left\{n\cdot(TS^{\lambda}(\sigma_i)-m_{\sigma_i})\right\}\overset{d}\to\left\{ \mathcal{N}(0,2 v_{\sigma_i}^2)\right\},\]
 where $\mathcal{N}(0,2 v_{\sigma_i}^2)$ is a normal random variable of variance $2 v_{\sigma_i}^2$.
\end{theorem}
Note that for a permutation $\sigma\in S_r$ the random distribution function $\mathcal{N}(0,2 v_{\sigma}^2)$ is degenerate in the case $v_{\sigma}=0$. Whether or not $v_{\sigma}=0$ is a nontrivial question. At the end of the section we study the values of $v_{\sigma}$ and $m_{\sigma}$ when $\sigma$ is an adjacent transposition, and we write the explicit values of $v_{\sigma}$ and $m_{\sigma}$ for permutations of $S_4$. In order to prove Theorem \ref{convergence of total sum}, we need some preliminary results.

\begin{proposition}
Let $\lambda\vdash n$ and $r\leq n$. There exists a bijection $\phi_r$ between
 \[\SYT(\lambda)\overset{\phi_r}\simeq \bigsqcup_{\nu\vdash r} \SYT(\nu)\times \SYT(\lambda/\nu)\]
\end{proposition}
\begin{proof}
 Let $T$ be a standard Young tableau of shape $\lambda$, the image $\phi_r(T)=(U,V)$ of $T$ is defined as follows: the boxes of $T$ whose entries are smaller or equal than $r$ identify a tableau $U$ of shape $\nu\subseteq \lambda$. Define now $V$ as a standard Young tableau of skew shape $\lambda/\nu$, and in each box write $a-r$, where $a$ is the value inside the corresponding box in $T$.
\end{proof}

\begin{example}
Here is an example for $\lambda=(6,4,3,3,3,1)$ and $r=8$:
 \[
\begin{array}{c}
 \young(12568\thirteen,379\sixteen,4\twelve\seventeen,\ten\fifteen\nineteen,\eleven\eighteen\twenty,\fourteen)
\end{array}
\leftrightarrow\left(\begin{array}{c} 
\young(12568,37,4)\end{array}
,\begin{array}{c}
  \young(:::::5,::18,:49,27\eleven,3\ten\twelve,6)
 \end{array}
\right)\]
\end{example}

\begin{lemma}\label{lemma: op}
 Let $\lambda\vdash n,$ and $\sigma\in S_r$ with $r\leq n$. For two standard Young tableaux $T,T'\in SYT(\lambda)$, set $\phi_r(T)=(U,V)$ and $\phi_r(T')=(U',V')$, then 
 \[\pi^{\lambda}(\sigma)_{T,T'}=\left\{\begin{array}{lcr}0&\mbox{if}&V\neq V'\\      \pi^{\nu}(\sigma)_{U,U'}          &\mbox{if}&V=V',     \end{array}\right.\]
 where $\nu=\sh(U)=\sh(U')$ is the shape of the tableau $U$ when $V=V'$.
 \end{lemma}
\begin{proof}
We prove the lemma by induction on the number of factors in the (minimal) decomposition of $\sigma$ into adjacent transpositions. Recall that if $\sigma=(k,k+1)$, $k<r$, then by Definition \ref{def young orthogonal} 
\[\pi^{\lambda}((k,k+1))_{T,T'}=\left\{\begin{array}{cr}
1/d_k(T)&\mbox{ if } T=T';\\
\sqrt{1-\frac{1}{d_k(T)^2}}&\mbox{ if } (k,k+1)T=T'.\\
0&\mbox{else.}\end{array}\right.
\]
Hence if $V\neq V'$ then $\pi^{\lambda}((k,k+1))_{T,T'}=0$, otherwise $\pi^{\lambda}((k,k+1))_{T,T'}=\pi^{\nu}((k,k+1))_{U,U'}$ since $d_k(T)=d_k(U)$.
\smallskip

Consider a general $\sigma\in S_r$ and write it as $\sigma=(k,k+1)\tilde{\sigma}$. Then
\[\pi^{\lambda}(\sigma)_{T,T'}=\sum_{S\in \SYT(\lambda)}\pi^{\lambda}((k,k+1))_{T,S}\pi^{\lambda}(\tilde{\sigma})_{S,T'}\]
\[\hspace{-0.6cm}=\frac{1}{d_k(T)}\pi^{\lambda}(\tilde{\sigma})_{T,T'}+\sqrt{1-\frac{1}{d_k(T)^2}}\pi^{\lambda}(\tilde{\sigma})_{(k,k+1)T,T'}\delta_{\{(k,k+1)T\in \SYT(\lambda)\}}.\]
We apply the inductive hypothesis on $\tilde{\sigma}$, obtaining 
\begin{itemize}
 \item $\frac{1}{d_k(T)}\pi^{\lambda}(\tilde{\sigma})_{T,T'}=\left\{\begin{array}{lcr}0&\mbox{if}&V\neq V'\\      \frac{1}{d_k(U)}\pi^{\nu}(\tilde{\sigma})_{U,U'} &\mbox{if}&V=V',     \end{array}\right.$
 \item $\sqrt{1-\frac{1}{d_k(T)^2}}\pi^{\lambda}(\tilde{\sigma})_{(k,k+1)T,T'}=\left\{\begin{array}{lcr}0&\mbox{if}&V\neq V'\\ \sqrt{1-\frac{1}{d_k(U)^2}}\pi^{\lambda}(\tilde{\sigma})_{(k,k+1)U,U'}&\mbox{if}&V=V',     \end{array}\right.$
\end{itemize}
since if $(k,k+1)T$ is a standard Young tableau, then $\phi_r((k,k+1)T)=((k,k+1)U,V)$. Similarly, $\delta_{\{(k,k+1)T\in \SYT(\lambda)\}}=\delta_{\{(k,k+1)U\in \SYT(\nu)\}}$. To conclude
\[\frac{1}{d_k(U)}\pi^{\lambda}(\tilde{\sigma})_{U,U'}+\sqrt{1-\frac{1}{d_k(U)^2}}\pi^{\lambda}(\tilde{\sigma})_{(k,k+1)U,U'}\delta_{\{(k,k+1)U\in \SYT(\nu)\}}=\pi^{\lambda}(\sigma)_{U,U'}\qedhere\]
\end{proof}

\begin{corollary}\label{corollary on total sum}
 Set $\lambda\vdash n$ and $\sigma\in S_r$ as before, then 
 \[TS^{\lambda}(\sigma)=\sum_{\nu\vdash r}TS^{\nu}(\sigma)\cdot\dim\nu\cdot\frac{\dim\lambda/\nu}{\dim\lambda}.\]
\end{corollary}
\begin{proof}
 By the previous lemma:
 \[TS^{\lambda}(\sigma):=\sum_{T,T'\in \SYT(\lambda)} \frac{\pi^{\lambda}(\sigma)_{i,j}}{\dim\lambda}=\sum_{\nu\vdash r}\sum_{U,U'\in\SYT(\nu)}\pi^{\nu}(\sigma)_{U,U'}\cdot\frac{\dim\lambda/\nu}{\dim\lambda},\]
 and the conclusion is immediate.
\end{proof}

\begin{lemma}\label{lemma on skew dimension}
 Let $|\nu|=r\leq n=|\lambda|$, then
 \begin{equation}\label{op}
  \dim\lambda/\nu=\frac{1}{r!}\sum_{\tau\in S_r}\chi^{\nu}(\tau)\chi^{\lambda}(\tau).
 \end{equation}

\end{lemma}
\begin{proof}
We prove the statement by induction on $n$. If $n=r$ then the left hand side of \eqref{op} is equal to $\delta_{\lambda,\nu}$, that is, is equal $1$ if $\lambda=\nu$ and $0$ otherwise, and same holds for the right hand side by the character orthogonality relation of the first kind.
 \smallskip
 
 Suppose $r=n-1$, then by Proposition \ref{decomposition of the representation matrix} we have $\chi^{\lambda}(\tau)=\sum_{\mu_j\nearrow\lambda}\chi^{\mu_j}(\tau)$ for each $\tau\in S_{n-1}$. Hence the right hand side of \eqref{op} can be written as 
 \[\frac{1}{r!}\sum_{\tau\in S_r}\chi^{\nu}(\tau)\chi^{\lambda}(\tau)=\sum_{\mu_j\nearrow\lambda}\left(\frac{1}{r!}\sum_{\tau\in S_r}\chi^{\nu}(\tau)\chi^{\mu_j}(\tau)\right)=\sum_{\mu_j\nearrow\lambda}\delta_{\mu_j=\nu}=\delta_{\nu\nearrow\lambda},\]
 which is equal to the left hand side of \eqref{op}.
 \smallskip
 
 By the inductive hypothesis we consider the statement true for each $\tilde{\lambda}\vdash n-1$. Hence 
 \begin{align*}
  \dim\lambda/\nu&=\sum_{\tilde{\lambda}\vdash n-1}\dim\lambda/\tilde{\lambda}\cdot\dim\tilde{\lambda}/\nu\\
  &=\sum_{\tilde{\lambda}\nearrow\lambda}\left(\frac{1}{r!}\sum_{\tau\in S_r}\chi^{\nu}(\tau)\chi^{\tilde{\lambda}}(\tau)\right)\\
  &=\frac{1}{r!}\sum_{\tau\in S_r}\chi^{\nu}(\tau)\chi^{\lambda}(\tau).\qedhere
 \end{align*} 
\end{proof}

\begin{proposition}\label{formula con traccia}
 Let $\sigma\in S_r$ and $\lambda\vdash n$, then
 \[TS^{\lambda}(\sigma)=\sum_{\tau\in S_r}\mathbb{E}_{\Pl}^{r}\left[\hat{\chi}^{\nu}(\tau) TS^{\nu}(\sigma)\right]\hat{\chi}^{\lambda}(\tau).\]
\end{proposition}

\begin{proof}
We apply Corollary \ref{corollary on total sum} and Lemma \ref{lemma on skew dimension}:
 \begin{align*}
 TS^{\lambda}(\sigma)&=\sum_{\nu\vdash r}TS^{\nu}(\sigma)\cdot\dim\nu\left(\frac{1}{r!}\sum_{\tau\in S_r}\chi^{\nu}(\tau)\hat{\chi}^{\lambda}(\tau)\right)\\
&=\sum_{\tau\in S_r}\left(\sum_{\nu\vdash r}\frac{\dim\nu^2}{r!}\hat{\chi}^{\nu}(\tau)TS^{\nu}(\sigma)\right)\hat{\chi}^{\lambda}(\tau)\\
\end{align*}
and we recognize inside the parenthesis the average $\mathbb{E}_{\Pl}^{r}\left[\hat{\chi}^{\nu}(\tau) TS^{\nu}(\sigma)\right]$ taken with the Plancherel measure of partitions of $r$.
\end{proof}

\begin{proof}[Proof of Theorem \ref{convergence of total sum}]
Consider a permutation $\sigma$, we write the previous proposition as 
\[TS^{\lambda}(\sigma)=m_{\sigma}+v_{\sigma}\cdot\hat{\chi}^{\lambda}((1,2))+\sum_{\substack{\tau\in S_r\\\wt(\tau)>2}}\mathbb{E}_{\Pl}^{r}\left[\hat{\chi}^{\nu}(\tau) TS^{\nu}(\sigma)\right]\hat{\chi}^{\lambda}(\tau),\]
where the first two terms correspond to respectively $\tau=\id$ and the sum of all transpositions. By Kerov's result (Theorem \ref{convergence of characters}), we have that 
\[n\sum_{\substack{\tau\in S_r\\\wt(\tau)>2}}\mathbb{E}_{\Pl}^{r}\left[\hat{\chi}^{\nu}(\tau) TS^{\nu}(\sigma)\right]\hat{\chi}^{\lambda}(\tau)\overset{p}\to 0\]
and $n \cdot v_{\sigma}\cdot \hat{\chi}^{\lambda}((1,2))\overset{d}\to v_{\sigma}\mathcal{N}(0,2)$. Consider now a sequence of permutations $\sigma_1,\sigma_2,\ldots$, then
 \[\left\{n\cdot(TS^{\lambda}(\sigma_i)-m_{\sigma_i})\right\}_{i\geq1}=\left\{n\cdot v_{\sigma_i}\hat{\chi}^{\lambda}((1,2))+o_P(1)\right\}_{i\geq1}\overset{d}\to\left\{ \mathcal{N}(0,2 v_{\sigma_i}^2)\right\}_{i\geq1}.\qedhere\]
\end{proof}

\begin{example}\label{ex wow}
 Let $\sigma=(r-1,r)$, $r>2$, be an adjacent transposition. In this example we show that $m_{\sigma}=\mathbb{E}_{\Pl}^{r}\left[TS^{\nu}(\sigma)\right]$ is strictly positive and that $v_{\sigma}:=\binom{r}{2}\mathbb{E}_{\Pl}^{r}\left[\hat{\chi}^{\nu}_{(2,1,\ldots,1)} TS^{\nu}(\sigma)\right]=1$. 
 
 First, notice that, for all $\nu\vdash r$ and for all $T,S$ standard Young tableaux of shape $\nu$, $d_{r-1}(T')=-d_{r-1}(T)$, where $T'$ is a tableau of shape $\nu'$ conjugated to $T$. This implies that
 \[\pi^{\nu}(\sigma)_{T',S'}=\left\{\begin{array}{lcr}
                                           -\pi^{\nu}(\sigma)_{T,S}&\mbox{if}& T=S\\
                                           \pi^{\nu}(\sigma)_{T,S}&\mbox{if}& T\neq S,
                                           \end{array}\right.\]
        and if $T\neq S$ then $ \pi^{\nu}(\sigma)_{T,S}\geq 0$. Hence
\[  m_{\sigma}=\sum_{\nu\vdash r}\frac{(\dim\nu)^2}{r!}TS^{\nu}(\sigma)=\frac{1}{2}\sum_{\nu\vdash r}\frac{(\dim\nu)^2}{r!}(TS^{\nu}(\sigma)+TS^{\nu'}(\sigma)).\]
Since $TS^{\lambda}(\sigma)=\hat{\chi}^{\lambda}(\sigma)+\sum_{\substack{T,S\in \SYT(\nu)\\T\neq S}}\frac{\pi^{\nu}(\sigma)_{T,S}}{\dim\nu}$, then
\[m_{\sigma}=\frac{1}{2}\sum_{\nu\vdash r}\frac{\dim\nu}{r!}\left(\sum_{\substack{T,S\in \SYT(\nu)\\T\neq S}}\pi^{\nu}(\sigma)_{T,S}+\pi^{\nu'}(\sigma)_{T',S'}\right)>0.\]
For $r>2$, at least one summand is nonzero.

Consider now $v_{\sigma}=\binom{r}{2}\mathbb{E}_{\Pl}^{r}\left[\hat{\chi}^{\nu}_{(2,1,\ldots,1)} TS^{\nu}(\sigma)\right]$ and decompose $TS^{\nu}(\sigma)$ as above:
 \[v_{\sigma}=\binom{r}{2}\left(\mathbb{E}_{\Pl}^{r}\left[(\hat{\chi}^{\nu}_{(2,1,\ldots,1)})^2\right]+\mathbb{E}_{\Pl}^{r}\left[\hat{\chi}^{\nu}_{(2,1,\ldots,1)}\sum_{\substack{T,S\in \SYT(\nu)\\T\neq S}}\frac{\pi^{\nu}(\sigma)_{T,S}}{\dim\nu}\right]\right),\]
 where we used the fact that $\hat{\chi}^{\nu}(\sigma)=\hat{\chi}^{\nu}_{(2,1,\ldots,1)}$ since $\sigma$ is a transposition. By the character relations of the second kind we get 
 \[\mathbb{E}_{\Pl}^{r}\left[(\hat{\chi}^{\nu}_{(2,1,\ldots,1)})^2\right]=\sum_{\nu\vdash r}\frac{(\dim\nu)^2}{r!}\left(\hat{\chi}^{\nu}_{(2,1,\ldots,1)}\right)^2=\frac{2}{r(r-1)}.\]
 On the other hand, using $T'$ for the conjugate of $T$ as above, 
 \[\mathbb{E}_{\Pl}^{r}\left[\hat{\chi}^{\nu}_{(2,1,\ldots,1)}\sum_{\substack{T,S\in \SYT(\nu)\\T\neq S}}\frac{\pi^{\nu}(\sigma)_{T,S}}{\dim\nu}\right]=\]
 \[\frac{1}{2}\sum_{\nu\vdash r}\frac{1}{r!}\sum_{\substack{T,S\in \SYT(\nu)\\T\neq S}}\left(\chi^{\nu}_{(2,1,\ldots,1)}\pi^{\nu}(\sigma)_{T,S}+\chi^{\nu'}_{(2,1,\ldots,1)}\pi^{\nu'}(\sigma)_{T',S'}\right)=0\]
 since $\pi^{\nu}(\sigma)_{T,S}=\pi^{\nu'}(\sigma)_{T',S'}$ if $T\neq S$ and $\chi^{\nu}_{(2,1,\ldots,1)}=-\chi^{\nu'}_{(2,1,\ldots,1)}$. Therefore, $v_{\sigma}=1$.
\end{example}
Here we write the values of $m_{\sigma}= \mathbb{E}_{\Pl}^{r}\left[TS^{\nu}(\sigma)\right]$ and $v_{\sigma}=\binom{r}{2}\mathbb{E}_{\Pl}^{r}\left[\hat{\chi}^{\nu}_{(2,1,\ldots,1)} TS^{\nu}(\sigma)\right]$ when $\sigma\in S_4$.
\begin{center}
 \scalebox{0.8}{
\begin{tabular}{ |c|ccccccccccccc|} 
 \hline
$\sigma$&$\id$&(3,4)&(2,3)&(2,3,4)&(2,4,3)&(2,4)&(1,2)&(1,2)(3,4)&(1,2,3)&(1,2,3,4)&(1,2,4,3)&(1,2,4)&(1,3,2)\\
$m$&1&1/2&2/3&5/12&1/6&-1/4&0&0&1/3&1/3&1/3&0&-1/3\\
$v$&0&1&1&1/2&4/3&13/6&1&1&0&0&2/3&1/3&0\\
 \hline
\end{tabular}}\smallskip

 \scalebox{0.8}{\begin{tabular}{|c|ccccccccccc|}
\hline
$\sigma$&(1,3,4,2)&(1,3)&(1,3,4)&(1,3)(2,4)&(1,3,2,4)&(1,4,3,2)&(1,4,2)&(1,4,3)&(1,4)&(1,4,2,3)&(1,4)(2,3)\\
$m$&-1/12&-2/3&-1/6&7/12&1/6&-1/12&0&-5/12&-1/4&-2/3&-7/12\\
$v$&1/2&1&0&-7/6&-1/3&-7/6&-4/3&-5/6&-1/6&1/3&1/6\\
\hline
\end{tabular}}
\end{center}

\subsection{Partial sum of the entries of an irreducible representation}
\begin{defi}
Let $\lambda\vdash n$, $\sigma\in S_n$ and $u\in \R$. Then the \emph{partial sum } associated to $\lambda,\sigma$ and $u$ is
\[PS_u^{\lambda}(\sigma):=\sum_{i,j\leq u \dim\lambda}\frac{\pi^{\lambda}(\sigma)_{i,j}}{\dim\lambda}.\]
 \end{defi}
 We can now argue in a similar way as we did in Section \ref{chapter the partial trace}: summing entries of $\pi^{\lambda}(\sigma)$ up to a certain index is equivalent to summing all the entries of the submatrices $\pi^{\mu_j}(\sigma)$ for $j<\bar{j}$ and the right choice of $\bar{j}$, plus a remainder which is again a partial sum. We present the analogous of Proposition \ref{first order decomposition of the partial trace}; we omit the proof, since it follows the same argument of the partial trace version:
\begin{proposition}\label{first order decomposition of partial sum}
 Fix $u\in [0,1]$, $\lambda\vdash n$ and $\sigma\in S_r$ with $r\leq n-1$. Set $\left(F_{\ctr}^{\lambda}\right)^{*}(u)=\frac{y_{\bar{j}}}{\sqrt{n}}=v^{\lambda}$. Define
 \[\bar{u}=\frac{\dim\lambda}{\dim\mu_{\bar{j}}}\left(u-\sum_{j<\bar{j}}\dim\mu_j\right)<1, \]
 then
\begin{equation}\label{first decomposition of partial sum}
 PS_u^{\lambda}(\sigma)=\sum_{y_j<v^{\lambda}\sqrt{n}}\frac{\dim\mu_j}{\dim\lambda}TS^{\mu_j}(\sigma)+\frac{\dim\mu_{\bar{j}}}{\dim\lambda}PS_{\bar{u}}^{\mu_{\bar{j}}}(\sigma).
\end{equation}
\end{proposition}
As before, we call the \emph{main term of the partial sum} $MS_u^{\lambda}(\sigma):=\sum_{y_j<v^{\lambda}\sqrt{n}}\frac{\dim\mu_j}{\dim\lambda}TS^{\mu_j}(\sigma)$ and \emph{remainder for the partial sum} $RS_u^{\lambda}(\sigma):=\frac{\dim\mu_{\bar{j}}}{\dim\lambda}PS_{\bar{u}}^{\mu_{\bar{j}}}(\sigma)$.

The connection between the main term of the partial trace and the main term of the partial sum is easily described by applying Proposition \ref{formula con traccia}:
\begin{align*}
 MS_u^{\lambda}(\sigma)&=\sum_{y_j<v^{\lambda}\sqrt{n}}\frac{\dim\mu_j}{\dim\lambda}TS^{\mu_j}(\sigma)\\
 &=\sum_{ y_j<v^{\lambda}\sqrt{n}}\frac{\dim\mu_j}{\dim\lambda}\sum_{\tau\in S_r}\mathbb{E}_{\Pl}^{r}\left[\hat{\chi}^{\nu}(\tau) TS^{\nu}(\sigma)\right]\hat{\chi}^{\mu_j}(\tau)\\
 &=\sum_{\tau\in S_r}\mathbb{E}_{\Pl}^{r}\left[\hat{\chi}^{\nu}(\tau) TS^{\nu}(\sigma)\right]\left(  \sum_{ y_j<v^{\lambda}\sqrt{n}}\frac{\dim\mu_j}{\dim\lambda}\hat{\chi}^{\mu_j}(\tau)\right)\\
 &=\sum_{\tau\in S_r}\mathbb{E}_{\Pl}^{r}\left[\hat{\chi}^{\nu}(\tau) TS^{\nu}(\sigma)\right]MT_u^{\lambda}(\tau),\numberthis\label{formula between G and F}
\end{align*}
We can thus apply Theorem \ref{convergence of transformed co-transition} on the convergence of the main term of the partial trace to describe the asymptotic of $MS_u^{\lambda}(\sigma)$; notice that the only term in the previous sum which does not disappear when $\lambda$ increases is the one in which $\tau$ is the identity; moreover, the second highest degree term correspond to $\tau$ being a transposition, and thus of order $O_P(n^{-1})$. We get:
\begin{corollary}\label{convergence of G}
Set $\sigma_1\in S_{r_1},\sigma_2\in S_{r_2},\ldots$, $u_1,u_2,\ldots\in[0,1]$ and let $\lambda\vdash n$ be a random partition distributed with the Plancherel measure. Consider as before
\[m_{\sigma_i}:= \mathbb{E}_{\Pl}^{r_i}\left[TS^{\nu}(\sigma_i)\right]\qquad \mbox{and}\qquad v_{\sigma_i}:=\binom{r_i}{2}\mathbb{E}_{\Pl}^{r_i}\left[\hat{\chi}^{\nu}_{(2,1,\ldots,1)} TS^{\nu}(\sigma_i)\right].\]
Then 
 \[\left\{n\cdot\left(MS_{u_i}^{\lambda}(\sigma_i)-u_i\cdot m_{\sigma_i}\right)\right\}\overset{d}\to \left\{u_i\cdot\mathcal{N}(0,2 v_{\sigma_i}^2)\right\}.\]
\end{corollary}
\begin{proof}
 For a generic $u$ and $\sigma$, we rewrite \eqref{formula between G and F} as:
 \begin{multline*}
   MS_u^{\lambda}(\sigma)=m_{\sigma}\cdot MT_u^{\lambda}(\id) + \sum_{\tau\in S_r:\wt(\tau)=2}\mathbb{E}_{\Pl}^{r}\left[\hat{\chi}^{\nu}(\tau) TS^{\nu}(\sigma)\right]MT_u^{\lambda}(\tau)\\
   + \sum_{\tau\in S_r:\wt(\tau)>2}\mathbb{E}_{\Pl}^{r}\left[\hat{\chi}^{\nu}(\tau) TS^{\nu}(\sigma)\right]MT_u^{\lambda}(\tau).
 \end{multline*}

 By Lemma \ref{behaviour of co-transition}, $MT_u^{\lambda}(\id)=(F_{\ctr}^{\lambda})^*(u)\overset{p}\to u$. By Theorem \ref{convergence of transformed co-transition}
 \[n\sum_{\substack{\tau\in S_r\\\wt(\tau)=2}}\mathbb{E}_{\Pl}^{r}\left[\hat{\chi}^{\nu}(\tau) TS^{\nu}(\sigma)\right]MT_u^{\lambda}(\tau)\overset{d}\to v_{\sigma}\cdot u\cdot \mathcal{N}(0,2);\]
 on the other hand 
 \[\sum_{\substack{\tau\in S_r\\\wt(\tau)>2}}\mathbb{E}_{\Pl}^{r}\left[\hat{\chi}^{\nu}(\tau) TS^{\nu}(\sigma)\right]MT_u^{\lambda}(\tau)\in o_P(n^{-1}).\]
Thus 
 \[\left\{n\cdot\left(MS_{u_i}^{\lambda}(\sigma_i)-u_i\cdot m_{\sigma_i}\right)\right\}\overset{d}\to \left\{u_i\cdot\mathcal{N}(0,2 v_{\sigma_i}^2)\right\}.\qedhere\]
\end{proof}
As mentioned in the introduction, the previous corollary implies (by setting $u=1$) Theorem \ref{convergence of total sum}, although they both rely on Proposition \ref{formula con traccia}.
 \medskip

Although we cannot show a satisfying result on the convergence of the partial trace since we cannot prove that $n^{\frac{wt(\rho)}{2}}RT_{u}^{\lambda}(\sigma)\overset{p}\to 0$, we are more lucky with the partial sum:

\begin{theorem}{\label{convergence of partial sum}}
 Set $\sigma\in S_r$ and let $\lambda\vdash n$. Set $m_{\sigma}:= \mathbb{E}_{\Pl}^{r}\left[TS^{\nu}(\sigma)\right]$
Then 
 \[ PS_u^{\lambda}(\sigma)\overset{p}\to u\cdot m_{\sigma}.\]
\end{theorem}

We prove the theorem after three lemmas. 
\begin{lemma}\label{bound on nonzero terms}
 Let as usual $\sigma\in S_r$ and $\lambda\vdash n$, with $n>r$. Then $\pi^{\lambda}(\sigma)_{i,j}=0$ for all $i,j$ such that $|i-j|>r!$
\end{lemma}
\begin{proof}
We iteratively use Proposition \ref{decomposition of the representation matrix}:
\[\pi^{\lambda}(\sigma)=\bigoplus_{\mu^{(n-1)}\nearrow \lambda}\pi^{\mu^{(n-1)}}(\sigma)=\bigoplus_{\substack{\mu^{(n-1)}\nearrow \lambda\\ \mu^{(n-2)}\nearrow \mu^{(n-1)}}}\pi^{\mu^{(n-2)}}(\sigma)=
\bigoplus_{\mu^{(r)}\nearrow \cdots\nearrow\lambda}\pi^{\mu^{(r)}}(\sigma).\]
Therefore $\pi^{\lambda}(\sigma)$ is a block matrix such that the only nonzero blocks are those on the diagonal. To conclude, notice that $\dim{\mu^{(r)}}\leq r!$
\end{proof}
\begin{lemma}\label{bound on entries}
Consider as usual $\lambda\vdash n$, $\sigma\in S_r$ and $u\in[0,1]$. Then
 \[|PS_u^{\lambda}(\sigma)|\leq 2 u\cdot r!\cdot 2^{l(\sigma)},\]
 where $l(\sigma)$ is the length of the reduced word of $\sigma$, \emph{i.e.} the minimal number of adjacent transpositions occurring in the decomposition of $\sigma$.
\end{lemma}
\begin{proof}
 We first prove a bound on the absolute value of an entry of the matrix $\pi^{\lambda}(\sigma)$:
 \begin{equation}\label{bound of an entry}
  |\pi^{\lambda}(\sigma)_{T,S}|\leq 2^{l(\sigma)}\mbox{ for each }T,S
 \end{equation}
 by induction on $l(\sigma)$. The initial case is $\sigma=(k,k+1)$, for which $|\pi^{\lambda}((k,k+1))_{T,S}|\leq 1$.\\
 Suppose $\sigma=(k,k+1)\tilde{\sigma}$, then
 \begin{align*}
  |\pi^{\lambda}(\sigma)_{T,S}|&=\left|\frac{1}{d_k(T)}\pi^{\lambda}(\tilde{\sigma})_{T,S}+\sqrt{1-\frac{1}{d_k(T)^2}}\pi^{\lambda}(\tilde{\sigma})_{(k,k+1)T,S}\delta_{(k,k+1)T\in SYT(\lambda)}\right|\\
  &\leq 2^{l(\tilde{\sigma})}+2^{l(\tilde{\sigma})}=2^{l(\sigma)}.
 \end{align*}

  We see that
\[ |PS_u^{\lambda}(\sigma)|=\left|\sum_{i,j\leq u \dim\lambda}\frac{\pi^{\lambda}(\sigma)_{i,j}}{\dim\lambda}\right|
 \leq 2u\dim\lambda\cdot r!\max_{T,S\in \SYT(\lambda)}\left|\frac{\pi^{\lambda}(\sigma)_{T,S}}{\dim\lambda}\right|\le2^{l(\sigma)+1} u\cdot r!, \]
which allows us to conclude. Notice that we used the previous lemma in the first inequality, which shows that the number of nonzero terms appearing in the sum is bounded by $2 u\dim\lambda\cdot r!$
\end{proof}
\begin{lemma}
 Let $\sigma\neq \id$ be a permutation in $S_r$, $u\in [0,1]$, and consider $\lambda$ to be a random partition of $n$ distributed with the Plancharel measure. Then the partial trace is asymptotically zero in probability:
 \[PT_u^{\lambda}(\sigma)\overset{p}\to 0.\]
\end{lemma}
\begin{proof}
 Recall the decomposition of the partial trace into main term and remainder (Proposition \ref{first order decomposition of the partial trace}):
 \[PT_u^{\lambda}(\sigma)=MT_u^{\lambda}(\sigma)+RT_u^{\lambda}(\sigma).\]
 By Theorem \ref{convergence of transformed co-transition}, if $\sigma\neq\id$ then $MT_u^{\lambda}(\sigma)\in O_P(n^{-\frac{\wt(\sigma)}{2}})\subseteq \smallo_P(1)$. On the other hand we can estimate the remainder through the bound of a singular entry:
 let $\mu_{\bar{j}}$ be the subpartition corresponding to the renormalized content $\frac{y_{\bar{j}}}{\sqrt{n}}=(F_{\ctr}^{\lambda})^*(u)$ and 
 \[\bar{u}=\frac{\dim\lambda}{\dim\mu_{\bar{j}}}\left(u-\sum_{j<\bar{j}}\dim\mu_j\right). \]
 Recall that, by definition
 \[RT_u^{\lambda}(\sigma)=\frac{\dim\mu_{\bar{j}}}{\dim\lambda}PT^{\mu_{\bar{j}}}_{\bar{u}}(\sigma),\]
 and by \eqref{bound of an entry} an entry of $\pi^{\mu_{\bar{j}}}(\sigma)$ is bounded by
 \[|\pi^{\mu_{\bar{j}}}(\sigma)_{i,j}|\leq 2^{l(\sigma)},\]
 where $l(\sigma)$ is the length of the reduced word of $\sigma$. Hence
 \[|RT_u^{\lambda}(\sigma)|\leq 2^{l(\sigma)}\cdot\frac{\dim\mu_{\bar{j}}}{\dim\lambda}\overset{p}\to 0.\qedhere\]
\end{proof}

\begin{proof}[Proof of Proposition \ref{convergence of partial sum}]
We claim that the partial sum $PS_u^{\lambda}(\sigma)$ and the following quantity are asymptotically close in probability
\[\sum_{\tau\in S_r}\mathbb{E}_{\Pl}^{r}\left[\hat{\chi}^{\nu}(\tau) TS^{\nu}(\sigma)\right]PT_u^{\lambda}(\tau),\]
that is, 
\[\left|PS_u^{\lambda}(\sigma)-\sum_{\tau\in S_r}\mathbb{E}_{\Pl}^{r}\left[\hat{\chi}^{\nu}(\tau) TS^{\nu}(\sigma)\right]PT_u^{\lambda}(\tau)\right|\overset{p}\to 0.\]
We substitute in the previous expression the decomposition formulas for the partial sum and partial trace, respectively Propositions \ref{first order decomposition of partial sum} and \ref{first order decomposition of the partial trace}, and we simplify according to the equality of Equation (\ref{formula between G and F}), so that we obtain:
\[\frac{\dim\mu_{\bar{j}}}{\dim\lambda}\left|PS_{\bar{u}}^{\mu_{\bar{j}}}(\sigma)-\sum_{\tau\in S_r}\mathbb{E}_{\Pl}^{r}\left[\hat{\chi}^{\nu}(\tau) TS^{\nu}(\sigma)\right]PT_{\bar{u}}^{\mu_{\bar{j}}}(\tau)\right|.\]
We recall from \eqref{bound of an entry} that $|\pi^{\lambda}(\tau)_{T,S}|\leq 2^{l(\tau)}$ for each $T,S$, which implies that 
\[\left|PT_{\bar{u}}^{\mu_{\bar{j}}}(\tau)\right|\leq 2^{l(\tau)}\cdot \bar{u}.\]
Hence
\begin{align*}
 &\frac{\dim\mu_{\bar{j}}}{\dim\lambda}\left|PS_{\bar{u}}^{\mu_{\bar{j}}}(\sigma)-\sum_{\tau\in S_r}\mathbb{E}_{\Pl}^{r}\left[\hat{\chi}^{\nu}(\tau) TS^{\nu}(\sigma)\right]PT_{\bar{u}}^{\mu_{\bar{j}}}(\tau)\right|\\
&\leq \frac{\dim\mu_{\bar{j}}}{\dim\lambda} \left(2 \bar{u}\cdot r!\cdot 2^{l(\sigma)}+\sum_{\tau\in S_r}\mathbb{E}_{\Pl}^{r}\left[\hat{\chi}^{\nu}(\tau) TS^{\nu}(\sigma)\right]2^{l(\tau)}\cdot \bar{u}\right)\overset{p}\to 0
\end{align*}
since the expression inside the parenthesis is bounded and $\dim\mu_{\bar{j}}/\dim\lambda\overset{p}\to 0$.

On the other hand 
\[\sum_{\tau\in S_r}\mathbb{E}_{\Pl}^{r}\left[\hat{\chi}^{\nu}(\tau) TS^{\nu}(\sigma)\right]PT_u^{\lambda}(\tau)=\mathbb{E}_{\Pl}^{r}\left[TS^{\nu}(\sigma)\right]PT_u^{\lambda}(\id)+\sum_{\substack{\tau\in S_r\\\tau\neq \id}}\mathbb{E}_{\Pl}^{r}\left[\hat{\chi}^{\nu}(\tau) TS^{\nu}(\sigma)\right]PT_u^{\lambda}(\tau),\]
and \[\sum_{\substack{\tau\in S_r\\\tau\neq \id}}\mathbb{E}_{\Pl}^{r}\left[\hat{\chi}^{\nu}(\tau) TS^{\nu}(\sigma)\right]PT_u^{\lambda}(\tau)\in \smallo_P(1)\]
by the previous lemma. Notice that $PT_u^{\lambda}(\id)=\frac{\lfloor u\dim\lambda\rfloor}{\dim\lambda}$, therefore
\[\sum_{\tau\in S_r}\mathbb{E}_{\Pl}^{r}\left[\hat{\chi}^{\nu}(\tau) TS^{\nu}(\sigma)\right]PT_u^{\lambda}(\tau)=\frac{\lfloor u\dim\lambda\rfloor}{\dim\lambda}\mathbb{E}_{\Pl}^{r}\left[TS^{\nu}(\sigma)\right]+\smallo_P(1)\overset{d}\to u\mathbb{E}_{\Pl}^{r}\left[TS^{\nu}(\sigma)\right]=u\cdot m_{\sigma}.\qedhere\]
\end{proof}

For the same reasons for which we cannot prove convergence of the partial trace, here we cannot show a second order asymptotic, indeed we know that the term $\frac{\dim\mu_{\bar{j}}}{\dim\lambda}PS_{\bar{u}}^{\mu_{\bar{j}}}(\sigma)$ disappears when $\lambda$ grows, but we do not know how fast. This will be discussed more deeply in the next section.

We can now generalize the concept of partial sum, and Lemma \ref{bound on entries} allows us to describe its asymptotics.
\begin{defi}
 Let $\lambda\vdash n$, $\sigma\in S_r$ and $ u_1,u_2\in \R$. Define the \emph{partial sum of the entries of the irreducible representation matrix associated to $\pi^{\lambda}(\sigma)$ stopped at $(u_1,u_2)$} as
\[PS_{u_1,u_2}^{\lambda}(\sigma):=\sum_{\substack{i\leq u_1 \dim\lambda\\j\leq u_2 \dim\lambda}}\frac{\pi^{\lambda}(\sigma)_{i,j}}{\dim\lambda}.\]
\end{defi}
\begin{corollary}
Let $u_1,u_2\in [0,1]$ and $\sigma\in S_r$. Consider a random partition $\lambda\vdash n$ distributed with the Plancharel measure. then
\[ PS_{u_1,u_2}^{\lambda}(\sigma)\overset{p}\to \min\{u_1,u_2\}\cdot m_{\sigma}\]
where, as before, $m_{\sigma}=\mathbb{E}_{\Pl}^{r}\left[TS^{\nu}(\sigma)\right]$
\end{corollary}
\begin{proof}
 Suppose $u_1<u_2$, we only need to prove that 
 \[|PS_{u_1,u_2}^{\lambda}(\sigma)-PS_{u_1}^{\lambda}(\sigma)|=\left|\sum_{\substack{i\leq u_1 \dim\lambda\\u_1 \dim\lambda<j\leq u_2 \dim\lambda}}\frac{\pi^{\lambda}(\sigma)_{i,j}}{\dim\lambda}\right|\overset{p}\to 0.\]
 The number of nonzero terms in the above sum is bounded by $(r!)^2$ by Lemma \ref{bound on nonzero terms}; moreover by \hbox{Equation \eqref{bound of an entry}} we have \hbox{$|\pi^{\lambda}(\sigma)_{i,j}|\leq 2^{l(\sigma)}$}. Therefore
  \[|PS_{u_1,u_2}^{\lambda}(\sigma)-PS_{u_1}^{\lambda}(\sigma)|\overset{p}\to 0,\]
 and the corollary follows.
\end{proof}

\section{A conjecture}\label{section conjecture}

\begin{conjecture}\label{conjecture 1}
 Set as usual $\lambda\vdash n$. We say that $\mu\subseteq \lambda$ if $\mu$ is a partition obtained from $\lambda$ by removing boxes.
 \smallskip
 
 We conjecture that there exists $\alpha>0$ such that, for all $s$,
\[P_{\Pl}^n\left(\left\{\lambda: \max_{\substack{\mu\colon \mu\subseteq \lambda\\|\mu|=|\lambda|-s}}\frac{\dim\mu}{\dim\lambda}>n^{-\alpha s}\right\}\right)\to 0.\]
 \end{conjecture}
 We run some tests which corroborate the conjecture for $s=1$; for example, if $\alpha=0.2$ and $n\leq 70$, we have
 \[P_{\Pl}^n\left(\left\{\lambda: \max_{\substack{\mu\colon \mu\subseteq \lambda\\|\mu|=|\lambda|-1}}\frac{\dim\mu}{\dim\lambda}>n^{-\alpha }\right\}\right)\leq \left\{\begin{array}{cc}
                                                                                                                                                                         0.2& \mbox{ if } n\geq 7\\
                                                                                                                                                                         0.1&\mbox{ if } n\geq 12\\
                                                                                                                                                                          0.05&\mbox{ if } n\geq 37.\\
                                                                                                                                                                        \end{array}\right.\]
 
Notice that, for $s>1$,
 \[\max_{\substack{\mu\colon \mu\subseteq \lambda\\|\mu|=|\lambda|-s}}\frac{\dim\mu}{\dim\lambda}\leq \max_{\mu^{(1)}\nearrow\mu^{(2)}\nearrow\ldots\nearrow \mu^{(s-1)}\nearrow\lambda}\prod_i\left(\max_{\substack{\mu\colon \mu\subseteq \mu^{(i)}\\|\mu|=|\mu^{(i)}|-1}}\frac{\dim\mu}{\dim\mu^{(i)}}\right),\]
so that it may seem that it is enough to prove the conjecture only for $s=1$. 

This is not true though, since the sequence $\mu^{(1)}\nearrow\ldots\nearrow\mu^{(s-1)}$ in the right hand side of the inequality above is not Plancherel distributed. Hence we need the conjecture in the more general form.
 \begin{proposition}
  If Conjecture \ref{conjecture 1} is correct, then for a sequence of permutations $\sigma_1,\sigma_2,\ldots$ and a sequence of real numbers $u_1,u_2,\ldots$ we have
  \[\left\{PT_{u_i}^{\lambda}(\sigma_i)\right\}\overset{d}\to\left\{ u_i\prod_{k\geq 2} k^{m_k(\rho_i)/2} \mathcal{H}_{m_k(\rho_i)}(\xi_k)\right\},\]
  where $\lambda\vdash n$ is distributed with the Plancherel measure.
 \end{proposition}
\begin{proof}
  Recall from Proposition \ref{decomposition of the partial trace} that for any $s<n-r$ there exists a sequence of real numbers $0\leq c_{0},\ldots,c_s<1$ and a sequence of partitions \hbox{$\mu^{(0)}\nearrow\mu^{(1)}\nearrow\ldots\nearrow \mu^{(s)}=\lambda$} such that
 \[PT_u^{\lambda}(\sigma)=\sum_{i=1}^{s}\frac{\dim\mu^{(i)}}{\dim\lambda}MT_{c^{(i)}}^{\mu^{(i)}}(\sigma)+\frac{\dim\mu^{(0)}}{\dim\lambda}PT_{c^{(0)}}^{\mu^{(0)}}(\sigma).\]
 We consider 
  \[n^{\frac{\wt(\sigma)}{2}}PT_u^{\lambda}(\sigma)=n^{\frac{\wt(\sigma)}{2}}\sum_{i=1}^{s}\frac{\dim\mu^{(i)}}{\dim\lambda}MT_{c^{(i)}}^{\mu^{(i)}}(\sigma)+n^{\frac{\wt(\sigma)}{2}}\frac{\dim\mu^{(0)}}{\dim\lambda}PT_{c^{(0)}}^{\mu^{(0)}}(\sigma).\]
In the first sum of the right hand side the term corresponding to $i=s$ is $n^{\frac{\wt(\sigma)}{2}}MT_{u}^{\lambda}(\sigma)$, which converges in distribution:
\[n^{\frac{\wt(\sigma)}{2}}MT_{u}^{\lambda}(\sigma)\overset{d}\to u \prod_{k\geq 2} k^{m_k(\rho)/2} \mathcal{H}_{m_k(\rho)}(\xi_k),\]
due to Theorem \ref{convergence of transformed co-transition}. On the other hand the other terms in the first sum of the right hand side are of the form $n^{\frac{\wt(\sigma)}{2}}\frac{\dim\mu^{(i)}}{\dim\lambda}MT_{c^{(i)}}^{\mu^{(i)}}(\sigma)$ for $i=1,\ldots,s-1$. It is easy to see that $(n-s+i)^{\frac{\wt(\sigma)}{2}}MT_{c^{(i)}}^{\mu^{(i)}}(\sigma)$ converges (because of Theorem \ref{convergence of transformed co-transition}), while $\frac{\dim\mu^{(i)}}{\dim\lambda}\in o_P(1)$ because of the convergence of the normalized co-transition distribution function towards an
atom free distribution function (see Corollary \ref{convergence of co-transition}).
\smallskip
We study thus the term $n^{\frac{\wt(\sigma)}{2}}\frac{\dim\mu^{(0)}}{\dim\lambda}PT_{c^{(0)}}^{\mu^{(0)}}(\sigma)$, and Conjecture \ref{conjecture 1} implies that
 \[\frac{\dim\mu^{(0)}}{\dim\lambda}\leq \max_{\substack{\mu\colon \mu\subseteq \lambda\\|\mu|=|\lambda|-s}}\frac{\dim\mu}{\dim\lambda}\leq n^{-\alpha s},\]
 with high probability. We choose $s$ such that $\wt(\sigma)/2<\alpha s$. An easy application of Inequality \eqref{bound of an entry} implies that $|PT_{c^{(0)}}^{\mu^{(0)}}(\sigma)|\leq c^{(0)}\cdot 2^{l(\sigma)}$, thus we have
 \[n^{\frac{\wt(\sigma)}{2}}\frac{\dim\mu^{(0)}}{\dim\lambda}PT_{c^{(0)}}^{\mu^{(0)}}(\sigma)\leq n^{\frac{\wt(\sigma)}{2}}n^{-\alpha s} c^{(0)}\cdot 2^{l(\sigma)}\to 0,\]
 which implies
 \[n^{\frac{\wt(\sigma)}{2}}\frac{\dim\mu^{(0)}}{\dim\lambda}PT_{c^{(0)}}^{\mu^{(0)}}(\sigma)\in o_P(1).\]
 Therefore 
  \[\left\{PT_{u_i}^{\lambda}(\sigma_i)\right\}\overset{d}\to\left\{ u_i\prod_{k\geq 2} k^{m_k(\rho_i)/2} \mathcal{H}_{m_k(\rho_i)}(\xi_k)\right\},\]
  which concludes the proof.
 \end{proof}

 \label{p:2}

\cleardoublepage

\chapter{Stanley polynomials for strict partitions}\label{ch: strict partitions}
\section{Introduction}
The theory of projective (or spin) representations was born with three fundamental papers of Issai Schur, \cite{schur1904darstellung}, \cite{schur1907untersuchungen} and \cite{schur1911darstellung}, at the beginning of the 20th century. The idea originated from an attempt of studying the connections between the linear representations of a finite group $G$ and its factor groups. It soon became evident that projective representation theory could enrich its older brother, linear representation theory, and was better suited for several problems arising from quantum mechanics. Already in 1911 Schur provided a well developed study of the irreducible projective representation of the symmetric group $S_n$. These irreducible representations are indexed by strict partitions of size $n$, that is, integer partitions of the from $\lambda=(\lambda_1,\ldots,\lambda_l)$ with $\lambda_1>\ldots>\lambda_l$ and $\sum\lambda_i=n$. A Plancherel measure on strict partitions arises naturally. 

While the linear representation theory of the symmetric group proliferated for the entire past century, its projective counterpart was mostly stale. Only in the second half of the 1980s the theory was brought back to life by Sergev \cite{sergeev1985tensor}, Worley \cite{worley1984theory} and Sagan \cite{sagan1987shifted}, who introduced shifted Young tableaux, which serve the study of projective representations in a similar fashion to Young tableaux for the classical case. Later, the works of Stembridge \cite{stembridge1989shifted} and Nazarov \cite{nazarov1988orthogonal} casted some light on the combinatorics of the projective representations of the symmetric group. Since then, many mathematicians have worked with the purpose of exploiting the similarities between the classical representation theory of $S_n$ and the projective one. A honorable mention is the development of the theory of supersymmetric functions, generated by power sums indexed by odd partitions (where each part is odd). It was inevitable that the russian school of Kerov, Olshanski, Ivanov and others would attempt a dual approach, which led to asymptotic results of projective characters and shifted Young diagrams by Ivanov (\cite{ivanov2004gaussian} and \cite{ivanov2006plancherel}), for Plancherel distributed random strict partitions. We mention also the works of Han and Xiong \cite{han2017new} and Matsumoto \cite{matsumoto2015polynomiality} for their studies on the polynomiality of certain functions on strict partitions, parallel to those of Panova \cite{panova2012polynomiality} and Stanley \cite{stanley2010some} for the classical case.

In 2003 Stanley proposed an innovative coordinate system for integer partitions, well suited for the study of characters, called \emph{multirectangular coordinates}. Consider a partition $\lambda=(\lambda_1,\ldots,\lambda_l)$ written in English notation. In Stanley's multirectangular coordinates the partition is written as $\lambda:=(\q\times\p):=(p_1,\ldots,p_m;q_1,\ldots,q_m)$, where $q_1\geq q_2\geq\ldots\geq q_m$ and
\begin{align*}
 &q_1=\lambda_1=\lambda_2=\ldots=\lambda_{p_1}; \\
 &q_2=\lambda_{p_1+1}=\ldots=\lambda_{p_1+p_2};\\
 &\ldots \\
 &q_m=\lambda_{p_1+\ldots+p_{m-1}+1}=\ldots=\lambda_{p_1+\ldots+p_m}.
\end{align*}

Note that this notation is not unique, but it can be made so by requiring in addition that $q_1>q_2>\ldots>q_m$. See Figure \ref{picture rectangular partitions} for an example. Alternatively, we can write $\lambda$ in exponential notation $\lambda=(1^{m_1(\lambda)},2^{m_2(\lambda)},\ldots)$ and set $q_1>\ldots >q_m$ to be the parts such that $m_{q_i}(\lambda)>0$. Then in multirectangular coordinates we have
\[\lambda=(m_{q_1}(\lambda),\ldots, m_{q_m}(\lambda);q_1,\ldots, q_m).\]

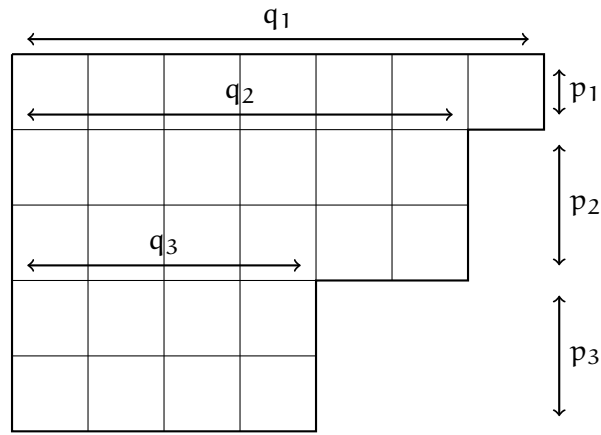
\begin{figure}
\begin{center}

\begin{tikzpicture}
 \draw [thick](0,5)--(0,0)--(4,0)--(4,2)--(6,2)--(6,4)--(7,4)--(7,5)--(0,5);
 \draw [<->, thick](0.2,5.2) -- node[anchor=south] {$q_1$}		(6.8,5.2);
 \draw [<->, thick](0.2,4.2) -- node[anchor=south] {$q_2$}		(5.8,4.2);
 \draw [<->, thick](0.2,2.2) -- node[anchor=south] {$q_3$}		(3.8,2.2);
 \draw [<->, thick](7.2,4.2) -- node[anchor=west]  {$p_1$}		(7.2,4.8);
 \draw [<->, thick](7.2,2.2) -- node[anchor=west]  {$p_2$}		(7.2,3.8);
 \draw [<->, thick](7.2,0.2) -- node[anchor=west]  {$p_3$}		(7.2,1.8);
 \draw [thin] (1,0)--(1,5);
 \draw [thin] (2,0)--(2,5);
 \draw [thin] (3,0)--(3,5);
 \draw [thin] (4,0)--(4,5);
 \draw [thin] (5,2)--(5,5);
 \draw [thin] (6,2)--(6,5);
 \draw [thin] (0,1)--(4,1);
 \draw [thin] (0,2)--(4,2);
 \draw [thin] (0,3)--(6,3);
  \draw [thin] (0,4)--(6,4);
 
\end{tikzpicture}\end{center}
\caption[Partition in multirectangular coordinates]{An example of the partition $\lambda=(7,6,6,4,4)$ written in multirectangular coordinates, that is, $\lambda=(7,6,4;1,2,2)$. Here $\q=(7,6,4)$ and $\p=(1,2,2)$}\label{picture rectangular partitions}
\end{figure}
Stanley \cite{stanley2003irreducible} conjectured that the normalized character $p^{\sharp}_k(\lambda)=n^{\downarrow k}\hat{\chi}^{\lambda}_{k,1^{n-k}}$, when written in function of the variables $-p_1,\ldots,-p_m,q_1,\ldots,q_m$, has nonpositive integer coefficients: let 
\[F_k(p_1,\ldots,p_m;q_1,\ldots,q_m)=n^{\downarrow k}\hat{\chi}^{\lambda}_{k,1^{n-k}},\]
then $-F_k(-p_1,\ldots,-p_m;q_1,\ldots,q_m)$ has nonnegative integer coefficients. The conjecture was extended in \cite{Stanley-preprint2006} into a formula which gives a combinatorial interpretation of the coefficients. This formula was later proved by F\'eray in \cite{feray2010stanley}.

In this chapter we introduce a strict partition counterpart of multirectangular coordinates. A conjecture arises, similar to the classical case, on the nonnegativity of the coefficients of the normalized character after some sign transformation: more precisely, let us call $\hat{\varkappa}^{\lambda}_{k,1^{n-k}}$ the \emph{projective} normalized character indexed by $\lambda$ and calculated on $(k,1^{n-k})$, where $\lambda$ is a strict partition and $k$ is odd. Set $F_k(p_1,\ldots,p_m;q_1,\ldots,q_m)=n^{\downarrow k}\hat{\varkappa}^{\lambda}_{k,1^{n-k}}$. Then we conjecture
\begin{conjecture}\label{conjecture: 1}
  The coefficients of the polynomial $-F_k(-p_1,\ldots,-p_m;q_1,\ldots,q_m)$ are nonnegative.
\end{conjecture}
Two remarks are in order: the coefficients of $p_1,\ldots, p_m,q_1,\ldots, q_m$ in $F_k$ are not integers, even in the smallest examples. On the other hand we prove in Proposition \ref{prop: for k even} that these coefficients are in $\Z/2$. Our second remark is that we lack a formula, parallel to the Stanley-F\'eray's formula, which gives a combinatorial interpretation of the coefficients. Notice that the proof of Stanley's original conjecture was obtained via the proof of the explicit formula, so it seems a natural step to first attempt to extend Stanley-F\'eray's formula to the projective case in order to solve Conjecture \ref{conjecture: 1}.

Following Stanley approach, we attack the problem starting from the leading term of $F_k$, called $L_k$, which is a homogeneous polynomial of degree $k+1$ in the variables $p_1,\ldots, p_m,q_1,\ldots, q_m$. Our main result is:
\begin{theorem}
  Let $k$ be odd. The leading term $-L_k(-\p\times\q)$ of $-F_k(-p_1,\ldots,-p_m;q_1,\ldots,q_m)$ has nonnegative coefficients in $\Z/2$.
\end{theorem}
Multirectangular coordinates and Stanley-F\'eray's formula were also fundamental for proving a positivity conjecture due to Kerov on the coefficients of the normalized character written in function of free cumulants. Free cumulants (defined in Section \ref{section: study main term}) are closely related to the Cauchy transform of the transition measure associated to the Plancherel measure on integer partitions. They have many interesting properties (see \cite{LassalleJackFreeCumulants}, \cite{FeraySniady2011} and \cite{Biane2003}) and they are especially useful for solving asymptotic questions. In a result credited to Kerov, Biane \cite{Biane2003} showed the existence of polynomials $K_k(x_1,\ldots,x_k)$, for $k\geq1$, such that 
\[n^{\downarrow k}\hat{\chi}^{\lambda}_k=K_k(R_2(\lambda),R_3(\lambda)\ldots,R_{k+1}(\lambda)),\]
where $R_j(\lambda)$ is the $j$-th free cumulant evaluated on $\lambda$. Kerov also conjectured that the polynomial $K_k(x_1,\ldots, x_k)$ has nonnegative integer coefficients. This conjecture was proven by F\'eray in \cite{F'eray2008}.
\smallskip

In a private communication, Matsumoto suggested to consider free cumulants $\tilde{R}_j(\lambda)$ associated to the transition measure on strict partitions. He conjectured the existence of projective versions of Kerov polynomials $\tilde{K}_k(x_1,\ldots, x_{\frac{k+1}{2}})$ such that
\[n^{\downarrow k}\hat{\varkappa}^{\lambda}_k=\tilde{K}_k(\frac{\tilde{R}_2(\lambda)}{2},\frac{\tilde{R}_4(\lambda)}{2}\ldots,\frac{\tilde{R}_{k+1}(\lambda)}{2}),\]
where $\hat{\varkappa}^{\lambda}_k$ is the normalized projective character. Matsumoto also conjectured that the coefficients of the polynomial $\tilde{K}_k$ are nonnegative integers, and that the leading term is $\tilde{R}_{k+1}(\lambda)/2$. In Proposition \ref{prop: leading term} we prove that the leading term is indeed $\tilde{R}_{k+1}(\lambda)/2$.
\medskip

In \cite{ivanov2004gaussian} and \cite{ivanov2006plancherel} Ivanov studied the asymptotic of shifted Young diagrams (which are projective versions of Young diagrams for strict partitions) and projective normalized characters for Plancherel distributed strict partitions. The classical version of his work can be found in \cite{ivanov2002kerov}. An almost immediate consequence of Ivanov's results is the convergence of shifted Young diagrams to its limit shape in the uniform topology. We include here the missing piece of the puzzle, but we stress out that the result is merely a corollary of Ivanov's work.
\bigskip

In Section \ref{section: preliminaries} we recall the theory of projective representations for finite groups, focusing on the symmetric group. We compute the Plancherel measure in this setting, defining it on strict partitions. In Section \ref{section: multirectangular coordinates} we introduce multirectangular coordinates for strict partitions, and we present a conjecture on the coefficients of the normalized character in function of those coordinates. We then prove this conjecture for the leading term of the normalized character. In Section \ref{section: asymptotics} we obtain a consequence of Ivanov's work on the asymptotic of shifted Young diagrams.

\section{Projective representation theory}\label{section: preliminaries}
\subsection{Projective representations of finite groups}
In this section we present an introduction to projective character theory, following \cite{stembridge1989shifted}, \cite{wan2011lectures}, \cite{karpilovsky1985projective} and \cite{curtis1990methods}.

Let $V$ be a $\C$-vector space and $\GL(V)$ the \emph{general linear group}, that is, the group of isomorphisms $V\to V$. Define the \emph{projective linear group} $\PGL(V)$ as the quotient $\GL(V)/\C^{\times}$.
\begin{defi}
 A \emph{projective representation} of a finite group $G$ is a homomorphism
 \[\tilde{\pi}\colon G\to\PGL(V).\] 
\end{defi}
Given a projective representation $\tilde{\pi}\colon G\to \PGL(V)$ we can lift it to a function $\pi\colon G\to \GL(V)$ as follows: we choose a section $q\colon \PGL(V)\to\GL(V)$, that is, a right inverse of the projection $\GL(V)\to\PGL(V)$, and we set $\pi=q\circ\tilde{\pi}$. Such a function $\pi$ is not a homomorphism, but there exists a map $c\colon G\times G\to\C^{\times}$ such that 
\begin{equation}\label{eq:4}
 \pi(x)\pi(y)=c(x,y)\pi(xy)\qquad \mbox{ for each }x,y\in G.
\end{equation}
The map $c$ is called a \emph{factor set}. By associativity
\begin{equation}\label{eq:2}
  c(x,y)\cdot c(xy,z)=c(x,yz)\cdot c(y,z)\qquad \mbox{ for each }x,y,z\in G.
\end{equation}
Then we derive an alternative definition of projective representation: a map $\pi\colon G\to\GL(V)$ such that $\pi(\id_G)=\id_V$ so that there exists a factor set $c\colon G\times G\to\C^{\times}$ such that \eqref{eq:4} and \eqref{eq:2} hold. Throughout the chapter we will alternate between the notations $\tilde{\pi}\colon G\to\PGL(V)$ and $\tilde{\pi}\colon G\to\GL(V)$ for projective representations.

If we fix a basis of $V$ which identifies $\GL(V)$ with $\GL_n(\C)$, then the map $\tilde{\pi}$ can be written as a matrix and it is called the \emph{projective matrix representation} of $G$ over $\C$ of degree $n$.
\begin{defi}
 Two projective representations $\tilde{\pi}_1\colon G\to\GL(V)$ and $\tilde{\pi}_2\colon G\to\GL(W)$ are said to be \emph{equivalent} if there exists an isomorphism $\varphi\colon V\to W$ and a map $b\colon G\to \C^{\times}$ such that 
 \[b(x)\cdot(\varphi\circ \tilde{\pi}_1(x)\circ\varphi^{\langle-1\rangle})=\tilde{\pi}_2(x)\]
 for each $x\in G$, where $(\cdot)^{\langle-1\rangle}$ means compositional inverse. Equivalent projective representations are said to have equivalent factor sets.
\end{defi}
It is not difficult to see that if two projective representations $\tilde{\pi}_1$ and $\tilde{\pi}_2$ with factor sets respectively $c_1$ and $c_2$ are equivalent then the map $b$ satisfies 
\[c_1(x,y)=\frac{b(x)b(y)}{b(xy)}c_2(x,y)\qquad\mbox{ for each }x,y\in G.\]
The group of factor sets modulo equivalence is abelian and it is called the \emph{Schur multiplier} $M(G)$ of $G$. This group is isomorphic to the second homology group $M(G)\cong H_2(G,\Z)$. Fore more on the topic, see \cite[Section 8]{curtis1990methods}.

The quotient map $q\colon \GL(V)\to\PGL(V)$ is a \emph{central extension}, that is, the kernel $\ker q$ is in the center of $\GL(V)$ (actually, in this particular case it is exactly the center). Schur showed that one can pull back the projective representation $\tilde{\pi}\colon G\to\PGL(V)$ through the quotient map, obtaining a linear representation $\pi\colon E\to \GL(V)$, where $E\overset{q'}\to G$ is a central extension of $G$. Thus there exists a map $q'\colon E\to G$ such that the following diagram commutes:
\[\xymatrix{
E\ar[r]^{\pi}\ar[d]_{q'}&\GL(V)\ar[d]^q\\
G\ar[r]^{\tilde{\pi}}&\PGL(V)}\]
In the original papers which began the theory of projective representations (\cite{schur1904darstellung}, \cite{schur1907untersuchungen} and \cite{schur1911darstellung}), Schur proved that there exists a finite central extension $E\overset{q'}\to G$ such that every projective representation of $G$ can be lifted to a linear representation of $E$. Moreover, there exists a surjective map $\ker q'\twoheadrightarrow M(G)$. When $\ker q'\cong M(G)\cong H_2(G,\Z)$ then the central extension is minimal and it is called a \emph{representation group}. Schur proved the existence of a representation group for each group $G$ (see \cite[Section 11E]{curtis1990methods}).
\smallskip

One can obtain information on the projective representations of $G$ by studying the linear representations of the representation group $E$. This was the original approach of Schur for $G=S_n$, the symmetric group, as we will see in the next section.
\subsection{Projective representations of the symmetric group}
In \cite{schur1911darstellung} the author proved the following:
\begin{theorem}
 The Schur multiplier for $G=S_n$ is
 \[M(S_n)=\left\{\begin{array}{ll}\Z_2&\mbox{ if }n\geq 4,\\\{1\}&\mbox{ if }n<4.\end{array}\right.\] 
\end{theorem}
Consider $n\geq 4$, then we deduce from the previous theorem that the representation groups of $S_n$ have order $2\cdot n!$. Let us write $\Z_2=\{1,z\}$ and consider for the rest of the section that $n\geq 4$. In the same article Schur showed that, up to isomorphism, there are only two representation groups of $S_n$ (\cite[p. 166]{schur1911darstellung}) which we call $\tilde{S}_n$ and $\tilde{S}_n'$. Recall that $S_n$ is generated by the adjacent transpositions $\tau_1,\ldots,\tau_{n-1}$, where $\tau_i=(i,i+1)$, with the following Coxeter relations:
\[\tau_j^2=1,\qquad (\tau_j\tau_k)^2=1\mbox{ if }|j-k|>2,\qquad (\tau_j\tau_{j+1})^3=1.\]
Then the groups $\tilde{S}_n$ and $\tilde{S}_n'$ have similar presentations: $\tilde{S}_n$ is generated by $z,s_1,\ldots,s_{n-1}$ such that 
\[s_j^2=z,\qquad (s_js_k)^2=z\mbox{ if }|j-k|>2,\qquad (s_js_{j+1})^3=z.\]
On the other hand, $\tilde{S}_n'$ is generated by $z,s_1',\ldots,s_{n-1}'$ such that 
\[s_j'^2=1,\qquad (s_j's_k')^2=z\mbox{ if }|j-k|>2,\qquad (s_j's_{j+1}')^3=1.\]

Consider $\C[S_n]$, $\C[\tilde{S}_n],$ $\C[\tilde{S}_n']$ the group algebras of, respectively, $S_n,\tilde{S}_n,\tilde{S}_n'$. It is clear from the presentations of these groups that $\C[S_n]\cong\C[\tilde{S}_n]/\langle z-1\rangle\cong\C[\tilde{S}_n']/\langle z-1\rangle$. Hence each linear representation of $\C[\tilde{S}_n]$ (or $\C[\tilde{S}_n']$) that sends $z$ to $\mathbb{1}_V$ is equivalent to a linear representation of $\C[S_n]$, and vice versa.

On the other hand, a linear representation $\pi\colon \tilde{S}_n\to \GL(V)$ (resp. $\pi\colon \tilde{S}_n'\to \GL(V)$) in which $z\mapsto -\mathbb{1}_V$ is called a \emph{spin representation} of $\tilde{S}_n$ (resp. of $\tilde{S}_n'$). The module $V$ is called a \emph{spin module} of $\tilde{S}_n$ (resp. of $\tilde{S}_n'$), and each linear representation of $\C[\tilde{S}_n]$ that sends $z$ to $-\mathbb{1}_V$ is equivalent to a linear representation of the \emph{spin group algebra} \[\C[S_n^-]:=\C[\tilde{S}_n]/\langle z+1\rangle.\]
From the presentations of $\tilde{S}_n$ and $\tilde{S}_n'$ it is obvious that $\C[\tilde{S}_n]/\langle z+1\rangle\cong\C[\tilde{S}_n']/\langle z+1\rangle.$
\smallskip

If $\pi$ is a spin representation of $\tilde{S}_n$ such that $\pi(s_j)=A_j\in \GL(V)$, then $\pi':s_j'\mapsto iA_j$ is a spin representation of $\tilde{S}_n'$, where $i=\sqrt{-1}$. Therefore the semigroup of spin representations of $\tilde{S}_n$ is isomorphic to the semigroup of spin representations of $\tilde{S}_n'$. For the rest of the chapter we shall focus thus only on the spin representations of $\tilde{S}_n$.
\subsection{Conjugacy classes and irreducible spin characters}
Let $\lambda=(\lambda_1,\ldots,\lambda_l)$ be a partition of $n$. If $\lambda_1>\lambda_2>\ldots>\lambda_l$ then $\lambda$ is called a \emph{strict (or distinct) partition}. Moreover, $\lambda$ is said to be \emph{even} if $n-l$ is even, and \emph{odd} otherwise. Similarly, a permutation $\sigma\in S_n$ is \emph{even} if the minimal number of adjacent transpositions occurring in the decomposition of $\sigma$ is even, and \emph{odd} otherwise. Note that if $\lambda$ is the cycle type of $\sigma$, then $\lambda$ is even if and only if $\sigma$ is even. We say that a spin representation $\pi$ is \emph{irreducible} if $\pi$ is irreducible as a linear representation, and we define similarly a spin irreducible module. Moreover, we call \emph{spin} character the character of a spin representation.

In \cite{schur1907untersuchungen} Schur gave a complete description of the conjugacy classes and the irreducible spin characters of $\tilde{S}_n$, which we recall: consider the projection of $\tilde{S}_n$ into $S_n$
\[\begin{array}{ccccc}\proj_{S_n}&\colon&\tilde{S}_n&\to& S_n\\&&s_j&\mapsto&\tau_j\\&&z&\mapsto&\id\end{array}\]
then for a partition $\rho$ we define 
\[\mathcal{C}_{\rho}:=\{\sigma \in \tilde{S}_n \mbox{ such that } \proj_{S_n}(\sigma) \mbox{ has cycle type }\rho\}.\]
Let $\sigma,\tau\in\mathcal{C}_\rho$ for some partition $\rho$ and suppose that $\proj_{S_n}(\sigma)$ is conjugated to $\proj_{S_n}(\tau)$ in $S_n$. Then $\tau$ is conjugated to either $\sigma$ or $z\sigma$. Hence if $\sigma$ is conjugated to $z\sigma$ it follows that $\mathcal{C}_{\rho}$ is a conjugacy class, and if $\sigma$ and $z\sigma$ are not conjugated then $\mathcal{C}_{\rho}$ is the union of two conjugacy classes of $\tilde{S}_n$, one containing $\sigma$ and the other containing $z\sigma$. In order to describe when $\mathcal{C}_{\rho}$ is a conjugacy class we fix some notation. For each $n$ set
\begin{itemize}
 \item $OP_n:=\{\lambda=(\lambda_1,\ldots,\lambda_l)\vdash n $ such that $\lambda_i$ is odd for all $i\}$;
 \item $DP^+_n:=\{\lambda=(\lambda_1,\ldots,\lambda_l)\vdash n $ such that $\lambda_1>\ldots>\lambda_l$ and $n-l$ is even$\}$;
 \item $DP^-_n:=\{\lambda=(\lambda_1,\ldots,\lambda_l)\vdash n $ such that $\lambda_1>\ldots>\lambda_l$ and $n-l$ is odd$\}$;
 \item $DP_n:=DP_n^+\cup DP_n^-.$ 
\end{itemize}
Notice that every partition in $OP_n$ is even. Recall that $P_n$ is the set of all partitions of $n$.
\begin{theorem}
 Let $\rho\in P_n$ be a partition, then 
 \begin{itemize}
  \item if $\rho\in OP_n\cup DP_n^-$ then $\mathcal{C}_{\rho}$ is the union of two conjugacy classes $\mathcal{C}_{\rho}^+$ and $\mathcal{C}_{\rho}^-$ of $\tilde{S}_n,$ and $\mathcal{C}_{\rho}^+=z\mathcal{C}_{\rho}^-.$
  \item if $\rho\notin (OP_n\cup DP_n^-)$ then $\mathcal{C}_{\rho}$ is a conjugacy class in $\tilde{S}_n$.
 \end{itemize}
 \end{theorem}
This theorem was proven originally in \cite[p. 172]{schur1911darstellung}, and a modern proof can be found in \cite[Theorem 2.1]{stembridge1989shifted}. 

Let $\pi$ be a spin representation, then $\pi(\sigma)=-\pi(z\sigma)$. Suppose $\sigma\in\mathcal{C}_{\rho}$ with $\rho\notin (OP_n\cup DP_n^-)$, then $\sigma$ is conjugated to $z\sigma$. Let $\varkappa$ be the spin character associated to $\pi$, then 
\[\varkappa(\sigma)=\varkappa(z\sigma)=-\varkappa(\sigma),\]
hence every spin character is automatically zero on $\mathcal{C}_{\rho}$ with $\rho\notin (OP_n\cup DP_n^-)$.
\begin{theorem}\label{thm: index of spin representations}
 To each $\lambda\in DP_n^+$ is associated an irreducible spin representation $\pi^{\lambda}$, and to each $\lambda\in DP_n^-$ is associated a pair of irreducible spin representations $\pi_1^{\lambda},$ $\pi_2^{\lambda}=\sgn\cdot \pi_1^{\lambda}$. Each irreducible spin representation can be written this way up to isomorphism. 
\end{theorem}
Let us check that the number of irreducible representations of $\tilde{S}_n$ is equal to the number of its conjugacy classes. As mentioned above, the irreducible representations of $\tilde{S}_n$ that send $z\mapsto \mathbb{1}_V$ are in correspondence with the irreducible representations of $S_n$, and are thus indexed by $P_n$. Hence the number of irreducible representations of $\tilde{S}_n$ is \[|P_n|+|DP_n^+|+2|DP_n^-|=|P_n|+|DP_n|+|DP_n^-|.\]
On the other hand the number of conjugacy classes of $\tilde{S}_n$ is 
\[2|OP_n\cup DP_n^-|+|\overline{OP_n\cup DP_n^-}|=|P_n|+|OP_n\cup DP_n^-|.\]
The theory of irreducible representations claims that the number of irreducible representations of a finite group is the same as the number of conjugacy classes of that group. The group $\tilde{S}_n$ is no exception, since $OP_n\cap DP_n^-=\emptyset$, so that $|P_n|+|OP_n\cup DP_n^-|=|P_n|+|OP_n|+| DP_n^-|,$ and it is well known that $|OP_n|=|DP_n|$.

We devote the next section to introduce the necessary objects to present the spin character table of $\tilde{S}_n$, that is, the table of values of spin characters on the conjugacy classes of $\tilde{S}_n$.
\subsection{The spin character table}
Set $\mathcal{P}=\bigcup P_n$, $DP=\bigcup DP_n$ and $OP=\bigcup OP_n$. Recall from Section \ref{section: symmetric functions} that $\Lambda$ is the algebra of symmetric functions. If $p_r$ is the $r$-th power sum symmetric function 
\[p_r(x_1,x_2,\ldots)=x_1^r+x_2^r+\ldots,\]
then $\{p_r\}_{r\geq 1}$ generates $\Lambda$. For a partition $\rho=(\rho_1,\ldots,\rho_{l(\rho)})$ set $p_{\rho}=p_{\rho_1}\cdot\ldots\cdot p_{\rho_{l(\rho)}}$. We call $\Gamma$ the \emph{algebra of supersymmetric functions} generated by $\{p_{2r+1}\}_{r\geq 1}$. Similarly to the classical case, set $\Gamma_m$ to be the algebra generated by $\{p_{2r+1}(x_1,\ldots x_m)\}$. The algebra $\Gamma$ can be seen as the projective limit $\Gamma=\lim\limits_{\leftarrow}\Gamma_m$, where the projection $\Gamma_{m+1}\to\Gamma_m$ sends \[p_r(x_1,\ldots,x_m,x_{m+1})\mapsto p_r(x_1,\ldots,x_m,0).\]
\begin{defi}
 Let $\lambda=(\lambda_1,\ldots,\lambda_l)\in DP$ and $m\in\N$. Define the supersymmetric polynomial $P_{\lambda|m}$ by
 \[P_{\lambda|m}:=\left\{\begin{array}{cc}
                          \frac{1}{(m-l)!}\sum\limits_{\sigma\in S_m}\sigma\left(x_1^{\lambda_1}\cdot\ldots\cdot x_l^{\lambda_l}\prod\limits_{\substack{i:1\leq i\leq l\\ j:i<j\leq m}}\frac{x_i+x_j}{x_i-x_j}\right)&\mbox{ if }m\geq l\\ 0&\mbox{ otherwise,}
                         \end{array}\right.\]
where, for a function $f(x_1,\ldots, x_m)$ and a permutation $\sigma\in S_m$ we define \[\sigma(f(x_1,\ldots,x_m))=f(x_{\sigma(1)},\ldots, x_{\sigma(l)}).\] 
\end{defi}
We call $P_{\lambda}$ the inverse limit of $P_{\lambda|m}$ for $m\to\infty$. It was proven by Schur in \cite[pag. 225]{schur1911darstellung} that $\{P_{\lambda}\}_{\lambda\in DP}$ forms a linear basis of $\Gamma$ (see \cite[Chapter III]{McDo}). The functions $P_{\lambda}$ are called \emph{Schur}-$P$ \emph{functions}. The functions $Q_{\lambda}:=2^{l(\lambda)}P_{\lambda}$ are called \emph{Schur-Q functions}, where $l(\lambda)$ is the number of parts of $\lambda$.

We present now the spin character table of $\tilde{S}_n$. For $\lambda\in DP^-_n$ and $\rho\in OP_n\cup DP_n^-$ there are two spin characters associated to $\lambda$, which we call $\varkappa^{\lambda}_\rho$ and $\varkappa'^{\lambda}_\rho$. Note that $\varkappa'^{\lambda}_\rho=(-1)^{\mbox{\tiny{parity of} }\rho}\varkappa^{\lambda}_\rho$, hence we will represent only $\varkappa^{\lambda}_{\rho}$ on the table. Recall from the observation before Theorem \ref{thm: index of spin representations} that the spin characters are $0$ on the conjugacy classes indexed by $\rho\notin OP_n\cup DP_m^-$.
\vspace{0.5cm}

\begin{tabular}{c||c|c}
 &$\varkappa^{\lambda},\lambda\in DP_n^+$&$\varkappa^{\lambda},\lambda\in DP_n^-$\\ 
 \hline\hline \\[-1em]
 $\substack{\mathcal{C}_{\rho}^+\\\rho\in OP_n\cup DP_m^-}$&\(\begin{array}{cc}2^{\frac{l(\lambda)+l(\rho)}{2}}\langle P_\lambda,p_\rho\rangle&\mbox{ if }\rho\in OP_n,\\0&\mbox{ otherwise } \end{array}\)&
 \(\begin{array}{cc}2^{\frac{l(\lambda)+l(\rho)-1}{2}}\langle P_\lambda,p_\rho\rangle&\mbox{ if }\rho\in OP_n,\\i^{\frac{n-l(\lambda)+1}{2}}\sqrt{\frac{z_{\lambda}}{2}}&\mbox{ if }\rho=\lambda\in DP_n^-,\\0&\mbox{ otherwise } \end{array}\)\\
\hline\\[-1em]
$\substack{\mathcal{C}_{\rho}^-\\\rho\in OP_n\cup DP_m^-}$&\(\begin{array}{cc}-2^{\frac{l(\lambda)+l(\rho)}{2}}\langle P_\lambda,p_\rho\rangle&\mbox{ if }\rho\in OP_n,\\0&\mbox{ otherwise } \end{array}\)&
\( \begin{array}{cc}-2^{\frac{l(\lambda)+l(\rho)-1}{2}}\langle P_\lambda,p_\rho\rangle&\mbox{ if }\rho\in OP_n,\\-i^{\frac{n-l(\lambda)+1}{2}}\sqrt{\frac{z_{\lambda}}{2}}&\mbox{ if }\rho=\lambda\in DP_n^-,\\0&\mbox{ otherwise } \end{array}\)\\
 \end{tabular}
\vspace{0.5cm}

 In the table, $\langle\cdot,\cdot\rangle$ is the Frobenious inner product described in Section \ref{section: repr theory}, $i=\sqrt{-1}$ and 
 \[z_\lambda=\prod_{i=1}^{l(\lambda)}\lambda_i\prod_{i=1}^\infty m_i(\lambda),\qquad\mbox{ where }\lambda=(\lambda_1,\lambda_2,\ldots,\lambda_{l(\lambda)})=(1^{m_1(\lambda)},2^{m_2(\lambda)},\ldots).\]

Let $\lambda\in DP_n$, in \cite[pag. 235]{schur1911darstellung} Schur showed a formula for the dimension of the representation $\pi^\lambda$, that is,
\begin{equation}\label{eq: schur strict dimension}
 \varkappa^\lambda_{(1^n)}=\frac{1}{\epsilon_\lambda}2^{\frac{n-l(\lambda)}{2}}g^\lambda,\mbox{ where }
\end{equation}
\begin{equation}\label{eq: def g}
 \epsilon_\lambda=\left\{\begin{array}{cc}\sqrt{2}&\mbox{ if }\lambda\in DP_n^-\\1&\mbox{ if }\lambda\in DP_n^+\end{array}\right.\qquad\mbox{ and }\qquad g^\lambda=\frac{n!}{\lambda_1!\lambda_2!\ldots\lambda_l!}\prod_{i<j}\frac{\lambda_i-\lambda_j}{\lambda_i+\lambda_j}.
\end{equation}

\subsection{The strict Plancherel measure}\label{section: bratteli stric partitions}
Recall that from representation theory (Section \ref{section: repr theory}) we obtain a Plancherel measure associated to the indices of irreducible characters of the group $\tilde{S}_n$. Let us say that the spin characters of $\tilde{S}_n$ are indexed by $\tilde{DP}_n:=DP_n^+\cup DP_n^-\cup\left(DP_n^-\right)'$, where $\left(DP_n^-\right)'$ is a copy of $DP_n^-$. The non spin characters are the irreducible characters of $S_n$, so they are indexed by the set of integer partitions $P_n$. The Plancherel measure \eqref{eq: plancherel measure} can be written thus as
\[1=\sum_{\lambda\in P_n}\frac{\varkappa^{\lambda}(\id_{\tilde{S}_n})^2}{2\cdot n!}+\sum_{\lambda\in \tilde{DP}_n}\frac{\varkappa^{\lambda}(\id_{\tilde{S}_n})^2}{2\cdot n!},\]
where the first sum involves non spin characters, and the second involves spin characters. From the representation theory of the symmetric group we know that 
\[\sum_{\lambda\in P_n}\frac{\varkappa^{\lambda}(\id_{\tilde{S}_n})^2}{ n!}=1,\]
hence
\[\sum_{\lambda\in \tilde{DP}_n}\frac{\varkappa^{\lambda}(\id_{\tilde{S}_n})^2}{n!}=1.\]
By substituting Schur's formula for the dimension of the spin characters (Equation \eqref{eq: schur strict dimension}) we obtain that, for $\lambda\in \tilde{DP}_n$,
\[\frac{\varkappa^{\lambda}(\id_{\tilde{S}_n})^2}{ n!}=\left\{\begin{array}{ll}\frac{2^{n-l(\lambda)}}{n!}(g^{\lambda})^2&\mbox{ if }\lambda\in DP_n^+\\\frac{2^{n-l(\lambda)-1}}{n!}(g^{\lambda})^2&\mbox{ if }\lambda\in DP_n^-\cup\left(DP_n^-\right)'. \end{array}\right.\]

By associating to each $\lambda\in DP_n$ the weight $\frac{2^{n-l(\lambda)}(g^\lambda)^2}{n!}$ we obtain a measure on the set of strict partitions $DP_n$:
\[P_{strict}^n(\lambda):=\frac{2^{n-l(\lambda)}(g^\lambda)^2}{n!},\]
which we call \emph{strict Plancherel measure}, and a partition $\lambda\in DP_n$ is strict Plancherel distributed when it is randomly chosen with this measure. Strict Plancherel distributed partitions have been studied, among others, by Ivanov (\cite{ivanov2004gaussian},\cite{ivanov2006plancherel}), Huan and Xiang (\cite{han2017new}), and Matsumoto (\cite{matsumoto2015polynomiality}). 

Let $\lambda\in DP$; the \emph{shifted Young diagram}, pictured in English style, is
\[\overline{S}(\lambda):=\{(i,j)\in \Z^2\mbox{ such that }1\leq i\leq l(\lambda), i\leq j\leq \lambda_i+i-1\}.\]
A \emph{box} of $\lambda$ is an element $\Box=(i,j)\in \overline{S}(\lambda)$. The shifted Young diagram is usually represented with the $y$-axis pointing downwards, see Figure \ref{fig: shifted tableau}.

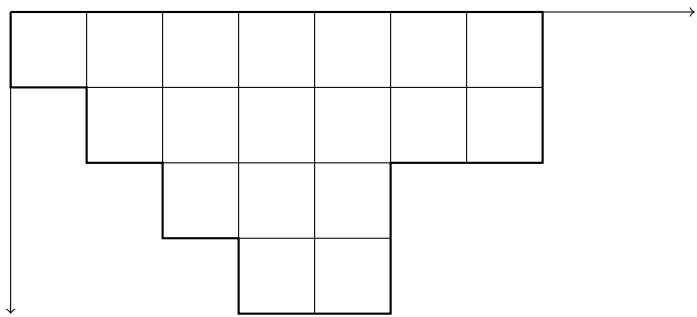
\begin{figure}
\begin{center}

\begin{tikzpicture}
\draw[->](0,4)--(9,4);
\draw[->](0,4)--(0,0);
 \draw [thick](0,4)--(0,3)--(1,3)--(1,2)--(2,2)--(2,1)--(3,1)--(3,0)--(5,0)--(5,2)--(7,2)--(7,4)--(0,4);
 \draw (3,1)--(5,1);
 \draw (2,2)--(5,2);
 \draw (1,3)--(7,3);
 \draw (1,3)--(1,4);
 \draw (2,2)--(2,4);
 \draw (3,1)--(3,4);
 \draw (4,0)--(4,4);
 \draw (5,2)--(5,4);
 \draw (6,2)--(6,4);
 \end{tikzpicture}
\end{center}
\caption[Shifted Young tableau]{Example of the shifted Young diagram associated to the partition $\lambda=(7,6,3,2)$.}\label{fig: shifted tableau}
\end{figure}

We can consider now the lattice of shifted Young diagrams, see Figure \ref{figure: shifted young lattice}, and notice that it is a Bratteli diagram, according to Definition \ref{defi: bratteli diagram}. The set of strict Plancherel measures $\{P_{strict}^n\}$ is a set of coherent measures, and the transition and co-transition measures are respectively, for $\lambda\in DP_n$ and $\Lambda\in DP_{n+1}$,
\[\tr(\lambda,\Lambda)=2^{\frac{l(\Lambda)-l(\lambda)+1}{2}}\frac{g^{\Lambda}}{(n+1)g^{\lambda}},\qquad \ctr(\lambda,\Lambda)=2^{\frac{l(\Lambda)-l(\lambda)-1}{2}}\frac{g^{\lambda}}{g^{\Lambda}}\]
 if $\lambda\nearrow\Lambda$.

 \begin{figure}
 \begin{center} 
 \scalebox{.7}{
  \[\xymatrix{\vuoto&\vuoto&\vuoto&\vuoto\\
  &\begin{array}{c}\young(\quad\quad\quad,:\quad\quad) \end{array}\ar[lu]\ar[u]&\begin{array}{c}\young(\quad\quad\quad\quad,:\quad) \end{array}\ar[lu]\ar[u]&\begin{array}{c}\young(\quad\quad\quad\quad\quad) \end{array}\ar[lu]\ar[u]\\
  &&\begin{array}{c}\young(\quad\quad\quad,:\quad) \end{array}\ar[lu]\ar[u]&\begin{array}{c}\young(\quad\quad\quad\quad) \end{array}\ar[lu]\ar[u]\\
  &&\begin{array}{c}\young(\quad\quad,:\quad) \end{array}\ar[u]&\begin{array}{c}\young(\quad\quad\quad) \end{array}\ar[lu]\ar[u]\\
  &&&\begin{array}{c}\young(\quad\quad) \end{array}\ar[lu]\ar[u]\\
  &&&\begin{array}{c}\young(\quad) \end{array}\ar[u]\\
  &&&\emptyset\ar[u]}\]}
  \caption[Beginning of the lattice of shifted Young diagrams]{Beginning of the lattice of shifted Young diagrams of size $\leq 5$.}\label{figure: shifted young lattice} \end{center}
 \end{figure}
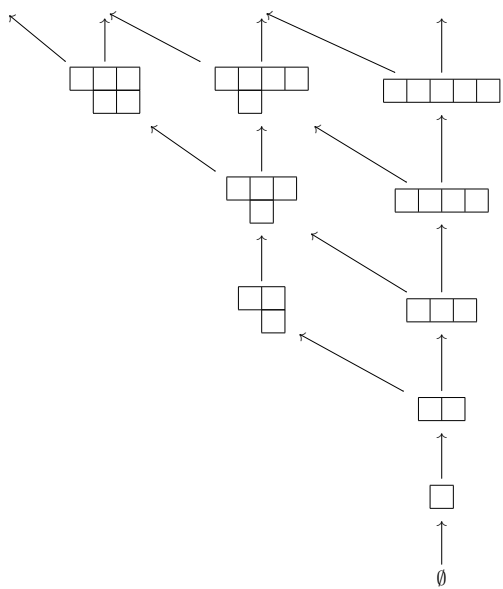

Let $\lambda$ be a strict partition, which we identify with its shifted Young diagram in English coordinates. For a box $\Box\in\lambda$ of coordinates $(i,j)$ we define the arm length
\[\arm_{\Box}=\arm_{(i,j)}=\mbox{ number of boxes in }\lambda\mbox{ exactly to the right of }\Box.\]
Similarly, the leg length is
\[\leg_{\Box}=\leg_{(i,j)}=\mbox{ number of boxes in }\lambda\mbox{ exactly below }\Box.\]
Set the \emph{shifted hook length} for $\Box=(i,j)$ in a shifted Young diagram $\lambda$ to be \[\tilde{h}_{\Box}=\tilde{h}_{(i,j)}=\arm_{(i,j)}+\leg_{(i,j)}+\lambda_{j+1}.\] We call $\tilde{\mathcal{H}}_{\lambda}$ the multiset of shifted hook lengths of $\lambda$, $\tilde{\mathcal{H}}_{\lambda}=\{\tilde{h}_\Box\mbox{ s.t. }\Box\in\lambda\}$, and $\tilde{H}_\lambda$ the \emph{shifted hook product:}
\[\tilde{H}_\lambda=\prod_{\Box\in\lambda}\tilde{h}_\Box.\]

The following formula is well known for a strict partition $\lambda\in DP_n$ (see for example \cite{bandlow2008elementary}, \cite{schur1911darstellung}, \cite{han2017new}):
\[g^{\lambda}=\frac{n!}{\tilde{H}_\lambda},\]
where $g^\lambda$ was defined in Equation \eqref{eq: def g}.


\subsection{An explicit formula for spin characters}

For the rest of the chapter we will focus on the values of $\varkappa^\lambda_\rho$ for $\lambda\in DP_n$ and $\rho\in OP_n$, bearing in mind that there are two (related) characters $\varkappa^\lambda_\rho$ and $\varkappa'^\lambda_\rho$ if $\lambda\in DP_n^-$ and $\varkappa^\lambda_\rho$ is easily computed if $\rho\notin OP_n$. We recall the work of Ivanov in \cite{ivanov2006plancherel}, equivalent to \cite{ivanov2002kerov} for the classical case. 

Consider $\lambda\in DP_n$, which we identify with its shifted Young diagram. We will be mostly concerned with shifted Young diagrams pictured in russian coordinates, that is
\[\tilde{S}(\lambda):=\left\{\left(i-j+\frac{1}{2},i+j+\frac{1}{2}\right)\mbox{ such that } (i,j)\in \overline{S}(\lambda)\right\}.\]
The transformation $\overline{S}(\lambda)\mapsto \tilde{S}(\lambda)$ correspond to multiplying both $i$ and $j$ by $\sqrt{2}$, rotating the points by $\pi/4$ counterclockwise, and adding $1/2$ to both coordinates. See for example the first picture of Figure \ref{fig: shifted tableau in russian}. As in the classical case, we identify the diagram of $\lambda$ with the piecewise linear function describing the border of the diagram, called by abuse of notation $\lambda(x)\colon \R_{\geq 0}\to\R_{\geq 0}$; we set $\lambda(x)=x$ if $x\geq \lambda_1+1/2.$ It will be convenient for us to consider the \emph{double diagram} $D(\lambda)\colon\R\to\R$ defined as $D(\lambda)=\lambda(|x|)$, as suggested by Matsumoto in a private communication. In this way we can embed $DP$ into the set of continual diagrams, introduced by Kerov and described in Section \ref{section: asymptotic of diagrams}. The function $D(\lambda)$ is completely determined by the set of maxima $y_{-m},\ldots,y_{-1},y_1,\ldots y_m$ and minima $x_{-m},\ldots,x_{-1},x_0,x_1,\ldots x_m$ with $x_0=0$. See the second picture of Figure \ref{fig: shifted tableau in russian} for an example of the function $D(\lambda)$ for $\lambda=(7,6,3,2)$.
\begin{figure}
\begin{center}
\begin{tikzpicture}[thick,>=stealth,scale=0.5]
      \draw[step=1cm,gray!30,very thin] (-1,-1) grid (8,10);  
      \begin{scope}
        \draw (0.5,0.5)--(7.5,7.5)--(5.5,9.5)--(3.5,7.5)--(1.5,9.5)--(-0.5,7.5)--(0.5,6.5)--(-0.5,5.5)--(0.5,4.5)--(-0.5,3.5)--(0.5,2.5)--(-0.5,1.5)--(0.5,0.5);
        \draw[thin] (0.5,2.5)--(6.5,8.5);
        \draw[thin] (0.5,4.5)--(3.5,7.5);
        \draw[thin] (0.5,6.5)--(2.5,8.5);
        \draw[thin] (1.5,1.5)--(0.5,2.5);
        \draw[thin] (0.5,4.5)--(2.5,2.5);
        \draw[thin] (0.5,6.5)--(3.5,3.5);
        \draw[thin] (0.5,8.5)--(4.5,4.5);
        \draw[thin] (3.5,7.5)--(5.5,5.5);
        \draw[thin] (4.5,8.5)--(6.5,6.5);
    \end{scope}

     \draw[thick,->]  (-1,0) -- (8,0);
     \draw[thick,->]  (0,-1) -- (0,10);
\end{tikzpicture}\qquad
\begin{tikzpicture}[thick,>=stealth,scale=0.5]
      \draw[step=1cm,gray!30,very thin] (-10,-1) grid (10,10);  
      \begin{scope}
        \draw (0.5,0.5)--(7.5,7.5)--(5.5,9.5)--(3.5,7.5)--(1.5,9.5)--(-0.5,7.5)--(0.5,6.5)--(-0.5,5.5)--(0.5,4.5)--(-0.5,3.5)--(0.5,2.5)--(-0.5,1.5)--(0.5,0.5);
        \draw[thin] (0.5,2.5)--(6.5,8.5);
        \draw[thin] (0.5,4.5)--(3.5,7.5);
        \draw[thin] (0.5,6.5)--(2.5,8.5);
        \draw[thin] (1.5,1.5)--(0.5,2.5);
        \draw[thin] (0.5,4.5)--(2.5,2.5);
        \draw[thin] (0.5,6.5)--(3.5,3.5);
        \draw[thin] (0.5,8.5)--(4.5,4.5);
        \draw[thin] (3.5,7.5)--(5.5,5.5);
        \draw[thin] (4.5,8.5)--(6.5,6.5);

        \draw[red] (-10,10)--(-7.5,7.5)--(-5.5,9.5)--(-3.5,7.5)--(-1.5,9.5)--(0,8)--(1.5,9.5)--(3.5,7.5)--(5.5,9.5)--(7.5,7.5)--(10,10);
        \node[red] at (8,9.5) {$D(\lambda)(x)$};
        \draw[dashed] (-7.5,7.5)--(0,0)--(.5,.5);
       
        \end{scope}

     \draw[thick,->]  (-8,0) -- (8,0);
     \draw[thick,->]  (0,-1) -- (0,10);
\end{tikzpicture}
\end{center}
 \caption[Shifted Young diagram in Russian coordinates]{On the left, the shifted Young diagram associated to the partition $\lambda=(7,6,3,2)$ represented with russian coordinates. On the right, the piecewise linear function $D(\lambda)(x)$ associated to $\lambda$ pictured in red. The maxima and minima are $x_{-2}=-7.5$, $y_{-2}=-5.5,$ $x_{-1}=-3.5,$ $y_{-1}=-1.5,$ $x_0=0,$ $y_1=1.5,$ $x_1=3.5,$ $y_2=5.5,$ $x_2=7.5$.}\label{fig: shifted tableau in russian}
\end{figure}
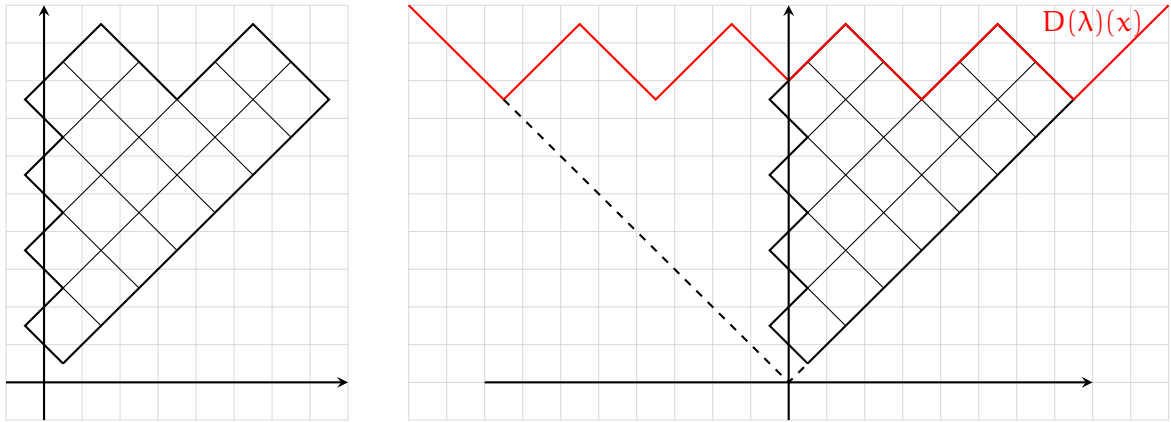

\begin{defi}
 Let $\lambda\in DP$, we call $\phi(z;\lambda)$ the \emph{generating function }of $\lambda$ defined as
 \[\phi(z;\lambda)=\prod_{i=1}^\infty\frac{z+\lambda_i}{z-\lambda_i},\qquad z\in\C.\] 
\end{defi}
\begin{defi}
 Let $k=1,3,5,\ldots$ and define the following function on $DP$:
 \[\tilde{p}^\sharp_k(\lambda):=\left\{\begin{array}{cc}n^{\downarrow k}2^{\frac{k-1}{2}}\frac{\varkappa^\lambda_{(k,1^{n-k})}}{\varkappa^\lambda_{(1^n)}}&\mbox{ if }n=|\lambda|\geq k\\0&\mbox{ otherwise,}\end{array}\right.\]
 where $n^{\downarrow k}$ is the falling factorial and $(k,1^{n-k})=(k,1,\ldots,1)\in OP_n$. 
\end{defi}
Note that, if $|\lambda|\geq k$, then
\[\tilde{p}^\sharp_k(\lambda)=n^{\downarrow k}\frac{\langle P_\lambda,p_{(k,1^{n-k})}\rangle}{\langle P_\lambda,p_{(1^n)}\rangle},\]
which is the definition which can be found, for example, in \cite{matsumoto2015polynomiality} or \cite{ivanov2006plancherel}.

We recall two results of Ivanov, respectively Propositions 2.6 and 3.3 in \cite{ivanov2006plancherel}. Note that our presentation is slightly different, since we are dealing with double diagrams.
\begin{proposition}\label{prop: phi with local extrema}
 Let $\lambda\in DP$ and let $y_{-m},\ldots,y_{-1},y_1,\ldots y_m$ and $x_{-m},\ldots,x_{-1},x_0,x_1,\ldots x_m$ be the local extrema of $D(\lambda)$. Then
 \[\frac{\phi(z-\frac{1}{2};\lambda)}{\phi(z+\frac{1}{2};\lambda)}=\prod_{\substack{i=-m\\i\neq 0}}^m\frac{z-y_i}{z-x_i}=\prod_{i=1}^m\frac{z^2-y_i^2}{z^2-x_i^2}, \]
 where the second equality is trivial by symmetry.
\end{proposition}
\begin{proposition}\label{prop: fundamental of ivanov}
Let $\lambda\in DP$ and $k=1,3,5,\ldots$. Set 
\[\psi_k(z;\lambda):=(2z-k)(z-1)^{\downarrow (k-1)}\frac{\phi(z;\lambda)}{\phi(z-k;\lambda)},\]
then $\tilde{p}^\sharp_k(\lambda)$ is the coefficient of $z^{-1}$ of the expansion of $-\frac{1}{4}\psi_k(z;\lambda)$ in descending powers of $z$ about the point $z=\infty$. That is,
\[\tilde{p}^\sharp_k(\lambda)=\Res(-\frac{1}{4k}\psi_k(z;\lambda);z=\infty),\]
where $\Res$ is the residue.
\end{proposition}

 We conclude with an explicit formula for $\tilde{p}^\sharp_k(\lambda)$.
\begin{proposition}\label{prop: before multirectangular}
 The following holds:
 \[p^{\sharp}_k( \lambda_1,\ldots, \lambda_l)=\sum_{i=1}^l  \lambda_i^{\downarrow k}\prod_{\substack{j=1\\j\neq i}}^l\frac{( \lambda_i- \lambda_j-k)( \lambda_i+ \lambda_j)}{( \lambda_i- \lambda_j)( \lambda_i+ \lambda_j-k)}.\]
\end{proposition}
\begin{proof}
From the definition of $\phi$,
\[\frac{\phi(z;\lambda)}{\phi(z-k;\lambda)}=\prod_{i=1}^l\frac{z+\lambda_i}{z-\lambda_i}\cdot\frac{z-\lambda_i-k}{z+\lambda_i-k}.\]
Hence 
\[  \psi_k(z;\lambda)=(2z-k)(z-1)^{\downarrow (k-1)}\prod_{i=1}^l\frac{(z- \lambda_i-k)(z+ \lambda_i)}{(z- \lambda_i)(z+ \lambda_i-k)}.\]
 By considering the $ \lambda_i$ as formal variables, the function $\psi_k(z; \lambda)$ has $2l$ simple poles, precisely $ \lambda_1,\ldots, \lambda_l,k- \lambda_1,\ldots,k- \lambda_l$. Hence
 \begin{equation}\label{eq:3}
\Res(\psi_k(z;  \lambda);z=\infty)=\sum_{i=1}^l\left(\Res(\psi_k(z;  \lambda);z= \lambda_i)+\Res(\psi_k(z;  \lambda);z=k- \lambda_i)\right).
   \end{equation}
  We evaluate the residues separately: for $z= \lambda_i$ we have 
  \begin{align*}
   \Res(\psi_k(z;  \lambda);z= \lambda_i)&= (2 \lambda_i-k)( \lambda_i-1)^{\downarrow (k-1)}\prod_{\substack{j=1\\j\neq i}}^l\frac{( \lambda_i- \lambda_j-k)( \lambda_i+ \lambda_j)}{( \lambda_i- \lambda_j)( \lambda_i+ \lambda_j-k)}\left(-\frac{2k \lambda_i}{2 \lambda_i-k}\right)\\
   &=-2k\cdot  \lambda_i^{\downarrow k}\prod_{\substack{j=1\\j\neq i}}^l\frac{( \lambda_i- \lambda_j-k)( \lambda_i+ \lambda_j)}{( \lambda_i- \lambda_j)( \lambda_i+ \lambda_j-k)};
  \end{align*}
while for $z=k- \lambda_i$,
\begin{align*}
 \Res(\psi_k(z;  \lambda);z=k- \lambda_i)&= (k-2 \lambda_i)(k- \lambda_i-1)^{\downarrow (k-1)}\times\\&\hspace{3cm}\times\prod_{\substack{j=1\\j\neq i}}^l\frac{(- \lambda_i- \lambda_j)(k- \lambda_i+ \lambda_j)}{(k- \lambda_i- \lambda_j)(- \lambda_i+ \lambda_j)}\left(-\frac{2k \lambda_i}{k-2 \lambda_i}\right)\\
   &=-2k\cdot  \lambda_i^{\downarrow k}\prod_{\substack{j=1\\j\neq i}}^l\frac{( \lambda_i- \lambda_j-k)( \lambda_i+ \lambda_j)}{( \lambda_i- \lambda_j)( \lambda_i+ \lambda_j-k)}
\end{align*}
By substituting into \eqref{eq:3} we obtain
\[\Res(\psi_k(z;  \lambda);z=\infty)=-4k\sum_{i=1}^l  \lambda_i^{\downarrow k}\prod_{\substack{j=1\\j\neq i}}^l\frac{( \lambda_i- \lambda_j-k)( \lambda_i+ \lambda_j)}{( \lambda_i- \lambda_j)( \lambda_i+ \lambda_j-k)}\]
therefore, from the previous proposition,
\[p^{\sharp}_k( \lambda_1,\ldots, \lambda_l)=\sum_{i=1}^l  \lambda_i^{\downarrow k}\prod_{\substack{j=1\\j\neq i}}^l\frac{( \lambda_i- \lambda_j-k)( \lambda_i+ \lambda_j)}{( \lambda_i- \lambda_j)( \lambda_i+ \lambda_j-k)},\]
and we conclude.
\end{proof}

\section{Multirectangular coordinates}\label{section: multirectangular coordinates}
We adapt the ideas of Stanley \cite{stanley2003irreducible} and Rattan \cite{rattan2007positivity} for the strict partitions case. In multirectangular coordinates, a strict partition $\lambda=(\lambda_1,\lambda_2,\ldots,\lambda_l)$ is written as $\mathbf{p}\times\mathbf{q}$, where $\q=(q_1,q_2,\ldots,q_m)$ and $\p=(p_1,p_2,\ldots,p_m)$, and 
\[(\lambda_1,\lambda_2,\ldots,\lambda_l)=(q_1,q_1-1,\ldots,q_1-p_1+1,q_2,q_2-1,\ldots q_2-p_2+1,\ldots,q_m-p_m+1).\]
By abuse of notation we call a partition \emph{rectangular}, and write it $p\times q$, if $m=1$. Notice that $q_i\geq p_i$ and $q_{i+1}\leq q_i-p_i$ for all $i$. See Figure \ref{picture rectangular strict partitions} for an example. The local extrema of the double diagram $D(\lambda)$ have an easy description in terms of the rectangular coordinates: for $i=1,\ldots,m$
\[y_i=q_{m-i+1}-p_{m-i+1}+\frac{1}{2},\qquad y_{-i}=-y_i,\]
\[x_i=q_{m-i+1}+\frac{1}{2},\qquad x_0=0,\qquad x_{-i}=-x_i.\]

\begin{figure}
\begin{center}

\begin{tikzpicture}
 \draw (0,4)--(0,3)--(1,3)--(1,2)--(2,2)--(2,1)--(3,1)--(3,0)--(6,0)--(6,4)--(0,4);
 \draw [<->](0.2,4.2) -- node[anchor=south] {$q$}		(5.8,4.2);
 \draw [<->](6.2,0.2) -- node[anchor=west]  {$p$}		(6.2,3.8);
\end{tikzpicture}
\begin{tikzpicture}
 \draw (0,4)--(0,3)--(1,3)--(1,2)--(2,2)--(2,1)--(3,1)--(3,0)--(5,0)--(5,2)--(7,2)--(7,4)--(0,4);
 \draw [<->](0.2,4.2) -- node[anchor=south] {$q_1$}		(6.8,4.2);
 \draw [<->](2.2,2.2) -- node[anchor=south] {$q_2$}		(4.8,2.2);
 \draw [<->](7.2,2.2) -- node[anchor=west]  {$p_1$}		(7.2,3.8);
 \draw [<->](7.2,0.2) -- node[anchor=west]  {$p_2$}		(7.2,1.8);
\end{tikzpicture}\end{center}
\caption[Shifted Young diagram in multirectangular coordinates]{On the left, an example of a rectangular strict partition $\lambda=(6,5,4,3)=6\times 4$. On the right we have the partition $(7,6,3,2)=(7,3)\times(2,2)$.}\label{picture rectangular strict partitions}
\end{figure}
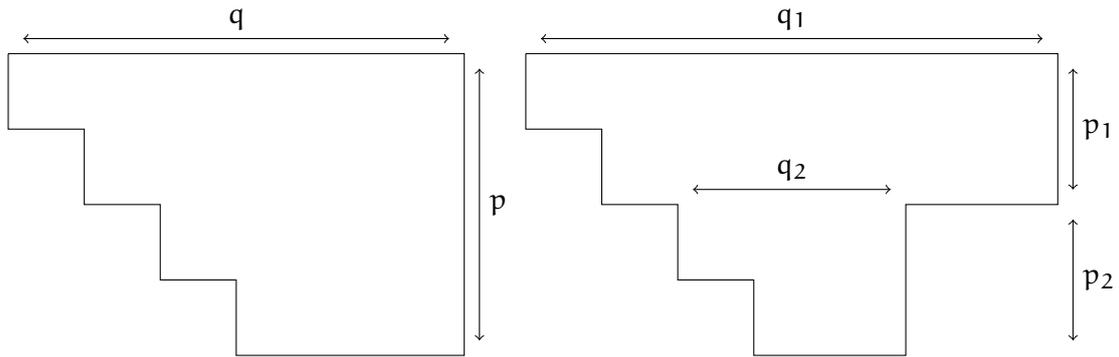

\begin{proposition}\label{proposizione iniziale}
Let $\lambda=\p\times \q$ be a strict partition with $\p=(p_1,\ldots, p_m)$ and $\q=(q_1,\dots,q_m)$. Then 
\[p_k^{\sharp}(\p\times\q)=-\frac{1}{4k}[z^{-1}]\left((2z-k)(z-1)^{\downarrow (k-1)}\prod_{i=1}^m\frac{(z-q_i-1)^{\downarrow k}(z+q_i)^{\downarrow k}}{(z-q_i+p_i-1)^{\downarrow k}(z+q_i-p_i)^{\downarrow k}}\right).\]
\end{proposition}
\begin{proof}
We substitute the rectangular coordinates into Proposition \ref{prop: phi with local extrema}, so that
 \[\frac{\phi(z-\frac{1}{2},\p\times\q)}{\phi(z+\frac{1}{2},\p\times\q)}=\prod_{i=1}^m\frac{(z-q_i-\frac{1}{2}+p_i)(z+q_i+\frac{1}{2}-p_i)}{(z-q_i-\frac{1}{2})(z+q_i+\frac{1}{2})}.\]
Hence
 \[\frac{\phi(z,\p\times\q)}{\phi(z-1,\p\times\q)}=\prod_{i=1}^m\frac{(z+q_i)(z-q_i-1)}{(z+q_i-p_i)(z-q_i+p_i-1)},\]
 and 
  \[\frac{\phi(z,\p\times\q)}{\phi(z-k,\p\times\q)}=\prod_{i=1}^m\frac{(z+q_i)^{\downarrow k}(z-q_i-1)^{\downarrow k}}{(z+q_i-p_i)^{\downarrow k}(z-q_i+p_i-1)^{\downarrow k}}.\]
 To conclude, we substitute the last formula in Proposition \ref{prop: fundamental of ivanov}. 
\end{proof}

\begin{proposition}\label{prop: for k even}
 For $k$ even, the right hand side of the formula in Proposition \ref{proposizione iniziale} is equal to $0$.
\end{proposition}
\begin{proof}
 Set 
 \[\psi_k(z;\p\times\q):=(2z-k)(z-1)^{\downarrow k-1}\prod_{i=1}^m\frac{(z-q_i-1)^{\downarrow k}(z+q_i)^{\downarrow k}}{(z-q_i+p_i-1)^{\downarrow k}(z+q_i-p_i)^{\downarrow k}}. \]
It is enough to prove that 
\begin{equation}\label{eq:1}
 (-1)^{k}\psi_k(z;\p\times\q)=\psi_k(-z+k;\p\times\q).
\end{equation}
Indeed $[z^{-1}]\psi_k(z;\p\times\q)=\Res(\psi_k(z;\p\times\q);z=\infty)=-\Res(\psi_k(-z+k;\p\times\q);z=\infty)$, and Equation \eqref{eq:1} would imply $(-1)^{k}[z^{-1}]\psi_k(z,\p\times\q)=-[z^{-1}]\psi_k(z,\p\times\q)=0$ for $k$ even.

\smallskip
Let us prove Equation \eqref{eq:1}:
\begin{equation}\label{eq: symmetry phi}
 \frac{\phi(-z+1,\p\times\q)}{\phi(-z,\p\times\q)}=\prod_{i=1}^m\frac{(-z+q_i+1)(-z-q_i)}{(-z+q_i-p_i+1)(-z-q_i+p_i)}=\frac{\phi(z,\p\times\q)}{\phi(z-1,\p\times\q)},
\end{equation}
hence 
\[\frac{\phi(z,\p\times\q)}{\phi(z-k,\p\times\q)}=\frac{\phi(-z+k,\p\times\q)}{\phi(-z,\p\times\q)}.\]
Therefore 
\[\psi_k(-z+k;\p\times\q)=(-2z+k)(-z+k-1)^{\downarrow (k-1)}\frac{\phi(z,\p\times\q)}{\phi(z-k,\p\times\q)}=(-1)^{k}\psi_k(z;\p\times\q),\]
where we used the equality $(-z+k-1)^{\downarrow (k-1)}=(-1)^{k-1}(z-1)^{\downarrow (k-1)}$. This concludes the proof of the claim and the proposition. 
\end{proof}
The next proposition is the projective version of \cite[Proposition 1]{stanley2003irreducible}. Let $\lambda=\p\times \q$, we define 
 \[F_k(p_1,\ldots,p_m;q_1,\ldots,q_m):=p^{\sharp}_k(\lambda).\]
\begin{proposition}
 The function $F_k(p_1,\ldots,p_m;q_1,\ldots,q_m)$ is a polynomial function of the $p_i$'s and $q_i$'s whose coefficients belong to $\Z/2$.
\end{proposition}
\begin{proof}
 Recall that 
 \[\frac{1}{z-a}=\frac{1}{z}+\frac{a}{z^2}+\frac{a^2}{z^3}+\ldots,\]
 hence the coefficient of $z^{-1}$ of 
 \[\psi_k(z;\p\times\q):=(2z-k)(z-1)^{\downarrow k-1}\prod_{i=1}^m\frac{(z-q_i-1)^{\downarrow k}(z+q_i)^{\downarrow k}}{(z-q_i+p_i-1)^{\downarrow k}(z+q_i-p_i)^{\downarrow k}}\]
 is a polynomial in the variables $p_1,\ldots,p_m,q_1,\ldots,q_m$ with integer coefficients, so that the coefficients of $F_k(p_1,\ldots,p_m;q_1,\ldots,q_m)$ are in $\Z/4k$.
 
 \smallskip
 We prove that the coefficients of $F_k(p_1,\ldots,p_m;q_1,\ldots,q_m)$ belong to $\Z/4$; we have to show that the coefficients of the polynomial 
 \begin{equation}\label{eq: on the parity}
  [z^{-1}]\left((2z-k)(z-1)^{\downarrow k}\frac{\phi(z;\lambda)}{\phi(z-k;\lambda)}\right)
 \end{equation}
 are divisible by $k$. But
 \[(2z-k)(z-1)^{\downarrow k}\frac{\phi(z;\lambda)}{\phi(z-k;\lambda)}\equiv (2z-k)(z-1)^{\downarrow k} \mod k\]
 and
\[[z^{-1}]\left((2z-k)(z-1)^{\downarrow k}\right)=0.\]
It remains to prove that the coefficients of $F_k(p_1,\ldots,p_m;q_1,\ldots,q_m)$ are in $\Z/2$ or, equivalently the coefficients of the polynomial \eqref{eq: on the parity} are divisible by $2$. Recall that since $k$ is odd proving that the coefficients of \eqref{eq: on the parity} are divisible by $2$ and $k$ is equivalent to proving that they are divisible by $2k$.
From Equation \eqref{eq: symmetry phi} we deduce that 
\[ \frac{\phi(z-1,\p\times\q)}{\phi(z,\p\times\q)}\equiv \frac{\phi(z,\p\times\q)}{\phi(z-1,\p\times\q)} \mod 2,\]
which implies 
\[\phi(z-1,\p\times\q)^2\equiv \phi(z,\p\times\q)^2\mod 2,\]
or, equivalently,
\[\phi(z-1,\p\times\q)\equiv \phi(z,\p\times\q)\mod 2.\]
Therefore 
\begin{align*}
 (2z-k)(z-1)^{\downarrow k}\frac{\phi(z;\lambda)}{\phi(z-k;\lambda)}&\equiv (z-1)^{\downarrow k-1}\frac{\phi(z;\lambda)}{\phi(z-1;\lambda)} \mod 2\\
 &\equiv (z-1)^{\downarrow k-1},
\end{align*}
and thus $\psi_k(z;\p\times\q)$ is congruent to a polynomial modulo $ 2$, and therefore \[[z^{-1}](\psi_k(z;\p\times\q)\equiv 0\mod 2.\]
 \end{proof}

\begin{example} For rectangular strict partitions ($\lambda=p\times q$):
 \begin{itemize}
  \item $p^{\sharp}_1(p\times q)=-\frac{1}{2}p^2 + \frac{1}{2}p(2q + 1)$;
  \item $p^{\sharp}_3(p\times q)=-p^4 + 2p^3(2q + 1) - \frac{1}{2}(9q^2 + 9q + 4)p^2 +  \frac{1}{2}(2q^3 + 3q^2 + 5q + 2)p$;
  \item $p^{\sharp}_5(p\times q)=-7/2p^6 + \frac{21}{2}p^5(2q + 1) - 5/2(18q^2 + 18q + 11)p^4 +  5/2(16q^3 + 24q^2 + 38q + 15)p^3 - \frac{1}{2}(25q^4 + 50q^3 + 185q^2 + 160q + 58)p^2 + \frac{1}{2}(2q^5 + 5q^4 + 40q^3 + 55q^2 + 66q + 24)p$.
 
 \end{itemize}
\end{example}
\begin{example} For double rectangular strict partitions ($\lambda=(p_1,p_2)\times (q_1,q_2)$):
 \begin{itemize}
  \item $p^{\sharp}_1(\p\times \q)=-\frac{1}{2}p_1^2 - \frac{1}{2}p_2^2 + p_1q_1 + p_2q_2 + \frac{1}{2}p_1 + \frac{1}{2}p_2$;
  \item $p^{\sharp}_3(\p\times \q)=-p_1^4 - p_2^4 + p_1q_1^3 + p_2q_2^3 + 2p_1^3 - \frac{1}{2}(3p_1^2 - 3p_1 + 4)p_2^2 + 2p_2^3 - \frac{3}{2}(3p_1^2 - p_1)q_1^2 - \frac{3}{2}(3p_2^2 - p_2)q_2^2 - 2p_1^2 + \frac{1}{2}(3p_1^2 - 3p_1 + 2)p_2 + \frac{1}{2}(8p_1^3 + 6p_1p_2^2 - 9p_1^2 - 6p_1p_2 + 5p_1)q_1 + \frac{1}{2}(8p_2^3 - 12p_1p_2q_1 + (6p_1^2 - 6p_1 + 5)p_2 - 9p_2^2)q_2 + p_1.$
 \end{itemize}

\end{example}
We will consider the polynomial $F_k(-p_1,\ldots,-p_m;q_1,\ldots,q_m)$, that is, the polynomial $F_k(p_1,\ldots,p_m;q_1,\ldots,q_m)$ where every variable $p_i$ is substituted by $-p_i$. We formulate a positivity conjecture involving the coefficients. In the classical case a similar conjecture was formulated by Stanley in \cite{stanley2003irreducible} and later generalized to a combinatorial formula in \cite{Stanley-preprint2006}. It was proved by F\'eray in \cite{F'eray2008} (see also the article \cite{feray2007asymptotics} of F\'eray and {\'S}niady).
\begin{conjecture}
 The coefficients of the polynomial $-F_k(-p_1,\ldots,-p_m;q_1,\ldots,q_m)$ are nonnegative.
\end{conjecture}
Computer calculations show the conjecture true for $k\leq 15$ for the rectangular case and $k\leq 9$ for the birectangular case.
\medskip

\subsection{The case \texorpdfstring{$q=p$}{Lg}}
\begin{proposition}
 Consider $k$ odd and set $k=2j-1$. By setting $q=p$ we obtain 
 \[F_k(p;p)=p^{\sharp}_{2j-1}(p\times p)=\frac{(-1)^j}{2}\cat(j-1)(p+j)^{\downarrow (2j)}.\]
\end{proposition}
We prove the proposition after the following lemma:
\begin{lemma}
 Consider $x$ a formal variable, then for any positive integer $j$ the following holds:
 \begin{multline*}
  -\frac{1}{4(2j-1)}\sum_{i=0}^{2j-1}(-1)^i\frac{2i-2j+1}{i!(2j-1-i)!}(i-x-1)^{\downarrow (2j-1)}(i+x)^{\downarrow(2j-1)}=\\\frac{(-1)^j}{2}\cat(j-1)(x+j)^{\downarrow (2j)}.
 \end{multline*}

\end{lemma}
\begin{proof}
 This can be done with a Maple program: evaluating the sum the program gives (after simplification) 
 \begin{multline*}
  -\frac{1}{4(2j-1)}\sum_{i=0}^{2j-1}(-1)^i\frac{2i-2j+1}{i!(2j-1-i)!}(i-x-1)^{\downarrow (2j-1)}(i+x)^{\downarrow(2j-1)}\\=-\frac{1}{8}\,{\frac {{4}^{j}\sin \left( \pi\,x \right) \Gamma \left( j+1+x \right) \Gamma \left( j-x \right) \Gamma \left( j-1/2 \right) }{{\pi}^{3/2}
\mbox{}\Gamma \left( j+1 \right) }},
 \end{multline*}
which gives the desired result once considered that 
\[\cat(j-1)=4^{j-1}\frac{\Gamma(j-1/2)}{\sqrt{\pi}\Gamma(j+1)};\qquad \Gamma(1-x)\Gamma(x)=\frac{\pi}{\sin(\pi x)}.\qedhere\]
\end{proof}

\begin{proof}[Proof of the proposition]
 We substitute $q=p$ in Proposition \ref{proposizione iniziale} and, after simplification, we obtain
 \[p^{\sharp}_{2j-1}(p\times p)=-\frac{1}{4(2j-1)}[z^{-1}]\left((2z-2j+1)\frac{(z-p-1)^{\downarrow(2j-1)}(z+p)^{\downarrow (2j-1)}}{z^{\downarrow (2j)}}\right).\]
 Set as before $ \psi_{2j-1}(z;p\times p)=(2z-2j+1)\frac{(z-p-1)^{\downarrow(2j-1)}(z+p)^{\downarrow (2j-1)}}{z^{\downarrow (2j)}},$ we notice that $\psi_{2j-1}$ has simple poles at the points $0,1,\ldots, 2j-1$. Hence
 \begin{align*}
  \Res(\psi_{2j-1};z=\infty)&=\sum_{i=0}^{2j-1}\Res(\psi_{2j-1};z=i)\\
  &=\sum_{i=0}^{2j-1}\lim_{z\to i}(2z-2j+1)\frac{(z-p-1)^{\downarrow(2j-1)}(z+p)^{\downarrow (2j-1)}}{\prod\limits_{\substack{0\leq l\leq 2j-1\\l\neq i}}(z-l)}\\
  &=\sum_{i=0}^{2j-1}(-1)^i\frac{2i-2j+1}{i!(2j-1-i)!}(i-p-1)^{\downarrow (2j-1)}(i+p)^{\downarrow(2j-1)},
 \end{align*}
and we can conclude with the previous lemma.
\end{proof}

\section{Study of the main term}\label{section: study main term}
Recall the Cauchy transform and the free cumulants of a probability measure $\mu$ on $\R$, respectively
\[C_{\mu}(z)=\int_{\R}\frac{1}{z-x}\mu(dx),\qquad R_k(\mu)=-\frac{1}{k-1}[z^{-1}]\frac{1}{C_{\mu}(z)^{k-1}}.\]
Moreover, recall from Section \ref{sec: the transition measure} that the Kerov transform of a diagram $\omega$ on the interval $[a,b]\subset\R$ is
\[K_{\omega}(z)=\frac{1}{z}\exp\frac{1}{2}\int_a^b\frac{d(\omega(t)-|t|)}{t-z}\qquad\mbox{ for each }z\in \C\setminus [a,b].\]
Krein and Nudelman showed in \cite{krein1977markov} that for each diagram $\omega$ there exists a measure $\mu$, called the \emph{transition measure} of $\omega$, such that $K_\omega(z)=C_\mu(z)$. As showed in Section \ref{sec: the transition measure}, if $\omega=\lambda(x)$ is a Young diagram pictured in russian coordinates, then $\mu$ is exactly the transition measure associated to $\lambda$.

The measure $\mu$ is unique and supported on the interval $[a,b]$. In \cite[Section 2.2]{kerov1993transition}, Kerov proves that that if $\omega$ is a rectangular diagram then the measure $\mu$ is discrete. If $\omega=D(\lambda)$ is the double diagram of a strict partition $\lambda$, then we call $\mathfrak{m}_{D(\lambda)}$ the corresponding transition measure; set $\{y_i\},\{x_i\}$ to be the local extrema of $D(\lambda)$, then $\mathfrak{m}_{D(\lambda)}$ has weight $\mu_i$ on the point $x_i$, with 
\[K_{D(\lambda)}(z)=\frac{\prod\limits_{i=1}^m(z^2-y_i^2)}{z\prod\limits_{i=1}^m(z^2-x_i^2)}=\sum_{i=-m}^m\frac{\mu_i}{z-x_i}=C_{\mathfrak{m}_{D(\lambda)}(z)}.\]
The weights can be formulated explicitly (\cite[Equation (2.4.2)]{kerov1993transition}):
 \[\mu_i=\frac{\prod\limits_{k=1}^m(x_i^2-y_k^2)}{2x_i^2\prod\limits_{\substack{1\leq k\leq m\\k\neq i}}(x_i^2-x_k^2)},\qquad \mu_{-i}=\mu_i\]
 if $i\neq 0$, and 
 \[\mu_0=\prod_{k=1}^m\frac{y_k^2}{x_k^2},\]
otherwise.

From Proposition \ref{prop: phi with local extrema} we notice that 
\[K_{D(\lambda)}(z)=\frac{1}{z}\frac{\phi(z-\frac{1}{2};\lambda)}{\phi(z+\frac{1}{2};\lambda)}.\]
In multirectangular coordinates $\lambda=\p\times\q$ it becomes
\[K_{\mathfrak{m}_{D(\p\times\q)}}(z,\p\times\q)=\frac{1}{z}\frac{\phi(z-\frac{1}{2};\p\times\q)}{\phi(z+\frac{1}{2};\p\times\q)}=\frac{1}{z}\prod_{i=1}^m\frac{(z+q_i-p_i+\frac{1}{2})(z-q_i+p_i-\frac{1}{2})}{(z+q_i+\frac{1}{2})(z-q_i-\frac{1}{2})}\]
In this section our goal is to study the leading terms of $F_k(p_1,\ldots,p_m;q_1,\ldots,q_m)$ seen as a polynomial in $p_1,\ldots,p_m,q_1,\ldots,q_m$, as Stanley did in \cite[Proposition 2]{stanley2003irreducible} for the classical case. Call $L_k(\p\times\q)$ these leading terms, which are homogeneous polynomials of degree $k+1$. Set
\begin{equation}\label{half free cumulants}
 R_k(\lambda)=\frac{1}{2}R_k(\mathfrak{m}_{D(\lambda)}).
\end{equation}
In a private communication Sho Matsumoto conjectured that, for $k$ odd, the function $\mathfrak{p}^{\sharp}_k$ can be written as a polynomial in terms of $R_2,$ $R_4,\ldots,$ $R_{k+1}$. Moreover 
\[\mathfrak{p}^{\sharp}_k=R_{k+1}+\mbox{ a polynomial in }R_2,\ldots R_{k-1},\]
and the coefficients of this polynomial are nonnegative integers. A similar result is known for the classical case was conjectured by Kerov, and proven in \cite{feray2010stanley}. In the following proposition we prove that the main term of $\mathfrak{p}^{\sharp}_k$ is indeed $R_{k+1}$. We state and prove our result in the multirectangular setting.
\begin{proposition}\label{prop: leading term}
 For $k$ odd we have that the leading term of $F_k(p_1,\ldots,p_m;q_1,\ldots,q_m)$, called $L_k(\p\times\q)$, is also the leading term of $R_{k+1}(\p\times\q)$, where $R_{k+1}(\p\times\q)$ is the normalized free cumulant evaluated on $\lambda=\p\times\q$. Moreover
  \[L(z;\p\times\q):=\frac{1}{z}+\sum_{\substack{k\geq 0\\k\mbox{ odd}}}L_k(\p\times\q)z^k=\frac{1}{2\left(z\prod\limits_{i=1}^m \frac{(1-z(q_i-p_i))(1+z(q_i-p_i))}{(1-zq_i)(1+zq_i)}\right)^{\langle-1\rangle}}.\]
\end{proposition}
\begin{proof}
 From Proposition \ref{proposizione iniziale} one deduces that the maximal term of $F_k(\p;\q)$  is
\[L_k(\p\times\q)=-\frac{1}{4k}[z^{-1}]\left(2\cdot z^k\prod_{i=1}^m\frac{(z-q_i)^k(z+q_i)^k}{(z-q_i+p_i)^k(z+q_i-p_i)^k}\right);\]
on the other hand this is also the maximal term of 
\begin{align*}
R_{k+1}(\p\times\q)&=-\frac{1}{2k}[z^{-1}]\left(\frac{1}{K_{\mathfrak{m}_{D(\lambda)}}(z,\p\times\q)^k}\right)\\
&=-\frac{1}{2k}[z^{-1}]\left( z^k\prod_{i=1}^m\frac{(z-q_i-\frac{1}{2})^k(z+q_i+\frac{1}{2})^k}{(z-q_i+p_i-\frac{1}{2})^k(z+q_i-p_i+\frac{1}{2})^k}\right).
\end{align*}
Set 
\[M(z;\p\times\q)=\prod_{i=1}^{m}\frac{(1-zq_i)(1+zq_i)}{(1-z(q_i-p_i))(1+z(q_i-p_i))},\]
 Then by the Lagrange inversion formula (\cite[Theorem 5.4.2]{Stanley:EC2})
 \[[z^{-1}]\left(2\cdot z^k\prod_{i=1}^m\frac{(z-q_i)^k(z+q_i)^k}{(z-q_i+p_i)^k(z+q_i-p_i)^k}\right)=-k[z^k]\frac{1}{\left(z/M(z;\p\times\q)\right)^{\langle -1\rangle}},\]
which concludes the proof. 
\end{proof}

\begin{theorem}
 Let $k$ be odd. The leading term $-L_k(-\p\times\q)$ of $-F_k(-p_1,\ldots,-p_m;q_1,\ldots,q_m)$ has positive coefficients.
\end{theorem}
\begin{proof}
We have seen that 
 \begin{align*}
  L_k(\p\times\q)&=-\frac{1}{2k}[z^{-1}]\left( z^k\prod_{i=1}^m\frac{(z-q_i)^k(z+q_i)^k}{(z-q_i+p_i)^k(z+q_i-p_i)^k}\right)\\
  &=\frac{1}{2k}[z^{-1-k}]\left(\prod_{i=1}^{m}\left(\frac{z^2-q_i^2}{z^2-(q_i+p_i)^2}\right)^k\right)\\
 &=\frac{1}{2k}[z^{-1-k}]\left(\prod_{i=1}^{m}\left(1+(p_i^2+2q_ip_i)\sum_{j\geq 1}\frac{(q_i+p_i)^{2j-2}}{z^{2j}}\right)^k\right).
 \end{align*}
It is clear that the coefficients of $z$ in 
\[\prod_{i=1}^{m}\left(1+(p_i^2+2q_ip_i)\sum_{j\geq 1}\frac{(q_i+p_i)^{2j-2}}{z^{2j}}\right)^k\]
are positive (seen as polynomials in $p_i,q_i$) for all powers of $z$, which concludes the proof. 
\end{proof}

\section{Asymptotics results for strict partitions}\label{section: asymptotics}
In \cite{ivanov2006plancherel} Ivanov proved some asymptotic results of the diagrams
\[\bar{\lambda}(x)=\frac{1}{\sqrt{n}}\lambda(\sqrt{n}x)\]
for strict Plancherel distributed $\lambda\in DP_n$. We recall and extend his results in terms of the double diagram $D(\lambda)$ and its renormalization
\[\overline{D(\lambda)}(x):=\frac{1}{\sqrt{n}}D(\lambda)(\sqrt{n}x).\]

Set
\[\Omega(x)=\left\{\begin{array}{ll}\frac{2}{\pi}(x\arcsin\frac{x}{2}+\sqrt{4-x^2}),&-2\leq x\leq 2\\|x|,&|x|>2, \end{array}\right.\]
 
and
\[\Delta_{\lambda}(x)=\frac{\sqrt{2|\lambda|}}{2}( \overline{D(\lambda)}(x)-\Omega(x)), \mbox{ for }x\in \R\mbox{ and }\lambda\in DP_n.\]
We consider $\Delta$ a random function $\Delta^{(n)}(x)$ on the probability space $(DP_n,P_{strict}^n)$. Consider a modified version of Chebyshev polynomials of the second kind:
\[u_k(x):=U_k(x/2)=\sum_{j=0}^{\lfloor k/2\rfloor}(-1)^j\binom{k-j}{j}x^{k-2j},\mbox{ for }k=0,1,\ldots.\]

\begin{theorem}[Central limit theorem for Young diagrams]
 For $k=1,2\ldots$ consider
 \[u_{k}^{(n)}:=\int_{\R}u_{k}(x)\Delta^{(n)}(x)\,dx,\]
 and let $\xi_1,\xi_2,\ldots$ be independent standard Gaussian variables. Then for $n\to\infty$
 \[\{u_{2k}^{(n)}\}_{k\geq 1}\overset{d}\to\left\{\frac{\xi_k}{\sqrt{2(2k+1)}}\right\}_{k\geq 1}\]
 and 
 \[u_{2k+1}^{(n)}= 0\mbox{ for all }k.\]
\end{theorem}
Notice that this theorem and \cite[Theorem 6.1]{ivanov2006plancherel} are equivalent since both $\overline{D(\lambda)}$ and $\Omega$ are even functions. To be more precise, Ivanov proves convergence of the moments of the random variable $u_k^{(n)}$, where the $l-$th moment of $u^{(n)}_k$ is
\[\mathbb{E}_{P_{strict}}^n[(u_k^{(n)})^l]=\sum\limits_{\lambda\in DP_n}P_{strict}^n(\lambda)\left(\int_{\R}u_{k}(x)\Delta_{\lambda}(x)\,dx\right)^l.\]
The fact that convergence of moments (together with some additional requirements) implies convergence in distribution is known as the moment method, see \cite[Proposition 6.2]{ivanov2006plancherel}.
\begin{corollary}[Law of large numbers, 1st form]\label{corol: law of large numbers 1}
  Let $\lambda\in DP_n$ be distributed with the strict Plancherel measure. Consider the border of the diagram $ \overline{D(\lambda)}$ as a piece-wise linear function $ \overline{D(\lambda)}(x)$ for $x\in\R$. Then for any $k=0,1,\ldots$
  \[\lim_{n\to\infty}\int_{\R}(\overline{D(\lambda)}(x)-\Omega(x))x^k\,dx=0\mbox{ almost surely,} \]
  where the probability space is considered to be the space of infinite paths on the Bratteli diagram of strict partitions defined by the system of strict Plancherel measures.
\end{corollary}
\begin{proof}
Call $I_k^{(n)}$ the random function that to $\lambda\in DP_n$ associates 
\[I_k^{\lambda}:=\int_{\R}(\overline{D(\lambda)}(x)-\Omega(x))x^k\,dx.\]
Then $\sqrt{n}I_k^{(n)}=\sqrt{2}\int_{\R}\Delta^{(n)}(x)\,dx$ can be written as a linear combination of $u_k^{(n)}$, since $\{u_k(x)\}_{k\geq 1}$ is a basis for $\R[x]$. In particular Ivanov's result implies that the moments of $I_k^{(n)}$ converge. We consider the fourth moment
\[\mathbb{E}_{P_{strict}}^n[(\sqrt{n}I_k^{(n)})^4]\to c_k\]
as $n\to\infty$ for some constant $c_k$. This implies that $\mathbb{E}_{P_{strict}}^n[(I_k^{(n)})^4]\sim c_k/n^2$, so that $\sum_{n\geq 1}\mathbb{E}_{P_{strict}}^n[(I_k^{(n)})^4]<\infty$. We can now apply the Borel-Cantelli lemma to deduce that $I_k^{(n)}\to0$ almost surely, and the corollary is proved.
\end{proof}
It is not difficult to see that strict Plancherel distributed Young diagrams converge in probability to the limit shape in the uniform topology, but it requires some additional work. Out of completeness, we include it here. This theorem is the strict partition equivalent of \cite[Theorem 5.5]{ivanov2002kerov} for the classic case, and we shall prove it similarly. Notice that in Ivanov-Olshanski's article they show convergence in probability, but with some minor adaptations one can prove almost sure convergence, so we establish directly the improved result.
\begin{theorem}[Law of large numbers, 2nd form]\label{theo: ivanov corollary}
 Let $\lambda\in DP_n$ be distributed with the strict Plancherel measure. Consider the border of the diagram $ \overline{D(\lambda)}$ as a piece-wise linear function $ \overline{D(\lambda)}(x)$ for $x\in\R$. Then
 \[\lim_{n\to\infty}\sup_{x\in\R}| \overline{D(\lambda)}(x)-\Omega(x)|=0\mbox{ almost surely,}\]
 where as before the limit is taken in the space of infinite paths on the Bratteli diagram of strict partitions.
\end{theorem}
We need two lemmas.
\begin{lemma}
 There exists a compact interval $I=[-r,r]\subset \R$ such that
 \[P_{strict}^n(\{\lambda\in DP_n\mbox{ s.t. }\overline{D(\lambda)}(x)>|x|\mbox{ for }x\notin I\})\leq 2Ce^{-c\sqrt{n}}\mbox{ for }n\to\infty,\]
 for some constants $C,c$ that depend on $r$ but not on $n$. In particular from the Borel-Cantelli lemma we have 
 \[\mathbb{1}(\{\lambda\in DP_n\mbox{ s.t. }\overline{D(\lambda)}(x)>|x|\mbox{ for }x\notin I\})\to 0 \mbox{ almost surely,}\]
 where $\mathbb{1}(A)$ is the indicator function for the event $A$.
\end{lemma}
\begin{proof}
Note that 
 \[P_{strict}^n(\{\lambda\in DP_n\mbox{ s.t. }\overline{D(\lambda)}(x)>|x|\mbox{ for }x\notin I\})=P_{strict}^n(\{\lambda\in DP_n\mbox{ s.t. }\lambda_1>r\sqrt{n}\}),\]
 where $\lambda_1$ is the longest part of $\lambda$.
 
 The proof relies on the strict partition version of the RSK correspondence, proved by Sagan in \cite{sagan1987shifted} (see also \cite[Chapter 13]{hoffman1992projective}). Sagan showed that the probability of a uniform random permutation $\pi\in S_n$ to have associated RSK $P$-tableaux of shape $\lambda$ is $P_{strict}^n(\lambda)$. Moreover, the algorithm allows us to easily compute the statistic $\lambda_1$: following Sagan notation, let $\pi=i_1\ldots i_n\in S_n$ be a partition written in one line notation and let $\overline{\pi}=i_n\ldots i_1$. Consider the concatenation $\overline{\pi}\pi=i_n\ldots i_1 i_1\ldots i_n$. For a sequence $s$ of integers let $a_1(s)$ be the length of the longest increasing subsequence of $s$, and $d_1(s)$ the length of the longest decreasing subsequence of $s$. Then Sagan proved in \cite[Corollary 5.2]{sagan1987shifted} that $\lambda_1=a_1(\overline{\pi}\pi)-1$, where $\lambda=(\lambda_1,\ldots,\lambda_m)$ is the shape of the shifted $P$-tableaux corresponding to $\pi$.
 
 It is clear that $a_1(\overline{\pi}\pi)-1\leq a_1(\pi)+d_1(\pi)$. If $P_{uniform}^n$ is the uniform probability in $S_n$ we have 
 \begin{align*}
  P_{strict}^n&(\{\lambda\in DP_n\mbox{ s.t. }\lambda_1>r\sqrt{n}\})=P_{uniform}(\pi\in S_n\mbox{ s.t. }a_1(\overline{\pi}\pi)-1>r\sqrt{n})\\
  &\leq P_{uniform}(\pi\in S_n\mbox{ s.t. }a_1(\pi)+d_1(\pi)>r\sqrt{n})\\
  &\leq P_{uniform}(\pi\in S_n\mbox{ s.t. }a_1(\pi)>\frac{r}{2}\sqrt{n})+P_{uniform}(\pi\in S_n\mbox{ s.t. }d_1(\pi)>\frac{r}{2}\sqrt{n}). 
 \end{align*}
The fact that the last two probabilities are exponentially small for $r>2e$ is a well known result (see for example \cite[Theorem 1.5]{romik2015surprising}). This concludes the lemma.
\end{proof}
The following lemma appears (with proof) in \cite[Lemma 5.7]{ivanov2002kerov}.
\begin{lemma}
 Fix an interval $I=[a,b]\subset \R$, and let $\Sigma$ denote the set of all real valued functions $\sigma(x)$ on $\R$, supported by $I$ and satisfying the Lipschitz condition $|\sigma(x_1)-\sigma(x_2)|\leq |x_1-x_2|$.
 
 \smallskip
 On the set $\Sigma$, the weak topology defined by the functionals 
 \[\sigma\mapsto \int\sigma(x)x^k\, dx,\qquad k=0,1,\ldots\]
 coincides with the uniform topology defined by the supremum norm $\Vert\sigma\Vert=\sup|\sigma(x)|.$ 
\end{lemma}

\begin{proof}[Proof of Theorem \ref{theo: ivanov corollary}]
The proof is now immediate from \ref{corol: law of large numbers 1} and the two previous lemmas.
\end{proof}
\section{Future research}\label{sec: future strict partitions}
 In this section we examine Conjecture \ref{conjecture: 1}, discussing how the classical version of this problem was attacked. All proofs in the classical case are based on the precise combinatorial interpretation of the coefficients of the renormalized character as a polynomial in $\p$ and $\q$. As mentioned above, this interpretation was conjectured by Stanley (\cite{stanley2003irreducible}) and proved by F\'eray (\cite{feray2007asymptotics}). Hence a first step for proving Conjecture \ref{conjecture: 1} could be to find a combinatorial formula $G_k(\p;\q)=F_k(\p;\q)$ such that 
 \[G_k(p_1,\ldots,p_m;q_1,\ldots,q_m)=-\frac{1}{2}\sum_{c\in\mathcal{C}_k^m}\prod_{i=1}^m p_i^{\stat_i (c)}\prod_{i=1}^m (-q_i)^{\stat_i'(c)}, \]
 where $\mathcal{C}_k^m$ is a combinatorial object and $\stat_i,\stat_i'$ are statistics on $\mathcal{C}_k^m$.

 Notice first that if $q_{i+1}=q_i-1$ for some $i$ then the parameters
 \[p_1,\ldots, p_m;q_1,\ldots,q_{i-1},q_i,q_i-1,q_{i+2},\ldots,q_m\]
 describe the same strict partition as
 \[p_1,\ldots,p_{i-1},p_i+p_{i+1},p_{i+2},\ldots,p_m;q_1,\ldots,q_i,q_{i+2},\ldots,q_m.\]
 Hence if the formula $G_k(p_1,\ldots,p_m;q_1,\ldots,q_m)$ exists then it must satisfy the following relation:
 \begin{multline}\label{eq: relation G}
   G_k(p_1,\ldots,p_m;q_1,\ldots,q_m)\vert_{ q_{i+1}=q_i-1}=\\G_k(p_1,\ldots,p_{i-1},p_i+p_{i+1},p_{i+2},\ldots,p_m;q_1,\ldots,q_i,q_{i+2},\ldots,q_m).
 \end{multline}
From this it follows that if one finds a formula $G_k$ that satisfies \eqref{eq: relation G} and $G_k(\mathbf{1};\q)=F_k(\mathbf{1};\q)$ then $G_k(\p;\q)=F_k(\p;\q)$ for any $\p$.

The initial step for the search of a formula $G_k$ is the rectangular case $m=1$, that is, $G_k(p;q)$. In the classical case Stanley proved in \cite{stanley2003irreducible} a combinatorial formula for rectangular partitions. Later, in \cite{Rattan2007}, Rattan found an elegant proof of the rectangular case using a combinatorial formula of shifted Schur functions proven by Okunkov and Olshanski in \cite{OkOl1998}. Although a projective version of the Okounkov-Olshanski formula is available (see \cite{ivanov2001combinatorial}) we do not manage to extrapolate from that a combinatorial formula for $F_k(p;q)$.

If the formula $G_k$ exists then the cardinality of the combinatorial set $\mathcal{C}_k^m$ can be retrieved by setting $\p=\underline{1},$ $\q=-\underline{1},$ where $\underline{1}=(1,\ldots,1):$
\[|\mathcal{C}_k^m|=-2G_k(\underline{1};-\underline{1}).\]

The value of $G_k(\underline{1};-\underline{1})=p_k^{\sharp}(\underline{1};-\underline{1})$ can be computed through Proposition \ref{proposizione iniziale}:
\[p_k^{\sharp}(\underline{1};-\underline{1})=-\frac{1}{4k}[z^{-1}]\left((2z-k)(z-1)^{\downarrow (k-1)}\frac{(z-k+1)^m(z-1)^m}{(z+1)^m(z-k-1)^m}\right).\]
In particular for $m=1$ we have
\[p_k^{\sharp}(1;-1)=-\frac{1}{2}\cdot2\cdot k!\]
so that $\mathcal{C}_k^1=2\cdot k!$. A natural candidate for $\mathcal{C}_k^1$ could be $\tilde{S}_k$, but we do not know what the corresponding statistic should be.
\medskip

Stanley proves (and conjectures) some formulas in \cite{stanley2003irreducible} and \cite{Stanley-preprint2006} that describe the renormalized character with multirectangular coordinates evaluated on a permutation of generic shape $\mu$, while in this chapter we consider projective characters evaluated on a single cycle of length $k$, with $k$ odd. A natural step would be thus to generalize our results for generic odd partitions. Unfortunately our results (and our numerical data) rely on Proposition \ref{prop: fundamental of ivanov}, so one should first generalize it to a generic shape $\mu\in DP_n$.

\label{p:3}

\cleardoublepage
\chapter{Supercharacter theory}\label{ch: supercharacters}
This chapter is an extended version of the article \cite{de2018plancherel}, which has been accepted by \emph{Advances in Applied Mathematics}.
\section{Introduction}
Let $p$ be a prime number, $q$ a power of $p$, and $\mathbb{K}$ the finite field of order $q$ and characteristic $p$. Consider $U_n=U_n(\mathbb{K})$ to be the group of upper unitriangular matrices with entries in $\mathbb{K}$, it is known that the description of conjugacy classes and complex irreducible characters of $U_n$ is a wild problem, in the sense described, for example, by Drodz in \cite{drozd1980tame}. From this perspective the problem is often considered intractable. To bypass the issue, Andr{\'e} \cite{andre1995basic} and Yan \cite{yan2010representations} set the foundations of what is now known as ``supercharacter theory'' (in the original works it was called ``basic character theory''). The idea is to meld together some irreducible characters and conjugacy classes (called respectively supercharacters and superclasses), in order to have characters which are easy enough to be tractable but still carry information of the group. In particular, one obtains a smaller character table, which is required to be a square matrix. As an application, in \cite{arias2004super}, Arias-Castro, Diaconis and Stanley described random walks on $U_n$ utilizing only the supercharacter table (usually the complete character table is required). In \cite{diaconis2008supercharacters}, Diaconis and Isaacs formalized the axioms of supercharacter theory, generalizing the construction from $U_n$ to algebra groups.\smallskip

Among the various supercharacter theories for $U_n$ a particular nice one, hinted in \cite{aguiar2012supercharacters} and described by Bergeron and Thiem in \cite{bergeron2013supercharacter}, has the property that the supercharacters take integer values on superclasses. This is particularly interesting because of a result of Keller \cite{keller2014generalized}, who proves that for each group $G$ there exists a unique finest supercharacter theory with integer values. Although it is not yet known if Bergeron and Thiem's theory is the finest integral one, it has remarkable properties which make it worth of a deeper analysis. In this theory the supercharacters of $U_n$ are indexed by set partitions of $\{1,\ldots,n\}$ and they form a basis for the Hopf algebra of superclass functions. This Hopf algebra is isomorphic to the algebra of symmetric functions in noncommuting variables. Recall that for the symmetric group $S_n$ the algebra of class functions , generated by the irreducible characters, is isomorphic to the Hopf algebra of symmetric functions (with \emph{commuting} variables). We will not investigate further this relation, which nevertheless provides a good motivation for a deeper study of Bergeron and Thiem's theory. See more on the topic in \cite{aguiar2012supercharacters}, \cite{baker2014antipode}, \cite{bergeron2013supercharacter} and \cite{bergeron2006grothendieck}.
\smallskip

In the theory introduced by Bergeron and Thiem, the characters depend on the following three statistics defined for a set partition $\pi$ of $[n]$:
\begin{itemize}
\item $d(\pi)$, the number of arcs of $\pi$;
 \item$\dim(\pi)$, that is, the sum $\sum \max(B)-\min(B)$, where the sum runs through the blocks $B$ of $\pi$;
 \item $\crs(\pi)$, the number of crossings of $\pi$.
\end{itemize}
More precisely, we have that if $\chi^{\pi}$ is the supercharacter associated to the set partition $\pi$ then the dimension is $\chi^{\pi}(\id_G)=q^{\dim(\pi)-d(\pi)}$ and $\langle\chi^{\pi},\chi^{\pi}\rangle=q^{\crs(\pi)}$. 

In the setting of probabilistic group theory one is interested in the study of statistics of the ``typical'' irreducible representation of the group. A natural probability distribution function is the uniform distribution function; in \cite{chern2014closed} and \cite{chern2015central} Chern, Diaconis, Kane and Rhoades study the statistics $\dim$ and $\crs$ for a uniform random set partition, proving formulas for the moments of $\dim(\pi)$ and $\crs(\pi)$ and, successively, a central limit theorem for these two statistics. These results imply that, for a uniform random set partition $\pi$ of $n$,
\[\dim(\pi)-d(\pi)= \frac{\alpha_n-2}{\alpha_n}n^2+O_P\left(\frac{n}{\alpha_n}\right),\qquad \crs(\pi)= \frac{2\alpha_n-5}{4\alpha_n^2}n^2+O_P\left(\frac{n}{\alpha_n}\right),\]
where $\alpha_n$ is the positive real solution of $u e^u=n+1$, so that $\alpha_n=\log n-\log\log n+o(1).$
\smallskip

In representation theory another natural distribution function is the Plancherel measure, which is a discrete probability measure associated to the irreducible characters of a finite group. The Plancherel measure has received vast coverage in the literature, especially in the case of the symmetric group $S_n$. Since the irreducible characters of $S_n$ are indexed by the partitions of $n$, the problem of investigating longest increasing subsequences of a uniform random permutation is equivalent to studying the first rows of a Plancherel distributed integer partition (see \cite{romik2015surprising}). This prompted the study of asymptotics of the Plancherel measure, and in 1977 a limit shape result for a random partition was proved independently by Kerov and Vershik \cite{kerov1977asymptotics} and Logan and Shepp \cite{logan1977variational}. The result was later improved to a central limit theorem  by Kerov \cite{ivanov2002kerov}. Moreover, it was proved by Borodin, Okounkov and Olshanski \cite{borodin2000asymptotics} that the rescaled limiting distribution function of the first $k$ rows of an integer partition coincides with the one describing asymptotics of the largest k eigenvalues of a GUE random matrix of growing
size. See also \cite{baik1999distribution},\cite{Okounkov2000} and \cite{johansson2001discrete} for further reading.
\smallskip

From the study of the Plancherel measure of $S_n$ has followed a theory regarding the Plancherel growth process. Indeed, there exist natural transition measures between partitions of $n$ and partitions of $n+1$, which generate a Markov process whose marginals are the Plancherel distributions. The transition measures have a nice combinatorial description, see \cite{kerov1993transition}.
\medskip

In this paper we generalize the notion of Plancherel measure  to adapt it to supercharacter theories. We call the measure associated to a supercharacter theory \emph{superplancherel measure}. We show that for a tower of groups $\{1\}=G_0\subseteq G_1\subseteq\ldots$, each group endowed with a consistent supercharacter theory, there exists a nontrivial transition measure which yields a Markov process; the marginals of this process are the superplancherel measures. In order to do so, we generalize a construction of superinduction for algebra groups to general finite groups. Such a construction was introduced by Diaconis and Isaacs in \cite{diaconis2008supercharacters} and developed by Marberg and Thiem in \cite{marberg2009superinduction}.
\smallskip

We then consider the superplancherel measure associated to the supercharacter theory of $U_n$ described by Bergeron and Thiem. In this setting, the superplancherel measure has an explicit formula depending on the statistics $\dim(\pi)$ and $\crs(\pi)$; we give a direct combinatorial construction of such a measure.
\smallskip

The main result of the paper is a limit shape for a random superplancherel distributed set partition. In order to formulate this result we immerse set partitions into the space of subprobabilities (\emph{i.e.}, measures with total weight less than or equal to $1$) of the unit square $[0,1]^2$ with some other properties. This embedding is similar to that of permutons for random permutations, see for example \cite{glebov2015finitely}. Given a set partition $\pi$ we refer to the corresponding subprobability as $\mu_{\pi}$. Define the measure $\Omega$ as the uniform measure on the set $\{(x,1-x)\mbox{ s.t. }x\in[0,1/2]\}$ of total weight $1/2$. Then
\begin{theorem}\label{main result 1}
  For each $n\geq 1$ let $\pi_n$ be a random set partition of $n$ distributed with the superplancherel measure $\SPl_n$, then
 \[\mu_{\pi_n}\to\Omega\mbox{ almost surely} \]
 where the convergence is the weak* convergence for measures, and the limit is taken in the space of infinite paths on the Bratteli diagram of set partitions defined by the system of superplancherel measures.
\end{theorem}
Informally, we can say that a set partition chosen at random with the superplancherel measure is asymptotically close to the the following shape:
 \[
 \begin{tikzpicture}[scale=0.5]
\draw [fill] (0,0) circle [radius=0.05];
\node [below] at (0,0) {1};
\draw [fill] (1,0) circle [radius=0.05];
\node [below] at (1,0) {2};
\draw [fill] (2,0) circle [radius=0.05];
\node [below] at (2,0) {3};
\draw [fill] (3.5,0) circle [radius=0.03];
\draw [fill] (3.75,0) circle [radius=0.03];
\draw [fill] (4,0) circle [radius=0.03];
\draw [fill] (5.5,0) circle [radius=0.05];
\draw [fill] (6.5,0) circle [radius=0.05];
\draw [fill] (7.5,0) circle [radius=0.05];
\node [below] at (7.5,0) {n};
\draw (0,0) to[out=70, in=110] (7.5,0);
\draw (1,0) to[out=70, in=110] (6.5,0);
\draw (2,0) to[out=70, in=110] (5.5,0);
\end{tikzpicture}\]
In the process, we obtain asymptotic results for the statistics $\dim(\pi)$ and $\crs(\pi)$ when $\pi$ is chosen at random with the superplancherel measure:
\begin{corollary}\label{main result 2}
  For each $n\geq 1$ let $\pi_n$ be a random set partition of $n$ distributed with the superplancherel measure $\SPl_n$. Consider as before the space of infinite paths on the Bratteli diagram of set partitions, then
 \[\frac{\dim(\pi)}{n^2}\to\frac{1}{4}\mbox{ a.s.},\qquad \crs(\pi)\in o_P(n^2).\] 
\end{corollary}
\smallskip

As mentioned, the main idea is to consider set partitions as particular measures of the unit square. With this transformation the statistics $\dim(\pi)$ and $\crs(\pi)$ can be seen as integrals of the measure $\mu_{\pi}$. We use an entropy argument to delimitate a set of set partitions of maximal probability. Finally, we relate the results on the entropy into the weak* topology of measures of $[0,1]^2.$
\smallskip

The combinatorial interpretation of the superplancherel measure for $U_n$ allows us to have computer generated superplancherel random set partitions $\pi\vdash[n]$ for fairly large $n$. In Figure \ref{fig:boat1} we present one of such $\mu_{\pi}$ for $\pi\vdash[200]$; we observe that it is indeed close to $\Omega$.
\begin{figure}[H]
  \begin{center}
  \[\includegraphics[width=4.3cm,height=4.3cm]{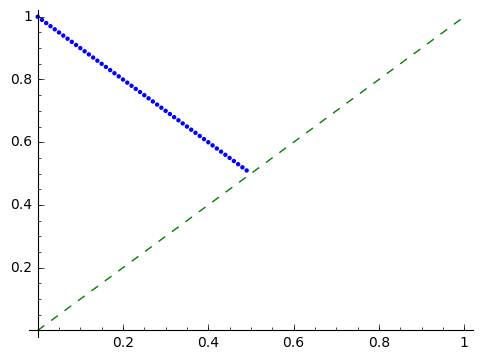}\qquad\qquad
\includegraphics[width=4.3cm,height=4.3cm]{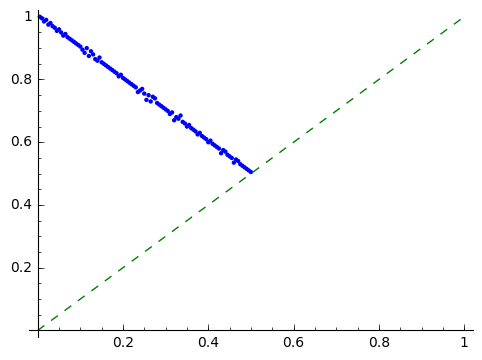}\]
  \caption[Limit shape for set partitions]{Description of a random superplancherel distributed set partition: the left image is the measure $\Omega$; on the right there is a computer generated random measure $\mu_{\pi}$ for $\pi\vdash[200]$. The algorithm we use for the program that generates a big random set partition is based on the combinatorial interpretation in Section \ref{section combinatorial interpretation}.}\label{fig:boat1}\end{center}
\end{figure}

\subsection{Outline of the paper}
In section \ref{section preliminries} we recall some basic notions of representation theory and supercharacter theory; we define the superplancherel measure and a transition measure; in section \ref{section supercharacter theory} we define the most important statistics of set partitions in the topic of the supercharacter theory of $U_n$, we find an explicit formula for the superplancherel measure, and we give a combinatorial interpretation. In section \ref{section set partitions as measures} we see set partitions as measures in $[0,1]^2$ and we study the statistics $\dim(\pi)$ and $\crs(\pi)$ in this setting. Finally, in section \ref{section final results} we prove the limit shape result for random set partitions and the result on the asymptotic behavior of $\dim(\pi)$ and $\crs(\pi)$ (respectively Theorem \ref{main result 1} and Corollary \ref{main result 2}).

\section{Preliminaries}\label{section preliminries}
\subsection{Supercharacter theory}
Recall that $\Irr(G)$ is the set of irreducible characters of $G$.
\begin{defi}
 Let $\chi,\xi$ be characters of $G$, with $\xi\in \Irr(G)$. We say that $\xi$ is a \emph{constituent} of $\chi$ if $\langle \chi,\xi\rangle\neq 0$. Moreover, we call $I(\chi):=\{\xi\in\Irr(G)\mbox{ s.t. }\langle\chi,\xi\rangle\neq 0\}$.
\end{defi}
 It is immediate to see that for each character $\chi$ of $G$ we have
\[\chi=\sum_{\xi\in\Irr(G)}\langle\chi,\xi\rangle\xi=\sum_{\xi\in I(\chi)}\langle\chi,\xi\rangle\xi\]

\begin{defi}\label{definition supercharacter theory}
 A \emph{supercharacter theory} of a finite group $G$ is a pair $(\scl(G),\sch(G))$ where $\scl(G)$ is a set partition of $G$ and $\sch(G)$ an orthogonal set of nonzero characters of $G$ (not necessarily irreducible) such that:
 \begin{enumerate}
  \item $|\scl(G)|=|\sch(G)|$;
  \item every character $\chi\in\sch(G)$ takes a constant value on each member $\mathcal{K}\in\scl(G)$;
  \item each irreducible character of $G$ is a constituent of one, and only one, of the characters $\chi\in\sch(G)$.
 \end{enumerate}
\end{defi}
The elements $\mathcal{K}\in\scl(G)$ are called \emph{superclasses}, while the characters $\chi\in\sch(G)$ are \emph{supercharacters}. It is easy to see that every element $\mathcal{K}\in \scl(G)$ is always a union of conjugacy classes. Since a supercharacter $\chi\in\sch(G)$ is always constant on superclasses we will sometimes write $\chi(\mathcal{K})$ instead of $\chi(g)$, where $\mathcal{K}\in\scl(G)$ is a superclass and $g\in\mathcal{K}$. Observe that irreducible character theory is a supercharacter theory.

The third condition of the definition of supercharacter theory can be substituted by equivalent ones, described in the next lemma. See \cite[Lemma 2.1]{diaconis2008supercharacters} for details.
\begin{lemma}\label{lemma: equivalent def diaconis}
 Let $G$ be a finite group, $\mathbf{K}$ a partition of $G$ and $\mathbf{X}$ an orthogonal set of characters of $G$. Suppose that $|\mathbf{K}|=|\mathbf{X}|$ and that the characters $\chi\in\mathbf{X}$ are constant on the sets $\mathcal{K}\in\mathbf{K}$. Then the following are equivalent:
 \begin{enumerate}
  \item Each irreducible character of $G$ is a constituent of one, and only one, of the characters $\chi\in\mathbf{X}$.
  \item The set $\{1\}$ is a member of $\mathbf{K}$.
  \item For every $\chi\in\mathbf{X}$ there is a $c(\chi)\in \C$ such that $c(\chi)\chi=\sum_{\xi\in I(\chi)}\xi(\id_G)\xi.$  
 \end{enumerate}
\end{lemma}
\begin{example}
For every finite group $G$ there are two trivial supercharacter theories.
\begin{itemize}
 \item The irreducible character theory, where $\scl(G)$ is the set of conjugacy classes of $G$ and $\sch(G)$ is the set of irreducible characters.
 \item The ``maximal'' supercharacter theory with two superclasses $\scl(G)=\{\{1_G\},G\setminus\{1_G\}\}$ and two supercharacters $\sch(G)=\{\id_G,\rho_G-\id_G\}$, where $\rho_G$ is the regular representation.
\end{itemize}
\end{example}
\begin{example}\label{example: construction via automorphisms}
 We recall here a construction of supercharacter theories via actions of other groups on $G$. Details on this construction can be found in \cite[Section 1.5]{andre2013supercharacters}. Let $G$ be a finite group and $A$ a group that acts via automorphisms on $G$, that is, there is a map $\phi\colon A\to\Aut(G)$. Then $A$ permutes both the set of conjugacy classes of $G$ and the set $\Irr(G)$ as follows. Let $\mathcal{C}$ be a conjugacy class of $G$ and set $\phi(a)(\mathcal{C}):=\{\phi(a)(g)\,,g\in\mathcal{C}\}$ for all $a\in A$. Then $\phi(a)(\mathcal{C})$ is also a conjugacy class. Define the set of superclasses $\scl(G)$ to be composed by the unions of the $A$-orbits of this action on the conjugacy classes. In other words, for a conjugacy class $\mathcal{C}_1$, set 
 \[\mathcal{K}_1:=\bigcup_{a\in A}\phi(a)(\mathcal{C}_1)\]
 to be a superclass; construct iteratively the other superclasses by taking $\mathcal{C}_2\neq\phi(a)(\mathcal{C}_1)$ for all $a\in A$ and considering $\mathcal{K}_2:=\bigcup_{a\in A}\phi(a)(\mathcal{C}_2)$, and so on.
 
 For the supercharacters, fix $a\in A$ and $\xi\in\Irr(G)$ and set $(a\xi)(g):=\xi(\phi(a)^{-1}(g))$ for all $g\in G$. For each $A$-orbit $\Omega\subseteq\Irr(G)$ define the supercharacter
 \[\chi_{\Omega}:=\sum_{\xi\in\Omega}\xi(\id_G)\xi\]
 and let $\sch(G)$ be the set of supercharacters defined this way. It is clear that $\sch(G)$ is a set of orthogonal characters, and these characters are constant on supeclasses by construction. Moreover, $\{1\}$ is obviously a superclass. Finally, a theorem of Brauer \cite[Theorem 6.32]{isaacs1994character} shows that $|\sch(G)|=|\scl(G)|$.
 
 This construction will be fundamental to build a particularly nice supercharacter theory for the group $U_n$. 
\end{example}

\subsection{Superplancherel measure}
\begin{defi}\label{definition of superplancherel}
Fix a supercharacter theory $T=(\scl(G),\sch(G))$ of $G$, we define the \emph{superplancherel measure} $\SPl_G$ of $T$ as follow: given $\chi\in\sch(G)$, then $\SPl_G^T(\chi):=\frac{1}{|G|}\frac{\chi(\id_G)^2}{\langle\chi,\chi\rangle}.$
\end{defi}
Notice that if $T$ is the irreducible character theory, then the superplancherel measure is equal to the usual Plancherel measure. We stress out that the definition of superplancherel measure depends on the supercharacter theory but we will omit it if it is clear from the context.
\smallskip

Let us show that $\SPl_G$ is indeed a probability measure; we prove first orthogonality relations of first and second kind. Fix a supercharacter theory $T=(\scl(G),\sch(G))$ for $G$, then by Lemma \ref{lemma: equivalent def diaconis} we know that for every supercharacter $\chi\in \sch(G)$ there exists $c(\chi)\in \C$ such that
\begin{equation}\label{equation supercharacter formula into irr characters}
 c(\chi)\chi=\sum_{\xi\in I(\chi)}\xi(\id_G)\xi.
\end{equation}
\begin{proposition}\label{dai}
 Set $\chi_1,\chi_2\in\sch(G)$, then
 \[\langle \chi_1,\chi_2\rangle=\frac{\chi_1(\id_G)}{c(\chi_1)}\delta_{\chi_1,\chi_2}. \]
\end{proposition}
\begin{proof}
 Consider
 \[c(\chi_1)\chi_1=\sum_{\xi\in I(\chi_1)}\xi(\id_G)\xi,\qquad c(\chi_2)\chi_2=\sum_{\xi\in I(\chi_2)}\xi(\id_G)\xi.\]
 Then
 \[\langle \chi_1,\chi_2\rangle=\frac{1}{c(\chi_1)c(\chi_2)}\sum_{\substack{\xi_1\in I(\chi_1)\\\xi_2\in I(\chi_2)}}\xi_1(\id_G)\xi_2(\id_G)\langle\xi_1,\xi_2\rangle.\]
 By the first orthogonality relations we have that $\langle\xi_1,\xi_2\rangle=\delta_{\xi_1,\xi_2}$; but if $\chi_1\neq\chi_2$ then $I(\chi_1)\cap I(\chi_2)=\emptyset$ by the third property of Definition \ref{definition supercharacter theory}. Hence 
 \[\langle \chi_1,\chi_2\rangle=\frac{1}{c(\chi_1)c(\chi_2)}\sum_{\substack{\xi_1\in I(\chi_1)\\\xi_2\in I(\chi_2)}}\xi_1(\id_G)\xi_2(\id_G)\delta_{\xi_1,\xi_2}=0\]
 if $\chi_1\neq\chi_2$. On the other hand, if $\chi_1=\chi_2$ then
 \[\langle \chi_1,\chi_2\rangle=\frac{1}{c(\chi_1)c(\chi_2)}\sum_{\substack{\xi_1\in I(\chi_1)\\\xi_2\in I(\chi_2)}}\xi_1(\id_G)\xi_2(\id_G)\delta_{\xi_1,\xi_2}=\frac{1}{c(\chi_1)^2}\sum_{\xi_1\in I(\chi_1)}\xi_1(\id_G)^2=\frac{\chi_1(\id_G)}{c(\chi_1)}.\]
Therefore we can conclude that $\langle \chi_1,\chi_2\rangle=\frac{\chi_1(\id_G)}{c(\chi_1)}\delta_{\chi_1,\chi_2}$.
\end{proof}
In the irreducible character theory, a direct consequence of the orthogonality relations of the first kind is the orthogonality relations of the second kind: if $g,h\in G$ then
\begin{equation}\label{orthogonality relations of the second kind}
 \sum_{\xi\in\Irr(G)}\xi(g)\overline{\xi(h)}=\frac{|G|}{|\mathcal{C}_g|}\delta_{\mathcal{C}_g,\mathcal{C}_h},
\end{equation}
where $\mathcal{C}_g,\mathcal{C}_h$ are the conjugacy classes of respectively $g$ and $h$. We adapt the proof of this result to the supercharacter theory, see for example \cite[Theorem 1.10.3]{sagan2013symmetric}
\begin{proposition}
 Let $\K_1,\K_2\in\scl(G)$, then
 \[\sum_{\chi\in\sch(G)}\frac{c(\chi)}{\chi(\id_G)}\chi(\K_1)\overline{\chi(\K_2)}=\frac{|G|}{|\K_1|}\delta_{\K_1,\K_2}.\]
\end{proposition}
\begin{proof}
The modified supercharacter table
 \[U=\left[ \sqrt{\frac{|\K|}{|G|}}\sqrt{\frac{c(\chi)}{\chi(\id_G)}}\chi(\K)\right]_{\chi\in\sch(G),\K\in\scl(G)}\]
 is unitary, that is, it has orthonormal rows, due to the previous proposition. This implies that it has also orthonormal columns, \emph{i.e.},
 \[\sum_{\chi\in\sch(G)}\frac{c(\chi)}{\chi(\id_G)}\frac{\sqrt{|\K_1|}\sqrt{|\K_2|}}{|G|}\chi(\K_1)\overline{\chi(\K_2)}=\delta_{\K_1,\K_2}.\qedhere\]
\end{proof}
\begin{proposition}
 For each group $G$ and supercharacter theory $T$ of $G$ the superplancherel measure $\SPl_G$ is a probability measure.
\end{proposition}
\begin{proof}
Since 
\[\langle\chi,\chi\rangle=\frac{1}{c(\chi)^2}\sum_{\xi\in I(\chi)}\xi(\id_G)^2=\frac{\chi(\id_G)}{c(\chi)},\]
then $\SPl_G^T(\chi):=\frac{1}{|G|}\frac{\chi(\id_G)^2}{\langle\chi,\chi\rangle}= \frac{c(\chi)}{|G|}\chi(\id_G).$
Recall from Lemma \ref{lemma: equivalent def diaconis} that $\K=\{1\}$ is always a superclass. In particular the previous proposition applied to $\K_1=\{1\}=\K_2$ gives:
\[\sum_{\chi\in \sch(G)}\frac{c(\chi)}{|G|}\chi(\id_G)=1\]
hence the superplancherel measure is indeed a probability measure.
\end{proof}                                                                                                      
\subsection{Superinduction and Frobenius reciprocity}\label{section superinduction}
In this section we extend the notion of Superinduction, defined by Diaconis and Isaacs in \cite{diaconis2008supercharacters} for algebra groups, to general finite groups, and we use it to define a transition measure. Let $G$ be a finite group, $H\leq G$ a subgroup and $(\scl(G),\sch(G))$ a supercharacter theory for $G$. Let $\phi\colon H\to\C$ be any function, we set $\phi^0\colon G\to\C$ to be $\phi^0(g)=\phi(g)$ if $g\in H$ and $\phi^0(g)=0$ otherwise. We define
\[\SInd_H^G(\phi)(g):=\frac{|G|}{|H|\cdot|[g]|}\sum_{k\in[g]}\phi^0(k),\]
where $[g]\in\scl(G)$ is the superclass containing $g$. By construction, $\SInd_H^G(\phi)$ is a superclass function. Since $\sch(G)$ is an orthogonal basis for the algebra of superclass functions (see \cite[Theorem 2.2]{diaconis2008supercharacters}), we can expand $\SInd_H^G(\phi)$ in this basis:
\begin{equation}\label{eq: expansion superinduction}
 \SInd_H^G(\phi)=\sum_{\chi\in\sch(G)}\frac{\langle \SInd_H^G(\phi),\chi\rangle}{\langle\chi,\chi\rangle}\chi.
\end{equation}

A supercharacter version of the Frobenius reciprocity holds: if $\psi$ is a superclass function then
\begin{align*}
 \langle \SInd_H^G(\phi),\psi\rangle&=\frac{|G|}{|H|}\frac{1}{|G|}\sum_{g\in G}\frac{1}{|[g]|}\sum_{k\in[g]}\overline{\phi^0(k)}\psi(g)\\
 &=\frac{1}{|H|}\sum_{\mathcal{K}\in\scl(G)}\sum_{g,k\in\mathcal{K}}\frac{\overline{\phi^0(k)}\psi(k)}{|\mathcal{K}|}\\
 &=\frac{1}{|H|}\sum_{\mathcal{K}\in\scl(G)}\sum_{k\in\mathcal{K}}\overline{\phi^0(k)}\psi(k)\\
 &=\frac{1}{|H|}\sum_{k\in G}\overline{\phi^0(k)}\psi(k)\\
 &=\frac{1}{|H|}\sum_{k\in H}\overline{\phi(k)}\psi(k)=\langle\phi,\Res_H^G(\psi)\rangle.
\end{align*}
Here $\Res_H^G(\psi)$ is the restriction of $\psi$ to $H$.
\subsection{Bratteli diagrams for supercharacter theories}\label{section bratteli supercharacter}
The theory of Bratteli diagrams inherited from representation theory (Section \ref{section: bratteli for representation}) can be extended to supercharacter theories, with the additional requirement that the supercharacter theories are consistent:
\begin{defi}
 Let $H\leq G$ and fix a supercharacter theory for each group: $(\scl(H),\sch(H))$, $(\scl(G),\sch(G))$. Then these supercharacter theories are \emph{consistent} if for each $\mathcal{H}\in\scl(H)$ there exists $\mathcal{K}\in\scl(G)$ such that $\mathcal{H}\subseteq\mathcal{K}$.
\end{defi}
Note that having consistent supercharacter theories is equivalent to the requirement that $\Res_H^G(\chi)$ is a superclass function on $H$ for each $\chi\in\sch(G)$ by \cite[Theorem 2.2]{diaconis2008supercharacters}.

Let $G_0=\emptyset\hookrightarrow G_1=\{\id_{G_1}\}\hookrightarrow G_2\hookrightarrow\ldots$ be a sequence of groups, each endowed with a supercharacter theory $(\scl(G_n),\sch(G_n))$ such that they are all consistent. For each $n$ let $\Gamma_n$ be a set of indices for $\sch(G_n)$, $\Gamma:=\bigcup\Gamma_n$ and set the multiplicity of an edge to be 
\[\kappa(\lambda,\Lambda)=\frac{\langle \SInd_{G_n}^{G_{n+1}}(\chi^{\lambda}),\chi^{\Lambda}\rangle}{\langle\chi^{\lambda},\chi^{\lambda}\rangle}=\frac{\langle\chi^{\lambda},\Res_{G_n}^{G_{n+1}}(\chi^{\Lambda})\rangle}{\langle\chi^{\lambda},\chi^{\lambda}\rangle}.\]
As in the classical case the dimension in the sense of Bratteli diagrams coincide with the supercharacter theory dimension, that is, we have
\[\chi^{\Lambda}(\id_{G_{n+1}})=\sum_{\emptyset\nearrow\ldots\nearrow\Lambda}\prod_{i=1}^n\kappa(\lambda_i,\lambda_{i+1}).\]
Indeed, it is enough to show that 
\[\chi^{\Lambda}(\id_{G_{n+1}})=\sum_{\lambda\in\Gamma_n}\chi^{\lambda}(\id_{G_n})\cdot\kappa(\lambda,\Lambda),\]
which is immediate since
\begin{align*}
 \chi^{\Lambda}(\id_{G_{n+1}})&=\Res_{G_n}^{G_{n+1}}(\chi^{\Lambda})(\id_{G_n})\\& =\sum_{\lambda\in\Gamma_n}\chi^{\lambda}(\id_{G_n})\frac{\langle\chi^{\lambda},\Res_{G_n}^{G_{n+1}}(\chi^{\Lambda})\rangle}{\langle\chi^{\lambda},\chi^{\lambda}\rangle}\\&=\sum_{\lambda\in\Gamma_n}\chi^{\lambda}(\id_{G_n})\cdot\kappa(\lambda,\Lambda).
\end{align*}

The second equality holds since $\Res_{G_n}^{G_{n+1}}(\chi^{\Lambda})$ is a superclass function on $G_n$, thanks to the consistency of the supercharacter theories.

As a corollary we obtain that the superplancherel measures form a set of coherent measures, according to Equation \eqref{eq: coherent measures}. Moreover, the transition and co-transition measures are 
\[\tr(\lambda,\Lambda)=\frac{|G_n|}{|G_{n+1}|}\frac{\dim\Lambda}{\dim\lambda}\frac{\langle\SInd_{G_n}^{G_{n+1}}(\chi^{\lambda}),\chi^{\Lambda}\rangle}{\langle\chi^{\Lambda},\chi^{\Lambda}\rangle},\qquad \ctr(\lambda,\Lambda)=\frac{\dim\lambda}{\dim\Lambda}\frac{\langle\SInd_{G_n}^{G_{n+1}}(\chi^{\lambda}),\chi^{\Lambda}\rangle}{\langle\chi^{\lambda},\chi^{\lambda}\rangle}.\]

In the next section we will describe a particular supercharacter theory for upper unitriangular matrices over a finite field $U_n(\mathbb{K})$. These theories will index the supercharacters by set partitions of $\{1,\ldots,n\}$, so that we obtain a Bratteli diagram whose vertices are set partitions.

\section{Supercharacter theory for unitriangular matrices}\label{section supercharacter theory}
\subsection{Set partitions}
We recall some basic definitions regarding set partitions. Let $n\in\N$ and set $[n]$ to be the set $\{1,\ldots,n\}$. A \emph{set partition} $\pi$ of $[n]$ is a family of non empty sets, called \emph{blocks}, which are disjoint and whose union is $[n]$. If $\pi$ is a set partition of $[n]$ we write $\pi\vdash[n]$. Conventionally the blocks of $\pi$ are ordered by increasing value of the smallest element of the block, and inside every block the elements are ordered with the usual order of natural numbers. If two numbers $i$ and $j$ are in the same block of the set partition $\pi\vdash[n]$ and there is no $k$ in that block such that $i<k<j$, then the pair $(i,j)$ is said to be an \emph{arc} of $\pi$. The set partition $\pi$ is uniquely determined by the set $D(\pi)$ of arcs. The \emph{standard representation} of $\pi\vdash[n]$ is the graph with vertex set $[n]$ and edge set $D(\pi)$, as in Figure \ref{picture standard representation}.
\begin{figure}[H]
 \[ \pi= \begin{array}{c}
 \begin{tikzpicture}[scale=0.5]
\draw [fill] (0,0) circle [radius=0.05];
\node [below] at (0,0) {1};
\draw [fill] (1,0) circle [radius=0.05];
\node [below] at (1,0) {2};
\draw [fill] (2,0) circle [radius=0.05];
\node [below] at (2,0) {3};
\draw [fill] (3,0) circle [radius=0.05];
\node [below] at (3,0) {4};
\draw [fill] (4,0) circle [radius=0.05];
\node [below] at (4,0) {5};
\draw [fill] (5,0) circle [radius=0.05];
\node [below] at (5,0) {6};
\draw [fill] (6,0) circle [radius=0.05];
\node [below] at (6,0) {7};
\draw [fill] (7,0) circle [radius=0.05];
\node [below] at (7,0) {8};
\draw [fill] (8,0) circle [radius=0.05];
\node [below] at (8,0) {9};
\draw (0,0) to[out=70, in=110] (4,0);
\draw (4,0) to[out=70, in=110] (6,0);
\draw (2,0) to[out=70, in=110] (3,0);
\draw (3,0) to[out=70, in=110] (8,0);
\draw (5,0) to[out=70, in=110] (7,0);
\end{tikzpicture}\end{array} \]
\caption[A set partition]{Example of the set partition $\pi=\left\{\{1,5,7\},\{2\},\{3,4,9\},\{6,8\}\right\}\vdash[9]$ in standard representation.}\label{picture standard representation}
\end{figure}
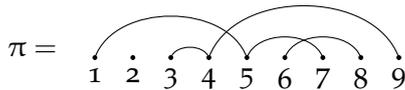
Fix $\pi\vdash[n]$, then define
\begin{itemize}
 \item the \emph{dimension} $\dim(\pi)$ as 
 \[\dim(\pi):=\sum_{(i,j)\in D(\pi)}j-i.\]
 For example, in the set partition of Figure \ref{picture standard representation}, the dimension is $\dim(\pi)=14$.
 \item The number of crossings $\crs(\pi)$ of $\pi$, where a \emph{crossing} is an unordered pair of arcs $\{(i,j),(k,l)\}$ in $D(\pi)$ such that $i<k<j<l$. Diagrammatically a crossing corresponds to the picture 
 \begin{figure}[H]
 \[ \begin{array}{c}
 \begin{tikzpicture}[scale=0.5]
\draw [fill] (0,0) circle [radius=0.05];
\node [below] at (0,0) {$i$};
\draw [fill] (1,0) circle [radius=0.05];
\node [below] at (1,0) {$k$};
\draw [fill] (2,0) circle [radius=0.05];
\node [below] at (2,0) {$j$};
\draw [fill] (3,0) circle [radius=0.05];
\node [below] at (3,0) {$l$};
\draw (0,0) to[out=70, in=110] (2,0);
\draw (1,0) to[out=70, in=110] (3,0);
\end{tikzpicture}\end{array} \]
\end{figure}
 In the example of Figure \ref{picture standard representation}, the number of crossings of $\pi$ is $\crs(\pi)=2$. 
 \item The number of nestings $\nst(\pi)$, where a \emph{nesting} is an unordered pair of arcs \[\{(i,j),(k,l)\}\subseteq D(\pi)\mbox{ such that }i<k<l<j\}.\] Diagrammatically a nesting corresponds to the picture 
 \begin{figure}[H]
 \[ \begin{array}{c}
 \begin{tikzpicture}[scale=0.5]
\draw [fill] (0,0) circle [radius=0.05];
\node [below] at (0,0) {i};
\draw [fill] (1,0) circle [radius=0.05];
\node [below] at (1,0) {$k$};
\draw [fill] (2,0) circle [radius=0.05];
\node [below] at (2,0) {$l$};
\draw [fill] (3,0) circle [radius=0.05];
\node [below] at (3,0) {$j$};
\draw (0,0) to[out=70, in=110] (3,0);
\draw (1,0) to[out=70, in=110] (2,0);
\end{tikzpicture}\end{array} \]
\end{figure}
 In the example of Figure \ref{picture standard representation}, the number of nestings of $\pi$ is $\nst(\pi)=3$. 
\end{itemize}
Fix $\pi\vdash[n]$ and $i,j$ with $i<j\leq n$. The pair $(i,j)$ is said to be $\pi$-regular if there exists no $k<i$ such that $(k,j)\in D(\pi)$ and there exists no $l>j$ such that $(i,l)\in D(\pi)$. The set of $\pi$-regular pairs is denoted $\reg(\pi)$. For example, given $\pi=\{\{1,4\},\{2,3,5\}\}=\begin{array}{c}\begin{tikzpicture}[scale=0.4]
\draw [fill] (0,0) circle [radius=0.05];\draw [fill] (1,0) circle [radius=0.05];\draw [fill] (2,0) circle [radius=0.05];\draw [fill] (3,0) circle [radius=0.05];\draw [fill] (4,0) circle [radius=0.05];
\draw (0,0) to[out=70, in=110] (3,0);
\draw (1,0) to[out=70, in=110] (2,0);
\draw (2,0) to[out=70, in=110] (4,0);
\end{tikzpicture}\end{array}$ then the set $\reg(\pi)$ is $\{(1,4),(1,5),(2,3),(2,5),(3,5)\}$; if an pair is not regular then it is called \emph{singular} and the set of $\pi$-singular pairs is denoted $\Sing(\pi)$. In the previous example thus $\Sing(\pi)=\{(1,2),(1,3),(2,4),(3,4),(4,5)\}$.
\smallskip

If $\pi\vdash[n]$ and $k<l\leq n$ then $\nst_{\pi}(k,l)=\sharp\{(i,j)\in D(\pi)\mbox{ s.t. }i<k<l<j\}$. If $\sigma\vdash[m]$ with $m\leq n$ then 
\[\nst_{\pi}(\sigma):=\sum_{(k,l)\in D(\sigma)}\nst_{\pi}(k,l).\]

\subsection{A supercharacter theory for \texorpdfstring{$U_n$}{Un}}
Let $\mathbb{K}$ be the finite field of order $q$ and characteristic $p$. The group $U_n=U_n(\mathbb{K})$ is the group of upper unitriangular matrices of size $n\times n$ and entries belonging to $\mathbb{K}$, that is,
\[U_n=U_n(\mathbb{K})=\left\{\left[\begin{array}{cccc}
1&a_{1,2}&\cdots&a_{1,n}\\&1&a_{2,3}&\vdots\\&&\ddots&a_{n-1,n}\\&&&1\\
\end{array}\right]\in M_{n\times n}(\mathbb{K})\right\}.\]
In \cite{bergeron2013supercharacter}, Bergeron and Thiem describe a supercharacter theory in which both $\sch(U_n)$ and $\scl(U_n)$ are in bijection with set partitions of $[n]=\{1,\ldots,n\}$. This supercharacter theory is easily obtained by the construction described in Example \ref{example: construction via automorphisms}; see also \cite[Section 3.6]{andre2013supercharacters} for more details.
Consider the group $A:=U_n\times U_n\times\mathbb{K}^{\times}$ and define an action on $U_n$:
\[\phi\colon A\to\Aut(U_n),\qquad \phi(g_1,g_2,t)(h):=\mathbb{1}+t g_1(h-\mathbb{1})g_2^{-1},\]
where $\mathbb{1}=\id_{U_n}$ is the identity matrix.

We describe the superclass $\mathcal{K}_{\pi}$ of $U_n$ indexed by the set partition $\pi\vdash [n]$: let $h\in \mathcal{K}_{\pi}$ and $i<j$, then
\begin{itemize}
 \item $h_{i,j}\neq 0$ for all $(i,j)\in D(\pi)$;
 \item $h_{i,j}=0$ if $(l,j)\notin D(\pi)$ for all $l\leq i$ and $(i,l)\notin D(\pi)$ for all $l\geq j$. 
\end{itemize}
For example, given $\pi=\{\{1,5,7\},\{2,3,8\},\{4\},\{6\}\}$, the following is a visual representation of the superclass $\mathcal{K}_{\pi}$:

\[\pi= \begin{array}{c}
 \begin{tikzpicture}[scale=0.5]
\draw [fill] (0,0) circle [radius=0.05];
\node [below] at (0,0) {1};
\draw [fill] (1,0) circle [radius=0.05];
\node [below] at (1,0) {2};
\draw [fill] (2,0) circle [radius=0.05];
\node [below] at (2,0) {3};
\draw [fill] (3,0) circle [radius=0.05];
\node [below] at (3,0) {4};
\draw [fill] (4,0) circle [radius=0.05];
\node [below] at (4,0) {5};
\draw [fill] (5,0) circle [radius=0.05];
\node [below] at (5,0) {6};
\draw [fill] (6,0) circle [radius=0.05];
\node [below] at (6,0) {7};
\draw [fill] (7,0) circle [radius=0.05];
\node [below] at (7,0) {8};
\draw (0,0) to[out=70, in=110] (4,0);
\draw (4,0) to[out=70, in=110] (6,0);
\draw (1,0) to[out=70, in=110] (2,0);
\draw (2,0) to[out=70, in=110] (7,0);
\end{tikzpicture}\end{array} 
\qquad
\mathcal{K}_{\pi}=\left[\begin{array}{cccccccc}1&0&\cdot&0&\star&\cdot&\cdot&\cdot\\
 &1&\star&\cdot&\cdot&\cdot&\cdot&\cdot\\
 &&1&0&0&0&\cdot&\star\\
 &&&1&0&0&\cdot&0\\
 &&&&1&0&\star&\cdot\\
 &&&&&1&0&0\\
 &&&&&&1&0\\
 &&&&&&&1\end{array}\right].\]

Here, $\mathcal{K}_{\pi}$ is the superclass of matrices with nonzero entries at $(1,5),(2,3),(3,8),(5,7)$ (represented by $\star$), arbitrary values on the right and above these entries (represented by $\cdot$), and zeros at the entry $(i,j)$ with $i>j$. In particular, if $\pi=\{\{1\},\{2\},\ldots,\{n\}\}$, then $\mathcal{K}_{\pi}=\{\mathbb{1}_n\}$.

Through the section, given set partitions $\pi,\sigma\vdash[n]$ we will write $\chi^{\pi}$ for the supercharacter corresponding to $\pi$ and $\K_{\sigma}$ for the superclass corresponding to $\sigma$.
\smallskip

This supercharacter theory has an explicit formula for the supercharacter values. See \cite[Theorem 3.9]{andre2013supercharacters} for a complete proof. Note that the formula in \cite{andre2013supercharacters} differs from ours of a factor $(q-1)^{d(\pi)},$ so that our supercharacter is integer valued.
\begin{proposition}
Let $\pi,\sigma\vdash[n]$, then 
\[\chi^{\pi}(\K_{\sigma})=\left\{\begin{array}{lr}
q^{\dim(\pi)-d(\pi)-\nst_{\pi}(\sigma)}\cdot(q-1)^{d(\pi)}\cdot (-\frac{1}{q-1})^{|D(\pi)\cap D(\sigma)|}&\mbox{if }D(\sigma)\subseteq\reg\pi;\\
0&\mbox{otherwise.}
\end{array}\right.\]
In particular, $\chi^{\pi}(\mathbb{1}_n)=(q-1)^{d(\pi)}\cdot q^{\dim(\pi)-d(\pi)}$.
\end{proposition}
Let $c(\pi)=c(\chi_{\pi})$ in \eqref{equation supercharacter formula into irr characters}. In \cite[Section 2.2]{baker2014antipode}, the authors describe $c(\pi)$ as
\[c(\pi)=\frac{q^{\crs(\pi)}(q-1)^{d(\pi)}}{\chi^{\pi}(\mathbb{1}_n)},\]
so that Proposition \ref{dai} gives 
\[\langle \chi^{\pi},\chi^{\pi}\rangle=(q-1)^{d(\pi)}q^{\crs(\pi)}.\]
\begin{corollary}
Set $\pi\vdash[n]$, then 
\[\SPl_n(\chi^{\pi}):=\SPl_{U_n}(\chi^{\pi})=\frac{1}{q^{\frac{n(n-1)}{2}}}\frac{(q-1)^{d(\pi)}\cdot q^{2\dim(\pi)-2 d(\pi)}}{q^{\crs(\pi)}}.\]
\end{corollary}
\begin{proof}
This is a direct consequence of Definition \ref{definition of superplancherel}, since $|U_n|=q^{n(n-1)/2}$.
\end{proof}

\subsection{Bratteli diagram for set partitions}
We consider the following inclusion of groups $U_n\overset{\iota}\hookrightarrow U_{n+1}$: if $A\in U_n$ then 
\[\iota(A)=\left[\begin{array}{cc}A&\underline{0}\\\underline{0}&1\end{array}\right].\]
Notice that the supercharacter theory defined above is consistent in the sense of Section \ref{section bratteli supercharacter} since it is defined directly on set partitions. In particular, we obtain a Bratteli diagram, which we describe in this section. By definition, the multiplicity of the edge $(\lambda,\Lambda)$ is $\kappa(\lambda,\Lambda)=\langle\SInd_{U_n}^{U_{n+1}}(\chi^{\lambda}),\chi^{\Lambda}\rangle/\langle\chi^{\lambda},\chi^{\lambda}\rangle.$ Recall that in our previous examples of Bratteli diagrams of integer partitions (Example \ref{example: bratteli integer partitions}) and of strict integer partitions (Section \ref{section: bratteli stric partitions}) the graphs are simple, that is, the multiplicity of each edge is $1$. This is not the case here. 

 An algorithm to compute $\SInd_{U_n}^{U_{n+1}}(\chi^{\lambda})$ as a linear combination of supercharacters of $U_{n+1}$ was described by Marberg and Thiem in \cite{marberg2009superinduction} for a finer supercharacter theory of $U_n$ in which the supercharacters and superclasses are indexed by colored set partitions, \emph{i.e.}, pairs $(\pi,\phi)$ where $\pi\vdash[n]$ and $\phi\in\col_{\mathbb{K}}(\pi)=\{\phi\colon D(\pi)\to\mathbb{K}^{\times}\}.$ The passage from the colored set partition supercharacter theory to the uncolored one is simple enough: if $(\pi,\phi)$ is a colored set partition, then (\cite[Theorem 3.9]{andre2013supercharacters})
\[\chi^{\pi}=\sum_{\phi\in\col_{\mathbb{K}}(\pi)}\chi^{\pi,\phi},\qquad\mathcal{K}_{\pi}=\bigcup_{\phi\in\col_{\mathbb{K}}(\pi)}\mathcal{K}_{\pi,\phi},\]
where, given $(\pi,\phi)$, the associated supercharacter is $\chi^{\pi,\phi}$ and the associated superclass is $\mathcal{K}_{\pi,\phi}$. The result of Marberg and Thiem is therefore easily translated on the supercharacter theory that we consider here. 

Before stating it, let us introduce an operation on the formal sum of set partitions of $[n]$: let $\pi\vdash[n]$ and define iteratively an operation $\pi\ast_i\{k\}$ for some positive integers $i,k\leq n$:
\[\pi\ast_i\{k\}:=\left\{\begin{array}{ll}
\pi&\mbox{ if }i=k,\\
q(\pi\ast_{i+1}\{k\})&\mbox{ if }\exists l>k\mbox{ with }(i,l)\in D(\pi),\\
\pi\vert_{D(\pi)\setminus (i,k)}\ast_{i+1}\{k\}&\mbox{ if }(i,k)\in D(\pi),\\
\pi\ast_{i+1}\{k\}+(q-1)\pi\vert_{D(\pi)\setminus (i,j)\cup (i,k)}\ast_{i+1}\{j\}&\mbox{ if }\exists j<k\mbox{ with }(i,j)\in D(\pi),\\
\pi\ast_{i+1}\{k\}+\pi\vert_{D(\pi)\cup (i,k)}&\mbox{ otherwise.}\end{array}\right.\]
Here we used the notation $\pi\vert_{D(\pi)\setminus (i,j)\cup (i,k)}$ to indicate the set partition equal to $\pi$, except that the arc $(i,j)$ is removed and the arc $(i,k)$ is added.
\begin{proposition}[Marberg and Thiem \cite{marberg2009superinduction}]
 Let $\pi\vdash [n]$ be a set partition and let $\pi\cup\{n+1\}$ be the set partition where $\{n+1\}$ is a block. Then 
 \[\SInd_{U_n}^{U_{n+1}}(\chi^{\pi})=\chi^{\pi\cup\{n+1\}}\ast_1\chi^{\{n+1\}},\]
 where 
 \[\chi^{\pi}\ast_i\chi^{\{k\}}=\sum_{\sigma}c^{\sigma}_{\pi,i,k}\chi^{\sigma}\qquad\mbox{ if }\qquad\pi\ast_i\{k\}=\sum_{\sigma}c^{\sigma}_{\pi,i,k}\sigma.\]
\end{proposition}
\begin{example}

Let $\pi=\{\{1,3\},\{2\}\}=\begin{array}{c}\begin{tikzpicture}[scale=0.4]
\draw [fill] (0,0) circle [radius=0.05];
\draw [fill] (1,0) circle [radius=0.05];
\draw [fill] (2,0) circle [radius=0.05];
\draw (0,0) to[out=70, in=110] (2,0);
\end{tikzpicture} \end{array}\vdash[3]$ and $\sigma=\{\{1,3\},\{2,4\}\}=\begin{array}{c}\begin{tikzpicture}[scale=0.4]
\draw [fill] (0,0) circle [radius=0.05];
\draw [fill] (1,0) circle [radius=0.05];
\draw [fill] (2,0) circle [radius=0.05];
\draw [fill] (3,0) circle [radius=0.05];
\draw (0,0) to[out=70, in=110] (2,0);
\draw (1,0) to[out=70, in=110] (3,0);
\end{tikzpicture}\end{array}\vdash[4]$. In this example we compute the multiplicity of the edge $\kappa(\pi,\sigma)$.

We start by computing $\pi\cup\{4\}\ast_1\{4\}$:
\begin{align*}
&\begin{array}{c}\begin{tikzpicture}[scale=0.4]
\draw [fill] (0,0) circle [radius=0.05];
\draw [fill] (1,0) circle [radius=0.05];
\draw [fill] (2,0) circle [radius=0.05];
\draw [fill] (3,0) circle [radius=0.05];
\draw (0,0) to[out=70, in=110] (2,0);
\end{tikzpicture} \end{array}\ast_1\{4\}=\begin{array}{c}\begin{tikzpicture}[scale=0.4]
\draw [fill] (0,0) circle [radius=0.05];
\draw [fill] (1,0) circle [radius=0.05];
\draw [fill] (2,0) circle [radius=0.05];
\draw [fill] (3,0) circle [radius=0.05];
\draw (0,0) to[out=70, in=110] (2,0);
\end{tikzpicture} \end{array}\ast_2\{4\}+(q-1)\begin{array}{c}\begin{tikzpicture}[scale=0.4]
\draw [fill] (0,0) circle [radius=0.05];
\draw [fill] (1,0) circle [radius=0.05];
\draw [fill] (2,0) circle [radius=0.05];
\draw [fill] (3,0) circle [radius=0.05];
\draw (0,0) to[out=70, in=110] (3,0);
\end{tikzpicture} \end{array}\ast_2\{3\}\\
&=\begin{array}{c}\begin{tikzpicture}[scale=0.4]
\draw [fill] (0,0) circle [radius=0.05];
\draw [fill] (1,0) circle [radius=0.05];
\draw [fill] (2,0) circle [radius=0.05];
\draw [fill] (3,0) circle [radius=0.05];
\draw (0,0) to[out=70, in=110] (2,0);
\end{tikzpicture} \end{array}\ast_3\{4\}+\begin{array}{c}\begin{tikzpicture}[scale=0.4]
\draw [fill] (0,0) circle [radius=0.05];
\draw [fill] (1,0) circle [radius=0.05];
\draw [fill] (2,0) circle [radius=0.05];
\draw [fill] (3,0) circle [radius=0.05];
\draw (0,0) to[out=70, in=110] (2,0);
\draw (1,0) to[out=70, in=110] (3,0);
\end{tikzpicture} \end{array}+(q-1)\begin{array}{c}\begin{tikzpicture}[scale=0.4]
\draw [fill] (0,0) circle [radius=0.05];
\draw [fill] (1,0) circle [radius=0.05];
\draw [fill] (2,0) circle [radius=0.05];
\draw [fill] (3,0) circle [radius=0.05];
\draw (0,0) to[out=70, in=110] (3,0);
\end{tikzpicture} \end{array}\ast_3\{3\}+(q-1)\begin{array}{c}\begin{tikzpicture}[scale=0.4]
\draw [fill] (0,0) circle [radius=0.05];
\draw [fill] (1,0) circle [radius=0.05];
\draw [fill] (2,0) circle [radius=0.05];
\draw [fill] (3,0) circle [radius=0.05];
\draw (0,0) to[out=70, in=110] (3,0);
\draw (1,0) to[out=70, in=110] (2,0);
\end{tikzpicture} \end{array}\\
&=\begin{array}{c}\begin{tikzpicture}[scale=0.4]
\draw [fill] (0,0) circle [radius=0.05];
\draw [fill] (1,0) circle [radius=0.05];
\draw [fill] (2,0) circle [radius=0.05];
\draw [fill] (3,0) circle [radius=0.05];
\draw (0,0) to[out=70, in=110] (2,0);
\end{tikzpicture} \end{array}\ast_4\{4\}+\begin{array}{c}\begin{tikzpicture}[scale=0.4]
\draw [fill] (0,0) circle [radius=0.05];
\draw [fill] (1,0) circle [radius=0.05];
\draw [fill] (2,0) circle [radius=0.05];
\draw [fill] (3,0) circle [radius=0.05];
\draw (0,0) to[out=70, in=110] (2,0);
\draw (2,0) to[out=70, in=110] (3,0);
\end{tikzpicture} \end{array}+\begin{array}{c}\begin{tikzpicture}[scale=0.4]
\draw [fill] (0,0) circle [radius=0.05];
\draw [fill] (1,0) circle [radius=0.05];
\draw [fill] (2,0) circle [radius=0.05];
\draw [fill] (3,0) circle [radius=0.05];
\draw (0,0) to[out=70, in=110] (2,0);
\draw (1,0) to[out=70, in=110] (3,0);
\end{tikzpicture} \end{array}+(q-1)\begin{array}{c}\begin{tikzpicture}[scale=0.4]
\draw [fill] (0,0) circle [radius=0.05];
\draw [fill] (1,0) circle [radius=0.05];
\draw [fill] (2,0) circle [radius=0.05];
\draw [fill] (3,0) circle [radius=0.05];
\draw (0,0) to[out=70, in=110] (3,0);
\end{tikzpicture} \end{array}+(q-1)\begin{array}{c}\begin{tikzpicture}[scale=0.4]
\draw [fill] (0,0) circle [radius=0.05];
\draw [fill] (1,0) circle [radius=0.05];
\draw [fill] (2,0) circle [radius=0.05];
\draw [fill] (3,0) circle [radius=0.05];
\draw (0,0) to[out=70, in=110] (3,0);
\draw (1,0) to[out=70, in=110] (2,0);
\end{tikzpicture} \end{array}\\
&=\begin{array}{c}\begin{tikzpicture}[scale=0.4]
\draw [fill] (0,0) circle [radius=0.05];
\draw [fill] (1,0) circle [radius=0.05];
\draw [fill] (2,0) circle [radius=0.05];
\draw [fill] (3,0) circle [radius=0.05];
\draw (0,0) to[out=70, in=110] (2,0);
\end{tikzpicture} \end{array}+\begin{array}{c}\begin{tikzpicture}[scale=0.4]
\draw [fill] (0,0) circle [radius=0.05];
\draw [fill] (1,0) circle [radius=0.05];
\draw [fill] (2,0) circle [radius=0.05];
\draw [fill] (3,0) circle [radius=0.05];
\draw (0,0) to[out=70, in=110] (2,0);
\draw (2,0) to[out=70, in=110] (3,0);
\end{tikzpicture} \end{array}+\begin{array}{c}\begin{tikzpicture}[scale=0.4]
\draw [fill] (0,0) circle [radius=0.05];
\draw [fill] (1,0) circle [radius=0.05];
\draw [fill] (2,0) circle [radius=0.05];
\draw [fill] (3,0) circle [radius=0.05];
\draw (0,0) to[out=70, in=110] (2,0);
\draw (1,0) to[out=70, in=110] (3,0);
\end{tikzpicture} \end{array}+(q-1)\begin{array}{c}\begin{tikzpicture}[scale=0.4]
\draw [fill] (0,0) circle [radius=0.05];
\draw [fill] (1,0) circle [radius=0.05];
\draw [fill] (2,0) circle [radius=0.05];
\draw [fill] (3,0) circle [radius=0.05];
\draw (0,0) to[out=70, in=110] (3,0);
\end{tikzpicture} \end{array}+(q-1)\begin{array}{c}\begin{tikzpicture}[scale=0.4]
\draw [fill] (0,0) circle [radius=0.05];
\draw [fill] (1,0) circle [radius=0.05];
\draw [fill] (2,0) circle [radius=0.05];
\draw [fill] (3,0) circle [radius=0.05];
\draw (0,0) to[out=70, in=110] (3,0);
\draw (1,0) to[out=70, in=110] (2,0);
\end{tikzpicture} \end{array}
\end{align*}

Hence 
\[\SInd_{U_3}^{U_4}(\chi^{\begin{tikzpicture}[scale=0.2]
\draw [fill] (0,0) circle [radius=0.05];
\draw [fill] (1,0) circle [radius=0.05];
\draw [fill] (2,0) circle [radius=0.05];
\draw (0,0) to[out=70, in=110] (2,0);
\end{tikzpicture} })=\chi^{\begin{tikzpicture}[scale=0.2]
\draw [fill] (0,0) circle [radius=0.05];
\draw [fill] (1,0) circle [radius=0.05];
\draw [fill] (2,0) circle [radius=0.05];
\draw [fill] (3,0) circle [radius=0.05];
\draw (0,0) to[out=70, in=110] (2,0);
\end{tikzpicture} }+\chi^{\begin{tikzpicture}[scale=0.2]
\draw [fill] (0,0) circle [radius=0.05];
\draw [fill] (1,0) circle [radius=0.05];
\draw [fill] (2,0) circle [radius=0.05];
\draw [fill] (3,0) circle [radius=0.05];
\draw (0,0) to[out=70, in=110] (2,0);
\draw (2,0) to[out=70, in=110] (3,0);
\end{tikzpicture} }+\chi^{\begin{tikzpicture}[scale=0.2]
\draw [fill] (0,0) circle [radius=0.05];
\draw [fill] (1,0) circle [radius=0.05];
\draw [fill] (2,0) circle [radius=0.05];
\draw [fill] (3,0) circle [radius=0.05];
\draw (0,0) to[out=70, in=110] (2,0);
\draw (1,0) to[out=70, in=110] (3,0);
\end{tikzpicture} }+(q-1)\chi^{\begin{tikzpicture}[scale=0.2]
\draw [fill] (0,0) circle [radius=0.05];
\draw [fill] (1,0) circle [radius=0.05];
\draw [fill] (2,0) circle [radius=0.05];
\draw [fill] (3,0) circle [radius=0.05];
\draw (0,0) to[out=70, in=110] (3,0);
\end{tikzpicture} }+(q-1)\chi^{\begin{tikzpicture}[scale=0.2]
\draw [fill] (0,0) circle [radius=0.05];
\draw [fill] (1,0) circle [radius=0.05];
\draw [fill] (2,0) circle [radius=0.05];
\draw [fill] (3,0) circle [radius=0.05];
\draw (0,0) to[out=70, in=110] (3,0);
\draw (1,0) to[out=70, in=110] (2,0);
\end{tikzpicture} }.\]

From Equation \eqref{eq: expansion superinduction} we know then that 
\[\frac{\langle \SInd_{U_3}^{U_4}(\chi^{\begin{tikzpicture}[scale=0.2]
\draw [fill] (0,0) circle [radius=0.05];
\draw [fill] (1,0) circle [radius=0.05];
\draw [fill] (2,0) circle [radius=0.05];
\draw (0,0) to[out=70, in=110] (2,0);
\end{tikzpicture} }),\chi^{\begin{tikzpicture}[scale=0.2]
\draw [fill] (0,0) circle [radius=0.05];
\draw [fill] (1,0) circle [radius=0.05];
\draw [fill] (2,0) circle [radius=0.05];
\draw [fill] (3,0) circle [radius=0.05];
\draw (0,0) to[out=70, in=110] (2,0);
\draw (1,0) to[out=70, in=110] (3,0);
\end{tikzpicture} }\rangle}{\langle\chi^{\begin{tikzpicture}[scale=0.2]
\draw [fill] (0,0) circle [radius=0.05];
\draw [fill] (1,0) circle [radius=0.05];
\draw [fill] (2,0) circle [radius=0.05];
\draw [fill] (3,0) circle [radius=0.05];
\draw (0,0) to[out=70, in=110] (2,0);
\draw (1,0) to[out=70, in=110] (3,0);
\end{tikzpicture} },\chi^{\begin{tikzpicture}[scale=0.2]
\draw [fill] (0,0) circle [radius=0.05];
\draw [fill] (1,0) circle [radius=0.05];
\draw [fill] (2,0) circle [radius=0.05];
\draw [fill] (3,0) circle [radius=0.05];
\draw (0,0) to[out=70, in=110] (2,0);
\draw (1,0) to[out=70, in=110] (3,0);
\end{tikzpicture} }\rangle}=1\]
and from the definition of $\kappa$ we obtain
\[\kappa(\pi,\sigma)=\frac{\langle \SInd_{U_3}^{U_4}(\chi^{\begin{tikzpicture}[scale=0.2]
\draw [fill] (0,0) circle [radius=0.05];
\draw [fill] (1,0) circle [radius=0.05];
\draw [fill] (2,0) circle [radius=0.05];
\draw (0,0) to[out=70, in=110] (2,0);
\end{tikzpicture} }),\chi^{\begin{tikzpicture}[scale=0.2]
\draw [fill] (0,0) circle [radius=0.05];
\draw [fill] (1,0) circle [radius=0.05];
\draw [fill] (2,0) circle [radius=0.05];
\draw [fill] (3,0) circle [radius=0.05];
\draw (0,0) to[out=70, in=110] (2,0);
\draw (1,0) to[out=70, in=110] (3,0);
\end{tikzpicture} }\rangle}{\langle\chi^{\begin{tikzpicture}[scale=0.2]
\draw [fill] (0,0) circle [radius=0.05];
\draw [fill] (1,0) circle [radius=0.05];
\draw [fill] (2,0) circle [radius=0.05];
\draw [fill] (3,0) circle [radius=0.05];
\draw (0,0) to[out=70, in=110] (2,0);
\draw (1,0) to[out=70, in=110] (3,0);
\end{tikzpicture} },\chi^{\begin{tikzpicture}[scale=0.2]
\draw [fill] (0,0) circle [radius=0.05];
\draw [fill] (1,0) circle [radius=0.05];
\draw [fill] (2,0) circle [radius=0.05];
\draw [fill] (3,0) circle [radius=0.05];
\draw (0,0) to[out=70, in=110] (2,0);
\draw (1,0) to[out=70, in=110] (3,0);
\end{tikzpicture} }\rangle}\cdot\frac{\langle\chi^{\begin{tikzpicture}[scale=0.2]
\draw [fill] (0,0) circle [radius=0.05];
\draw [fill] (1,0) circle [radius=0.05];
\draw [fill] (2,0) circle [radius=0.05];
\draw [fill] (3,0) circle [radius=0.05];
\draw (0,0) to[out=70, in=110] (2,0);
\draw (1,0) to[out=70, in=110] (3,0);
\end{tikzpicture} },\chi^{\begin{tikzpicture}[scale=0.2]
\draw [fill] (0,0) circle [radius=0.05];
\draw [fill] (1,0) circle [radius=0.05];
\draw [fill] (2,0) circle [radius=0.05];
\draw [fill] (3,0) circle [radius=0.05];
\draw (0,0) to[out=70, in=110] (2,0);
\draw (1,0) to[out=70, in=110] (3,0);
\end{tikzpicture} }\rangle}{\langle\chi^{\begin{tikzpicture}[scale=0.2]
\draw [fill] (0,0) circle [radius=0.05];
\draw [fill] (1,0) circle [radius=0.05];
\draw [fill] (2,0) circle [radius=0.05];
\draw (0,0) to[out=70, in=110] (2,0);
\end{tikzpicture} },\chi^{\begin{tikzpicture}[scale=0.2]
\draw [fill] (0,0) circle [radius=0.05];
\draw [fill] (1,0) circle [radius=0.05];
\draw [fill] (2,0) circle [radius=0.05];
\draw (0,0) to[out=70, in=110] (2,0);
\end{tikzpicture} }\rangle}=1\cdot \frac{t^2\cdot q}{t}=tq\]
since $\langle\chi^{\pi },\chi^{\pi }\rangle=t^{d(\pi)}q^{\crs(\pi)}$.

\end{example}

In Figure \ref{figure: bratteli diagram for set partitions} we show the beginning of the Bratteli diagram for set partitions.
 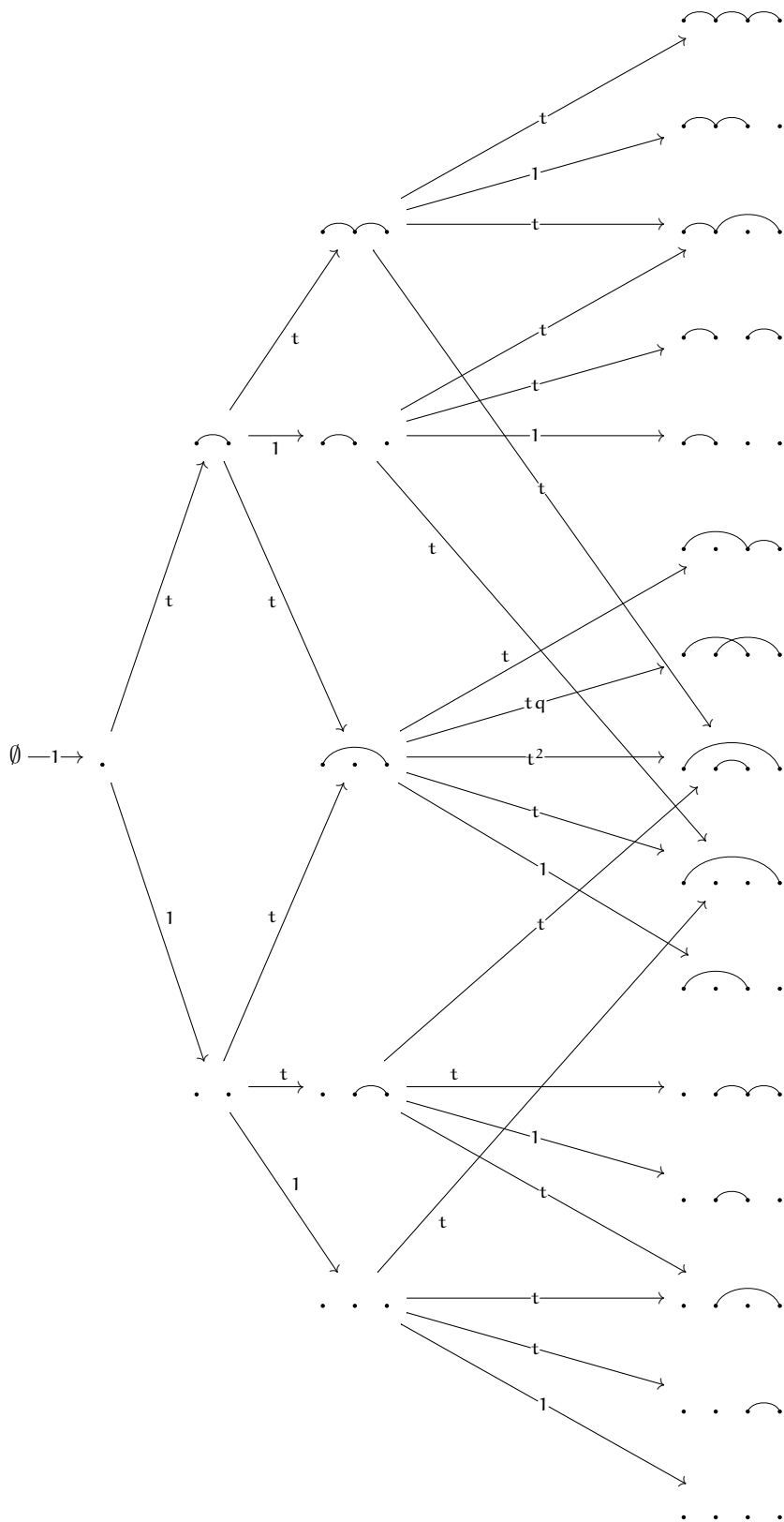
\begin{figure}\scalebox{.9}{
  \[\xymatrix{&&&&&&&\begin{array}{c}
 \begin{tikzpicture}[scale=0.5]
\draw [fill] (0,0) circle [radius=0.05];
\draw [fill] (1,0) circle [radius=0.05];
\draw [fill] (2,0) circle [radius=0.05];
\draw [fill] (3,0) circle [radius=0.05];
\draw (0,0) to[out=70, in=110] (1,0);
\draw (1,0) to[out=70, in=110] (2,0);
\draw (2,0) to[out=70, in=110] (3,0);
\end{tikzpicture}\end{array} \\
&&&&&&&\begin{array}{c}
 \begin{tikzpicture}[scale=0.5]
\draw [fill] (0,0) circle [radius=0.05];
\draw [fill] (1,0) circle [radius=0.05];
\draw [fill] (2,0) circle [radius=0.05];
\draw [fill] (3,0) circle [radius=0.05];
\draw (0,0) to[out=70, in=110] (1,0);
\draw (1,0) to[out=70, in=110] (2,0);
\end{tikzpicture}\end{array} \\
&&&\begin{array}{c}
 \begin{tikzpicture}[scale=0.5]
\draw [fill] (0,0) circle [radius=0.05];
\draw [fill] (1,0) circle [radius=0.05];
\draw [fill] (2,0) circle [radius=0.05];
\draw (0,0) to[out=70, in=110] (1,0);
\draw (1,0) to[out=70, in=110] (2,0);
\end{tikzpicture}\end{array}\ar[rrrruu]|-{t}\ar[rrrru]|-{1}\ar[rrrr]|-{t}\ar[rrrrddddd]|-{t}&&&&\begin{array}{c}
 \begin{tikzpicture}[scale=0.5]
\draw [fill] (0,0) circle [radius=0.05];
\draw [fill] (1,0) circle [radius=0.05];
\draw [fill] (2,0) circle [radius=0.05];
\draw [fill] (3,0) circle [radius=0.05];
\draw (0,0) to[out=70, in=110] (1,0);
\draw (1,0) to[out=70, in=110] (3,0);
\end{tikzpicture}\end{array}\\
&&&&&&&\begin{array}{c}
 \begin{tikzpicture}[scale=0.5]
\draw [fill] (0,0) circle [radius=0.05];
\draw [fill] (1,0) circle [radius=0.05];
\draw [fill] (2,0) circle [radius=0.05];
\draw [fill] (3,0) circle [radius=0.05];
\draw (0,0) to[out=70, in=110] (1,0);
\draw (2,0) to[out=70, in=110] (3,0);
\end{tikzpicture}\end{array} \\
&&\begin{array}{c}
 \begin{tikzpicture}[scale=0.5]
\draw [fill] (0,0) circle [radius=0.05];
\draw [fill] (1,0) circle [radius=0.05];
\draw (0,0) to[out=70, in=110] (1,0);
\end{tikzpicture}\end{array}\ar[ruu]_-{t}\ar[r]_-{1}\ar[rddd]_-{t} &\begin{array}{c}
 \begin{tikzpicture}[scale=0.5]
\draw [fill] (0,0) circle [radius=0.05];
\draw [fill] (1,0) circle [radius=0.05];
\draw [fill] (2,0) circle [radius=0.05];
\draw (0,0) to[out=70, in=110] (1,0);
\end{tikzpicture}\end{array}\ar[rrrr]|-{1}\ar[rrrru]|-{t}\ar[rrrruu]|-{t}\ar[rrrrdddd]_<<<<<<<<<<<<<<<<{t}&&&&\begin{array}{c}
 \begin{tikzpicture}[scale=0.5]
\draw [fill] (0,0) circle [radius=0.05];
\draw [fill] (1,0) circle [radius=0.05];
\draw [fill] (2,0) circle [radius=0.05];
\draw [fill] (3,0) circle [radius=0.05];
\draw (0,0) to[out=70, in=110] (1,0);
\end{tikzpicture}\end{array} \\
&&&&&&&\begin{array}{c}
 \begin{tikzpicture}[scale=0.5]
\draw [fill] (0,0) circle [radius=0.05];
\draw [fill] (1,0) circle [radius=0.05];
\draw [fill] (2,0) circle [radius=0.05];
\draw [fill] (3,0) circle [radius=0.05];
\draw (0,0) to[out=70, in=110] (2,0);
\draw (2,0) to[out=70, in=110] (3,0);
\end{tikzpicture}\end{array} \\
&&&&&&&\begin{array}{c}
 \begin{tikzpicture}[scale=0.5]
\draw [fill] (0,0) circle [radius=0.05];
\draw [fill] (1,0) circle [radius=0.05];
\draw [fill] (2,0) circle [radius=0.05];
\draw [fill] (3,0) circle [radius=0.05];
\draw (0,0) to[out=70, in=110] (2,0);
\draw (1,0) to[out=70, in=110] (3,0);
\end{tikzpicture}\end{array} \\
\emptyset\ar[r]|-{1}&\begin{array}{c}
 \begin{tikzpicture}[scale=0.5]
\draw [fill] (0,0) circle [radius=0.05];
\end{tikzpicture}\end{array}\ar[ruuu]_{t}\ar[rddd]^{1}&&\begin{array}{c}
 \begin{tikzpicture}[scale=0.5]
\draw [fill] (0,0) circle [radius=0.05];
\draw [fill] (1,0) circle [radius=0.05];
\draw [fill] (2,0) circle [radius=0.05];
\draw (0,0) to[out=70, in=110] (2,0);
\end{tikzpicture}\end{array} \ar[rrrruu]^<<<<<<<<<<<<<<<<<<<<{t}\ar[rrrru]|-{tq}\ar[rrrr]|-{t^2}\ar[rrrrd]|-{t}\ar[rrrrdd]|-{1}&&&&\begin{array}{c}
 \begin{tikzpicture}[scale=0.5]
\draw [fill] (0,0) circle [radius=0.05];
\draw [fill] (1,0) circle [radius=0.05];
\draw [fill] (2,0) circle [radius=0.05];
\draw [fill] (3,0) circle [radius=0.05];
\draw (0,0) to[out=70, in=110] (3,0);
\draw (1,0) to[out=70, in=110] (2,0);
\end{tikzpicture}\end{array} \\
&&&&&&&\begin{array}{c}
 \begin{tikzpicture}[scale=0.5]
\draw [fill] (0,0) circle [radius=0.05];
\draw [fill] (1,0) circle [radius=0.05];
\draw [fill] (2,0) circle [radius=0.05];
\draw [fill] (3,0) circle [radius=0.05];
\draw (0,0) to[out=70, in=110] (3,0);
\end{tikzpicture}\end{array} \\
&&&&&&&\begin{array}{c}
 \begin{tikzpicture}[scale=0.5]
\draw [fill] (0,0) circle [radius=0.05];
\draw [fill] (1,0) circle [radius=0.05];
\draw [fill] (2,0) circle [radius=0.05];
\draw [fill] (3,0) circle [radius=0.05];
\draw (0,0) to[out=70, in=110] (2,0);
\end{tikzpicture}\end{array} \\
&&\begin{array}{c}
 \begin{tikzpicture}[scale=0.5]
\draw [fill] (0,0) circle [radius=0.05];
\draw [fill] (1,0) circle [radius=0.05];
\end{tikzpicture}\end{array}\ar[ruuu]^{t}\ar[r]^{t}\ar[rdd]^{1} &\begin{array}{c}
 \begin{tikzpicture}[scale=0.5]
\draw [fill] (0,0) circle [radius=0.05];
\draw [fill] (1,0) circle [radius=0.05];
\draw [fill] (2,0) circle [radius=0.05];
\draw (1,0) to[out=70, in=110] (2,0);
\end{tikzpicture}\end{array}\ar[rrrr]^<<<<<<<<{t}\ar[rrrrd]|-{1}\ar[rrrrdd]|-{t}\ar[rrrruuu]|-{t}&&&&\begin{array}{c}
 \begin{tikzpicture}[scale=0.5]
\draw [fill] (0,0) circle [radius=0.05];
\draw [fill] (1,0) circle [radius=0.05];
\draw [fill] (2,0) circle [radius=0.05];
\draw [fill] (3,0) circle [radius=0.05];
\draw (1,0) to[out=70, in=110] (2,0);
\draw (2,0) to[out=70, in=110] (3,0);
\end{tikzpicture}\end{array} \\
&&&&&&&\begin{array}{c}
 \begin{tikzpicture}[scale=0.5]
\draw [fill] (0,0) circle [radius=0.05];
\draw [fill] (1,0) circle [radius=0.05];
\draw [fill] (2,0) circle [radius=0.05];
\draw [fill] (3,0) circle [radius=0.05];
\draw (1,0) to[out=70, in=110] (2,0);
\end{tikzpicture}\end{array} \\
&&&\begin{array}{c}
 \begin{tikzpicture}[scale=0.5]
\draw [fill] (0,0) circle [radius=0.05];
\draw [fill] (1,0) circle [radius=0.05];
\draw [fill] (2,0) circle [radius=0.05];
\end{tikzpicture}\end{array} \ar[rrrr]|-{t}\ar[rrrrd]|-{t}\ar[rrrrdd]|-{1}\ar[rrrruuuu]_<<<<<<<<<<<<<{t}&&&&\begin{array}{c}
 \begin{tikzpicture}[scale=0.5]
\draw [fill] (0,0) circle [radius=0.05];
\draw [fill] (1,0) circle [radius=0.05];
\draw [fill] (2,0) circle [radius=0.05];
\draw [fill] (3,0) circle [radius=0.05];
\draw (1,0) to[out=70, in=110] (3,0);
\end{tikzpicture}\end{array} \\
&&&&&&&\begin{array}{c}
 \begin{tikzpicture}[scale=0.5]
\draw [fill] (0,0) circle [radius=0.05];
\draw [fill] (1,0) circle [radius=0.05];
\draw [fill] (2,0) circle [radius=0.05];
\draw [fill] (3,0) circle [radius=0.05];
\draw (2,0) to[out=70, in=110] (3,0);
\end{tikzpicture}\end{array} \\
&&&&&&&\begin{array}{c}
 \begin{tikzpicture}[scale=0.5]
\draw [fill] (0,0) circle [radius=0.05];
\draw [fill] (1,0) circle [radius=0.05];
\draw [fill] (2,0) circle [radius=0.05];
\draw [fill] (3,0) circle [radius=0.05];
\end{tikzpicture}\end{array} }\]}
 \caption[Beginning of the Bratteli diagram for set partitions]{The Bratteli diagram of set partitions of size $\leq 4$, with the weights of the edges written explicitly. Here $t=q-1$.}\label{figure: bratteli diagram for set partitions} \end{figure}

\subsection{A combinatorial interpretation of the superplancherel measure}\label{section combinatorial interpretation}
We associate to $\pi\vdash[n]$ the following set $\mathcal{J}_{\pi}\subseteq U_n$: a matrix $A$ belongs to $\mathcal{J}_{\pi}$ if and only if
\begin{itemize}
 \item if $(i,j)\in D(\pi)$ then $A_{i,j}\in \mathbb{K}\setminus\{0\}$;
 \item if $(i,j)\in \reg(\pi)\setminus D(\pi)$ then $A_{i,j}=0$;
 \item if $(i,j)\in \Sing(\pi)$ then $A_{i,j}\in \mathbb{K}$.
\end{itemize}

 \begin{figure}[H]\begin{center}
\[\mathcal{J}_{\pi}=\left[\begin{array}{ccccc}
       1&\bullet&\bullet&\ast&0\\
       0&1&\ast&\bullet&0\\
       0&0&1&\bullet&\ast\\
       0&0&0&1&\bullet\\
       0&0&0&0&1
      \end{array}\right]\]
\end{center}
 \caption[Example of $\mathcal{J}_{\pi}$]{Example of $\mathcal{J}_{\pi}$, where  $\pi=\{\{1,4\},\{2,3,5\}\}$. Here $\ast$ means that in that position there is an element of $\mathbb{K}^{\times}$, and $\bullet$ is an element of $\mathbb{K}.$}\label{example of J_pi}
\end{figure}
We say that a matrix $A$ in $\mathcal{J}_{\pi}$ is \emph{canonical} if
\[A_{i,j}=\left\{\begin{array}{lcl}1&\mbox{if}&(i,j)\in D(\pi)\mbox{ or }i=j\\0&&\mbox{otherwise}\end{array}\right.\]
In this section we show that the sets $\mathcal{J}_{\pi}$ partition of the group $U_n$ and that given a matrix $A\in U_n$ the set partition $\pi$ such that $A\in\mathcal{J}_{\pi}$ can be computed efficiently. We stress out that $\mathcal{J}_{\pi}$ are not the superclasses for this supercharacter theory, and in general they are not even union of conjugacy classes.

The following algorithm takes as input a matrix $A\in U_n$ and gives as output a canonical matrix $\tilde{A}\in\mathcal{J}_{\pi}$ for some $\pi$. The algorithm will consists of $n$ steps, at the step $k$ we will consider the $k$-th diagonal $d_k$ of $A^{k-1}$ starting from the upper-right corner, where 
\[d_k(A)=\{A_{1,n-k+1},A_{2,n-k+2},\ldots, A_{k,n}\}.\]

\begin{description}\label{lemma combinatorial algorithm}
\item STEP $0$: set $A^0=A.$
 \item STEP $1$: if $A^0_{1,n}=0$ set $A^1=A^0$;\\
  if $A^0_{1,n}\neq 0$ set $A^1_{1,n}=1$ and all other entries in the same row on the left and on the same column below $A^1_{1,n}$, up to the diagonal, equal to $0$. Set $A^1_{i,j}=A^0_{i,j}$ for all other entries $(i,j)$.
  \item STEP $k$: for all $A^{k-1}_{i,n-k+i}\in d_k(A^{k-1})$ do the following: if $A^{k-1}_{i,n-k+i}=0$ set $A^k_{i,j}=A^{k-1}_{i,j}$ for each $j=1,\ldots, n$; if $A^{k-1}_{i,n-k+i}\neq 0$ set $A^k_{i,n-k+i}=1$ and all other entries in the same row on the left and on the same column below $A^k_{i,n-k+i}$, up to the diagonal, equal to $0$. All other entries of $A^k$ are equal to the ones of $A^{k-1}$. 
\end{description}
For an example of the algorithm see Figure \ref{figure matrices}.

\begin{lemma}
 Given a matrix $A\in U_n$, there exists a unique $\pi$ such that $A\in \mathcal{J}_{\pi}$. In other words, $U_n=\bigsqcup\limits_{\pi\vdash[n]}\mathcal{J}_{\pi}$.
\end{lemma}
\begin{proof}
We start the proof with the following two observations:
\begin{itemize}
 \item consider $A\in\mathcal{J}_{\pi}$ and $(i,j)\in D(\pi)$, so that $A_{i,j}\neq0$. The matrix $\tilde{A}$ which is equal to $A$ except in the entry $\tilde{A}_{i,j}$, in which we still have $\tilde{A}_{i,j}\neq 0$, still belongs to $\mathcal{J}_{\pi}$;
 \item consider $A\in\mathcal{J}_{\pi}$ and $(i,j)\in D(\pi)$. Define $\tilde{A}$ such that all entries are the same as the entries of $A$, but those in the $i$-th row on the left of $(i,j)$, up to the diagonal, and those on the $j$-th column below $(i,j)$, up to the diagonal. These are the entries $\tilde{A}_{k,l}$ which correspond to the pairs $(k,l)\in \Sing(\pi)$. Hence $\tilde{A}\in\mathcal{J}_{\pi}$.
\end{itemize}
From these observations it is clear that $A^{k-1}\in\mathcal{J}_{\pi}$ if and only if $A^k\in\mathcal{J}_{\pi}$. Since the output of the algorithm is a canonical representative of $\mathcal{J}_{\pi}$, then it follows that for each matrix $A\in U_n$ there exists a unique $\pi\vdash[n]$ such that $A\in\mathcal{J}_{\pi}$.
\end{proof}

\begin{figure}
\begin{center}
 \[A^0=\left[\begin{array}{ccccc}1&0&5&2&1\\&1&2&0&0\\&&1&5&0\\&&&1&4\\&&&&1\end{array}\right], A^1=\left[\begin{array}{ccccc}1&0&0&0&1\\&1&2&0&0\\&&1&5&0\\&&&1&0\\&&&&1\end{array}\right]= A^2=A^3, A^4=\left[\begin{array}{ccccc}1&0&0&0&1\\&1&1&0&0\\&&1&1&0\\&&&1&0\\&&&&1\end{array}\right]\]
 \end{center}
 \caption[Example of the algorithm of Section \ref{section combinatorial interpretation}]{Example of the algorithm described above; we start from a matrix $A=A^0\in U_n$ and we obtain a matrix $A^4$ which is canonical for the set $\mathcal{J}_{\pi}$ for $\pi=\{\{1,5\},\{2,3,4\}\}$. In general, if during the algorithm we find a non-zero term on the $k$-th diagonal then this corresponds necessarily to an arc of $\pi$ and not to a singular pair.}\label{figure matrices}
\end{figure}
\begin{proposition}
 For any $\pi\vdash[n]$, \[\SPl(\pi)=\frac{|\mathcal{J}_{\pi}|}{|U_n|}.\] Equivalently, the superplancherel measure of $\pi$ is the probability of choosing a random matrix in $U_n$ which belongs to $\mathcal{J}_{\pi}$.
\end{proposition}
\begin{proof}
 It is enough to prove that \[|\mathcal{J}_{\pi}|=\frac{(q-1)^{d(\pi)}\cdot q^{2\dim(\pi)-2 d(\pi)}}{q^{\crs(\pi)}};\] in order to do so we calculate $|\Sing(\pi)|$. Given a pair $(i,j)$ we write $\sigma^n_{(i,j)}\vdash[n]$ for the set partition such that $D(\sigma_{(i,j)}^n)=\{(i,j)\}$. Then \[\Sing(\sigma^n_{(i,j)})=\{(i,i+1),\ldots,(i,j-1),(i+1,j),\ldots,(j-1,j)\}.\] The cardinality $|\Sing(\sigma^n_{(i,j)})|$ is clearly $2(j-i-1)$. It is immediate to see that 
\[\Sing(\pi)=\bigcup_{(i,j)\in D(\pi)}\Sing(\sigma^n_{(i,j)}).\]
We use the inclusion-exclusion formula to calculate $|\Sing(\pi)|$: notice that for a pair of different arcs $(i,j),(k,l)$ with $i<k$ then 
\[\Sing(\sigma^n_{(i,j)})\cap\Sing(\sigma^n_{(k,l)})=\left\{\begin{array}{lcl}\{(j,k)\}&\mbox{if}&i<k<j<l  \\ \emptyset&&\mbox{otherwise}\end{array}\right.\]
Moreover, for any triplet of different arcs $(i,j),(k,l),(r,s)\in D(\pi)$ we have \[\Sing(\sigma^n_{(i,j)})\cap\Sing(\sigma^n_{(k,l)})\cap\Sing(\sigma^n_{(r,s)})=\emptyset.\] Thus
\[|\Sing(\pi)|=\sum_{(i,j)\in D(\pi)}|\Sing(\sigma^n_{(i,j)})|-\sum_{\substack{(i,j)\neq(k,l)\\\mbox{in }D(\pi)}}|\Sing(\sigma^n_{(i,j)})\cap\Sing(\sigma^n_{(k,l)})|,\]
hence
\begin{align*}
 |\Sing(\pi)|&=\sum_{(i,j)\in D(\pi)}2(j-i-1)-\sharp\{(i,j),(l,k)\in D(\pi)\mbox{ s.t. }i<l<j<k\}\\&=2(\dim(\pi)-d(\pi))-\crs(\pi),
\end{align*}

which concludes the proof.
\end{proof}
We use this interpretation to generate the second picture of Figure \ref{fig:boat1}: we generate a random matrix $A\in U_n$, then we apply the algorithm described above in order to reduce $A$ to a canonical matrix $\tilde{A}$. We define $\pi$ as the set partition whose arcs are exactly the non zero entries of $\tilde{A}$, so that $\tilde{A}\in\mathcal{J}_{\pi}$. Such a set partition is randomly distributed with the superplancherel measure.
\smallskip

This construction lead to a combinatorial interpretation of the superplancherel measure. In the classic setting of irreducible representations of the symmetric group the combinatorial interpretation of the Plancherel measure corresponds to the RSK algorithm, which gives a bijection between permutations and pair of standard Young tableaux; see \cite{sagan2013symmetric} for an extensive introduction.

\section{Set partitions as measures on the unit square}\label{section set partitions as measures}
In this section we will describe an embedding of set partitions into particular measures on a subset $\Delta=\{(x,y)\in [0,1]^2\mbox{ s.t. }y\geq x\}$ of the unit square. We settle first some notation: if $A\subseteq \R^2$ is measurable then we write $\lambda_A$ for the uniform measure on $A$ of total mass equal to $1$, that is, $\int_A\,d\lambda_A=1$; given $n\in\N, i<j\leq n$, set 
\[A_{i,j}=\left\{(x,y)\in\R^2\mbox{ s.t. }\frac{i-1}{n}\leq x\leq \frac{i}{n}, \frac{j-1}{n}\leq y\leq \frac{j}{n}\right\}\subseteq\Delta;\]
for $\pi\vdash[n]$ we will write $A_{\pi}:=\bigcup_{(i,j)\in D(\pi)}A_{i,j}$ and $\mu_{\pi}=\frac{1}{n}\sum_{(i,j)\in D(\pi)}\lambda_{A_{i,j}}$. An example is given in Figure \ref{example of the measure mu_pi}.

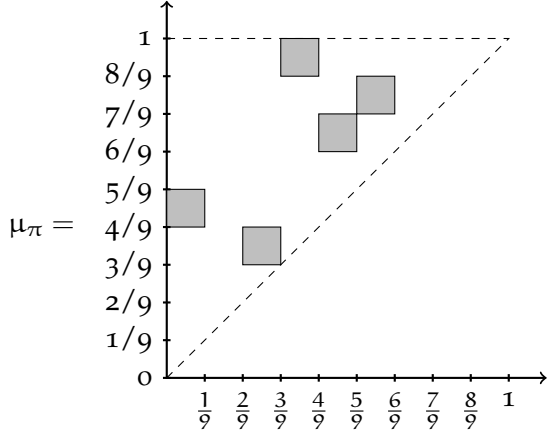
\begin{figure}[H]
 \begin{center}
 \[\mu_{\pi}=\begin{array}{c}
\begin{tikzpicture}[scale=0.5]
\draw [thick, <->] (0,10) -- (0,0) -- (10,0);
\draw [dashed](0,0)--(9,9)--(0,9);
\draw [fill=lightgray](0,4) rectangle (1,5);
\draw [fill=lightgray](4,6) rectangle (5,7);
\draw [fill=lightgray](2,3) rectangle (3,4);
\draw [fill=lightgray](3,8) rectangle (4,9);
\draw [fill=lightgray](5,7) rectangle (6,8);
\draw [thick] (-.1,9) node[left]{1} -- (.1,9);
\draw [thick] (-.1,8) node[left]{8/9} -- (.1,8);
\draw [thick] (-.1,7) node[left]{7/9} -- (.1,7);
\draw [thick] (-.1,6) node[left]{6/9} -- (.1,6);
\draw [thick] (-.1,5) node[left]{5/9} -- (.1,5);
\draw [thick] (-.1,4) node[left]{4/9} -- (.1,4);
\draw [thick] (-.1,3) node[left]{3/9} -- (.1,3);
\draw [thick] (-.1,2) node[left]{2/9} -- (.1,2);
\draw [thick] (-.1,1) node[left]{1/9} -- (.1,1);
\draw [thick] (-.1,0) node[left]{0} -- (.1,0);
\draw [thick] (9,-.1) node[below]{1} -- (9,.1);
\draw [thick] (8,-.1) node[below]{$\frac{8}{9}$} -- (8,.1);
\draw [thick] (7,-.1) node[below]{$\frac{7}{9}$} -- (7,.1);
\draw [thick] (6,-.1) node[below]{$\frac{6}{9}$} -- (6,.1);
\draw [thick] (5,-.1) node[below]{$\frac{5}{9}$} -- (5,.1);
\draw [thick] (4,-.1) node[below]{$\frac{4}{9}$} -- (4,.1);
\draw [thick] (3,-.1) node[below]{$\frac{3}{9}$} -- (3,.1);
\draw [thick] (2,-.1) node[below]{$\frac{2}{9}$} -- (2,.1);
\draw [thick] (1,-.1) node[below]{$\frac{1}{9}$} -- (1,.1);
 \end{tikzpicture}\end{array}\]
 \end{center}
 \caption[Measure associated to a set partition]{Example of the measure $\mu_{\pi}$ on $\Delta$ corresponding to $\pi=\left\{\{1,5,7\},\{2\},\{3,4,9\},\{6,8\}\right\}\vdash[9]$ of Figure \ref{picture standard representation}. Everywhere but the gray areas has zero weight, while the gray areas represent where the measure has uniform weight. Each square has total weight $\frac{1}{n}=\frac{1}{9}$, so that the total weight is $\int_{\Delta}\,d\mu=\frac{5}{9}$.}\label{example of the measure mu_pi}
 \end{figure}
\begin{defi}
Let $X\subseteq \R^2$, set $\pi_1$ (resp. $\pi_2$) the projection into the first (resp. the second) coordinate. A measure $\mu$ on $X$ is said to have \emph{uniform marginals} if for each interval $I\subseteq \pi_1(X)$ and $J\subseteq\pi_2(X)$
\[\mu(I\times \pi_2(X))=|I|,\]
\[\mu(\pi_1(X)\times J)=|J|.\]
Similarly, the measure $\mu$ has \emph{subuniform marginals} if for each interval $I\subseteq \pi_1(X)$ and $J\subseteq\pi_2(X)$
\[\mu(I\times \pi_2(X))\leq|I|,\]
\[\mu(\pi_1(X)\times J)\leq|J|.\]
\end{defi}
As a measure on $\Delta$, $\mu_{\pi}$ has subuniform marginal and in particular $\int_{\Delta}\,d\mu\leq 1$. We call \emph{subprobability} a positive measure with total weight less than or equal to $1$, so that $\mu_{\pi}$ is a subprobability. We will sometimes deal with measures $\mu$ of $\Delta$ as measures on the whole square unit interval $[0,1]^2$, assuming that $\mu([0,1]^2\setminus \Delta)=0$.

\subsection{Statistics of set partitions approximated by integrals}
We define the following space of measures:
\[\Gamma:=\{\mbox{subprobabilities }\mu \mbox{ on }\Delta\mbox{ s.t. }\mu\mbox{ has subuniform marginals}\};\]
In this new setting we can describe the values of $d(\pi), \dim(\pi),\crs(\pi)$ as follows:
\begin{lemma}
Let $\pi\vdash[n]$, so that $\mu_{\pi}\in\Gamma$, then 
\begin{enumerate}
\item $d(\pi)\in O(n)$;
\item $\dim(\pi)=n^2\int_{\Delta}(y-x)\,d\mu_{\pi}(x,y);$
\item $\crs(\pi)=n^2\int_{\Delta^2}\mathbb{1}[x_1<x_2<y_1<y_2]\,d\mu_{\pi}(x_1,y_1)\,d\mu_{\pi}(x_2,y_2) + O(n).$
\end{enumerate}
\end{lemma}
\begin{proof}
 \begin{enumerate}
  \item It is immediate to see that $d(\pi)\leq n-1$, with equality if and only if there is only one block: $\pi=\{\{1,2,\ldots,n\}\}$.
  \item Since $\mu_{\pi}=\sum_{(i,j)\in D(\pi)}\mu_{\{\{i,j\}\}}$, it is enough to prove the statement for $\pi=\sigma^n_{(i,j)}$ (recall that $D(\sigma_{(i,j)}^n)=\{(i,j)\}$). Notice moreover that for $f\colon \R^2\to \R$ measurable
  \[\int_{A_{i,j}} f(x,y)\,d\lambda_{A_{i,j}}=\frac{\int_{A_{i,j}} f(x,y)\,dx\,dy}{\int_{A_{i,j}} \,dx\,dy}=n^2\int_{A_{i,j}} f(x,y)\,dx\,dy.\]
Therefore
  \begin{align*}
   n^2\int_{\Delta}(y-x)\,d\mu_{\pi}(x,y)&=n\int_{\Delta}(y-x)\,d\lambda_{A_{i,j}}\\&= n^3\int_{A_{i,j}} (y-x)\,dx\,dy\\&=j-i\\&=\dim(\pi).
  \end{align*}
\item Similarly as before, we have that for $A,B$ bounded subsets of $\R^2$ and $f\colon \R^4\to \R$
\[\int_{A\times B} f(x_1,y_1,x_2,y_2)\,d\lambda_A(x_1,y_1)\,d\lambda_B(x_2,y_2)=\frac{\int_{A\times B} f(x_1,y_1,x_2,y_2)\,dx_1\,dy_1\,dx_2\,dy_2}{\int_A \,dx\,dy\int_B\,dx\,dy}.\]
We see that
\begin{align*}
 &n^2\int_{\Delta^2}\mathbb{1}[x_1<x_2<y_1<y_2]\,d\mu_{\pi}(x_1,y_1)\,d\mu_{\pi}(x_2,y_2)\\
 =&\sum_{(i,j),(k,l)\in D(\pi)}\int_{\Delta^2}\mathbb{1}[x_1<x_2<y_1<y_2]\,d\lambda_{A_{(i,j)}}(x_1,y_1)\,d\lambda_{A_{(k,l)}}(x_2,y_2)\\
 =&n^4\sum_{(i,j),(k,l)\in D(\pi)}\int_{A_{(i,j)}\times A_{(k,l)}}\mathbb{1}[x_1<x_2<y_1<y_2]\,dx_1\,dy_1\,dx_2\,dy_2.\\
\end{align*}
Suppose $(i,j),(k,l)\in D(\pi)$ and call 
\[\mathcal{I}[i,j,k,l]=\int_{A_{(i,j)}\times A_{(k,l)}}\mathbb{1}[x_1<x_2<y_1<y_2]\,dx_1\,dy_1\,dx_2\,dy_2,\]
direct computations show that
\begin{itemize}
 \item if $i<k<j<l$ then $\mathcal{I}[i,j,k,l]=\frac{1}{n^4}$;
 \item if $i<j=k<l$ then $\mathcal{I}[i,j,k,l]=\frac{1}{2n^4}$;
 \item if $i=k<j=l$ then $\mathcal{I}[i,j,k,l]=\frac{1}{4n^4}$;
 \item in any other case $\mathcal{I}[i,j,k,l]=0$.
\end{itemize}
Hence
\begin{multline*}
 \crs(\pi)=n^2\int_{\Delta^2}\mathbb{1}[x_1<x_2<y_1<y_2]\,d\mu_{\pi}(x_1,y_1)\,d\mu_{\pi}(x_2,y_2)\\-\frac{1}{4}d(\pi)-\frac{1}{2}\#\{\mbox{adjacent arcs of }\pi\},
\end{multline*}

where a pair of arcs $(i,j),(k,l)\in D(\pi)$ is \emph{adjacent} if $j=k$. Since the number of adjacent arcs is obviously less than the number of arcs, the second and the third terms of the RHS are $O(n)$, and we conclude.\qedhere\end{enumerate}
 \end{proof}

\subsection{Maximizing the entropy}
In this section our goal is to maximize the superplancherel measure. From the previous lemma we see that 
\begin{align*}
 &\SPl_{n}(\chi^{\pi})=\frac{1}{q^{\frac{n(n-1)}{2}}}\frac{q^{2\dim(\pi)-2 d(\pi)}}{(q-1)^{d(\pi)}q^{\crs(\pi)}}=\\
 &\exp\left(\log q\left(-\frac{n^2}{2}+\frac{n}{2}+\frac{\log (q-1)}{\log q}d(\pi)-2d(\pi)+2\dim(\pi)-\crs(\pi)\right)\right)=\\
 &\begin{multlined}[t][10.5cm]
   \left.\exp\left(-n^2\log q\left(\frac{1}{2}-2\int_{\Delta}(y-x)\,d\mu_{\pi}(x,y) +\right.\right.\right.\\
   \left.\left.\int_{\mathrlap{\Delta^2}}\mathbb{1}[x_1<x_2<y_1<y_2]\,d\mu_{\pi}(x_1,y_1)\,d\mu_{\pi}(x_2,y_2)\right)+ O(n)\right). 
  \end{multlined}
\end{align*}
For each measure $\mu\in\Gamma$ we set thus 
\begin{itemize}
 \item $I_1(\mu):=\int_{\Delta}(y-x)\,d\mu$;
 \item $I_2(\mu):=\int_{\Delta^2}\mathbb{1}[x_1<x_2<y_1<y_2]\,d\mu(x_1,y_1)\,d\mu(x_2,y_2)$;
 \item $I(\mu):=\frac{1}{2}-2I_1(\mu)+I_2(\mu)$.
\end{itemize}
Hence for $\pi\vdash[n]$ we have
\begin{equation}\label{formula of superplancherel with entropy}
 \SPl_{n}(\chi^{\pi})=\exp\left(-n^2\log q\cdot I(\mu_{\pi})+O(n)\right).
\end{equation}
We set
\[\tilde{\Gamma}:=\{\mbox{subprobabilities }\mu \mbox{ on }[0,1/2]\times[1/2,1]\mbox{ s.t. }\mu\mbox{ has uniform marginals}\}.\]
Recall that for a measurable function $f$ and a measure $\mu$ the push forward is
\[f_{\ast}\mu(A):=\mu(f_{\ast}^{-1}(A))\]
for each $A$ measurable. Consider $f(x)=1-x$ and the Lebesgue measure $\Leb([0,1/2])$ on the interval $[0,1/2]$. Define $\Omega$ as $\Omega:=f_{\ast}\Leb([0,1/2])$ and notice that $\Omega\in\tilde{\Gamma}$. The goal of this section is to prove the following proposition:
\begin{proposition}\label{minimizing I}
 Consider $\mu\in\Gamma$, then $I(\mu)=0$ if and only if $\mu=\Omega$.
\end{proposition}
We will prove the proposition after studying the two functionals $I_1$ and $I_2$.
\begin{lemma}\label{maximizing I_1}
 Let $\mu\in\Gamma$, then $I_1(\mu)=\int_{\Delta}(y-x)\,d\mu\leq 1/4$, with equality if and only if $\mu\in\tilde{\Gamma}$.
\end{lemma}
\begin{proof}
 We show that for each $\mu\in\Gamma$ there exists a measure $\tilde{\mu}\in\Gamma$ and intervals $I_{\mu}\subseteq[0,1]$ and $J_{\mu}\subseteq[0,1]$ such that $\tilde{\mu}$ has uniform marginals as a measure of $I_{\mu}\times J_{\mu}$ and $I_1(\mu)\leq I_1(\tilde{\mu})$. This will be proved by ``squeezing'' the measure $\mu$ toward the top-left corner of $\Delta$. Set $I_{\mu}=[0,\mu(\Delta)]$, $J_{\mu}=[1-\mu(\Delta),1]$, $f_{\mu}(x)=\mu([0,x]\times [0,1])\leq x$ and $g_{\mu}(y):=1-\mu([0,1]\times[1-y,1])\geq y$. We define the measure $\tilde{\mu}$ as the push forward of $\mu$ by the function $(x,y)\to(f_{\mu}(x),g_{\mu}(y))$. It is evident that $\tilde{\mu}([0,1]^2\setminus\Delta)=0$. By construction, $\tilde{\mu}$ has uniform marginals on $I_{\mu}\times J_{\mu}$. Therefore
 \begin{align*}
  I_1(\tilde{\mu})&=\int_{\Delta}(v-u)\,d\tilde{\mu}(u,v)\\
  &=\int_{I_{\mu}\times J_{\mu}}(v-u)\,d\tilde{\mu}(u,v)\\
  &=\int_\Delta(g_{\mu}(y)-f_{\mu}(x))\,d\mu(x,y)\geq\int_{\Delta}(y-x)\,d\mu(x,y)=I_1(\mu),  
 \end{align*}
where the inequality comes from $g_{\mu}(y)\geq y$ and $f_{\mu}(x)\leq x$. Notice that we have $I_1(\tilde{\mu})=I(\mu)$ if and only if $f_{\mu}(x)=x$ and $g_{\mu}(y)=y$ almost everywhere according to the marginal of $\mu$ in, respectively, the $x$ and $y$ coordinates. This is equivalent to $\tilde{\mu}=\mu$. For an example of this construction, see Figure \ref{picture  for the proof of I_1}.
 \smallskip
 
 Set $l_{\mu}=\mu(\Delta)$. We show that $I_1(\tilde{\mu})=l_{\mu}(1-l_{\mu})$. We write $I_1(\tilde{\mu})=\int_{\Delta}y\,d\tilde{\mu}-\int_{\Delta}x\,d\tilde{\mu}$ and consider the two integrals separately. Observe that the $y$-marginal of $\tilde{\mu}$ is $\Leb([1-l_{\mu},1])$, the Lebesgue measure on the interval $[1-l_{\mu},1]$; hence
 \[\int_{y=0}^{y=1}y\int_{x=0}^{x=1}\,d\tilde{\mu}(x,y)=\int_{1-l_{\mu}}^1 y\,dy=l_{\mu}-\frac{l_{\mu}^2}{2}.\]
Similarly (the $x$-marginal of $\tilde{\mu}$ is $\Leb([0,l_{\mu}])$)
 \[\int_{x=0}^{x=1}x\int_{y=0}^{y=1}\,d\tilde{\mu}(x,y)=\int_0^{l_{\mu}}x\,dx=\frac{l_{\mu}^2}{2}.\]
 Therefore $I_1(\tilde{\mu})=\int_{\Delta}y\,d\tilde{\mu}-\int_{\Delta}x\,d\tilde{\mu}=l_{\mu}-l_{\mu}^2.$

Since $l_{\mu}\leq 1$, the maximal value of $l_{\mu}(1-l_{\mu})$ is obtained when $l_{\mu}=1/2$, in which case $I_1(\tilde{\mu})=1/4$. We showed thus that for each measure $\mu\in\Gamma$ there exists a measure $\tilde{\mu}\in\Gamma$ such that $I_1(\mu)\leq I_1(\tilde{\mu})\leq 1/4$, with equality if and only if $\mu\in\tilde{\Gamma}$, which concludes the proof.
\end{proof}

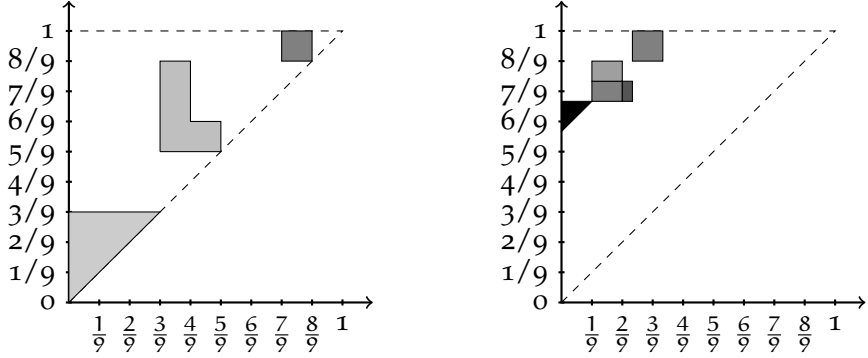
\begin{figure}
 \begin{center}
 \[ \begin{tikzpicture}[scale=0.4]
\draw [thick, <->] (0,10) -- (0,0) -- (10,0);
\draw [dashed](0,0)--(9,9)--(0,9);
\draw [thick] (-.1,9) node[left]{1} -- (.1,9);
\draw [thick] (-.1,8) node[left]{8/9} -- (.1,8);
\draw [thick] (-.1,7) node[left]{7/9} -- (.1,7);
\draw [thick] (-.1,6) node[left]{6/9} -- (.1,6);
\draw [thick] (-.1,5) node[left]{5/9} -- (.1,5);
\draw [thick] (-.1,4) node[left]{4/9} -- (.1,4);
\draw [thick] (-.1,3) node[left]{3/9} -- (.1,3);
\draw [thick] (-.1,2) node[left]{2/9} -- (.1,2);
\draw [thick] (-.1,1) node[left]{1/9} -- (.1,1);
\draw [thick] (-.1,0) node[left]{0} -- (.1,0);
\draw [thick] (9,-.1) node[below]{1} -- (9,.1);
\draw [thick] (8,-.1) node[below]{$\frac{8}{9}$} -- (8,.1);
\draw [thick] (7,-.1) node[below]{$\frac{7}{9}$} -- (7,.1);
\draw [thick] (6,-.1) node[below]{$\frac{6}{9}$} -- (6,.1);
\draw [thick] (5,-.1) node[below]{$\frac{5}{9}$} -- (5,.1);
\draw [thick] (4,-.1) node[below]{$\frac{4}{9}$} -- (4,.1);
\draw [thick] (3,-.1) node[below]{$\frac{3}{9}$} -- (3,.1);
\draw [thick] (2,-.1) node[below]{$\frac{2}{9}$} -- (2,.1);
\draw [thick] (1,-.1) node[below]{$\frac{1}{9}$} -- (1,.1);
\draw [fill={rgb:black,1;white,4}] (0,0) -- (3,3) -- (0,3) -- (0,0);
\draw [fill={rgb:black,1;white,3}] (3,5) -- (5,5) -- (5,6) -- (4,6) -- (4,8) -- (3,8) -- (3,5);
\draw [fill={rgb:black,1;white,1}] (7,8) rectangle (8,9);
\end{tikzpicture} \qquad\qquad \begin{tikzpicture}[scale=0.4]
\draw [thick, <->] (0,10) -- (0,0) -- (10,0);
\draw [dashed](0,0)--(9,9)--(0,9);
\draw [thick] (-.1,9) node[left]{1} -- (.1,9);
\draw [thick] (-.1,8) node[left]{8/9} -- (.1,8);
\draw [thick] (-.1,7) node[left]{7/9} -- (.1,7);
\draw [thick] (-.1,6) node[left]{6/9} -- (.1,6);
\draw [thick] (-.1,5) node[left]{5/9} -- (.1,5);
\draw [thick] (-.1,4) node[left]{4/9} -- (.1,4);
\draw [thick] (-.1,3) node[left]{3/9} -- (.1,3);
\draw [thick] (-.1,2) node[left]{2/9} -- (.1,2);
\draw [thick] (-.1,1) node[left]{1/9} -- (.1,1);
\draw [thick] (-.1,0) node[left]{0} -- (.1,0);
\draw [thick] (9,-.1) node[below]{1} -- (9,.1);
\draw [thick] (8,-.1) node[below]{$\frac{8}{9}$} -- (8,.1);
\draw [thick] (7,-.1) node[below]{$\frac{7}{9}$} -- (7,.1);
\draw [thick] (6,-.1) node[below]{$\frac{6}{9}$} -- (6,.1);
\draw [thick] (5,-.1) node[below]{$\frac{5}{9}$} -- (5,.1);
\draw [thick] (4,-.1) node[below]{$\frac{4}{9}$} -- (4,.1);
\draw [thick] (3,-.1) node[below]{$\frac{3}{9}$} -- (3,.1);
\draw [thick] (2,-.1) node[below]{$\frac{2}{9}$} -- (2,.1);
\draw [thick] (1,-.1) node[below]{$\frac{1}{9}$} -- (1,.1);
\draw [fill={rgb:black,1}] (0,17/3) -- (1,20/3) -- (0,20/3) -- (0,17/3);
\draw [fill={rgb:black,1;white,1}] (1,20/3) rectangle (2,22/3);
\draw [fill={rgb:black,3;white,5}] (1,22/3) rectangle (2,8);
\draw [fill={rgb:black,2;white,1}] (2,20/3) rectangle (7/3,22/3);
\draw [fill={rgb:black,1;white,1}] (7/3,8) rectangle (10/3,9);
\end{tikzpicture}\]
 \end{center}
 \caption[Transformation of $\mu$ into $\tilde{\mu}$]{Example of the transformation of $\mu$ (left image) into $\tilde{\mu}$ (right image). The weights of the measure $\mu$ are the following: the triangle shape has a total weight of $1/9$, the \textbf{L} shape has a weight of $4/27$, and the top right box has weight $1/9$. Notice that the measure $\mu$ restricted to the top right box has already uniform marginals, hence the box will not be squeezed by the transformation into $\tilde{\mu}$, but will just shift.}\label{picture for the proof of I_1}
 \end{figure}
An immediate consequence of the previous lemma is that $I(\mu)\geq 0$, and we have $I(\mu)=0$ if and only if 
\begin{itemize}
 \item $\mu\in\tilde{\Gamma}$,
 \item $I_2(\mu)=0.$
\end{itemize} 
\smallskip 

\begin{lemma}\label{proposition mu=lambda_omega}
 Let $\mu\in\tilde{\Gamma}$ such that $I_2(\mu)=0.$ Then $\mu=\Omega$. 
\end{lemma}
\begin{proof}
 Consider a variation of the distribution function for a measure $\rho\in\tilde{\Gamma}$:
 \[F_{\rho}(a,b):=\rho([0,a]\times[1-b,1]),\]
 for $a,b\in[0,1/2]$. To prove the lemma it is enough to show that 
 \[F_{\mu}(a,b)=F_{\Omega}(a,b)=\min(a,b).\]
Suppose $a\leq b$ (the other case is done similarly), and consider the three sets $S=[0,a]\times[1/2,1-b]$, $T=[0,a]\times [1-b,1],$ $Q=[a,1/2]\times[1-b,1]$ as in Figure \ref{example for proof of I_2}. We claim that $\int_S\,d\mu=0$; suppose this is not the case, then $\int_S\,d\mu>0$. Notice that since $\mu$ has uniform marginals on the square $[0,1/2]\times [1/2,1]$ then 
\[a=\int_{[0,a]\times [1/2,1]}\,d\mu=\int_{S\cup T}\,d\mu=\int_S\,d\mu+\int_T\,d\mu.\]
By a similar argument we have $b=\int_T\,d\mu+\int_Q\,d\mu,$ therefore $\int_Q\,d\mu=b-a+\int_S\,d\mu>0$. We consider thus 
\begin{align*}
 I_2(\mu)&=\int_{\Delta^2}\mathbb{1}[x_1<x<2<y_1<y_2]\,d\mu(x_1,y_1)\,d\mu(x_1,y_2)\\
 &\geq \int_{S\times Q}\mathbb{1}[x_1<x<2<y_1<y_2]\,d\mu(x_1,y_1)\,d\mu(x_1,y_2).
\end{align*}
Observe that the characteristic function $\mathbb{1}[x_1<x_2<y_1<y_2]$ is equal to $1$ on the set $S\times Q$, hence $I_2(\mu)\geq\int_{S\times Q}\,d\mu\otimes d\mu=\mu(S)\cdot\mu(Q)>0$, which is a contradiction. Thus $\int_S\,d\mu=0$ as claimed. This implies that
\[F(a,b)=\mu(T)=\mu(T\cup S)=a\]
and the proof is concluded. 
\end{proof}
\begin{figure}

 \begin{center}
 \[\begin{tikzpicture}[scale=0.5]
\draw [thick, <->] (0,10) -- (0,0) -- (10,0);
\draw [dashed](0,0)--(9,9)--(0,9);
\draw (0,4.5) rectangle (2,6);
\draw [fill=lightgray](0.2,4.7) rectangle (0.45,5.1);
\node at (1,5){S};
\draw (0,6) rectangle (2,9);
\node at (1,7){T};
\draw (2,6) rectangle (4.5,9);
\draw [fill=lightgray](2.5,6.3) rectangle (2.9,6.7);
\node at (3.25,7){Q};
\draw [thick] (-.1,4.5) node[left]{1/2} -- (.1,4.5);
\draw [thick] (-.1,6) node[left]{1-b} -- (.1,6);
\draw [thick] (4.5,-.1) node[below]{$\frac{1}{2}$} -- (4.5,.1);
\draw [thick] (2,-.1) node[below]{a} -- (2,.1);
 \end{tikzpicture}\]
 \end{center}
 \caption[Example of the area division in the proof of Lemma \ref{proposition mu=lambda_omega}]{Example of the area division in the proof of Lemma \ref{proposition mu=lambda_omega}. Here we picture $a=2/9\leq b=3/9$. If the measure $\mu$ has non zero weight inside $S$ (here is pictured as the gray area), then it has also non zero weight in $Q$, and therefore $I_2(\mu)\neq 0$.}\label{example for proof of I_2}
 \end{figure}
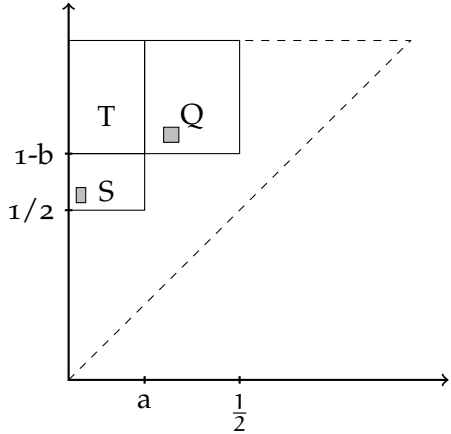
\begin{proof}[Proof of Proposition \ref{minimizing I}]
 It is easy to see that $I(\Omega)=0$. Suppose on the other hand that $I(\mu)=0$, then $I_1(\mu)=\frac{1}{4}+\frac{I_2(\mu)}{2}\leq \frac{1}{4}$ by Lemma \ref{maximizing I_1}. This implies that $I_2(\mu)=0$ and thus $I_1(\mu)=1/4$; hence $\mu\in\tilde{\Gamma}$ by Lemma \ref{maximizing I_1}, and we can apply Lemma \ref{proposition mu=lambda_omega} to conclude that $\mu=\Omega$.
\end{proof}

\section{Convergence in the weak* topology}\label{section final results}
In this section we prove the main result of the paper, that is, that in the weak* topology $\mu_{\pi^{(n)}}$ converges almost surely to $\Omega$ when $\pi^{(n)}$ is a random set partition distributed with the superplancherel measure. In order to do this we show some necessary lemmas, which relate the entropy $I$ to the L\'evy-Prokhorov metric on measures. We proceed as following: we show that the space $M^{\leq 1}(\Delta):=\{\mu\mbox{ measure on }\Delta\mbox{ s.t. }\int_{\Delta}\,d\mu\leq 1\}$ of subprobabilities on $\Delta$ is compact, and then we verify that $\Gamma$ is closed in $M^{\leq 1}(\Delta)$. We check then that both $I_1$ and $I_2$ are continuous as functions $\Gamma\to\R$. The proofs are mostly based on known theorems regarding probabilities, adapted in our case to subprobabilities.

Throughout the section, consider $(X,|\cdot|)$ a metric space and let $\mathcal{C}(X,\R)$ the set of continuous bounded functions $X\to\R$.
\begin{defi}
Let $(X,|\cdot|)$ be a metric space, then $M^{\leq 1}(X)$ and $M^1(X)$ are respectively the space of subprobabilities on $X$ and probabilities on $X$.
\end{defi}
We endow both $M^{\leq 1}(X)$ and $M^1(X)$ with the weak* topology, that is, consider $\{\mu_n\}_{n\in\mathbb{N}}\subseteq M^{\leq 1}(X),$ $\mu\in M^{\leq 1}(X)$ (resp. $\{\mu_n\}_{n\in\mathbb{N}}\subseteq M^1(X),\mu\in M^1(X)$), then we say that $\mu_n\overset{w^*}\to\mu$ in $M^{\leq 1}(X)$ (resp. $M^1(X)$) if
\[\int f(x) \,d\mu_n(x)\to\int f(x) \,d\mu(x)\]
for each $f\in\mathcal{C}(X,\R)$.
\smallskip

For a subset $Y\subseteq X$ and $\epsilon>0$ the $\epsilon-$neighborhood of $Y$ is
\[Y^{\epsilon}:=\{x\in X\mbox{ s.t. there exists }y\in Y\mbox{ with }\|x-y\|<\epsilon\}.\]
The L\'evy-Prokhorov metric is defined as
\begin{multline*}
 d_{L-P}(\mu,\nu)=\inf\{\epsilon>0\mbox{ s.t. }\mu(Y)\leq \nu(Y^{\epsilon})+\epsilon\mbox{ and }\\\nu(Y)\leq \mu(Y^{\epsilon})+\epsilon\mbox{ for each }Y\subseteq X\mbox{ measurable}\}.
\end{multline*}

It is well known that convergence with the L\'evy-Prokhorov metric is equivalent to the weak* convergence. For an introduction on the subject, see \cite{billingsley2013convergence}.
\smallskip 

As shown, for example, in \cite[Theorem 29.3]{billingsley2012probability}, if $X$ is compact then $M^1(X)$ is compact. The same is true for $M^{\leq 1}(X)$.
\begin{lemma}
Let $X$ be a compact metric space, then $M^{\leq 1}(X)$ is compact according to the weak* topology.
\end{lemma}
\begin{proof}
We can encode subrobabilities on $X$ as probabilities on $X\cup\{\partial\}$, where $\partial\notin X$ is called a \emph{cemetery} point, as follows:
\[\phi\colon M^{\leq 1}(X)\to M^{1}(X\cup\{\partial\})\]
\[\phi(\mu)(A)=\mu(A)\mbox{ if }A\subseteq X\mbox{ and }\phi(\mu)(\partial)=1-\mu(X).\]
Then $\phi$ is clearly a homeomorphism (with the obvious topology on $X\cup\{\partial\}$). Since $X\cup\{\partial\}$ is compact, then so is $M^1(X\cup\{\partial\})$, and thus also $M^{\leq 1}(X).$
\end{proof}

\begin{lemma}
 The set $\Gamma:=\{\mbox{subprobabilities }\mu \mbox{ on }\Delta\mbox{ s.t. }\mu\mbox{ has subuniform marginals}\}$ is closed in the set of subprobabilities $M^{\leq 1}(\Delta)$. In particular, $\Gamma$ is compact.
\end{lemma}
\begin{proof}
 Let $\{\mu_n\}$ be a sequence in $\Gamma$ converging to $\mu\in M^{\leq 1}(\Delta)$, we prove that $\mu\in\Gamma$. Suppose the contrary, we set without loss of generality that $\mu$ is not subuniform in the $x$-coordinates. Then there exists $(a,b)\in\Delta$ such that 
 \[\int_{[a,b]\times[0,1]}\,d\mu=b-a+\delta\]
 for $\delta>0$. Set $K=[a,b]\times[0,1]$ and $U=(a-\delta/3,b+\delta/3)\times[0,1]$. 
 By Uyshion's Lemma there exists a function $f\in\mathcal{C}(\Delta,[0,1])$ such that $\mathbb{1}_K(x,y)\leq f(x,y)\leq \mathbb{1}_U(x,y)$ for each $(x,y)\in\Delta$. Hence 
 \[\int_{\Delta}f(x,y)\,d\mu_n\leq\int_{[a-\frac{\delta}{3},b+\frac{\delta}{3}]\times[0,1]}\,d\mu_n\leq b-a+\delta-\frac{2}{3}.\]
 Since $\int f\,d\mu_n\to\int f\,d\mu$, this implies that $\int_{\Delta}f(x,y)\,d\mu\leq b-a+\delta-2/3$, contradiction. 
\end{proof}
The following lemma can be found in \cite[Theorem 29.1]{billingsley2012probability}:
\begin{lemma}
 Let $f\in\mathcal{C}(\Delta,\R)$ be bounded, then the functional that maps $\mu$ to $\int f\,d\mu$ is continuous. In particular, $I_1(\mu)=\int_{\Delta}(y-x)\,d\mu$ is continuous.
\end{lemma}
To prove the continuity of $I_2$ we need the following proposition, which can be found in \cite[Theorem 29.2]{billingsley2012probability}:
\begin{proposition}
 Suppose that $h\colon\R^k\to\R^j$ is measurable and that the set $D_h$ of its discontinuities is measurable. If $\nu_n\to\nu$ in $\R^k$ and $\nu(D_h)=0$, then $h_{\ast}\nu_n\to h_{\ast}\nu$ in $\R^j$.
\end{proposition}
\begin{lemma}
 The functional $I_2(\mu)=\int_{\Delta^2}\mathbb{1}[x_1<x<2<y_1<y_2]\,d\mu(x_1,y_1)\,d\mu(x_1,y_2)$ is continuous.
\end{lemma}
\begin{proof}
 We prove that $I_2$ is sequentially continuous, \emph{i.e.}, if $\mu_n\to\mu$ in weak* topology then $I_2(\mu_n)\to I_2(\mu)$. It is well known that in metric spaces continuity is equivalent to sequential continuity. Consider thus $\mu_n\to\mu$, then $\mu_n\otimes\mu_n\to\mu\otimes\mu$. Define the function $h\colon\Delta\times\Delta\to\R$, \[h(x_1,y_1,x_2,y_2):=\mathbb{1}[x_1<x_2<y_1<y_2].\] We claim that if $D_h$ is the set of discontinuities of $h$ then $\mu\otimes\mu(D_h)=0$. Indeed 
 \[D_h=\{(x_1,y_1,x_2,y_2)\in\Delta\times\Delta\mbox{ s.t. }x_1=x_2\mbox{ or }x_2=y_1\mbox{ or }y_1=y_2\};\]
Consider for example the set $\{(x_1,y_1,x_2,y_2)\in\Delta\times\Delta\mbox{ s.t. }x_1=x_2\}$. Since $\mu$ has subuniform marginals, $x_1$ and $x_2$ chosen with $\mu$ will be almost surely different and thus $\{(x_1,y_1,x_2,y_2)\in\Delta\times\Delta\mbox{ s.t. }x_1=x_2\}$ has measure $0$. The same holds for the cases $x_2=y_1$ and $y_1=y_2$.
\smallskip

By applying the previous proposition we have therefore that $h_{\ast}(\mu_n\otimes\mu_n)\to h_{\ast}(\mu\otimes\mu)$. In particular
\[I_2(\mu_n)=h_{\ast}(\mu_n\otimes\mu_n)(1)\to h_{\ast}(\mu\otimes\mu)(1)=I_2(\mu).\qedhere\]
 \end{proof}
 As a consequence, the functional $I(\mu):=\frac{1}{2}-2I_1(\mu)+I_2(\mu)$ is continuous.
 \begin{theorem}
  We have
  \[\SPl(\{\pi\vdash [n]\mbox{ s.t. }d_{L-P}(\mu,\Omega)>\epsilon\})\to 0.\]
 \end{theorem}
\begin{proof}
We claim that for each $\epsilon>0$ there exists $\delta>0$ such that if $d_{L-P}(\mu,\Omega)>\epsilon$ then $|I(\mu)|> \delta$. Fix $\epsilon>0$ and suppose the claim not true, so that for each $\delta>0$ there is $\mu_{\delta}$ with $d_{L-P}(\mu_\delta,\Omega)>\epsilon$ and $|I(\mu)|\leq \delta$. Set $\delta=1/n$, we obtain a sequence $(\mu_n)$ with $|I(\mu_n)|\leq 1/n$. Since $\Gamma$ is compact there exists a converging subsequence $(\mu_{i_n})$. Call $\overline{\mu}$ the limit of this subsequence. Since $I$ is continuous we have $I(\overline{\mu})=\lim_n I(\mu_{i_n})=0$. This is a contradiction, since $\Omega$ is the unique measure in $\Gamma$ with $I(\Omega)=0$, and the claim is proved.
\smallskip

 Fix $\epsilon>0$, then there exists $\delta>0$ such that if $d_{L-P}(\mu,\Omega)>\epsilon$ then $|I(\mu)|> \delta$. Define the set $N^n_{\epsilon}:=\{\pi\vdash [n]\mbox{ s.t. }d_{L-P}(\mu,\Omega)>\epsilon\}$, then 
 \[\SPl(N^n_{\epsilon})=\sum_{\pi\in N^n_{\epsilon}}\exp(-n^2\log q I(\mu_{\pi})+O(n)).\]
 Recall that the number of set partitions of $n$, called the Bell number, is bounded from above by $n^n$; therefore
 \[\SPl(N^n_{\epsilon})\leq n^n\sup_{\pi\in N^n_{\epsilon}}\exp(-n^2\log q I(\mu_{\pi})+O(n))<\exp(-n^2\delta\log q+O(n\log n))\to 0.\qedhere\] 
\end{proof}
We prove now Theorem \ref{main result 1} and Corollary \ref{main result 2}:
\begin{theorem}
   For each $n\geq 1$ let $\pi_n$ be a random set partition of $n$ distributed with the superplancherel measure $\SPl_n$, then
 \[\mu_{\pi_n}\to\Omega\mbox{ almost surely} \]
 where the limit is taken in the space of infinite paths on the Bratteli diagram of set partitions defined by the system of superplancherel measures.
\end{theorem}
\begin{proof}
 As before, set $N^n_{\epsilon}:=\{\pi\vdash [n]\mbox{ s.t. }d_{L-P}(\mu,\Omega)>\epsilon\}$, so that \[\SPl(N^n_{\epsilon})<\exp(-n^2\delta\log q+O(n\log n)).\] Thus $\sum_n\SPl_n(N_n^{\epsilon})<\infty$ and we can apply the first Borel Cantelli lemma, which implies that $\limsup_nN_n^{\epsilon}$ has measure zero for each $\epsilon>0$, and therefore $\mu_{\pi_n}\to\Omega$ almost surely. 
\end{proof}

\begin{corollary}
  For each $n\geq 1$ let $\pi_n$ be a random set partition of $n$ distributed with the superplancherel measure $\SPl_n$, then
 \[\frac{\dim(\pi)}{n^2}\to\frac{1}{4}\mbox{ a.s.},\qquad \crs(\pi)\in o_P(n^2),\]
 where as before the limit is taken in the space of infinite paths on the Bratteli diagram of set partitions.
\end{corollary}
\begin{proof}
 Define $N^n_{\epsilon,\dim}:=\{\pi\vdash [n]\mbox{ s.t. }|\frac{\dim(\pi)}{n^2}-\frac{1}{4}|>\epsilon\}$, then for each $\pi\in N^n_{\epsilon,\dim}$ we have $I(\mu_{\pi})>\epsilon$. Hence $\SPl(N^n_{\epsilon,\dim})<\exp(-n^2\epsilon\log q+O(n\log n))\to 0$. As before, this implies $\sum_n\SPl_n(N_n^{\epsilon,\dim})<\infty$ and thus $\frac{\dim(\pi)}{n^2}\to\frac{1}{4}$ almost surely. The crossing case is done similarly.
\end{proof}

\section{Future research}\label{sec: future supercharacters}
\subsection{More precise asymptotic of set partitions}
A natural step following our result on the limit shape of superplancherel distributed set partitions would be  studying the second order asymptotic. Unfortunately, our method is not a viable path for this problem. Also the statistics $\dim\pi$ and $\crs\pi$ deserve a more precise asymptotic. Our algorithm of Section \ref{section combinatorial interpretation} allows us to obtain some insight for these problems.

We associate to each random superplancherel distributed set partition $\pi\vdash [n]$ a plot $\mathcal{P}_{\pi}$ as follows: if $(i,j)\in D(\pi)$, we draw the point
\[\left(\frac{i}{n},\frac{j-(n-i)}{n^{\frac{1}{6}}}\right).\]
 The term $n^{1/6}$ comes from empirical evidence, and we are not certain about it. On the other hand  it is a term that appears in other instances of asymptotic representation theory: indeed, it is the order of magnitude of the second order asymptotic of the first row of a random Poissonized Plancherel partition. Equivalently, $n^{1/6}$ is the order of magnitude of the fluctuations of the length of the longest increasing subsequence of a uniform random permutation. In figure \ref{fig: second order asymptotic SPl} we show two plots of $\mathcal{P}_{\pi}$ for a random superplancherel distributed set partition $\pi\vdash [800]$ and $\pi\vdash [8000]$.

\begin{figure}
  \begin{center}
  \[\includegraphics[width=7cm,height=7cm]{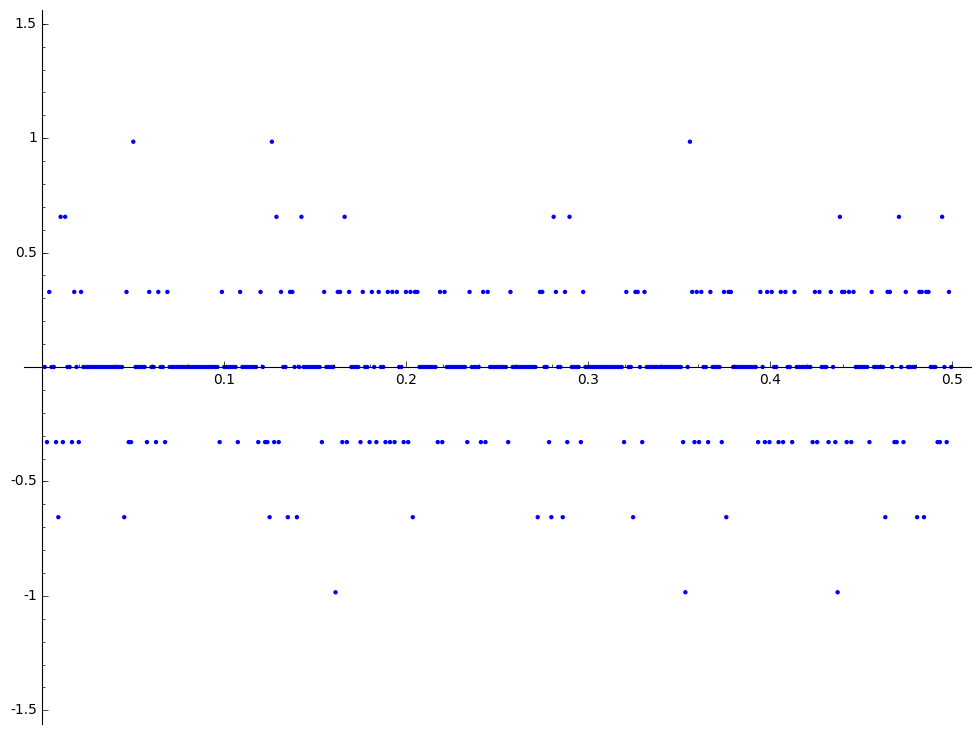}\qquad\qquad
\includegraphics[width=7cm,height=7cm]{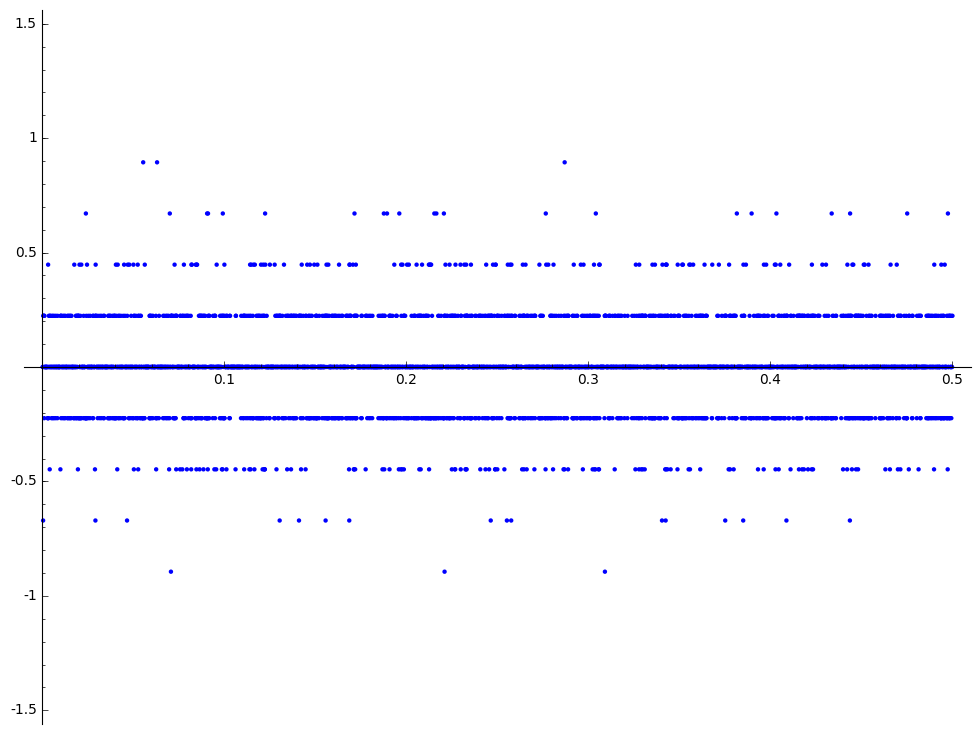}\]
  \caption[Second order asymptotic for set partitions]{Computer generated second order asymptotic for random superplancherel distributed set partitions: the left image is the measure $\nu_{\pi}$ for $\pi\vdash [800]$; on the right there is $\nu_{\pi}$ for $\pi\vdash [8000]$.}\label{fig: second order asymptotic SPl}\end{center}
\end{figure}

In Figure \ref{fig: second order asymptotic dim} we represent the random statistic corresponding to the second order asymptotic of $\dim\pi$. Set $\dim_2(\pi)=n^2(\frac{1}{4}-\frac{\dim\pi}{n^2})$ for $\pi\vdash[n]$, where again the coefficient $n^2$ seems the right one from computer simulations. Given $n,k\in\N$, we randomly generate $k$ set partitions $\pi_1,\ldots,\pi_k$ of $[n]$ and compute the statistic $\dim_2(\pi_i)$, $i=1,\ldots k$. In the Figure \ref{fig: second order asymptotic dim} we present two histograms with on the abscissa the values of $\dim_2(\pi)$ and on the ordinate the frequency of said values. We set $k=500$ and $n=500$ on the left and $n=5000$ on the right.

\begin{figure}
  \begin{center}
  \[\includegraphics[width=7cm,height=7cm]{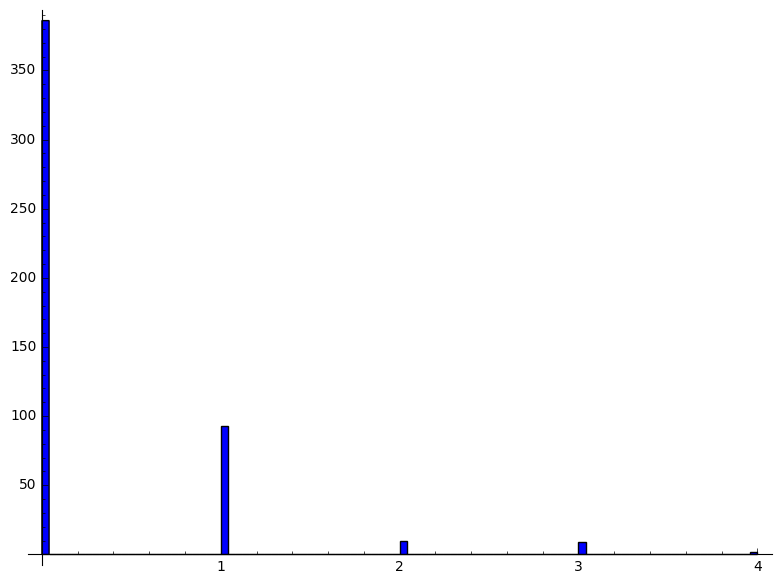}\qquad\qquad
\includegraphics[width=7cm,height=7cm]{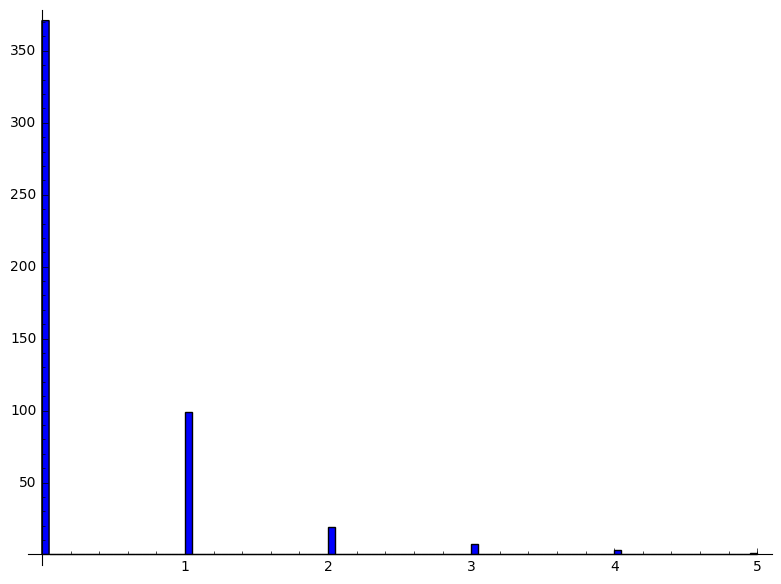}\]
  \caption[Second order asymptotic for $\dim$]{Computer generated second order asymptotic for the statistic $\dim$ for superplancherel distributed set partitions: the left image corresponds to $\pi\vdash [500]$ and on the right for $\pi\vdash [5000]$.}\label{fig: second order asymptotic dim}\end{center}
\end{figure}
In Figure \ref{fig: second order asymptotic crs} we show the statistic $\frac{\crs\pi}{n}$ with a similar procedure as the dimension $\dim\pi$. 

We believe that the normalization factors are the correct ones to get a nontrivial limit distribution, since the figures are consistent for the two values of $n$. In the case of the dimension $\dim$ the limit could be a Poisson or a geometric distribution, while in the case of the number of crossings we do not attempt (yet) a guess. Further tests could suggest precise conjectures.
\begin{figure}
  \begin{center}
  \[\includegraphics[width=7cm,height=7cm]{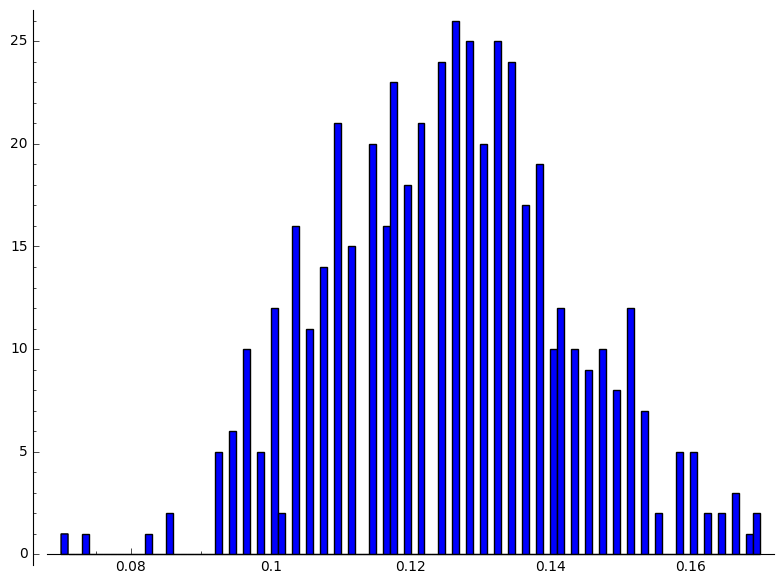}\qquad\qquad
\includegraphics[width=7cm,height=7cm]{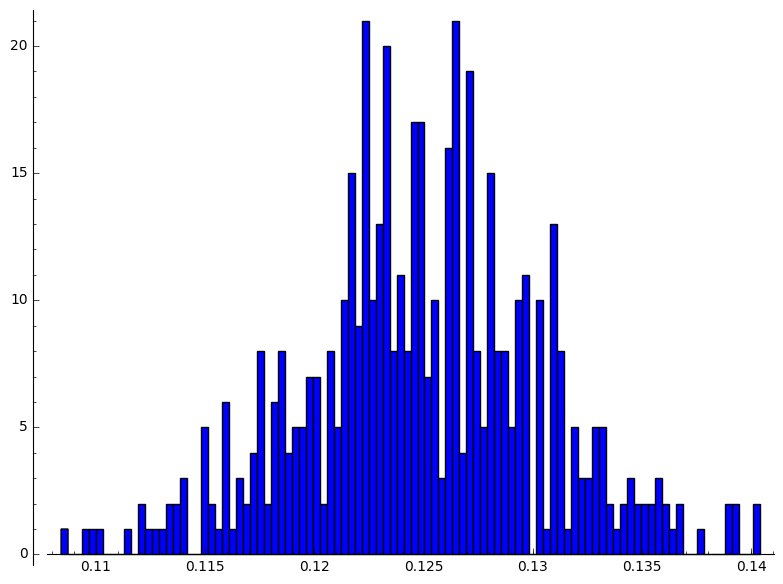}\]
  \caption[Second order asymptotic for $\crs$]{Computer generated second order asymptotic for the statistic $\crs$ for superplancherel distributed set partitions: the left image corresponds to $\pi\vdash [500]$ and on the right for $\pi\vdash [5000]$.}\label{fig: second order asymptotic crs}\end{center}
\end{figure}
\subsection{A new (old) approach to supercharacter theory}
At the origin of representation theory the problem submitted by Dedekind to Frobenius was to factorize the group determinant. Let $G$ be a finite group and set $\{X_g\}_{g\in G}$ to be a set of formal variables indexed by the elements of $G$. The \emph{group matrix} is the matrix $[X_{gh^{-1}}]$. 
\begin{defi}
 The \emph{group determinant} $\Theta(G)$ of $G$ is the determinant of the group matrix, considered as a polynomial in the variables $\{X_g\}_{g\in G}$. More precisely
 \[\Theta(G):=\det[X_{gh^{-1}}]=\sum_{\sigma\in S_G}\sgn\sigma\prod_{g\in G}X_{g\cdot\sigma(g)^{-1}},\]
 where $S_G$ is the group that permutes the elements of $G$.
\end{defi}
Notice that $\Theta(G)$ is a homogeneous polynomial of degree $|G|$.

Frobenius described the factorization of the group determinant in a series of letters sent to Dedekind in April 1896. See \cite{conrad1998origin} and \cite{curtis1992representation} for modern treatments of the subject. 

For a representation $\pi\colon G\to \GL_d(\C)$ of degree $d$ define
\[\Theta_{\pi}(G):=\det[\sum_{g\in G}X_g\pi_g].\]
By abuse of notation we will sometimes write $\Theta_{\chi}$ rather than $\Theta_{\pi}$, where $\chi\colon G\to\C$ is the character associated to $\pi$. Frobenius proved the following statements (which can be found in \cite{conrad1998origin})
\begin{enumerate}
 \item If $\pi$ is the regular representation then $\Theta_{\pi}(G)=\Theta(G)$.
 \item If $\pi=\pi_1\oplus\pi_2$ then $\Theta_{\pi}(G)=\Theta_{\pi_1}(G)\cdot\Theta_{\pi_2}(G).$
 \item The polynomial $\Theta_{\pi}(G)$ is irreducible over $\C$ if and only if $\pi$ is an irreducible representation.
\end{enumerate}
Hence a complete factorization of $\Theta(G)$ over $\C$ is given by 
\[\Theta(G)=\prod_{\chi\in\Irr(G)}\Theta_{\chi}(G)^{\dim\chi}.\]
The character $\chi$ is completely determined by $\Theta_{\chi}(G)$. Indeed, by writing $\Theta_{\chi}(G)$ as a polynomial in $X_{\id_G}$ we have
\[\Theta_{\chi}(G)=X_{\id_G}^{\deg\chi}+\sum_{g\in G}\chi(g)X_{\id_G}^{\deg\chi-1}X_g+O(X_{\id_G}^{\deg\chi-2}).\]
A more precise formula is available (see \cite[Section 5]{conrad1998origin})
\begin{equation}\label{eq: formula Theta}
 \Theta_{\chi}(G)=\sum_{\lambda\vdash d}\frac{(-1)^{d-l(\lambda)}}{z_{\lambda}}\prod_{k=1}^d\left(\sum_{(g_1,\ldots,g_k)\in G^k}\chi(g_1\cdot\ldots\cdot g_k)X_{g_1}\cdot\ldots\cdot X_{g_k}\right)^{m_k(\lambda)}, 
\end{equation}
where $d=\chi(\id_G)$ and for a partition $\lambda$ the number $m_k(\lambda)$ represent the multiplicity of $k$ in the parts of $\lambda$ and $l(\lambda)$ is the length of $\lambda$. Notice how the previous formula resembles the decomposition of the elementary symmetric function $e_d(x_1,x_2,\ldots)$ into the power sum symmetric function:
\[e_d(x_1,x_2,\ldots)=\sum_{\lambda\vdash d}\frac{(-1)^{d-l(\lambda)}}{z_{\lambda}}\prod_{k=1}^d p_k(x_1,x_2,\ldots)^{m_k(\lambda)}.\]
 In \cite{keller2014generalized} Keller proved that each group $G$ has a unique minimal integral supercharacter theory, that is, each supercharacter takes integer values, and the theory is the coarsest with this property. A direct method to identify the minimal integral supercharacter theory for $G$ is still unknown, and it is an open problem whether the theory studied in this chapter is the minimal one for $U_n$. Maybe this approach with the group determinant can cast some light on the problem, in the following sense: consider a fixed supercharacter theory $(\scl(G),\sch(G))$ for $G$. Recall that supercharacters themselves are also characters, hence the previous formulas apply to them. If $(\scl(G),\sch(G))$ is an integral supercharacter theory then for each $\chi\in\sch(G)$ the polynomial $\Theta_{\chi}(G)$ belongs to $\Z[\{X_g\}_{g\in G}]$ by Equation \eqref{eq: formula Theta}. Moreover if this supercharacter theory is not minimal integral then $\Theta_{\chi}(G)$ should factorize over the integers for some supercharacter $\chi$ by Property 2. On the other hand we would expect that if $\Theta_{\chi}(G)$ factorizes over the integers, then the factors should correspond to the supercharacters of a coarser supercharacter theory. This would be a ``supercharacter'' analogue of Property 3.
 
If what conjectured here is true we would obtain a method to verify if a supercharacter theory is the minimal integral one by checking the irreducibility of the polynomials $\Theta_{\chi}(G)$ for each $\chi\in\sch(G)$. 

More generally, it would be interesting to study whether supercharacter theories represent factorizations of $\Theta(G)$ over certain domains. If this is the case then supercharacter theory would appear as a natural generalization of character theory, and not as an \emph{ad hoc} construction for groups for which the full character theory is not available.

\label{p:4}

\cleardoublepage
\manualmark
\markboth{\spacedlowsmallcaps{\bibname}}{\spacedlowsmallcaps{\bibname}} 
\refstepcounter{dummy}
\addtocontents{toc}{\protect\vspace{\beforebibskip}} 
\addcontentsline{toc}{chapter}{\tocEntry{\bibname}}
\bibliographystyle{plainnat}
\label{app:bibliography} 
\bibliography{courant}{}

\end{document}